\documentclass{amsart}

\usepackage{amssymb, amsmath,mathtools}
\usepackage{mathrsfs}
\usepackage{amscd}
\usepackage{verbatim}
\usepackage{stmaryrd}
\usepackage[dvipsnames]{xcolor}

\usepackage{enumerate}

\usepackage[colorlinks,linkcolor={blue},citecolor={blue},urlcolor={purple},]{hyperref}

\usepackage[colorinlistoftodos,prependcaption,textsize=tiny]{todonotes}

\usepackage{forest}

\usepackage{tabularx}
\newcolumntype{L}{>{\arraybackslash}X}
\usepackage{multirow}

\theoremstyle{plain}
\newtheorem{theorem}{Theorem}[section]
\theoremstyle{remark}
\newtheorem{remark}[theorem]{Remark}
\newtheorem{roadmap}[theorem]{\em \textbf{Roadmap}}

\theoremstyle{plain}
\newtheorem{corollary}[theorem]{Corollary}
\newtheorem{lemma}[theorem]{Lemma}
\newtheorem{proposition}[theorem]{Proposition}
\newtheorem{definition}[theorem]{Definition}

\newtheorem{assumption}[theorem]{Assumption}

\numberwithin{equation}{section}

\def\N{{\mathbb N}}

\def\R{{\mathbb R}}

\newcommand{\E}{{\mathbb E}}
\renewcommand{\P}{{\mathbb P}}
\newcommand{\F}{{\mathscr F}}
\newcommand{\Filtr}{\mathbb{F}}

\newcommand{\g}{\gamma}
\renewcommand{\d}{\delta}
\newcommand{\om}{\omega}
\renewcommand{\O}{\Omega}

\newcommand{\Ito}{{\hbox{\rm It\^o}}}
\renewcommand{\a}{\kappa}

\newcommand{\aaa}{\alpha}

\newcommand{\Dom}{\mathscr{O}}
\newcommand{\U}{\mathcal{U}}
\newcommand{\V}{\mathcal{V}}
\newcommand{\I}{I}

\newcommand{\Tor}{\mathbb{T}}
\newcommand{\W}{\mathcal{O}}
\newcommand{\ee}{\mathcal{I}}

\newcommand{\A}{{\mathcal A}}
\newcommand{\X}{\mathfrak{X}}
\newcommand{\y}{\mathcal{Z}}
\newcommand{\Y}{\mathcal{Y}}

\newcommand{\loc}{\mathrm{loc}}

\newcommand{\calL}{{\mathscr L}}

\newcommand{\z}{\mathcal{Z}}
\newcommand{\yT}{\Upsilon}

\newcommand{\hz}{\prescript{}{0}{H}}
\newcommand{\Sz}{\prescript{}{0}{\mathsf{MR}}_X}

\newcommand{\Wz}{\prescript{}{0}{W}}

\newcommand{\K}{\wt{K}}

\newcommand{\ez}{\prescript{}{0}{\mathsf{E}}}

\newcommand{\Sol}{\mathscr{R}}

\DeclareMathOperator*{\esssup}{\textup{ess\,sup}}

\newcommand{\Do}{\mathsf{D}}

\newcommand{\wt}{\widetilde}
\renewcommand{\ll}{\llbracket}
\newcommand{\rr}{\rrbracket}
\newcommand{\rParen}{\rrparenthesis}
\newcommand{\lParen}{\llparenthesis}
\newcommand{\llo}{\lParen}
\newcommand{\rro}{\rParen}
\newcommand{\rhos}{\rho^{\star}}

\usepackage{stmaryrd}

\newcommand{\MRtas}{\mathcal{SMR}_{p,\a}^{\bullet}(s,T)}

\newcommand{\MRtasj}{\mathcal{SMR}_{p,\a}^{\bullet}(s_j,T)}

\newcommand{\MRta}{\mathcal{SMR}_{p,\a}^{\bullet}(T)}

\newcommand{\MRts}{\mathcal{SMR}_p^{\bullet}(s,T)}

\newcommand{\MRtasigma}{\mathcal{SMR}_{p,\a}^{\bullet}(\sigma,T)}

\newcommand{\MRttwosigma}{\mathcal{SMR}_{2,0}^{\bullet}(\sigma,T)}
\newcommand{\MRttwosigmaz}{\mathcal{SMR}_{2,0}(\sigma,T)}
\newcommand{\MRtatau}{\mathcal{SMR}_{p,\a}^{\bullet}(\tau,T)}
\newcommand{\MRttau}{\mathcal{SMR}_{p}^{\bullet}(\tau,T)}
\newcommand{\MRtasigmaz}{\mathcal{SMR}_{p,\a}(\sigma,T)}

\newcommand{\MRtsigma}{\mathcal{SMR}^{\bullet}_p(\sigma,T)}
\newcommand{\MRtmu}{\mathcal{SMR}^{\bullet}_p(\mu,T)}
\newcommand{\MRtsigmazero}{\mathcal{SMR}^{\bullet}_{p,0}(\sigma,T)}
\newcommand{\MRtsigmaz}{\mathcal{SMR}_p(\sigma,T)}
\newcommand{\MRtsigmazeroz}{\mathcal{SMR}_{p,0}(\sigma,T)}

\newcommand{\Tr}{\mathsf{Tr}}

\newcommand{\Xap}{X^{\mathsf{Tr}}_{\a,p}}

\newcommand{\Xellp}{X^{\mathsf{Tr}}_{\ell,p}}

\newcommand{\Xapcrit}{X^{\mathsf{Tr}}_{\a_{\crit},p}}
\newcommand{\Xp}{X^{\mathsf{Tr}}_{p}}

\newcommand{\Xr}{X^{\mathsf{Tr}}_{r}}
\newcommand{\Xzp}{X^{\mathsf{Tr}}_{0,p}}

\newcommand{\Yaaar}{Y^{\mathsf{Tr}}_{\aaa,r}}

\newcommand{\Yr}{Y^{\mathsf{Tr}}_{r}}

\newcommand{\one}{{{\bf 1}}}
\newcommand{\embed}{\hookrightarrow}

\newcommand{\dps}{\displaystyle}

\newcommand{\B}{B}

\newcommand{\p}{r}

\newcommand{\deter}{\mathsf{det}}
\newcommand{\crit}{\mathsf{crit}}
\newcommand{\stoc}{\mathsf{sto}}

\newcommand{\norm}[1]{{\left\vert\kern-0.25ex\left\vert\kern-0.25ex\left\vert #1
    \right\vert\kern-0.25ex\right\vert\kern-0.25ex\right\vert}}

\newcommand{\Hip}{\normalfont{\textbf{(H)}}}
\newcommand{\Hiep}{\normalfont{\textbf{H}}}
\renewcommand{\emptyset}{\varnothing}

\renewcommand{\ggg}{g_{B}}
\newcommand{\ff}{f_{A}}

\newcommand{\Progress}{\mathscr{P}}
\newcommand{\MeasurableP}{\mathscr{A}}

\newcommand{\wh}{\widehat}

\newcommand{\maxSigma}{\Xi}

\newcommand{\Borel}{\mathscr{B}}
\newcommand{\Constant}{R}

\newcommand{\pownoise}{\eta}

\newcommand{\vone}{v}
\newcommand{\vtwo}{v'}

\def\XXint#1#2#3{{\setbox0=\hbox{$#1{#2#3}{\int}$ }
\vcenter{\hbox{$#2#3$ }}\kern-.6\wd0}}

\newcommand{\stopp}{\lambda}

\newcommand{\nonlinearity}{\mathcal{N}}
\newcommand{\closed}{\mathcal{C}}

\allowdisplaybreaks

\begin{document}

\author{Antonio Agresti}
\address{Institute of Science and Technology Austria (IST Austria)\\ Am Campus 1\\ 3400 Klosterneuburg\\ Austria.} \email{antonio.agresti92@gmail.com}

\author{Mark Veraar}
\address{Delft Institute of Applied Mathematics\\
Delft University of Technology \\ P.O. Box 5031\\ 2600 GA Delft\\The
Netherlands.} \email{M.C.Veraar@tudelft.nl}

\thanks{The second author is supported by the VIDI subsidy 639.032.427 of the Netherlands Organisation for Scientific Research (NWO)}

\date\today

\title[parabolic stochastic evolution equations in critical spaces II]{Nonlinear parabolic stochastic evolution \\ equations in critical spaces part II \\\small{\emph{B\lowercase {low-up criteria and instantaneous regularization}}}}



\begin{abstract}
This paper is a continuation of Part I of this project, where we developed a new local well-posedness theory for nonlinear stochastic PDEs with Gaussian noise. In the current Part II we consider blow-up criteria and regularization phenomena. As in Part I we can allow nonlinearities with polynomial growth, and rough initial values from critical spaces.

In the first main result we obtain several new blow-up criteria for quasi- and semilinear stochastic evolution equations. In particular, for semilinear equations we obtain a Serrin type blow-up criterium, which extends a recent result of Pr\"uss-Simonett-Wilke (2018) to the stochastic setting. Blow-up criteria can be used to prove global well-posedness for SPDEs. As in Part I, maximal regularity techniques and weights in time play a central role in the proofs.

Our second contribution is a new method to bootstrap Sobolev and H\"older regularity in time and space, which does not require smoothness of the initial data. The blow-up criteria are at the basis of these new methods. Moreover, in applications the bootstrap results can be combined with our blow-up criteria, to obtain efficient ways to prove global existence. This gives new results even in classical $L^2$-settings, which we illustrate for a concrete SPDE.

In future works in preparation we apply the results of the current paper to obtain global well-posedness results, and regularity for several concrete SPDEs. These include stochastic Navier-Stokes equations, reaction diffusion equations and the Allen-Cahn equation. Our setting allows to put these SPDEs into a more flexible framework, where less restrictions on the nonlinearities are needed, and we are able to treat rough initial values from critical spaces. Moreover, we will obtain higher order regularity results.
\end{abstract}

\maketitle
\newpage
\tableofcontents

\newpage

\section{Introduction}

The field of stochastic evolution equations has attracted a lot of attention in the past decades. Adding noise to existing models can provide a lot of new insights. At the same time there is still a large gap between the understanding of the deterministic theory and the stochastic theory. Broadly speaking, we and many others are trying to close and understand this gap, and this paper can be seen as another step in this direction.

The paper is a continuation of our recent work \cite{AV19_QSEE_1}, which we refer to as Part I. In Part I we obtained a systematic treatment of nonlinear stochastic evolution equations by means of maximal regularity techniques, and the main result in Part I was a new local well-posedness theory which we successfully applied to classes of quasi- and semilinear SPDEs.
The aim of the current paper
is to obtain blow-up criteria, global existence, and instantaneous regularization. As far as possible we made Part II independent of Part I. However, for this the reader has to take the local well-posedness theorem of Part I  (see Theorem \ref{thm:wellintro} below) for granted.

The equations we consider are of the following form:
\begin{equation}
\label{eq:QSEE_intro}
\begin{cases}
du + A(\cdot, u) u \,dt = F(\cdot, u)dt + (B(\cdot, u) u +G(\cdot, u)) dW,\\
u(0)=u_{0}.
\end{cases}
\end{equation}
Here $(A,B)$ are the leading operators and are of quasilinear type which means that for each $v$ in a suitable interpolation space $A(v):X_1\to X_0$ and $B(v):X_{1}\to X_{1/2}$ are bounded linear operators.
In the above $X_0$ and $X_1$ are Banach spaces such that $X_1$ densely embeds into $X_0$, and for $\theta\in (0,1)$, $X_{\theta} = [X_0, X_1]_{\theta}$ denotes the complex interpolation space. In applications all of these spaces will be suitable Sobolev and Besov spaces.
In the semilinear case $(A(u) u, B(u)u)$ is replaced by $(\bar{A} u, \bar{B} u)$, where $(\bar{A}, \bar{B})$ does not depend on $u$. The noise term $W$ is a cylindrical Brownian motion. The nonlinearities $F$ and $G$ are of semilinear type which means they are defined on suitable interpolation spaces. Many examples of SPDEs fit in the above framework.

As in Part I when comparing
to other papers we would like to point out the main novelties of our theory:
\begin{enumerate}[(a)]
\item\label{it:noveltya} it is an $L^p(\text{time};L^q(\text{space}))$-theory;
\item\label{it:noveltyb} ellipticity or coercivity conditions are formulated for the couple $(A,B)$;
\item\label{it:noveltyc} no smallness assumption on $B$ is needed;
\item\label{it:noveltyd} measurable dependence in $(t,\om)$ for the couple $(A,B)$ is allowed;
\item\label{it:noveltye} rough initial data $u_0\in (X_0,X_1)_{1-\frac{1+\a}{p},p}$ with $\a\in [0,\frac{p}{2}-1)$ are allowed;
\item\label{it:noveltyf} $F,G$ can be defined in $[X_0,X_1]_{1-\varepsilon}$ for $\varepsilon>0$ with polynomial growth;
\item\label{it:noveltyg} we can identify intrinsically defined critical spaces which reflect the scaling properties of concrete SPDEs.
\end{enumerate}
Individual cases of these points have been addressed in several papers, but having the combination of all of them is one of the strengths of our theory. Moreover, our paper seems to be the first to deal with \eqref{it:noveltye} and \eqref{it:noveltyg}, but also the first to deal with \eqref{it:noveltyb} and \eqref{it:noveltyc} in an abstract $L^p$-framework. Let us give some comparison for some of the classical theories for SPDEs. In the monotone operator approach to SPDEs (see \cite{LR15} and references therein) \eqref{it:noveltyb}-\eqref{it:noveltyd} are included as well, but a variational Hilbert space framework is required. On the other hand, with the monotone operator approach to SPDEs many types of nonlinearities can be treated which are not of quasi-linear type and it gives global existence immediately. In Krylov's $L^p$-theory \cite{Kry} \eqref{it:noveltya} with $p=q$, \eqref{it:noveltyb}-\eqref{it:noveltyd} are included, but only in the case $A$ is a linear second order differential operator. Moreover, the nonlinearities are assumed to be Lipschitz and of semilinear type. Although some of the initial ideas of our theory can be traced back to the semigroup approach to SPDEs \cite{Brz2, DPZ, Hornung, NVW11eq}, many of the above points, such as \eqref{it:noveltyb}-\eqref{it:noveltyd}, cannot be addressed if one is limited to this method in its pure form. Finally, the rough path theory approach to SPDEs is developing quickly and is the only available theory in the case renormalization is required \cite{GubImkPer, Hairer}, but at the moment we think it is unnatural to compare it to our theory.

Our work can be seen as a stochastic version of the theory of evolution equations in critical spaces of Pr\"{u}ss, Simonett and Wilke \cite{CriticalQuasilinear} (see also \cite{PrussWeight1,addendum}). As is clear from \cite{AV19_QSEE_1} it is nontrivial to extend the deterministic theory to the stochastic setting. From the results of this paper it turns out that it is also highly involved to extend blow-up criteria to the stochastic setting. In addition, we also obtain several results which are even new in the deterministic setting.

In the theory of evolution equation it is standard to combine blow-up criteria with energy estimates to obtain global existence (see the monographs \cite{Lun, pruss2016moving, TayPDE3, Yagibook}). These methods usually rely on Sobolev embeddings combined with $L^p$-energy estimates. In the stochastic case $L^p$-theory is much less developed than in the deterministic case, and this might be one of the reasons that blow-up criteria have not played a major role in the theory yet. Of course another advantage of $L^p$-theory is that it allows to study Navier-Stokes equations, Allen-Cahn equation, reaction-diffusion equations, with gradient nonlinearities, and in higher dimensions, which seems not possible when only $L^2$-theory is available.

The results of the current paper and \cite{AV19_QSEE_1} have already been applied in several situations and have already provided many new results and insights. In \cite{AV20_NS} we develop a new $L^p$-theory for stochastic Navier-Stokes equations with transport noise. In particular, for $d=2$ new global well-posedness and regularity properties have been obtained using our new temporal weighted methods, and we do not know how to derive these results without our new methods. In \cite{AHHS21} global well-posedness for the stochastic primitive equations with transport noise in $3d$ is obtained by applying our local existence theory from \cite{AV19_QSEE_1} and blow-up criteria Theorem \ref{thm:semilinear_blow_up_Serrin_refined}. In \cite{AV19_QSEE_3,AV22_reaction_diffusion} we will use our theory to build an $L^p$-theory for stochastic reaction-diffusion equations with transport noise. Moreover, we do expect that our theory will be useful for many other SPDEs. 

We kindly invite the interested readers to apply our theory to their favorite quasi-linear or semi-linear SPDE. After rewriting it as \eqref{eq:QSEE_intro} one can try to check the mapping properties of the nonlinearities stated in Section \ref{ss:quasi_revised} by taking $X_0$ and $X_1$ suitable Sobolev spaces. Here one has to try out which nonlinearities to put in which terms. Several classical examples can already be found in \cite{AV19_QSEE_1}. In many cases (and at first we would recommend to skip checking this), the abstract ``stochastic maximal regularity'' condition on the leading operators of the SPDE satisfy the conditions of Theorem \ref{t:local_s}, and thus local well-posedness follows. Quite likely this will already shine new light on the SPDE: more flexibility concerning initial values and nonlinearities, new regularity properties of the solution, and a greater variety of function spaces in which the SPDE can be analyzed. After that the doors to our theory are completely open, and the reader can try to follow Roadmap \ref{roadcomplete} to prove higher order regularity of local solutions and try to establish global well-posedness by showing energy estimates which are strong enough to apply one of our blow-up criteria. At the moment we have several papers in preparation where this program is followed for several classes of SPDEs, including some of the examples in \cite{AV19_QSEE_1}.

In the next subsections we give a laymen's summary of the results of our paper. We will start by recalling some of the local well-posedness theory of \cite{AV19_QSEE_1}.

\subsection{Local well-posedness and critical spaces\label{ss:introlocal}}
As already mentioned, the main result of Part I (see \cite{AV19_QSEE_1}) is a local well-posedness result, which we will now briefly recall in an informal way.  Details can be found in Section \ref{ss:quasi_revised}, and, in particular, the precise statement can be found in Theorem \ref{t:local_s} below. The result gives suffices conditions for existence and uniqueness of maximal solutions $(u,\sigma)$. Here $\sigma$ is a certain maximal stopping time, and $u$ a stochastic process on $[0,\sigma)$.

\begin{theorem}\label{thm:wellintro}
Let $p\in [2, \infty)$ and $w_{\a}(t) =  t^{\a}$ with $\a\in [0,\frac{p}{2}-1)$ (set $\a=0$ if $p=2$). Under maximal $L^p$-regularity assumptions on the pair $(A,B)$, and local Lipschitz conditions and polynomial growth conditions on $A$, $B$, $F$ and $G$, and assuming $u_0\in \Xap$ a.s., there exists a unique $L^p_{\a}$-maximal solution $(u,\sigma)$ to \eqref{eq:QSEE_intro}, and the paths of $u$ almost surely satisfy
\begin{equation}\label{eq:introreg}
u\in L^p_{\rm loc}([0,\sigma),w_{\a};X_{1}) \cap C([0,\sigma);\Xap)\cap C((0,\sigma);\Xp).
\end{equation}
\end{theorem}
In the above $\sigma>0$ a.s., and $\Xp = (X_0,X_1)_{1-\frac{1}{p},p}$ and $\Xap := (X_0,X_1)_{1-\frac{1+\a}{p},p}$. A loose introduction to maximal regularity can be found in \cite[Section 1.2]{AV19_QSEE_1}.
By analyzing the precise polynomial growth conditions of $F$ and $G$ we obtain conditions on $(p,\kappa)$ for criticality of the space $\Xap$. Of course, this condition also depends on the choice of the spaces $X_0$ and $X_1$. However, the corresponding `trace space' $\Xap$ in the critical case is usually independent of the choice of the scale (see \cite[Section 2.4]{CriticalQuasilinear} and the applications in \cite[Section 5-7]{AV19_QSEE_1}).
A brief  introduction to criticality in this context can be found in \cite[Section 1.1]{AV19_QSEE_1}.

The well-posedness result of Theorem \ref{thm:wellintro} is easy to state, and in \cite{AV19_QSEE_1} we showed that it leads to new results for many of the classical SPDEs, including Burger’s equation, the Allen-Cahn equation, the Cahn-Hilliard equation, reaction–diffusion
equations, and the porous media equation.
In the case of additive noise it is often possible to reduce \eqref{eq:QSEE_intro} to the deterministic setting of \cite{CriticalQuasilinear} which in turn can be analyzed pathwise. This is carried out for the so-called primitive equation in fluid dynamics in \cite{hieber2020primitive}.

\subsection{Blow-up criteria}
Next we state one of our blow-up criteria for quasilinear equations. More details and other criteria can be found in Theorem \ref{t:blow_up_criterion}.
To formulate our blow-up criteria we introduce the following notation:
$$
\nonlinearity^{\a}(u;t):=\|F(\cdot,u)\|_{L^p(0,t,w_{\a};X_0)}+\|G(\cdot,u)\|_{L^p(0,t,w_{\a};\g(H,X_{1/2}))}.
$$
\begin{theorem}[Quasi-linear case]\label{thm:blowupintro}
Under suitable conditions, the maximal solution $(u,\sigma)$ of Theorem \ref{thm:wellintro} satisfies
\begin{enumerate}[{\rm(1)}]
\item\label{it:blowintroquas1} $\dps\P\Big(\sigma<\infty,\,\lim_{t\uparrow \sigma}u(t) \text{ exists in }\Xap,\,\nonlinearity^{\a}(u;\sigma)<\infty\Big)=0$;
\item\label{it:blowintroquas_not_critical} $\dps \P\Big(\sigma<\infty,\,\lim_{t\uparrow \sigma} u(t)\text{ exists in }\Xap\Big)=0$ if $\Xap$ is not critical for \eqref{eq:QSEE_intro};
\item\label{it:blowintroquas2} $\dps\P\Big(\sigma<\infty,\,\lim_{t\uparrow \sigma} u(t)\text{ exists in}\;\Xap,\,
\|u\|_{L^p(0,\sigma;X_{1-\frac{\a}{p}})}<\infty\Big)=0$.
\end{enumerate}
\end{theorem}
The blow-up criteria of Theorem \ref{thm:blowupintro} can often be used to prove global existence. For instance, in case $\Xap$ is not critical, for this one needs a suitable a priori bound for the solution which implies the existence of the limit $\lim_{t\uparrow \sigma} u(t)$ in the `trace space' $\Xap$ on the set $\{\sigma<\infty\}$. According to Theorem \ref{thm:blowupintro}\eqref{it:blowintroquas_not_critical} this can only happen if $\P(\sigma<\infty) =0$, and thus $\sigma = \infty$ a.s. Similar considerations hold for Theorem \ref{thm:blowupintro}\eqref{it:blowintroquas1} and \eqref{it:blowintroquas2}.
Of course to obtain an a priori bound or energy estimate we need to use structural properties of a given SPDE. To obtain such bounds, one can typically use It\^o's formula, combined with one-sided growth conditions of $F$, and subtle regularity results for linear SPDEs.

In the semilinear case much more can be said (see Theorems \ref{thm:semilinear_blow_up_Serrin} and \ref{thm:semilinear_blow_up_Serrin_refined} for the precise statements).
\begin{theorem}[Semi-linear case]\label{thm:semilinblow}
Under suitable conditions, the maximal solution $(u,\sigma)$ of Theorem \ref{thm:wellintro} satisfies
\begin{enumerate}[{\rm(1)}]
\item\label{it:introuXap2} $\dps \P\Big(\sigma<\infty,\,\sup_{t\in [0,\sigma)}\|u(t)\|_{\Xap}<\infty\Big)=0$ if $\Xap$ is non-critical for \eqref{eq:QSEE_intro};
\item\label{it:introuXap3} $\dps \P\Big(\sigma<\infty,\,\sup_{t\in [0,\sigma)}\|u(t)\|_{\Xap}+\|u\|_{L^p(0,\sigma;X_{1-\frac{\a}{p}})}<\infty\Big)=0$;
\item\label{it:introuXap4}
$\dps \P\Big(\sigma<\infty,\,\|u\|_{L^p(0,\sigma;X_{1-\frac{\a}{p}})}<\infty\Big)=0$ under extra conditions on $\a$.
\end{enumerate}
\end{theorem}
The above results extend the blow-up criteria in \cite[Corollaries 2.2, 3.3 and Theorem 2.4]{CriticalQuasilinear} to the stochastic setting. The criterium \eqref{it:introuXap4} is a Serrin type blow-up condition, and probably the deepest of the criteria stated here. It seems that our results are the first systematic approach to blow-up criteria in the stochastic case. The global existence results for stochastic Navier-Stokes equations in $d=2$, and equations of reaction diffusion type in \cite{AV20_NS, AV19_QSEE_3}, will all be based on these new criteria. Let us mention that some of the criteria we obtain are even new in the deterministic setting.

The advantage of our approach is that for a given concrete SPDE, the local well-posedness theory, and blow-up criteria can be used as a black box. So to prove global existence one only needs to prove energy estimates (which can be hard). However, the rest of the argument can be completed in a rather soft way. We summarize this in the following roadmap, of which a more extensive version can be found in Roadmap \ref{roadcomplete}. An illustration of the results will be discussed in Subsection \ref{ss:introillus}.
\begin{roadmap}
\
\begin{enumerate}[{\rm(a)}]
\item\label{it:roadmaplocsimple} Prove local well-posedness and regularity with Theorem \ref{thm:wellintro};
\item\label{it:roadmapenergysimple} prove an energy estimate;
\item\label{it:roadmapthmapplsimple} combine the energy estimate with Theorem \ref{thm:blowupintro} or \ref{thm:semilinblow} to prove $\sigma=\infty$.
\end{enumerate}
\end{roadmap}
Moreover, instantaneous regularization (see Subsection \ref{ss:introIR}) can help in the above scheme as often the extra regularity enables to prove energy estimates.

In \cite{Brz2, BrzElw, carroll1999stochastic, Hornung, NVW3, NVW11eq} an abstract settings appear in which global existence is proved using an argument which resembles a blow-up criterium.
Of these results the one in \cite[Theorem 4.3]{Hornung} comes closest to the one in Theorem \ref{thm:blowupintro}\eqref{it:blowintroquas2} if $\a=0$. However, in that case the result of Theorem \ref{thm:blowupintro}\eqref{it:blowintroquas1}-\eqref{it:blowintroquas_not_critical} is applicable and actually easier to check as we do not consider $\|u\|_{L^p(0,\sigma;X_1)}<\infty$ in the criteria. There are many other differences, and in particular, the assumptions on the nonlinearities and initial data in \cite{Hornung} are much more restrictive. We refer to the introduction of \cite{AV19_QSEE_1} for a detailed comparison.

\subsection{Instantaneous regularization}\label{ss:introIR}
In order to introduce the reader to instantaneous regularization, in this section we let $X_1\subsetneqq X_0$, which is usual in applications to SPDEs.
From \eqref{eq:introreg} one sees that if $\a>0$, then the solution to \eqref{eq:QSEE_intro} instantaneously regularizes `in space' as the regularity of $u$ for $t>0$ is better than the one in $t=0$ since $\Xp\subsetneqq \Xap$. This simple but central result is the key behind our new bootstrapping method. It we will now explain in the special setting of Corollary \ref{cor:regularization_X_0_X_1}, which requires the conditions $p>2$ and $\a>0$. The case $p=2$ or $\a=0$ can be studied as well, see Proposition \ref{prop:adding_weights} and the text below it.

Fix $s>0$ and $r\in (p,\infty)$. Since $\a>0$, we can choose $\alpha\in [0,\frac{r}{2}-1)$ such that $\frac{1}{p}<\frac{1+\alpha}{r}<\frac{1+\a}{p}$. By \eqref{eq:introreg} we have $u(s)\in \Xp\hookrightarrow X_{\alpha,r}^{\Tr}$ a.s.\ and one can construct a maximal local solution to \eqref{eq:QSEE_intro} starting at $s$ with initial data $u(s)\in X_{\alpha,r}^{\Tr}$ by Theorem \ref{thm:wellintro}. This gives a maximal local solution $(v,\tau)$ on $[s,\infty)$ and by \eqref{eq:introreg},
\begin{equation}\label{eq:introregv}
v\in L^r_{\rm loc}([s,\tau),w_{\alpha};X_{1}) \cap C([s,\tau);X^{\Tr}_{\alpha,r})\cap C((s,\tau);\Xr) \ \ \text{almost surely}.
\end{equation}
Since $r>p$, $\Xr\subsetneqq \Xp$ and hence the regularity of $v$ seems to be better than the one of $u$ in \eqref{eq:introreg}. Now if we could show that $\tau=\sigma$ and $u=v$ on $[s,\sigma)$, then this would improve the regularity of $u$ significantly. By choosing $r$ large one can even obtain H\"older regularity in time (see Corollary \ref{cor:regularization_X_0_X_1} for details).

To prove $u=v$, first note that by using the regularity of $v$ and the uniqueness of the maximal local solution $(u,\sigma)$, one can obtain $\tau\leq \sigma$ a.s.\ and $v=u$ on $[ s,\tau)$. This is not surprising since $v$ is `more regular' than $u$ and therefore one expects that $v$ blows-up before $u$. The key step is proving that $\sigma=\tau$.
To show this, note that on the set $\{\tau<\sigma\}$,
$
v=u\in C((s,\tau];\Xp) \subseteq C((s,\tau];X^{\Tr}_{\alpha,r}),
$
and hence
\begin{align*}
\P(\tau<\sigma)&=\P\Big(\{\tau<\sigma\}\cap \{\tau<\infty\}
\cap \Big\{\lim_{t\uparrow\tau} v(t) \text{ exists in }X^{\Tr}_{\alpha,r}\Big\}\Big)\\
&\leq
\P\Big(\tau<\infty,\, \lim_{t\uparrow\tau} v(t) \text{ exists in }X^{\Tr}_{\alpha,r}\Big)=0,
\end{align*}
which follows from the blow-up criterium of Theorem \ref{thm:blowupintro}\eqref{it:blowintroquas_not_critical} applied to $(v,\tau)$.

The above gives an abstract bootstrap mechanism to obtain time regularization of solutions to \eqref{eq:QSEE_intro}. A variation of this strategy can be used to bootstrap regularity in space. This requires two Banach couples $(Y_0,Y_1)$ and $(\wh{Y}_0,\wh{Y}_1)$ in which the equation \eqref{eq:QSEE_intro} can be considered as well. Next we state this result. The precise assumptions are too technical to state here, but the conditions to be checked seem to be natural in all examples we have considered in \cite{AV20_NS, AV19_QSEE_3}.
For details we refer to Theorem \ref{t:regularization_z}.
\begin{theorem}\label{thm:intregintro}
Let $(Y_0,Y_1)$ and $(\wh{Y}_0,\wh{Y}_1)$ be couples of Banach spaces such that \[Y_1\hookrightarrow Y_0, \qquad \wh{Y}_1\hookrightarrow \wh{Y}_0, \ \ \ \  \text{and}   \ \  \ \  \wh{Y}_i\hookrightarrow Y_i.\]
Let $\wh{r}\geq r\geq p>2$, $\alpha\in [0, \frac{r}{2}-1)$, and $\wh{\alpha}\in [0, \frac{\wh{r}}{2}-1)$. Let $(u,\sigma)$ be the $L^p_{\a}$-maximal solution of Theorem \ref{thm:wellintro}. Under suitable conditions, the following implication holds:
\[u\in \bigcap_{\theta\in [0,1/2)} H^{\theta,r}_{\emph{loc}}(0,{\sigma};Y_{1-\theta}) \ \ \text{a.s.} \ \Longrightarrow \ u\in \bigcap_{\theta\in [0,1/2)} H^{\theta,\wh{r}}_{\emph{loc}}(0,{\sigma};\wh{Y}_{1-\theta}) \ \ \text{a.s.}\]
\end{theorem}
We emphasize that we do not need additional regularity for the initial data $u_0$, since the arguments all take place on $[s,\infty)$ with $s>0$.
The main idea of the theorem is that regularity in the $(Y_0, Y_1, r, \alpha)$-setting be transferred to the $(\wh{Y}_0,\wh{Y}_1,\wh{r}, \wh{\alpha})$-setting. Since we can freely choose the spaces $Y_i$ and $\wh{Y}_i$, we can iterate the above to gain regularity. The regularity class $\bigcap_{\theta\in [0,1/2)}  H^{\theta,r}_{\loc}(0,{\sigma};Y_{1-\theta})$ seems rather obscure at first sight. However, it is the one that contains all information concerning stochastic maximal $L^r$-regularity. Its deterministic analogue $L^r_{\loc}(0,\sigma;Y_1)\cap W^{1,r}_{\loc}(0,\sigma;Y_0)$ is much simpler to deal with.

The extra regularity obtained by bootstrapping, is, of course interesting from a theoretical point of view, but it can also assist in proving global existence. Indeed, due to the extra smoothness and integrability, often one can prove energy estimates on an interval $[s,\sigma)$ with $s>0$ by applying \Ito's formula and integration by parts.

In classical bootstrapping arguments, one argues in a completely different way. Given the maximal solution $(u,\sigma)$ one investigates what regularity $f:=F(\cdot, u)$ and $g:=G(\cdot, u)$ have, and combines this with regularity estimate for {\em linear equations} with inhomogeneities $f$ and $g$ to (hopefully) find more space and time regularity for $u$. With the new information on $u$, one can repeat this argument over and over again. This method is of course very important, but it also has some disadvantages. First of all it requires a smooth initial value. Moreover, in case of critical nonlinearities or unweighted situations, it is often not possible to use this argument as $F(\cdot, u)$ or $G(\cdot, u)$ does not have the right integrability/regularity properties. A 1d example where this occurs is discussed in Subsection \ref{ss:introillus}, and in \cite{AV20_NS} we showed that the same holds for 2d Navier-Stokes equations.

In order to deal with the critical and unweighted case (in particular, if $p=2$), we will prove a further variation of the bootstrapping result of Theorem \ref{thm:intregintro} in Proposition \ref{prop:adding_weights}. Here the idea is to exchange some of the space regularity to create a weighted setting out of an unweighted one. As soon as the weight is there, the loss of integrability and regularity can be recovered with Theorem \ref{thm:intregintro}.

Finally, we mention that in deterministic theory one can often use the implicit function theorem to prove higher order regularity in time and space. This method is referred to as the {\em parameter trick} and usually attributed to \cite{Angenent91,Angenent90Ann}. It can be used to prove differentiability and even real analyticity in time. For further details on this method we refer to \cite[Chapter 5]{pruss2016moving} and to the notes of that chapter for further historical accounts. Of course, differentiability in time is completely out of reach in the stochastic setting, since already Brownian motion itself is not differentiable. Therefore, it seems impossible to extend this method to the stochastic framework.

\subsection{Illustration}\label{ss:introillus}
In Section \ref{s:1D_problem} we illustrate our main results in a simple toy example. Here the point is not the strength of the results, but more the techniques to prove it. Indeed, the same methods can be used to study more serious examples (see \cite{AV20_NS,AV19_QSEE_3,AV22_reaction_diffusion}).

Consider the following stochastic PDE:
\begin{equation}
\label{eq:1D_problem_intro}
\begin{cases}
du-\partial_{x}^2 u\,dt =\partial_x( u^3 )dt+  |u|^{h}dw^c_t,& \ \ \text{on }\Tor,\\
u(0)=u_0,& \ \ \text{on }\Tor,
\end{cases}
\end{equation}
where $u:[0,\infty)\times \O\times \Tor \to \R$ is the unknown process and $w^c_t$ is a colored noise on $\Tor$, i.e.\ an $H^{\lambda}(\Tor)$-cylindrical Brownian motion with $\lambda\in (\frac{1}{2},1)$, and $h\in [1,3)$. The following is a special case of Theorems \ref{t:local_1D} and \ref{t:global_1D} below, where $u^3$ and $|u|^{h}$ are replaced by more general nonlinearities.

\begin{theorem}
\label{t:local_1D_intro}
Let Assumption \ref{ass:1D_stochastic} be satisfied. Then for any $u_0\in L^0_{\F_0}(\O;L^2(\Tor))$ the following holds.
\begin{enumerate}[{\rm(1)}]
\item\label{it:1D_problem_intro} There exists a unique maximal solution $(u,\sigma)$ to \eqref{eq:1D_problem} on $[0,\infty)$ such that
\begin{equation}\label{eq:uregintro}
u\in L^2_{\rm loc}([0,\sigma);H^{1}(\Tor))\cap C([0,\sigma);L^2(\Tor)) \ \ \text{ a.s. }
\end{equation}
\item\label{it:intro_regularity_L2}
(H\"older regularity) $u\in C_{{\rm loc}}^{\theta_1,\theta_2}((0,{\sigma})\times \Tor)$ for all $\theta_1\in (0,1/2)$ and $\theta_2\in (0,1)$.
\item\label{it:intro_global_theta_1} (Global existence) If $h=1$, then $\sigma=\infty$ a.s.
\end{enumerate}
\end{theorem}

As in \cite[Theorem 2.7]{AV20_NS}, by further bootstrapping, one can even get $\theta_2\in (0,2)$ in \eqref{it:intro_regularity_L2}. Moreover, this can be further extended by (for instance) applying Schauder theory. Furthermore, a version of the above result also holds for $u_0$ with negative smoothness in a suitable Besov scale (see Theorem \ref{t:1d_rough_initial_data}).

\eqref{it:1D_problem_intro} follows from our local well-posedness theory of Part I \cite{AV19_QSEE_1}. It seems that it cannot be deduced from other standard results unless $h=1$. The proof of \eqref{it:intro_regularity_L2} requires the full strength of our bootstrapping results and is complicated because of the criticality of the nonlinearity. Some further explanations what this means can be found below. Using \eqref{it:intro_regularity_L2}, the proof of \eqref{it:intro_global_theta_1} follows by deriving an energy estimate from \Ito's formula and applying the blow-up criterium of Theorem \ref{thm:semilinblow}\eqref{it:introuXap3}.

It seems that Theorem \ref{t:local_1D_intro}\eqref{it:intro_regularity_L2} \emph{cannot} be proved by standard bootstrap methods even if $|u|^{h}d w^c$ is omitted. Indeed, following the usual strategy (viewing $\partial_x(u^3)$ as a inhomogeneity), one is tempted to prove that $\partial_x(u^3)\in L^{r}_{\rm loc}([0,\sigma);H^{-1,q}(\Tor))$, where at least one of the integrability exponents $r$ and $q$ is {\em strictly} larger than $2$. In turn, one has to prove that $u\in L^{3r}_{\rm loc}([0,\sigma);L^{3q}(\Tor))$ a.s.\ However, the regularity in \eqref{eq:uregintro} is not strong enough for this. Indeed, the following embeddings are sharp
\begin{align}\label{eq:critproblemex}
C([0,t];L^2(\Tor))\cap L^2(0,t;H^1( \Tor)) \ \hookrightarrow \
 L^{6}(0,t;H^{1/3}(\Tor))
\ \hookrightarrow \
L^{6}(0,t;L^6(\Tor)).
\end{align}
Therefore, using \eqref{eq:uregintro} we merely obtain $\partial_x(u^3)\in L^{2}_{\rm loc}([0,\sigma);H^{-1,2}(\Tor))$ which is useless if we want to improve regularity by standard estimates.

The above is related to the fact that $\partial_x(u^3)$ is a so-called {\em critical nonlinearity} for \eqref{eq:1D_problem}. Roughly speaking this means that the typical energy bound (see Lemma \ref{l:1D_energy_estimates}) and nonlinearity are of the same order, and this makes this type of bootstrapping argument impossible. Unfortunately, we can also not use our Theorem \ref{thm:intregintro} because it is not applicable in case $p=2$. Instead we use the previously mentioned variation of the latter (see Proposition \ref{prop:adding_weights}) where the key point in this example is to create a weighted setting out of the unweighted setting. After that we can apply Theorem \ref{thm:intregintro} to gain further regularity.

Let us stress that the way the results in Section \ref{s:regularization} are applied is summarized in Roadmap \ref{roadcriticalreg}. This Roadmap has some universality in the sense that it can be applied almost verbatim in other situations where a $L^2(L^2)$-theory is available. For instance, this appears in reaction-diffusion equations \cite{AV22_reaction_diffusion} and the 2d stochastic Navier-Stokes equations \cite{AV20_NS}.

\subsection{Other approaches to quasilinear evolution equations}
In this subsection we discuss a selection of other approaches to quasilinear stochastic evolution equations. Recall that our definition of quasilinearity is as explained below \eqref{eq:QSEE_intro}. We will restrict ourselves to the case of strong solutions (in the probabilistic sense) to parabolic equations with multiplicative colored or white Brownian noise. As general references we refer the reader to the monographs \cite{DPZ, LR15, Rozov}. Below we will not discuss the part of the literature where only additive noise, or where kinetic or weak (probabilistic) solutions are considered, as the list for this is too numerous, and these works seem less connected to our paper.

A well-known method to study quasi- and semilinear stochastic evolution equations is to approximate the original equations, derive uniform energy estimates, prove convergence of the approximate solutions, and identify the limit as the solution to the original equation. In particular, one could approximate the semilinear terms $F$ and $G$ by functions which are globally Lipschitz.

In a Hilbert space framework one can also project the SPDE into a finite-dimensional setting and use the above scheme, and in particular this is the general idea in the monotone operator approach (see \cite{LR15} and references therein). Another type of approximation appears in \cite{HZ17}, where the quasilinear equation is regularized by the heat semigroup, a version of the above approximation approach is carried out in $L^2$, and higher order regularity is bootstrapped via $L^p$-theory.

In \cite{KuehnNeamtu2020} the authors use the theory of pathwise mild solutions of \cite{PronkVer14}, to obtain local well-posedness of quasilinear parabolic stochastic evolution equations. This is done in a Hilbert space setting, but could also be done in an $L^p$-setting. One  difficulty with the approach is that one needs generation of evolution families for $A(u)\in \calL(X_1, X_0)$ for a large class of processes $u$.

\subsubsection*{Overview}

\begin{itemize}
\item In Section \ref{s:preliminaries} we discuss the relevant background including: weighted function spaces and stochastic analysis;
\item in Section \ref{s:stochastic_maximal_regularity} we present some extensions of results on stochastic maximal $L^p$-regularity where the initial time is random. Moreover, we prove a perturbation result which plays a central role in blow-up criteria;
\item in Section \ref{s:blow_up} we state our main results on blow-up criteria, and give a simple application to a problem with linear growth conditions;
\item in Section \ref{s:proofs_blow_up_criteria} we prove the blow-up criteria;
\item in Section \ref{s:regularization} we present our results on instantaneous regularization;
\item in Section \ref{s:1D_problem} we illustrate how the results of Sections \ref{s:blow_up} and \ref{s:regularization} can be used.
\end{itemize}

\subsubsection*{Notation}
\begin{itemize}
\item For any $T\in (0,\infty]$, we set $\I_T=(0,T)$. Thus, $\overline{I}_T=[0,T]$ if $T<\infty$ and $\overline{I}_{\infty}=[0,\infty)$;
\item $\R_+ =(0,\infty)$, $\N = \{0, 1, 2, \ldots\}$;
\item we write $A \lesssim_P B$ (resp.\ $A\gtrsim_P B$), whenever there is a constant $C$ only depending on the parameter $P$ such that $A\leq C B$ (resp.\ $A \geq C B$). Moreover, we write $A \eqsim_P B$ if $A \lesssim_P B$ and $A \gtrsim_P B$;
\item we write $C(q_1,\dots,q_N)>0$ or $C=C(q_1,\dots,q_N)>0$ if the positive constant $C$ depends only on the parameters $q_1,\dots,q_N$;
\item for any metric space $(Y,d_Y)$, $y\in Y$, $\eta>0$, we denote by $\B_{Y}(x,\eta):=\{y'\in Y\,:\,d_Y(y,y')<\eta\}$. If $Y$ is a vector space, then we set $\B_{Y}(\eta):=\B_{Y}(0,\eta)$;
\item if $X,Y$ is an interpolation couple of Banach space, we endow the intersection $X\cap Y$ with the norm $\|\cdot\|_{X\cap Y}:=\|\cdot\|_X+\|\cdot\|_Y$;
\item $(\cdot,\cdot)_{\theta,p}$, $[\cdot,\cdot]_{\theta}$ denotes the real and the complex interpolation method, respectively (see \cite{BeLo,InterpolationLunardi});
\item $X^{\mathsf{Tr}}_{\a,p} = (X_0, X_1)_{1-\frac{1+\a}{p},p}$, $X^{\mathsf{Tr}}_{p} = X^{\mathsf{Tr}}_{0,p}$, and $X_{\theta} = [X_0, X_1]_{\theta}$ (see Assumption \ref{ass:X});
\item $w_{\a}^a(t)=|t-a|^{\a}$ see \eqref{eq:weight_shifted};
\item $\A(I,w_{\a}^a;X)$ for $\A\in\{L^p,H^{\theta,p},W^{1,p},\hz^{\theta,p},L_{{\rm loc}},H^{\theta,p}_{{\rm loc}}\dots\}$ denote weighted function spaces (see Subsection \ref{ss:weighted_function_spaces});
\item $C^{\theta_1,\theta_2}$ anisotropic H\"{o}lder space (see the text before Theorem \ref{t:local_1D});
\item $\MRtasigma$ and $\MRtsigmazero=\MRtsigma$, etc.\ denote the couples with stochastic maximal $L^p$-regularity (see Definitions \ref{def:SMRgeneralized} and \ref{def:SMRthetacase});
\item Hypothesis \hyperref[H:hip]{$\Hip$}, see Subsection \ref{ss:quasi_revised};
\item \textit{Critical spaces} see the text below Hypothesis \hyperref[H:hip]{$\Hip$};
\item Hypothesis \hyperref[assum:HY]{\Hiep}$(Y_0,Y_1,\alpha,r)$, see Assumption \ref{assum:HY};
\item $\X$ space associated with the nonlinearities $F$ and $G$, see \eqref{eq:def_X_space}; 
\item the $(Y_0, Y_1, r, \alpha)$-setting, see below Assumption \ref{assum:HY}.
\end{itemize}

\subsubsection*{Acknowledgements}
The authors thank Emiel Lorist for helpful comments. The authors thank the anonymous referees for their helpful remarks to improve the presentation.

\section{Preliminaries}
\label{s:preliminaries}
In this section we collect several known results and fix our notation. Below $X$ denotes a Banach space, and further assumptions will be given when needed.

Let $\Dom\subseteq \R^d$ for some $d\geq 1$. For $k\in\N $,
$C^k(\Dom;X)$ denotes the set of all maps $f:\Dom\to X$ such that $\partial^{\alpha}f$ are continuous on $\Dom$ for all $\alpha\in \N^d$ satisfying $\sum_{j=1}^d \alpha_j=k$.
If $\Dom$ is compact, then $C^k(\Dom;X)$ is endowed with the norm $\|u\|_{C^k(\Dom;X)}:=\sum_{|\alpha|\leq k}\sup_{x\in \Dom}\|\partial^\alpha u(x)\|_X$.

\subsection{Weighted function spaces}
\label{ss:weighted_function_spaces}
Let $p\in (1,\infty)$, $\a\in (-1,p-1)$ and for $ a\geq 0$, we denote by $w_{\a}^a$ the shifted power weight
\begin{equation}
\label{eq:weight_shifted}
w_{\a}^a(t):=|t-a|^{\a}, \quad t\in \R,\quad\qquad w_{\a}:=w_{\a}^0.
\end{equation}

For $\I=(a,b)$ where $0\leq a<b\leq \infty$ and $\theta\in (0,1)$, define the following spaces:
\begin{itemize}
\item $L^p(\I,w_{\a}^a;X)$ is the set of all strongly measurable functions $f:\I\to X$ such that
$\|f\|_{L^p(\I,w_{\a}^a;X)}:=(\int_a^b \|f(t)\|_{X}^p w_{\a}^a(t)dt)^{1/p}<\infty$;
\item $W^{1,p}(\I,w_{\a}^a;X)$ is the set of all  $f\in L^p(\I,w_{\a}^a;X)$ such that the weak derivative satisfies $f'\in L^p(\I,w_{\a}^a;X)$. This space is endowed with the norm:
$$
\|f\|_{W^{1,p}(\I,w_{\a}^a;X)}:=\|f\|_{L^{p}(\I,w_{\a}^a;X)}+\|f'\|_{L^{p}(\I,w_{\a}^a;X)};
$$
\item $\Wz^{1,p}(\I,w_{\a}^a;X)= \{f\in W^{1,p}(\I,w_{\a}^a;X): f(a)=0\}$;
\item $H^{\theta,p}(\I,w_{\a}^a;X) = [L^p(\I,w_{\a}^a;X),W^{1,p}(\I,w_{\a}^a;X)]_{\theta}$ (complex interpolation);

\item $\hz^{\theta,p}(\I,w_{\a}^a;X) = [L^p(\I,w_{\a}^a;X_1),\Wz^{1,p}(\I,w_{\a}^a;X)]_{\theta}$;
\item For an interval $J\subseteq \overline{I}$ and $\A\in \{L^p,H^{\theta,p},W^{1,p}\}$, we denote by $\A_{{\rm loc}}(J,w_{\a}^{a};X)$ the set of all strongly measurable maps $f:J\to X$ such that $f\in\A(J',w_{\a}^{a};X)$ for all
bounded intervals $J'$ with $\overline{J'}\subseteq J$.
\end{itemize}
If $\a=0$, then the weight will be omitted from the notation.
For the definition of $H^{\theta,p}$ and $\hz^{\theta,p}$ we used complex interpolation. For details on interpolation theory we refer to \cite{BeLo,Tr1} and \cite[Appendix C]{Analysis1}.

In the case $\theta<\frac{1+\a}{p}$, the main result of \cite[Section 6.2]{LMV18} (see \cite[Theorem 2.6]{AV19_QSEE_1} for the case of bounded intervals) states that for all $0\leq a<b\leq \infty$,
\begin{equation}\label{eq:LMVresult}
\hz^{\theta,p}(a,b,w_{\a}^{a};X) = H^{\theta,p}(a,b,w_{\a}^{a};X) \ \ \
\end{equation}
with equivalent norms. This already played an important role in \cite{AV19_QSEE_1}. In the current paper it will play a key role in Subsection \ref{ss:thmblowuphardespart_Serrin}.

The following will be used many times in the manuscript. For each $f\in L^p(a,b,w_{\a}^a;X)$, $f$ is integrable on $(a,b)$ since $\kappa\in (-1,p-1)$. Moreover, in the case that $\a\geq 0$
for any $c\in (a,b)$ one has
\begin{equation*}
\|f\|_{L^p(c,b;X)}\leq |c-a|^{-\a/p} \|f\|_{L^p(c,b,w_{\a}^a;X)}.
\end{equation*}
Let us collect some useful results in the following proposition.

\begin{proposition}
\label{prop:change_p_q_eta_a}
Let $X$ be a Banach space and $p\in (1, \infty)$. Let $0\leq a\leq c<d\leq b< \infty$, $\a\in (-1,p-1)$, $\theta\in (0,1)$ and $\A\in \{\hz,H\}$. The following assertions hold.
\begin{enumerate}[{\rm(1)}]
\item\label{it:loc_embedding} If $\a\geq 0$, then for each $f \in \A^{\theta,p}(a,b,w_{\a}^a;X)$ the following estimates hold:
    \begin{align*}
\|f\|_{\A^{\theta,p}(c,d,w_{\a}^a;X)}&\leq \|f\|_{\A^{\theta,p}(a,b,w_{\a}^a;X)},\\
\|f\|_{\A^{\theta,p}(c,b;X)}&\leq (c-a)^{-\a/p}\|f\|_{\A^{\theta,p}(a,b,w_{\a}^a;X)},\\
\|f\|_{H^{\theta,p}(c,b;X)}&\leq (c-a)^{-\a/p} \|f\|_{\A^{\theta,p}(a,b,w_{\a}^a;X)}.
\end{align*}
In particular $H^{\theta,p}(a,b,w_{\a}^a;X)\hookrightarrow H^{\theta,p}_{\loc}(a,b;X)$;
\item\label{it:ext_by_0_h_z} let $a>0$. Let $\ez_{a}$ be the zero-extension operator from $\Wz^{1,p}(a,b,w_{\a}^a;X)$ to $\Wz^{1,p}(0,b,w_{\a}^a;X)$. Then $\ez_{a}$ induces a contractive mapping
$$
\ez_{a}: \hz^{\theta,p}(a,b,w_{\a}^a;X)\to\hz^{\theta,p}(0,b,w_{\a}^a;X);
$$
\item\label{it:change_a_p_eta_a} let $1<p\leq q<\infty$ and $\eta\in (-1,q-1)$. Assume that $\frac{1+\a}{p}>\frac{1+\eta}{q}$. Then
$
\A^{\theta,q}(a,b,w_{\eta}^a;X)\hookrightarrow \A^{\theta,p}(a,b,w_{\a}^a;X),
$
and for all $f\in \A^{\theta,p}(a,b,w_{\eta}^a;X)$,
$$
\|f\|_{\A^{\theta,p}(a,b,w_{\a}^a;X)}\lesssim |b-a|^{\left(\frac{\a+1}{p}-\frac{\eta+1}{q}\right)}
\|f\|_{\A^{\theta,q}(a,b,w_{\eta}^a;X)},
$$
where the implicit constant depends only on $p,q,\eta,\a$;
\item\label{it:Sob_embedding} let $1< p_0\leq p_1<\infty$, $\theta_0, \theta_1\in (0,1)$ and $\a_i\in (-1,p_i-1)$ for $i\in \{0,1\}$. Assume $\frac{\a_1}{p_1} \leq \frac{\a_0}{p_0}$ and $\theta_0-\frac{1+\a_0}{p_0}\geq \theta_1-\frac{1+\a_1}{p_1}$. Then for all $f\in \A^{\theta_0,p_0}(a,b,w_{\a_0}^a;X)$,
\[
\|f\|_{\A^{\theta_1,p_1}(a,b,w_{\a_1}^a;X)}\lesssim \|f\|_{\A^{\theta_0,p_0}(a,b,w_{\a_0}^a;X)}.
\]
Finally, if $\A=\hz$, then the implicit constant in the above estimate can be chosen independently of $b-a$. 
\end{enumerate}
\end{proposition}

\begin{proof}
\eqref{it:loc_embedding}: This follows by interpolation (see \cite[Proposition 2.3]{AV19_QSEE_1} for details).

\eqref{it:ext_by_0_h_z}: One can check that $\ez_{a}: \hz^{j,p}(a,b,w_{\a}^a;X)\to \hz^{j,p}(0,b,w_{\a}^a;X)$ is contractive for $j\in\{0,1\}$. Therefore, the claim follows by interpolation.

\eqref{it:change_a_p_eta_a}:
We may assume $a=0$ and $T:=b<\infty$. Then for $f\in L^q(\I_T,w_{\eta};X)$,
\begin{align*}
\|f\|_{L^p(\I_T,w_{\a};X)}^p&=\int_0^T \big(t^{\frac{\eta}{q}}\|f(t)\|_{X}\big)^p t^{\a-\eta\frac{p}{q}}  dt \\
&\leq \Big(\int_0^T t^{\frac{\a q-\eta p}{q-p}} dt\Big)^{\frac{q-p}{q}}
\Big( \int_0^T t^{\eta}\|f(t)\|^{q}_{X}dt\Big)^{\frac{p}{q}}\\
&= C_{p,q,\a,\eta} T^{p\left(\frac{1+\a}{p}-\frac{1+\eta}{q}\right)}\|f\|_{L^q(\I_T,w_{\eta};X)}^p,
\end{align*}
where we used the H\"{o}lder's inequality. Clearly, the same estimate holds for the first order Sobolev space. The general case follows by interpolation.

\eqref{it:Sob_embedding}: See \cite[Proposition 2.7]{AV19_QSEE_1}.
\end{proof}

\begin{remark}
\label{r:embedding_elementary_false_thresold}
Proposition \ref{prop:change_p_q_eta_a}\eqref{it:change_a_p_eta_a} is false in the limiting case $(1+\a)/p=(1+\eta)/q$. This can be seen by taking $f(t) = t^{-\alpha}$ for an appropriate $\alpha$.
\end{remark}
We state a simple consequence of the above result.
\begin{corollary}\label{cor:technicalweight}
Let $X$ be a Banach space and $1\leq q\leq p<\infty$. Let $0\leq a<b< \infty$.
Let $\varepsilon>0$ and suppose that $\a\in (-1,p-1)$, $\eta\in (-1,q-1)$ and $\frac{1+\a}{p}<\varepsilon+\frac{1+\eta}{q}$.
Then for each $\A\in \{\hz,H\}$,
\[\A^{\theta,p}(a,b,w_{\a}^a;X) \hookrightarrow \A^{\theta-\varepsilon,q}(a,b,w_{\eta}^a;X), \ \ \ \ \theta\in [\varepsilon,1].\]
The case $\frac{1+\a}{p}=\varepsilon+\frac{1+\eta}{q}$ is allowed provided $p=q$.
\end{corollary}

\begin{proof}
It suffices to consider $a=0$ and $b=T$. Let $\A\in \{\hz,H\}$ and let $\theta\in [\varepsilon,1]$ be fixed. We distinguish two cases.

\emph{Case $(i)$}: $\varepsilon<\frac{1+\a}{p}$. Let $\wt{\a}:=\a-\varepsilon p$. Note that $\wt{\a}\leq\a<p-1$ and $\wt{\a}>-1$ since $\varepsilon<\frac{1+\a}{p}$. Therefore, Proposition \ref{prop:change_p_q_eta_a}\eqref{it:Sob_embedding} and \eqref{it:change_a_p_eta_a} (using  $\frac{1+\wt{\a}}{p}= \frac{1+\a}{p}-\varepsilon<\frac{1+\eta}{q}$) give
\begin{equation}
\label{eq:A_first_embedding_case_small_epsilon}
\A^{\theta,p}(\I_T,w_{\a};X)
\hookrightarrow
\A^{\theta-\varepsilon,p}(\I_T,w_{\wt{\a}};X)
\hookrightarrow
\A^{\theta-\varepsilon,q}(\I_T,w_{\eta};X).
\end{equation}

\emph{Case $(ii)$}: $\varepsilon\geq \frac{1+\a}{p}$. Take $\wt{\varepsilon}\in (0,\frac{1+\a}{p})$ such that $\frac{1+\a}{p}<\frac{1+\eta}{q}+\wt{\varepsilon}$. By the previous case, we have
$$
\A^{\theta,p}(\I_T,w_{\a};X)
\hookrightarrow
\A^{\theta-\wt{\varepsilon},q}(\I_T,w_{\eta};X)
\hookrightarrow
\A^{\theta-\varepsilon,q}(\I_T,w_{\eta};X)
$$
where the last inclusion follows by Proposition \ref{prop:change_p_q_eta_a}\eqref{it:Sob_embedding} and $\varepsilon>\wt{\varepsilon}$.

To prove the last claim, note that $\varepsilon<\frac{1+\a}{p}$ due to the assumption. Now the claim follows as in Case $(i)$ if we omit the last embedding of \eqref{eq:A_first_embedding_case_small_epsilon}.
\end{proof}

Next we recall some useful interpolation estimates.
\begin{lemma}[Mixed derivative inequality]
\label{l:mixed_derivative}
Let $(X_0,X_1)$ be an interpolation couple of UMD spaces.
Let $p_i\in (1,\infty)$, $\a_i\in (-1,p_i-1)$ and $s_i\in [0,1]$ for $i\in\{0,1\}$. For $\theta\in (0,1)$ set
\[s:=s_0(1-\theta)+s_1\theta, \ \ \ \
\frac{1}{p}:= \frac{1-\theta}{p_0} + \frac{\theta}{p_1}, \ \ \ \
\a := (1-\theta) \frac{p}{p_0} \a_0 + \theta \frac{p}{p_1} \a_1.\]
Then there exists a constant $C>0$ s.t.\ for all $f\in \A_0(\I_T,w_{\a_0};X_0)\cap \A_1(\I_T,w_{\a_1};X_1)$
\[\|f\|_{\A(\I_T,w_{\a};[X_0,X_1]_{\theta})}\leq
C\|f\|_{\A_0(\I_T,w_{\a_0};X_0)}^{1-\theta}\|f\|_{\A_1(\I_T,w_{\a_1};X_1)}^{\theta}, \]
in each of the following cases:
\begin{enumerate}[{\rm(1)}]
\item\label{it:mixed_derivative_Hz} $\A=H^{s,p}$, $\A_i=H^{s_i,p_i}$ and $s_i\in (0,1)$ for $i\in\{0,1\}$;
\item\label{it:mixed_derivative_H} $\A=\hz^{s,p}$, $\A_i=\hz^{s_i,p_i}$ and $s_i\in (0,1)$ for $i\in\{0,1\}$, provided $s\neq \frac{1+\a}{p}$;
\item\label{it:mixed_derivative_W} $\A=H^{s,p}$, $\A_0=L^{p_0}$, $\A_1=W^{1,p_1}$, $s_0=0$, $s_1=1$;
\end{enumerate}
where in case \eqref{it:mixed_derivative_H} the constant $C$ can be chosen independently of $T$.
\end{lemma}

\begin{proof}
\eqref{it:mixed_derivative_Hz}: See \cite[Proposition 2.8]{AV19_QSEE_1}. The other cases follow by the same argument if one uses the extension operator of \cite[Proposition 2.5]{AV19_QSEE_1}.
\end{proof}

To conclude, let us recall the trace embedding. Here in the limiting case we write $H^{1,p} = W^{1,p}$  and $\hz^{1,p} = \Wz^{1,p}$. We also use the notation $X_{\theta}=[X_0,X_1]_{\theta}$ for $\theta\in [0,1]$ for the complex interpolation spaces.
\begin{proposition}
\label{prop:continuousTrace}
Assume that $p\in (1,\infty)$, $\a\in [0,p-1)$, $\theta\in (0,1]$ and let $0\leq a<b< \infty$. Let $X_0,X_1$ be Banach spaces such that $X_1\hookrightarrow X_0$. Then the following hold:
\begin{enumerate}[{\rm(1)}]
\item\label{it:trace_with_weights_Xap} If $\theta>(1+\a)/p$, then
$$
H^{\theta,p}(a,b,w_{\a}^a;X_{1-\theta})\cap L^p(a,b,w_{\a}^a;X_1)\hookrightarrow
C([a,b];(X_0,X_1)_{1-\frac{1+\a}{p},p});
$$
\item\label{it:trace_without_weights_Xp} If $\theta>1/p$, then for any $\delta\in(0,b-a)$,
$$
H^{\theta,p}(a,b,w_{\a}^a;X_{1-\theta})\cap L^p(a,b,w_{\a}^a;X_1)\hookrightarrow
C([a+\delta,b];(X_0,X_1)_{1-\frac{1}{p},p}).
$$
\end{enumerate}
Moreover, if $H^{\theta,p}$ is replaced by $\hz^{\theta,p}$, the constants in the embeddings in \eqref{it:trace_with_weights_Xap} and \eqref{it:trace_without_weights_Xp} can be taken independent of $a,b$.
\end{proposition}

The above result follows from \cite{MV14,AV19} (see also \cite[Proposition 2.10]{AV19_QSEE_1}) provided $X_1=\Do(A)$ and $A$ is a sectorial operator on $X$. The more general case will be considered in \cite{ALV20}, but is not needed here.

\subsection{Probabilistic preliminaries}
\label{ss:stopping_time}
Throughout the paper $(\O,\MeasurableP,\P)$ denotes a probability space, and $\Filtr=(\F_t)_{t\geq 0}$ is a filtration.
Below $W_H:L^2(\R_+;H)\rightarrow L^2(\Omega)$ denotes an  {\em $H$-cylindrical Brownian motion} with respect to $\Filtr$ (see e.g.\ \cite[Definition 2.11]{AV19_QSEE_1}). The stochastic integration theory of \cite{BNVW08,NVW1,NVW13} is summarized in Part I of our work \cite[Section 2.3]{AV19_QSEE_1}.

Let $X$ be a Banach space. A process $\phi:[0,T]\times\Omega \to X$ is called {\em strongly progressively measurable} if for all $t\in [0,T]$, $\phi|_{[0,t]}$ is strongly $\Borel([0,t])\otimes \F_t$-measurable (here $\Borel$ denotes the Borel $\sigma$-algebra). The $\sigma$-algebra generated by the strongly progressively measurable processes on $\O\times [0,T]$ will be denoted by $\Progress_T$ (or simply $\Progress$) and is a subset of $\Borel([0,T])\otimes \F_{T}$. In the following, for any $p\in [0,\infty)$ and any Banach space $Y$, $L^p_{\Progress}(\I_T\times\O;Y)$ denotes the closed subspace of all strongly progressively measurable processes.

A stopping time (or random time) $\tau$ is a measurable map $\tau:\O\to [0,T]$ such that $\{\tau\leq t\}\in \F_t$ for all $t\in [0,T]$.
We denote by $\ll 0,\sigma\rr$ the stochastic interval
$$
\ll 0,\sigma\rr:=\{(t,\omega)\in [0,T]\times\Omega\,:\,0 \leq t\leq \sigma(\om)\}.
$$
Analogously definitions hold for $\ll 0,\sigma\rro$, $\llo 0,\sigma\rro$ etc.

In accordance with the previous notation, for $A\subseteq \O$ and $\tau,\mu$ two stopping times such that $\tau\leq \mu$, we set
\begin{equation*}
[0,T]\times\Omega\supseteq [\tau,\mu]\times A:=\{(t,\om)\in [0,T]\times A\,:\,\tau(\om)\leq t\leq \mu(\om)\}.
\end{equation*}
Similar definitions are employed for $[\tau,\mu)\times A$, $(\tau,\mu)\times A$ etc. 
In particular, $\ll 0,\sigma\rr=[0,\sigma]\times\Omega$.

Let $X$ be a Banach space and let $A\in \MeasurableP$. Let $\sigma,\mu$ be stopping time such that $\mu\leq \sigma$ a.s. Let $\ee\in \{[\mu,\sigma],(\mu,\sigma)\}$. We say that $u:\ee\times A\to X$ is strongly measurable (resp.\ strongly progressively measurable) if the process
\begin{equation}
\label{eq:strong_measurable_process_cutted}
\one_{\ee \times A}\,u:=
\begin{cases}
u,\qquad \text{on }\ee \times  A,\\
0,\qquad \text{otherwise},
\end{cases}
\end{equation}
is strongly measurable (resp.\ strongly progressively measurable).

To each stopping time $\tau$ we can associate the $\sigma$-algebra of the $\tau$-past,
\begin{equation}
\label{eq:tau_past_sigma_algebra}
\F_{\tau}:=\{A\in \MeasurableP\,:\,\{\tau\leq t\}\cap A\in \F_t,\;\forall t \in [0,T]\}.
\end{equation}

The following well-known results will be used frequently in the paper without further mentioning (see \cite[Lemmas 7.1 and 7.5]{Kal}).
\begin{proposition}
\label{prop:F_sigma_algebra_stopping_times}
Let $\tau$ be a stopping time. Then $\F_{\tau}$ is a $\sigma$-algebra and the following properties hold.
\begin{itemize}
\item If $\tau=t$ a.s.\ for some $t\in [0,T]$, then $\F_{\tau}=\F_t$;
\item if $X:[0,T]\times\Omega\to X$ is a strongly progressively measurable process, then the random variable $X_{\tau}(\om):=X(\tau(\om),\om)$ is strongly $\F_{\tau}$-measurable;
\item if $\sigma$ is a stopping time, then $\{\tau<\sigma\}$ is in $\F_{\tau}\cap \F_{\sigma} = \F_{\tau\wedge \sigma}$.
\end{itemize}
\end{proposition}

We will need the following approximation result for a sequence of stopping times.
\begin{lemma}
\label{l:stopping_k}
Let $(\sigma_n)_{n\geq 1}$ and $\sigma$ be stopping times such that $0\leq \sigma_n<\sigma\leq T$ a.s.\ for all $n\geq 1$ and $\lim_{n\to \infty} \sigma_n=\sigma$ a.s. Then for each $\varepsilon>0$, there exists a sequence of stopping times $(\wt{\sigma}_n)_{n\geq 1}$ such that the following assertion holds for each $n\geq 1$:
\begin{enumerate}[{\rm(1)}]
\item $\wt{\sigma}_n$ takes values in a finite subset of $[0,T]$;
\item $\wt{\sigma}_{n-1}\leq \wt{\sigma}_n$ and $\wt{\sigma}_n\geq \sigma_n$ a.s.;
\item $\P(\wt{\sigma}_n\geq  \sigma)\leq \varepsilon$.
\end{enumerate}
Finally, if $\sup_{\O}\sigma_n<T$ for all $n\geq 1$, then one can choose $(\wt{\sigma}_n)_{n\geq 1}$ such that $\sup_{\O} \wt{\sigma}_n<T$ for all $n\geq 1$.
\end{lemma}
\begin{proof}
Let $\varepsilon>0$. We approximate each $\sigma_n$ from above in a suitable way.
Since for each $n\geq 1$, one has $0=\P(\sigma_n\geq \sigma) = \lim_{j\to \infty} \P(\sigma_n+1/j>\sigma)$ it follows that there exists a $j_n\geq 1$ (also depending on $\varepsilon$)
such that
\begin{equation}\label{eq:sigma_n_P_frac_2}
\P(\sigma_n+1/j_n> \sigma)\leq 2^{-n}\varepsilon.
\end{equation}
Let $0=t_0^n<t_1^n<\ldots<t_{N_n}^n=T$ be such that $|t_i^n-t_{i+1}^n|<1/(2j_n)$ for each $i=0,\dots,N_n-1$. Set $\U_i^n:=\{t_i^n\leq \sigma_n< t_{i+1}^n\}\in \F_{t_{i+1}^n}$ for $i=0,\dots,N_n-1$. Let $\tau_n$ be the stopping time defined by
\[
\tau_n:=\sum_{i=0}^{N_n-1} t_{i+1}^n \one_{\U_i^n}.
\]
Thus, $\sigma_n\leq \tau_n< \sigma_n+1/j_n$ a.s.\ and by \eqref{eq:sigma_n_P_frac_2}, $\P(\tau_n \geq \sigma)\leq \P(\sigma_n+1/j_n> \sigma)<2^{-n}\varepsilon$. Now for each $n\geq 1$, set $\wt{\sigma}_n:=\tau_n\vee \tau_{n-1}\vee \dots\vee \tau_1$.
Then each $\wt{\sigma}_n$ takes values in a finite set. Moreover, $\wt{\sigma}_n\geq \wt{\sigma}_{n-1}$ and $\wt{\sigma}_n\geq \tau_n\geq \sigma_n$ for all $n\geq 1$ a.s. Finally,
$$
\P(\wt{\sigma}_n\geq  \sigma)\leq \sum_{i=1}^n \P(\tau_i\geq  \sigma)\leq \sum_{i=1}^{\infty}2^{-i}\varepsilon= \varepsilon.
$$

The last claim follows from the above construction by choosing $j_n$ in \eqref{eq:sigma_n_P_frac_2} small enough so that $1/j_n < T-\sup_{\O}\sigma_n$.
\end{proof}

Next, we introduce some stopped versions of the spaces we introduced in Section \ref{ss:weighted_function_spaces}. As these definitions sometimes lead to measurability problems, we need to be careful here. This issue already appears in \cite{AV19_QSEE_1} (see Lemma 2.15 and Definition 2.16 there) for processes with random endpoint. The latter definitions can be extended easily to the case of random initial time $\sigma$ provided it takes values in a finite set. In particular, a natural norm can be defined for the space
$$
L^p(\O;\A^{\theta,p}(\sigma,\tau,w_{\a}^{\sigma};X))\ \ \text{ for }\ \ \A\in \{\hz,H\},\ \  \ \text{ and } \ \  \
L^p(\O;C([\sigma,\tau];X))
$$
where $X$ is a Banach space. Additionally, we will need some spaces of processes starting at a more general random time and this is defined below. If the starting random time takes values in a finite set, the definitions coincide.

We say that $f\in L^p(\O;\hz^{\theta,p}(\sigma,T;X))$ if $f:\ll \sigma,T\rr \to X$ is strongly measurable, $f\in \hz^{\theta,p}(\sigma,T;X)$ a.s.\ and $\ez_{\sigma}f\in L^p(\O;\hz^{\theta,p}(\I_T;X))$, where $\ez_{\sigma}$ is the $0$-extension operator in Proposition \ref{prop:change_p_q_eta_a}\eqref{it:ext_by_0_h_z}.  Moreover, we set $\|f\|_{L^p(\O;\hz^{\theta,p}(\sigma,T;X))}:=\|\ez_{\sigma}f\|_{L^p(\O;\hz^{\theta,p}(\I_T;X))}$.

Let $(Y_t)_{t\in \I_T}$ be a family of function spaces such that for any $f:\I_T \to Y$ and any $t\in \I_T$, $f|_{(t,T)}\in Y_t$ and $t\mapsto \|f|_{(t,T)}\|_{Y_t}$ is continuous, we say $f\in L^p(\O;Y_{\sigma})$ if there exists a strongly measurable map $\wt{f}:\ll 0,T\rr\to X$ such that $\wt{f}|_{\llo \sigma,T\rro}=f$ and
$\|f\|_{L^p(\O;Y_{\sigma})}:=(\E\|\wt{f}|_{(\sigma,T)}\|_{Y_{\sigma}}^p)^{1/p}$. The main spaces we need this for are $Y_t = C([t,T];X)$ and $Y_t = L^r(t,T;X)$. In the above, we add the subscript $\Progress$ if in addition the process $f$ is strongly progressively measurable.
One can check that the spaces $L^p(\O;Y_{\sigma})$ with $(Y_{t})_{t\in \I_T}$ as above, are Banach spaces.

\section{Stochastic maximal $L^p$-regularity}
\label{s:stochastic_maximal_regularity}

The theory of stochastic maximal $L^p$-regularity has been developed in many recent works and is still ongoing research. We refer to \cite{Kry,MaximalLpregularity,VP18} and references therein for an overview.
A loose introduction into the topic of (stochastic) maximal regularity can be found in \cite[Section 1.2]{AV19_QSEE_1}.

As in \cite[Section 3]{AV19_QSEE_1}, the following assumption will be made throughout Sections \ref{s:stochastic_maximal_regularity}-\ref{s:regularization} where the abstract theory is studied.

\begin{assumption}
\label{ass:X}
Let $X_0,X_1$ be UMD Banach spaces with type 2 and $X_1\hookrightarrow X_0$ densely.
Assume one of the following holds:
\begin{itemize}
\item $p\in (2,\infty)$ and $\a\in [0,\frac{p}{2}-1)$;
\item $p=2$, $\a=0$ and $X_0,X_1$ are Hilbert spaces.
\end{itemize}
For $\theta\in (0,1)$ and $p,\a$ as above, we set
$$
X_{\theta}:=[X_0,X_1]_{\theta}, \qquad \Xap:=(X_0,X_1)_{1-\frac{1+\a}{p},p}, \qquad \Xp:=\Xzp.
$$
\end{assumption}

Let us recall that, in the case $p=2$ and $\a=0$, by \cite[Corollary C.4.2]{Analysis1} we have $X_{\frac{1}{2}}=(X_0,X_1)_{\frac{1}{2},2}=X^{\Tr}_2$. This is one of the reasons we only consider Hilbert spaces if $p=2$ and it will be used without further mentioning it.

\subsection{Definitions and foundational results}
In this section we recall and extend some definition given in \cite[Section 3]{AV19_QSEE_1}. Let us begin with the following assumption which will be in force throughout Section \ref{s:stochastic_maximal_regularity}.

\begin{assumption}
\label{ass:AB_boundedness}
Let $T\in (0,\infty]$ and $\sigma:\O\to [0,T]$ be a stopping time. The maps $A:\ll\sigma,T\rr \to \calL(X_1,X_0)$, $B:\ll\sigma,T\rr\to \calL(X_1,\g(H,X_{1/2}))$ are strongly progressively measurable. Moreover, there exists a constant $C_{A,B}>0$ such that for a.a.\ $\om\in \O$ and for all $t\in (\sigma(\om),T)$,
$$
\|A(t,\om)\|_{\calL(X_1,X_0)} + \|B(t,\omega)\|_{ \calL(X_1,\g(H,X_{1/2}))}\leq C_{A,B}.
$$
\end{assumption}

Stochastic maximal $L^p$-regularity is concerned with the optimal regularity estimate for the linear abstract stochastic Cauchy problem:
\begin{equation}
\label{eq:diffAB_s}
\begin{cases}
du(t) +A(t)u(t)dt=f(t) dt+ (B(t)u(t)+g(t))dW_H(t),\quad t\in \ll \sigma,T\rr,\\
u(\sigma)=u_{\sigma}.
\end{cases}
\end{equation}
A new feature is that we consider the initial condition at a random time $\sigma$ since this will be needed in some of the proofs below.

The definition of strong solution to \eqref{eq:diffAB_s} is as follows:
\begin{definition}
\label{def:strong_linear}
Let Assumptions \ref{ass:X} and \ref{ass:AB_boundedness} be satisfied.
Let $\tau$ be a stopping time such that $\sigma\leq \tau\leq T$ a.s. Let
$$u_{\sigma}\in L^0_{\F_{\sigma}}(\O;X_0),\; f\in L^0_{\Progress}(\O;L^1(\sigma,\tau;X_0)),\; g\in L^0_{\Progress}(\O;L^2(\sigma,\tau;\g(H,X_{1/2}))).$$
A strongly progressive measurable map $u:\ll\sigma,\tau \rr\to X_1$ is called a {\em strong solution} to \eqref{eq:diffAB_s} on $\ll\sigma,\tau\rr$ if $u\in L^2(\sigma,\tau;X_1)$ a.s.\ and
a.s.\ for all $t\in [\sigma,\tau]$,
\begin{equation*}
u(t)-u_{\sigma}+\int_{\sigma}^t A(s)u(s) ds=\int_{\sigma}^t
 f(s) ds+ \int_0^t\one_{\ll \sigma,\tau\rr}(B(s)u(s)+g(s))dW_H(s).
\end{equation*}
\end{definition}
If $\sigma,\tau$ are constants, then we simply write $u$ is a strong solution to \eqref{eq:diffAB_s} on $[\sigma,\tau]$ instead of $\ll \sigma,\tau\rr$.
As in \cite{AV19_QSEE_1}, we follow \cite{VP18} to define stochastic maximal $L^p$-regularity. Recall that norms at random times are defined in Subsection \ref{ss:stopping_time}.
\begin{definition}[Stochastic maximal $L^p$-regularity]
\label{def:SMRgeneralized}
Let Assumptions \ref{ass:X} and \ref{ass:AB_boundedness} be satisfied. We write $(A,B)\in \MRtasigmaz$ if for every
\begin{equation}
\label{eq:f_g_maximal_reg_sigma}
f\in L^p_{\Progress}(\llo\sigma,T\rro,w_{\a}^{\sigma};X_0),
\qquad
g\in  L^p_{\Progress}(\llo\sigma,T\rro,w_{\a}^{\sigma};\g(H,X_{1/2}))
\end{equation}
there exists a strong solution $u$ to \eqref{eq:diffAB_s} with $u_{\sigma}=0$ such that $u\in L^p_{\Progress}(\llo \sigma,T\rro, w_{\a}^{\sigma};X_{1})$, and, moreover, for all stopping time $\tau$, such that $\sigma\leq \tau \leq T$ a.s., and each strong solution $u\in L^p_{\Progress}(\llo \sigma,\tau\rro, w_{\a}^{\sigma};X_{1})$ to \eqref{eq:diffAB_s} on $\ll \sigma,\tau\rr$ the following estimate holds
$$
\|u\|_{L^p(\llo \sigma,\tau\rro,w_{\a}^{\sigma};X_1)}
\leq C(\|f\|_{L^p(\llo\sigma,\tau\rro,w_{\a}^{\sigma};X_0)}+
\|g\|_{L^p(\llo\sigma,\tau\rro,w_{\a}^{\sigma};\g(H,X_{1/2}))}),
$$
where $C>0$ is independent of $(f,g,\tau)$. We set $\MRtsigmaz:=\MRtsigmazeroz$. Moreover, we write $A\in \MRtasigmaz$ if $(A,0)\in \MRtasigmaz$.
\end{definition}
If $(A,B)\in \MRtasigmaz$, then by the stated estimate a strong solution to \eqref{eq:diffAB_s} in $L^p(\llo \sigma,T\rro,w_{\a}^{\sigma};X_1)$ is unique.

\begin{definition}\label{def:SMRthetacase}
Let Assumptions \ref{ass:X} and \ref{ass:AB_boundedness} be satisfied.
\begin{enumerate}[{\rm(1)}]
\item\label{it:SMR_regularity_0} Let $p\in (2,\infty)$. In case $\a>0$, suppose that $\sigma$ takes values in a finite set. We write $(A,B)\in \MRtasigma$ if $(A,B)\in \MRtasigmaz$ and for each $f,g$ as in \eqref{eq:f_g_maximal_reg_sigma} the strong solution $u$ to \eqref{eq:diffAB_s} with $u_{\sigma}=0$ satisfies
\begin{align*}
\|u\|_{L^p(\O;\hz^{\theta,p}(\sigma,T,w_{\a}^{\sigma};X_{1-\theta}))}\leq C_{\theta}(
\|f\|_{L^p(\llo\sigma,T\rro,w_{\a}^{\sigma};X_0)}+
\|g\|_{L^p(\llo\sigma,T\rro,w_{\a}^{\sigma};\g(H,X_{1/2}))}),
\end{align*}
for each $\theta\in [0,\frac{1}{2})\setminus \{\frac{1+\a}{p}\}$, where $C_{\theta}$ is independent of $(f,g,\tau)$.
\item Let $p=2$ and $\a=0$. We write $(A,B)\in \MRttwosigma$ if $(A,B)\in \MRttwosigmaz$ and for each $f,g$ as in \eqref{eq:f_g_maximal_reg_sigma} the strong solution $u$ to \eqref{eq:diffAB_s} with $u_{\sigma}=0$ satisfies
\begin{equation*}
\|u\|_{L^2(\O; C([\sigma,T];X_{1/2}))}\leq C(\|f\|_{L^2(\llo\sigma,T\rro;X_0)}+
\|g\|_{L^2(\llo\sigma,T\rro;\g(H,X_{1/2}))}),
\end{equation*}
where $C$ is independent of $(f,g,\tau)$.
\end{enumerate}
Finally, we say that $A\in \MRtasigma$ if $(A,0)\in \MRtasigma$ and we set
$$
\MRtsigma:=\MRtsigmazero.
$$
\end{definition}

In the above setting we consider the \emph{solution operator}
\begin{equation}
\label{eq:soloperatorR1}
u = \Sol_{\sigma,(A,B)}(0,f,g).
\end{equation}
In \eqref{eq:soloperatorR2} below the mapping $\Sol_{\sigma,(A,B)}(0,\cdot,\cdot)$ will be extended to nonzero initial data. For $p>2$ and $\a\in [0,\frac{p}{2}-1)$ (where $\sigma$ takes values in a finite set in the case $\a>0$), and $\theta\in [0,\frac{1}{2})\setminus \{\frac{1+\a}{p}\}$, $\Sol_{\sigma,(A,B)}(0,\cdot,\cdot)$ defines a mapping
$$
L^{p}_{\Progress}(\llo \sigma, T\rro,w_{\a}^{\sigma};X_0) \times
L^{p}_{\Progress}(\llo \sigma, T\rro,w_{\a}^{\sigma};\g(H,X_{1/2}))\to
L^p(\O;\hz^{\theta,p}(\sigma,T;X_{1-\theta})).
$$
Moreover, we define the constants
\begin{equation*}
\begin{aligned}
C_{(A,B)}^{\det,\theta,p,\a}(\sigma,T) & = \|\Sol_{\sigma,(A,B)}(0,\cdot,0)
\|_{L^p_{\Progress}(\llo \sigma,T\rro,
w_{\a}^{\sigma};X_0)\to
L^p(\O;\hz^{\theta,p}(\sigma,T,w_{\a}^{\sigma};X_{1-\theta}))},
\\
C_{(A,B)}^{\stoc,\theta,p,\a}(\sigma,T) & = \|\Sol_{\sigma,(A,B)}(0,0,\cdot)\|_{L^p_{\Progress}(\llo \sigma,T\rro,
w_{\a}^{\sigma};\g(H,X_{1/2}))\to L^p(\O;\hz^{\theta,p}(\sigma,T,w_{\a}^{\sigma};X_{1-\theta}))},
\end{aligned}
\end{equation*}
where in the case $p=2$, $\a=0$, $\theta\in (0,\frac{1}{2})$, we replace the range space by $L^2(\O;C([\sigma,T];X_{1/2}))$ (which is constant in $\theta\in (0,\frac{1}{2})$). Moreover, we set
\begin{equation}
\label{eq:constants_SMR}
\begin{aligned}
K_{(A,B)}^{\det,\theta,p,\a}(\sigma,T)& =C_{(A,B)}^{\det,\theta,p,\a}(\sigma,T)+C_{(A,B)}^{\det,0,p,\a}(\sigma,T), \\
K_{(A,B)}^{\stoc,\theta,p,\a}(\sigma,T)&=C_{(A,B)}^{\stoc,\theta,p,\a}+C_{(A,B)}^{\stoc,0,p,\a}(\sigma,T).
\end{aligned}
\end{equation}

\begin{remark}\
\label{r:SMR}
\begin{enumerate}[{\rm(1)}]
\item Definition \ref{def:SMRthetacase} \eqref{it:SMR_regularity_0} reduces to the one in \cite[Section 3]{AV19_QSEE_1} in case $\sigma=0$. The case $\theta= \frac{1+\a}{p}$ is not considered, since the concrete description of the interpolation space is more complicated and not considered in \cite[Theorem 6.5]{LMV18}. In case $\sigma = 0$, this case was included in \cite{AV19_QSEE_1} as it can be obtained by complex interpolation. This becomes unclear for random times $\sigma$;
\item the stopping time $\sigma$ in Definition \ref{def:SMRthetacase}\eqref{it:SMR_regularity_0} is assumed to take only takes finitely many values and this is sufficient for our purposes. We avoid the general case due to nontrivial measurability problems.
\end{enumerate}
\end{remark}

The following basic result can be proved in a similar way as in \cite[Proposition 3.8]{AV19_QSEE_1}. It allows us to focus on proving $(A,B)\in \MRtasigmaz$ and obtain the stronger result $(A,B)\in \MRtasigma$ almost for free.
\begin{proposition}[Transference of stochastic maximal regularity]
\label{prop:time_transference}
Let Assumptions \ref{ass:X} and \ref{ass:AB_boundedness} be satisfied. Let $\sigma:\O\to [0,T]$ be a stopping time which takes values in a finite set if $\a>0$. If $(A,B)\in \MRtasigmaz$ and $\MRtasigma$ is nonempty, then $(A,B)\in \MRtasigma$.
\end{proposition}

In order to check that $\MRtasigma$ is nonempty we can often use the following result of \cite{MaximalLpregularity} (the extension to power weights can be found in \cite{AV19}). More general weights and extrapolation have been considered in \cite{LoVer}. Details on $H^\infty$-functional calculus can be found in \cite{Haase:2} and \cite[Chapter 10]{Analysis2}.
\begin{theorem}
\label{t:SMR_H_infinite}
Let Assumption \ref{ass:X} be satisfied. Assume that $X_0$ is isomorphic to a closed subspace of an $L^q$-space for some $q\in [2,\infty)$ on a $\sigma$-finite measure space. Let $A$ be a closed operator on $X_0$ such that $\Do(A)=X_1$. Assume that there exists a $\lambda\in \R$ such that $\lambda
+A$ has a bounded $H^{\infty}$-calculus on $X_0$ of angle $<\pi/2$. Then
$$
A\in \MRtas, \ \ \ \text{ for all }0\leq s<T<\infty.
$$
\end{theorem}

As noticed in \cite{NVW11eq}, Theorem \ref{t:SMR_H_infinite} easily extends to $\F_s$-measurable operators $A:\O\to \calL(X_1,X_0)$ as long as the estimates and the angle for the $H^{\infty}$-calculus are uniform in $\om\in\O$. Combining Theorem \ref{t:SMR_H_infinite} with \cite[Example 2.1]{AV19_QSEE_1}, one immediately obtains a large class of example of operators in $\MRtas$. For additional examples, see \cite[Section 3]{AV19_QSEE_1}.

\subsection{Random initial times}
\label{ss:solution_operator}
In this section we study the role of random initial times. We will start by considering linear problem \eqref{eq:diffAB_s} for non-trivial initial data at a random initial time. A similar result was proved in \cite[Proposition 3.11]{AV19_QSEE_1} for fixed times, but without taking care of the dependence on the length of the time interval, which turns out to be important here. Therefore, we have to provide a detailed proof. The construction in the proof below will be used later on.

\begin{proposition}[Nonzero initial data]
\label{prop:start_at_s}
Suppose that Assumptions \ref{ass:X} and \ref{ass:AB_boundedness} be satisfied. Let $(A,B)\in \MRtasigmaz$. Then for any $u_{\sigma}\in L^p_{\F_{\sigma}}(\O;\Xap)$, $f\in L^p_{\Progress}(\llo\sigma,T\rro,w_{\a}^{\sigma};X_0)$, and $g\in L^p_{\Progress}(\llo\sigma,T\rro,w_{\a}^{\sigma};\g(H,X_{1/2}))$, there exists a unique strong solution $u\in L^p_{\Progress}((\sigma ,T)\times \O,w_{\a}^{\sigma};X_1)$ to \eqref{eq:diffAB_s} on $\ll \sigma,T\rr$, and
\begin{equation}\label{eq:MaxRegu_0random}
\begin{aligned}
\|u\|_{L^p((\sigma ,T)\times \O,w_{\a}^{\sigma};X_{1})}
& \leq C\|u_{\sigma}\|_{L^p(\O;\Xap)}
\\ &\  +
 C\|f\|_{L^p(\llo\sigma,T\rro,w_{\a}^{\sigma};X_0)}
+C\|g\|_{L^p(\llo\sigma,T\rro,w_{\a}^{\sigma};\g(H,X_{1/2}))}
\end{aligned}
\end{equation}
where $C$ only depends on $p,\a,K^{j,\theta,p,\a}_{(A,B)}(\sigma,T)$.

If additionally $(A,B)\in \MRtasigma$ and $\sigma$ takes finitely many values if $\a>0$, then the left-hand side of \eqref{eq:MaxRegu_0random} can be replaced by
\begin{enumerate}[{\rm (1)}]
\item\label{it:start_at_s1} $\|u\|_{L^p(\Omega;C([\sigma,T];\Xap))}$;
\item\label{it:start_at_s2} $\|u\|_{L^p(\Omega;C([\sigma+\delta,T];\Xp))}$ if $\delta>0$;
\item\label{it:start_at_s3} $\|u\|_{L^p(\Omega;H^{\theta,p}(\sigma,T,w_{\a}^{\sigma};X_{1-\theta}))}$ if $\theta\in [0,\frac12)\setminus \{\frac{1+\a}{p}\}$ and $p\in (2, \infty)$;
\item\label{it:start_at_s4}    $\|u\|_{L^p(\Omega;\hz^{\theta,p}(\sigma,T,w_{\a}^{\sigma};X_{1-\theta}))}$ if $\theta\in (0,\frac{1+\a}{p})$ and $p\in (2, \infty)$;
\end{enumerate}
where $C$ also depends on $\delta$ (resp.\ $\theta$) in \eqref{it:start_at_s2} (resp.\ \eqref{it:start_at_s3} and \eqref{it:start_at_s4}).
\end{proposition}

The norms on random intervals in the above result are defined as at the end of Subsection \ref{ss:stopping_time}. Part \eqref{it:start_at_s4} is an estimate in terms of the $\hz$-norms, and will only play a role in the proof of Theorem \ref{thm:semilinear_blow_up_Serrin}\eqref{it:blow_up_semilinear_serrin_Pruss_modified}.
Note that due to \eqref{eq:LMVresult}, no trace restrictions is needed in \eqref{it:start_at_s4}.

\begin{proof}
For the reader's convenience, we split the proof into several steps. We only consider the case $p>2$, since the case $p=2$ is simpler.

Below we will use the so-called trace method of interpolation. By \cite[Theorem 3.12.2]{BeLo} or \cite[Theorem 1.8.2, p.\ 44]{Tr1}, $\Xap$ is the set of all $x\in X_0+X_1$ such that $x=h(0)$ for some $h\in W^{1,p}(\R_+,w_{\a};X_0)\cap L^p(\R_+,w_{\a};X_1)$ and
\begin{equation}
\label{eq:equivalence_norm}
C_{p,\a}^{-1}\|x\|_{\Xap}\leq \inf \|h\|_{W^{1,p}(\R_+,w_{\a};X_0)\cap L^p(\R_+,w_{\a};X_1)}\leq C_{p,\a}\|x\|_{\Xap},
\end{equation}
where the infimum is taken over all $h$ as above and where $C$ only depends on $p,\kappa$.

{\em Step 1: the case $(A,B)\in \MRtasigmaz$}.
Uniqueness is clear from $(A,B)\in \MRtasigmaz$. By completeness and density (see \cite[Proposition 3.11]{AV19_QSEE_1}), it is enough to prove the claim for $u_{\sigma}=\sum_{j=1}^N \one_{\U_j} x_j$ where $N\geq 1$, $x_1, \ldots, x_N\in \Xap$, and $(\U_i)_{i=1}^N\subseteq\F_{\sigma}$ is a partition of $\Omega$.
By \eqref{eq:equivalence_norm} there exist $h_j \in W^{1,p}(\R_+,w_{\a};X_0)\cap L^p(\R_+,w_{\a};X_1)$ such that $h_j(0)=x_j$ and
\[\|h_j\|_{W^{1,p}(\R_+,w_{\a};X_0)\cap L^p(\R_+,w_{\a};X_1)}\leq 2C_{p,\a}\|x_j\|_{\Xap}, \ \ \ j\in \{1, \ldots, N\}.\]
Then  $v_1:=\sum_{j=1}^N \one_{\U_j} h_j(\cdot-\sigma)$ on $\ll \sigma,\infty\rro$ is strongly progressively measurable.
It follows that $u$ is a strong solution to \eqref{eq:diffAB_s} if and only if $v_2:=u-v_1$ is a strong solution to
\begin{equation}
\label{eq:stochastic_equation_v_2}
\begin{cases}
dv_2+A v_2 dt=[f- \dot{v}_1-A v_1]dt+ (B v_2+B v_1+g)dW_{H}, \\
v_2(\sigma)=0.
\end{cases}
\end{equation}
Let $\A^{\theta,p} = H^{\theta,p}$, or $\A = \hz^{\theta,p}$ if $\theta\in (0,(1+\a)/p)$.  Note that on $\U_j$,
\begin{align*}
\|v_1\|_{\A^{\theta,p}(\sigma,T,w_{\a}^{\sigma};X_{1-\theta})}& \stackrel{(i)}{\leq}
\|t\mapsto h_j (t-\sigma)\|_{\A^{\theta,p}(\sigma,\infty,w_{\a}^{\sigma};X_{1-\theta})}
\\ & \stackrel{(ii)}{\leq} \|h_j\|_{\A^{\theta,p}(\R_+,w_{\a};X_{1-\theta})}
\\ & \stackrel{(iii)}{\leq} C_{\theta}\|h_j\|_{H^{\theta,p}(\R_+,w_{\a};X_{1-\theta})}
\\ & \stackrel{(iv)}{\leq} C_{\theta} \|h_j\|_{L^p(\R_+,w_{\a};X_1)} + C_{\theta}\|h_j\|_{W^{1,p}(\R_+,w_{\a};X_0)}
\\ & \stackrel{(v)}{\leq}  2C_{\theta} C_{p,\a}\|u_{\sigma}\|_{\Xap},
\end{align*}
where in $(i)$ we used Proposition \ref{prop:change_p_q_eta_a}\eqref{it:loc_embedding}, and in $(ii)$ a translation argument. In $(iii)$ we used  \eqref{eq:LMVresult} if $\A^{\theta,p}=\hz^{\theta,p}$ and in $(iv)$ Lemma \ref{l:mixed_derivative}. Finally, in $(v)$ we used the choice of $h_j$.
Note that $(iii)$ can be avoided if $\A^{\theta,p} = H^{\theta,p}$.

In the following we set $C_{\theta}':=2C_{\theta}C_{p,\a}$.
If $\theta\in \{0,1\}$, then taking $L^p(\O)$-norms we obtain
\begin{equation}
\label{eq:estimate_v_1_continuity_sigma}
\|v_1\|_{L^p(\O;H^{\theta,p}(\sigma,T,w_{\a}^{\sigma};X_{1-\theta}))}\leq C'_{\theta} \|u_{\sigma}\|_{L^p_{\F_{\sigma}}(\O;\Xap)},
\end{equation}
Since $v_2$ satisfies \eqref{eq:stochastic_equation_v_2} and $(A,B)\in \MRtasigmaz$, setting $C_1^{\theta}= C_{(A,B)}^{\det,\theta,p,\a}(\sigma,T)$ and $C_2^{\theta} = C_{(A,B)}^{\stoc,\theta,p,\a}(\sigma,T)$  it follows that
\begin{align*}
\|v_2& \|_{L^p(\llo\sigma,T\rro,w_{\a}^{\sigma};X_{1})}
\\ & \leq C_1^0\|f+\dot{v}_1-A v_1\|_{L^p(\llo \sigma,T\rro,w_{\a}^{\sigma};X_0)} + C_2^0\|B v_1+g\|_{L^p(\llo \sigma,T\rro,w_{\a}^{\sigma};\g(H,X_{1/2}))}
\\ & \leq \wt{C}_0 C_1^0
\|u_{\sigma}\|_{L^p(\O;\Xap)}+ C_1^0 \|f\|_{L^p(\llo \sigma,T\rro,w_{\a}^{\sigma};X_0)}+ C_2^0\|g\|_{L^p(\llo \sigma,T\rro,w_{\a}^{\sigma};\g(H,X_{1/2}))},
\end{align*}
where in the last step we used \eqref{eq:estimate_v_1_continuity_sigma}. Combining the estimates for $v_1$ and $v_2$, we obtain \eqref{it:start_at_s1}.

Next suppose $(A,B)\in \MRta$.
In the same way as in Step 1, by \eqref{eq:estimate_v_1_continuity_sigma} for each $\theta\in [0,\frac{1}{2})\setminus \{\frac{1+\a}{p}\}$,
\begin{equation}\label{eq:v2estHtheta}
\begin{aligned}
\|v_2\|_{L^p(\O;\hz^{\theta,p}(\sigma,T,w_{\a}^{\sigma};X_{1-\theta}))}
&\leq \wt{C}_0C_1^{0}\|u_{\sigma}\|_{L^p(\O;\Xap)}+ C_1^{\theta}\|f\|_{L^p(\llo \sigma,T\rro,w_{\a}^{\sigma};X_0)}\\
&\qquad \qquad \qquad\qquad+C_2^{\theta}\|g\|_{L^p(\llo \sigma,T\rro,w_{\a}^{\sigma};\g(H,X_{1/2}))}.
\end{aligned}
\end{equation}
Combining the estimates, we obtain \eqref{it:start_at_s3} and \eqref{it:start_at_s4}, where for \eqref{it:start_at_s3} one additionally needs to use that $\hz^{\theta,p}\hookrightarrow H^{\theta,p}$, contractively.

The maximal estimate in \eqref{it:start_at_s1} follows by considering $v_1$ and $v_2$ separately again.
Indeed, by Proposition \ref{prop:continuousTrace}\eqref{it:trace_with_weights_Xap} applied to each $h_j$ we obtain
$$
\|v_1\|_{C([\sigma,T];\Xap)} \leq C\|u_{\sigma}\|_{\Xap}, \ \ \text{ a.s.\ }
$$
The estimate for $v_2$ follows by combining \eqref{eq:v2estHtheta} for $\theta=0$ and $\theta\in ((1+\a)/p,1/2)$ with Proposition \ref{prop:continuousTrace}. To prove \eqref{it:start_at_s2}, one can argue similarly using Proposition \ref{prop:continuousTrace}\eqref{it:trace_without_weights_Xp} instead of Proposition \ref{prop:continuousTrace}\eqref{it:trace_with_weights_Xap}.
\end{proof}

By the above we can extend the solution operator, defined for $u_{\sigma}=0$ in \eqref{eq:soloperatorR1}, to nonzero initial values by setting
\begin{equation}
\label{eq:soloperatorR2}
\Sol_{\sigma,(A,B)}(u_{\sigma},f,g):=u,
\end{equation}
where $u$ is the unique strong solution to \eqref{eq:diffAB_s} on $\ll\sigma,T\rr$. This defines a mapping from
\begin{equation}
\label{eq:u_sigma_f_g_max_reg}
L^{p}_{\F_{\sigma}}(\O;\Xap)\times
L^{p}_{\Progress}(\llo \sigma, T\rro,w_{\a}^{\sigma};X_0) \times
L^{p}_{\Progress}(\llo \sigma, T\rro,w_{\a}^{\sigma};\g(H,X_{1/2}))
\end{equation}
into $L^p(\llo \sigma,T\rro,w_{\a}^{\sigma};X_1)$.

The following result can be obtained as in \cite[Proposition 3.13]{AV19_QSEE_1}.
\begin{proposition}[Localization and causality]
\label{prop:causality_phi_revised_2}
Let Assumptions \ref{ass:X} and \ref{ass:AB_boundedness} be satisfied.
Let $(A,B)\in \MRtasigmaz$.
Let $\sigma,\tau$ be stopping times such that $\sigma\leq \tau\leq T$ a.s.
Assume that $(u_{\sigma},f,g)$ belongs to the space in \eqref{eq:u_sigma_f_g_max_reg} and $u:=\Sol_{\sigma,(A,B)}(u_{\sigma},f,g)$. Then the following holds:
\begin{enumerate}[{\rm(1)}]
\item\label{it:localization_F_sigma_s} If $\sigma$ is a stopping time with values in $\I_T$, then for any $F\in \F_{\sigma}$ one has
\begin{equation*}
\one_{F}u=\one_{F}\Sol_{\sigma,(A,B)}(\one_{F}u_{\sigma},\one_{F}f,\one_{F}g),\qquad \text{a.s.\ on }\ll s,T\rr.
\end{equation*}
\item\label{it:causality_sigma_tau} For any strong solution $v\in L^p(\llo \sigma,\tau\rro,w_{\a}^{\sigma};X_1)$ to \eqref{eq:diffAB_s} on $\ll \sigma,\tau\rr$, one has
$$
v=u|_{\ll \sigma,\tau\rr}
=\Sol_{\sigma,(A,B)}(u_{\sigma},\one_{\llo \sigma,\tau\rro}f,\one_{\llo \sigma,\tau\rro}g),
\qquad \text{a.s.\ on } \ll \sigma,\tau\rr,
$$
where we can replace $\llo \sigma,\tau\rro$ by its half open versions or closed version as well.
\end{enumerate}
\end{proposition}
Note that the last assertion follows from the fact that the symmetric difference of the different sorts of intervals has zero Lebesgue measure.

Next we show that one can combine operators $(A_j, B_j)$ at discrete random times to obtain an operator in $\MRtasigma$.
\begin{proposition}[Sufficient conditions for SMR at random initial times]
\label{prop:start_at_sigma_random_time}
Let Assumption \ref{ass:X} be satisfied. Suppose that $\sigma=\sum_{j=1}^N \one_{\U_j} s_j$, for $N\in \N$, where $(s_j)_{j=1}^N$ is in $(0,T)$, and $(\U_{j})_{j=1}^N$ is a partition of $\Omega$ with $\U_j\in \F_{s_j}$ for $j\in \{1, \ldots, N\}$. Given $(A_j,B_j)\in \MRtasj$ for $j\in\{1,\dots,N\}$ satisfying Assumption \ref{ass:AB_boundedness}, set
\[A:=\sum_{j=1}^N \one_{\U_j\times[ s_j, T]} A_j, \ \ \text{and} \ \ B:=\sum_{j=1}^N \one_{\U_j\times [ s_j, T]} B_j\]
Then $(A,B)\in \MRtasigma$ and for each $i\in \{\deter, \stoc\}$ and $\theta\in  [0,\frac{1}{2})\setminus \{\frac{1+\a}{p}\}$
\begin{align*}
K_{(A,B)}^{i,\theta,p,\a}(\sigma,T)& \leq \max_{j\in\{1,\dots,N\}} K_{(A_j,B_j)}^{i,\theta, p, \a}(\sigma,T).
\end{align*}
\end{proposition}

\begin{proof}
We only consider $p>2$. Let $\Sol_j:=\Sol_{s_j,(A_j,B_j)}$ be the solution operator associated to $(A_j,B_j)$. Using Proposition \ref{prop:causality_phi_revised_2}, one can check that the unique solution to \eqref{eq:diffAB_s} with $f,g$ as in \eqref{eq:f_g_maximal_reg_sigma} and $u_{\sigma}=0$ is given by
\begin{equation}
\label{eq:gluing_u}
u:=\sum_{j=1}^N\one_{\U_j} \Sol_{j}(0,f,g)=\sum_{j=1}^N\one_{\U_j}\Sol_{j}(0,\one_{\U_j}f,\one_{\U_j}g),
\end{equation}
Therefore, if $f=0$, setting $K = \max_{j\in \{1, \ldots, N\}}K_{(A_j,B_j)}^{\theta,i, p, \a}(s_j, T)$ we obtain
\begin{align*}
\|u\|_{L^p_{\Progress}(\O;\hz^{\theta,p}(\sigma,T,w_{\a}^{\sigma};X_{1-\theta}))}^p &=
\sum_{j=1}^N \E[\one_{\U_j}
\|\Sol_{j}(0,0,\one_{\U_j}g)\|_{\hz^{\theta,p}(s_j,T,w_{\a}^{s_j};X_{1-\theta})}^p]\\
&\leq K^p\Big( \sum_{j=1}^N \E[\one_{\U_j}\|g\|_{L^p(s_j,T,w_{\a}^{s_j};\g(H,X_{1/2}))}^p]
\Big)
\\ &= K^p\|g\|_{L^p(\llo \sigma,T\rro ,w_{\a}^{\sigma};\g(H,X_{1/2}))}^p.
\end{align*}
This proves the required estimate for $K_{(A,B)}^{\stoc,\theta,p,\a}$. The other case is similar.
\end{proof}

We end this subsection with a result which shows that weighted maximal regularity implies unweighted maximal regularity for a shifted problem. Although in applications it is usually obvious that the latter holds, from a theoretical perspective it has some interest that weighted maximal regularity can be sufficient. It can be used to check Assumption \ref{H_a_stochnew} for blow-up criteria.

\begin{proposition}
\label{prop:change_initial_time}
Let Assumptions \ref{ass:X} and \ref{ass:AB_boundedness} be satisfied. Let $\tau$ be a stopping time such that $\sigma\leq \tau \leq T$ a.s. Assume that one of the following conditions holds:
\begin{itemize}
\item $\a=0$.
\item $\a>0$, $\sigma$ takes values in a finite set, where we suppose $s:=\inf\{\tau(\omega)-\sigma(\omega)\,:\,\omega\in \Omega\}>0$ and $r:=\sup\{T-\sigma(\om)\,:\,\omega\in \Omega\}$.
\end{itemize}
If $(A,B)\in \MRtasigma$, then $(A,B)\in \MRttau$, and
$$K^{j,\theta,p,0}_{(A,B)}(\tau,T)\leq s^{-\a/p} r^{\a/p} K^{j,\theta,p,\a}_{(A,B)}(\sigma,T),$$
for $j\in \{\deter,\stoc\}$ and $\theta\in [0,{1}/{2})$.
\end{proposition}
\begin{proof}
To prove the proposition, we only consider the case $p>2$ as the other case is simpler.
Let $u=\Sol_{\sigma,(A,B)}(0,\one_{\ll\tau,T\rr}f,\one_{\ll\tau,T\rr}g)$.
Then Proposition \ref{prop:causality_phi_revised_2} and the assumption on $(A,B)$ imply that $u|_{\ll \tau,T\rr}$ is the unique strong solution to
\begin{equation*}
du +Au dt =fdt+(Bu+g)dW_{H},\qquad u(\tau)=0,
\end{equation*}
and $u=0$ on $\ll \sigma,\tau\rr$.

If $\sigma<\tau$, then combining two estimates in Proposition \ref{prop:change_p_q_eta_a}\eqref{it:loc_embedding} we obtain
\[\|u\|_{\hz^{\theta,p}(\tau,T;X_{1-\theta})} \leq s^{-\a/p}\|u\|_{\hz^{\theta,p}(\tau,T, w_{\a}^{\sigma};X_{1-\theta})}
\leq s^{-\a/p}\|u\|_{\hz^{\theta,p}(\sigma,T,w_{\a}^{\sigma};X_{1-\theta})}.
\]
Clearly, the latter still holds with constant one if $\sigma=\tau$ and $\a=0$.

Therefore, by the assumption on $(A,B)$ for each $\theta\in  [0,\frac{1}{2})\setminus \{\frac{1+\a}{p}\}$, we obtain
\begin{align*}
\|u&\|_{L^p(\O;\hz^{\theta,p}(\tau,T;X_{1-\theta}))}
\\ & \leq s^{-\a/p}\|u\|_{L^p(\O;\hz^{\theta,p}(\sigma,T,w_{\a}^{\sigma};X_{1-\theta}))}
\\ & \leq s^{-\a/p} ( K_1\|\one_{\ll\tau,T\rr}f\|_{L^p(\llo \sigma,T\rro,w_{\a}^{\sigma};X_0)}+K_2\|\one_{\ll\tau,T\rr}g\|_{L^p\llo \sigma,T\rro, w_{\a}^{\sigma};\g(H,X_{1/2}))})
\\ & \leq s^{-\a/p} r^{\a/p}( K_1\|f\|_{L^p(\llo \tau,T\rro;X_0)}+K_2\|g\|_{L^p\llo \tau,T\rro;\g(H,X_{1/2}))}),
\end{align*}
where $K_1 = K^{\deter,\theta,p,\a}_{(A,B)}(\sigma,T)$ and $K_2 = K^{\stoc,\theta,p,\a}_{(A,B)}(\sigma,T)$ using the notation of \eqref{eq:constants_SMR}.
\end{proof}

\subsection{Perturbations}
In this subsection we will discuss a simple perturbation result which will be needed in Theorem \ref{t:blow_up_criterion} on blow-up criteria.
It is based on a version of the method of continuity, which extends the result \cite[Proposition 3.18]{VP18}  of Portal and the second author in several ways.

To simplify the notation for stopping times $\sigma,\tau$ such that $s\leq \tau\leq \sigma\leq T$, we set
\begin{align*}
E_\theta(\sigma,\tau) &= L^p_{\Progress}(\llo\sigma,\tau\rro,w_{\a}^{\sigma};X_\theta) \ \ \text{and}  \ \ E_\theta^{\gamma}(\sigma,\tau)  = L^p_{\Progress}(\llo\sigma,\tau\rro,w_{\a}^{\sigma};\gamma(H,X_\theta)), \ \ \ \theta\in [0,1].
\end{align*}

\begin{proposition}[Method of continuity]\label{prop:methodcont}
Let Assumptions \ref{ass:X} and \ref{ass:AB_boundedness} be satisfied.
Suppose that $(A,B)\in \mathcal{SMR}_{p,\a}(\sigma,T)$, where $\sigma$ is a stopping time with values in $[s,T]$.
Let $\wh{A}:\ll \sigma,T\rr\to\calL(X_1,X_0)$, $\wh{B}:\ll \sigma,T\rr\to \calL(X_1,\g(H,X_{1/2}))$ be strongly progressively measurable and assume there exists a constant $\wh{C}$ such that
\begin{align*}
\|\wh{A}(t,\omega)x\|_{X_0} +\|\wh{B}(t,\omega)x\|_{\g(H,X_{1/2})}\leq \wh{C}_{A,B} \|x\|_{X_1}, \ \ \  (t,\omega)\in \ll \sigma,T\rr,\  x\in X_1.
\end{align*}
Let \[A_{\lambda} = (1-\lambda)A + \lambda \wh{A} \ \ \text{and} \ \ B_{\lambda} = (1-\lambda)B + \lambda \wh{B}, \ \  \lambda\in [0,1].
\]
Suppose that there exist constants $C_{\deter},C_{\stoc}>0$ such that for all $\lambda\in [0,1]$, for all stopping time $\tau$ such that $\sigma\leq \tau\leq T$, $f\in E_0(\sigma,T)$, $g\in  E_{1/2}^{\gamma}(\sigma,T)$
and each $u\in E_1(\sigma,\tau)$ which is a strong solution on $\ll \sigma,\tau\rr$ to
\begin{equation}
\label{eq:diffAB_scont}
\begin{cases}
du(t) +A_{\lambda}udt=f dt+ (B_{\lambda}u+g)dW_H,\quad \text{on} \ \ll \sigma,T\rr,\\
u(\sigma)=0,
\end{cases}
\end{equation}
the following estimate holds
\begin{align}\label{eq:aprioriestcont}
\|u\|_{E_1(\sigma,\tau)}
\leq C_{\deter}\|f\|_{E_0(\sigma,\tau)}+ C_{\stoc}
\|g\|_{E_{1/2}^{\gamma}(\sigma,\tau)}.
\end{align}
Then $(\wh{A},\wh{B})\in \mathcal{SMR}_{p,\a}(\sigma,T)$ and $C_{(A,B)}^{j,0,p,\a}(\sigma,T)\leq C_{j}$ for $j\in \{\deter,\stoc\}$.
\end{proposition}
Of course the above result can be combined with Proposition \ref{prop:time_transference} to find a similar result for $\MRtasigma$.
In case $\MRtasigma\neq \emptyset$, it is enough to prove \eqref{eq:aprioriestcont} assuming also optimal space-time regularity for $u$. The latter observation can be useful in certain situations (see e.g.\ \cite[Proposition 4.1]{AHHS21}). For convenience we give a more precise formulation in the following remark.

\begin{remark}
Let the assumptions of Proposition \ref{prop:methodcont} be satisfied and suppose that $\MRtasigma\neq \emptyset$. Then $(A,B)\in \MRtasigma$ if the estimate \eqref{eq:aprioriestcont} holds with constants independent of $\lambda$ and any strong solution $u$ on $\ll 0,\tau\rr$ to \eqref{eq:diffAB_scont} satisfying
\begin{itemize}
\item $u\in L^p(\O;\hz^{\theta,p}(\sigma,\tau,w_{\a}^{\sigma};X_{1-\theta}))$ for all $\theta\in [0,\frac{1}{2})$ if $p>2$;
\item $u\in L^2(\O\times (\sigma,\tau);X_{1}) \cap L^2(\O;C([\sigma,\tau];X_{1/2}))$ if $p=2$.
\end{itemize}
To see this one can repeat the argument in \cite[Proposition 3.8]{AV19_QSEE_1} to show that if $u\in E_1(\sigma,\tau)$  satisfies \eqref{eq:diffAB_scont} and $\MRtasigma\neq \emptyset $, then $u$ also has the above space-time regularity.
\end{remark}

\begin{proof}
Uniqueness of the solution to \eqref{eq:diffAB_scont} is clear from \eqref{eq:aprioriestcont}. It remains to show the existence of strong solution on $\ll \sigma,T\rr$. Let $\Lambda\subseteq [0,1]$ denote the set  of all $\lambda$ such that for all $f\in E_0(\sigma,T)$ and $g\in E^{\gamma}_{1/2}(\sigma,T)$, \eqref{eq:diffAB_scont} has a strong solution $u\in E_1(\sigma,T)$. Since $(\wh{A},\wh{B})\in \mathcal{SMR}_{p,\a}(\sigma,T)$, one has $0\in \Lambda$ and it is enough to check that $1\in \Lambda$. For this it is enough to show that there exists a $\varepsilon_0>0$  such that for any $\lambda\in \Lambda$ one has $[\lambda, \lambda+\varepsilon_0]\cap [0,1]\subseteq \Lambda$.

Let $\varepsilon_0 = [2C(\wh{C}_{A,B}+C_{A,B})]^{-1}$ where $C_{A,B}$ and $C$ are as in Assumption \ref{ass:AB_boundedness} and \eqref{eq:aprioriestcont}, respectively. Let $\lambda\in \Lambda$ and $\varepsilon\in (0,\varepsilon_0]$ be such that $\lambda+\varepsilon\leq 1$. It is enough to show $\lambda+\varepsilon \in \Lambda$. Given $v\in E_1(\sigma,T)$, let $L_{\varepsilon}(v) = u$, where $u$ is the unique strong solution to \eqref{eq:diffAB_scont} with $(f,g)$  replaced by
$
(f + \varepsilon(Av - \wh{A}v),g+ \varepsilon(\wh{B}v - Bv)).
$

Since $\lambda\in \Lambda$, $L_{\varepsilon}$ defines a mapping on $E_1(\sigma,T)$. By definition, for $v_1, v_2\in E_1(\sigma,T)$ one has that $u_{1,2} := L_{\varepsilon}(v_1) - L_{\varepsilon}(v_2)$ satisfies
\eqref{eq:diffAB_scont} with $(f,g)$ replaced by $(\varepsilon(Av - \wh{A}v), g + \varepsilon(\wh{B}v - Bv))$, where $v = v_1-v_2$. Therefore, by \eqref{eq:aprioriestcont},
\begin{align*}
\|L_{\varepsilon}(v_1) - L_{\varepsilon}(v_2)\|_{E_1(\sigma,T)} & = \|u_{1,2}\|_{E_1(\sigma,T)}
\\ & \leq C(\|\varepsilon(Av - \wh{A}v)\|_{E_0(\sigma,T)}+\|\varepsilon(\wh{B}v - Bv)\|_{E_{1/2}^{\gamma}(\sigma,T)})
\\ & \leq C (\wh{C}_{A,B}+C_{A,B}) \varepsilon \|v_1-v_2\|_{E_1(\sigma,T)}
\\ & \leq \frac12 \|v_1-v_2\|_{E_1(\sigma,T)}.
\end{align*}
By the Banach contraction principle it follows that there exists a unique $u\in E_1(\sigma,T)$, such that $L_{\varepsilon}(u) = u$, and thus $u$ is the unique strong solution of \eqref{eq:diffAB_scont} with $\lambda$ replaced by $\lambda+\varepsilon$. From this we can conclude that $\lambda+\varepsilon\in \Lambda$.

The final estimate is immediate from \eqref{eq:aprioriestcont} for $\lambda=1$.
\end{proof}

Now we are able to state and proof our perturbation result, where the main novelty is that we can allow initial random times. The perturbation is assumed to be small in terms of the maximal regularity constants $C^{\deter,0,p,\kappa}_{(A,B)}$ and $C^{\stoc,0,p,\kappa}_{(A,B)}$ introduced below \eqref{eq:soloperatorR1}, but this will be sufficient for the proof of the blow-up criteria of Theorem \ref{t:blow_up_criterion}. Other perturbation results allowing lower order terms can be found in \cite{VP18} and will be discussed in \cite{AV21_torus}.
\begin{corollary}[Perturbation]\label{cor:Pert2}
Let Assumptions \ref{ass:X} and \ref{ass:AB_boundedness} be satisfied.
Let $\sigma:\O\to [0,T]$ be a stopping time which takes values in a finite set if $\a>0$. Assume that $(A,B)\in \MRtasigma$. Let $\wh{A}:\ll \sigma,T\rr\to\calL(X_1,X_0)$, $\wh{B}:\ll \sigma,T\rr\to \calL(X_1,\g(H,X_{1/2}))$ be strongly progressively measurable such that for some positive constants $C_A,C_B,L_A,L_B$ and for all $x\in X_1$, a.s.\ for all $t\in (\sigma,T)$,
\begin{align*}
\|A(t,\omega)x - \wh{A}(t,\om)x\|_{X_0}&\leq C_A \|x\|_{X_1}, \ \
\|B(t,\omega)x - \wh{B}(t,\om)x\|_{\g(H,X_{1/2})}\leq C_{B} \|x\|_{X_1}.
\end{align*}
If $\delta_{A,B}:=C^{\deter,0,p,\kappa}_{(A,B)}(\sigma,T)C_{A}+C^{\stoc,0,p,\kappa}_{(A,B)}(\sigma,T) C_{B}<1$, then $(\wh{A},\wh{B})\in \MRtasigma$.
\end{corollary}
\begin{proof}
By Proposition \ref{prop:time_transference} and $(A,B)\in \MRtasigma$, it suffices to prove $(\wh{A},\wh{B})\in \MRtasigmaz$, and actually the proof shows that we only need $\MRtasigmaz$ for the latter. We will use the method of continuity of Proposition \ref{prop:methodcont}. In the notation introduced there, let $\lambda\in [0,1]$, and let $u\in E_1(\sigma,\tau)$ be a strong solution to \eqref{eq:diffAB_scont} on $\ll\sigma,\tau\rr$. It suffices to prove the a priori estimate  \eqref{eq:aprioriestcont}. Since $u(\sigma) = 0$,
\[du(t) +Audt=\big[f + \lambda (A- \wh{A})u  ]dt+ \big[Bu+ g + \lambda (\wh{B} - B)u\big]dW_H \ \ \text{on} \  \ll \sigma,\tau],\]
and $(A,B)\in \MRtasigmaz$, it follows from Proposition \ref{prop:causality_phi_revised_2} that a.s.\ on $\ll \sigma,\tau\rr$
\begin{align*}
u = \Sol_{\sigma,(A,B)}(0, \one_{\ll \sigma,\tau\rr} f + \lambda (A- \wh{A}) \one_{\ll \sigma,\tau\rr}u, \one_{\ll \sigma,\tau\rr}g + \lambda (\wh{B} - B)\one_{\ll \sigma,\tau\rr}u).
\end{align*}
Therefore, by the properties of $\Sol_{\sigma,(A,B)}$ we obtain
\begin{align*}
 \|u\|_{E_1(\sigma,\tau)}
&\leq C^{\deter,0,p,\kappa}_{(A,B)}(\sigma,T)\|\one_{\ll \sigma,\tau\rr}f +  \lambda (A- \wh{A})\one_{\ll \sigma,\tau\rr}u\|_{E_0(\sigma,T)}\\ & \quad +
C^{\stoc,0,p,\kappa}_{(A,B)}(\sigma,T) \|\one_{\ll \sigma,\tau\rr}g +\lambda (\wh{B} - B)\one_{\ll \sigma,\tau\rr}u\|_{E_{1/2}^{\gamma}(\sigma,\tau)}
\\ & \leq C^{\deter,0,p,\kappa}_{(A,B)}(\sigma,T)\|f\|_{E_0(\sigma,\tau)} + C^{\stoc,0,p,\kappa}_{(A,B)}(\sigma,T) \|g\|_{E_{1/2}^{\gamma}(\sigma,\tau)} +  \delta_{A,B}\|u\|_{E_1(\sigma,\tau)}.
\end{align*}
Therefore, \eqref{eq:aprioriestcont} follows, and this completes the proof.
\end{proof}

\section{Blow-up criteria for stochastic evolution equations}
\label{s:blow_up}
In this section we present blow-up criteria for parabolic stochastic evolution equations of the form:
\begin{equation}
\label{eq:QSEE}
\begin{cases}
du + A(\cdot,u) u dt = (F(\cdot,u)+f)dt + (B(\cdot,u) u +G(\cdot,u)+g) dW_{H},\\
u(s)=u_{s},
\end{cases}
\end{equation}
where $s\geq 0$. Moreover, in the main theorems below we will actually consider \eqref{eq:QSEE} on a finite time interval $[s,T]$ where $T\in (s,\infty)$ is fixed. Extensions to $T=\infty$ are straightforward consequences. Moreover, by using uniqueness and combining solutions and can always reduce to the finite interval case (see Subsection \ref{ss:globalgeneral}). Before stating our main results we first review local existence results for \eqref{eq:QSEE} proven in \cite{AV19_QSEE_1}.

\subsection{Nonlinear parabolic stochastic evolution equations in critical spaces}
\label{ss:quasi_revised}
To prove local existence for \eqref{eq:QSEE}, we need the following assumptions taken from \cite{AV19_QSEE_1}. We will refer to the condition as:

\textbf{Hypothesis} $\Hip$\label{H:hip}.

\let\ALTERWERTA\theenumi
\let\ALTERWERTB\labelenumi
\def\theenumi{(HA)}
\def\labelenumi{(HA)}
\begin{enumerate}
\item\label{HAmeasur}
Suppose Assumption \ref{ass:X} holds.
Let $A:[s,T]\times\Omega \times \Xap\to \calL(X_1,X_0)$ and $B:[s,T]\times\Omega\times \Xap \to \calL(X_1,\g(H,X_{1/2}))$. Assume that for all $x\in \Xap$ and $y\in X_1$, the maps $(t,\om)\mapsto A(t,\om,x)y$ and $(t,\om)\mapsto B(t,\om,x)y$ are strongly progressively measurable.

Moreover, for all $n\geq 1$, there exist $C_n,L_n\in \R_+$ such that for all $x,y\in \Xap$ with $\|x\|_{\Xap},\|y\|_{\Xap}\leq n$, $t\in [s,T]$, and a.a.\ $\om\in \O$
\begin{align*}
\|A(t,\om,x)\|_{\calL(X_1,X_0)}&\leq C_n (1+\|x\|_{\Xap}),\\
\|B(t,\om,x)\|_{\calL(X_1,\g(H,X_{1/2}))}&\leq C_n(1+ \|x\|_{\Xap}),\\
\|A(t,\om,x)-A(t,\om,y)\|_{\calL(X_1,X_0)}&\leq L_n \|x-y\|_{\Xap},\\
\|B(t,\om,x)-B(t,\om,y)\|_{\calL(X_1,\g(H,X_{1/2}))}&\leq L_n \|x-y\|_{\Xap}.
\end{align*}
\end{enumerate}
\let\theenumi\ALTERWERTA
\let\labelenumi\ALTERWERTB

\let\ALTERWERTA\theenumi
\let\ALTERWERTB\labelenumi
\def\theenumi{(HF)}
\def\labelenumi{(HF)}
\begin{enumerate}
\item\label{HFcritical} The map $F:[s,T]\times\Omega\times X_1 \to X_0$ decomposes as $F:=F_{\Tr}+F_{c}$, where
for all $x\in X_1$ the mappings $(t,\omega)\mapsto F_{\Tr}(t,\om,x)$ and $(t,\omega)\mapsto F_{c}(t,\om,x)$ are strongly progressively measurable. Moreover, $F_{\Tr}$  and $F_{c}$ satisfy the following estimates.
\begin{enumerate}[(i)]
\item\label{it:F_c} There exist $m_{F}\geq 1$, $\varphi_j \in (1-(1+\a)/p,1)$, $\beta_j\in (1-(1+\a)/p,\varphi_j]$, $\rho_j\geq 0$ for $j\in \{1,\dots,m_F\}$ such that
$F_c:[s,T]\times \O\times X_{1}\to X_0$ and for each $n\geq 1$ there exist $C_{c,n},L_{c,n}\in \R_+$ for which
\begin{equation}
\label{eq:F_c_estimates}
\begin{aligned}
\|F_c(t,\om,x)\|_{X_{0}}&\leq C_{c,n} \sum_{j=1}^{m_F}(1+\|x\|_{X_{\varphi_j}}^{\rho_j})\|x\|_{X_{\beta_j}}+C_{c,n},\\
\|F_c(t,\om,x)-F_c(t,\om,y)\|_{X_{0}}&\leq L_{c,n} \sum_{j=1}^{m_F}(1+\|x\|_{X_{\varphi_j}}^{\rho_j}+\|y\|_{X_{\varphi_j}}^{\rho_j})\|x-y\|_{X_{\beta_j}},
\end{aligned}
\end{equation}
a.s.\ for all $x,y\in X_{1}$, $t\in [s,T]$ such that $\|x\|_{\Xap},\|y\|_{\Xap}\leq n$. Moreover, $\rho_j,\varphi_j,\beta_j,\a$ satisfy
\begin{equation}
\label{eq:HypCritical}
\rho_j \Big(\varphi_j-1 +\frac{1+\a}{p}\Big) + \beta_j \leq 1, \qquad j\in \{1,\dots,m_{F}\}.
\end{equation}
\item For each $n\in \N$ there exist $L_{\Tr,n},C_{\Tr,n}\in \R_+$ such that the mapping $F_{\Tr}:[s,T]\times\Omega\times\Xap\to X_0$ satisfies
\begin{align*}
\|F_{\Tr}(t,\om,x)\|_{X_0}&\leq C_{\Tr,n}(1+\|x\|_{\Xap}),\\
\|F_{\Tr}(t,\om,x)-F_{\Tr}(t,\om,y)\|_{X_0}&\leq L_{\Tr,n}\|x-y\|_{\Xap},
\end{align*}
for a.a.\ $\omega\in \O$, for all $t\in [s,T]$ and $\|x\|_{\Xap},\|y\|_{\Xap}\leq n$.
\end{enumerate}

\end{enumerate}
\let\theenumi\ALTERWERTA
\let\labelenumi\ALTERWERTB

\let\ALTERWERTA\theenumi
\let\ALTERWERTB\labelenumi
\def\theenumi{(HG)}
\def\labelenumi{(HG)}
\begin{enumerate}
\item\label{HGcritical} The map $G:[s,T]\times\Omega\times X_1 \to \g(H,X_{1/2})$ decomposes as $G:=G_{\Tr}+G_{c}$, where for all $x\in X_1$ the mappings $(t,\omega)\mapsto G_{\Tr}(t,\om,x)$ and $(t,\omega)\mapsto G_{c}(t,\om,x)$ are strongly progressively measurable. Moreover, $G_{\Tr}$ and $G_{c}$
    satisfy the following estimates.
\begin{enumerate}[(i)]
\item\label{it:G_c} There exist $m_G\geq 1$, $\varphi_j \in (1-(1+\a)/p,1)$, $\beta_j\in (1-(1+\a)/p,\varphi_j]$, $\rho_j\geq  0$  for $j=m_F+1,\dots,m_F+m_G$  such that
$G_c:[s,T]\times \O\times X_{1}\to X_0$ and for each $n\geq 1$ there exist $C_{c,n},L_{c,n}\in \R_+$ for which
\begin{equation}
\label{eq:G_c_estimates}
\begin{aligned}
\|G_c(t,\om,x)\|_{\g(H,X_{1/2})}&\leq C_{c,n} \sum_{j=m_F+1}^{m_F+m_G}(1+\|x\|_{X_{\varphi_j}}^{\rho_j})\|x\|_{X_{\beta_j}}+C_{c,n},\\
\|G_c(t,\om,x)-G_c(t,\om,y)\|_{\g(H,X_{1/2})}&\leq  L_{c,n}\sum_{j=m_F+1}^{m_F+m_G} (1+\|x\|_{X_{\varphi_j}}^{\rho_j}+\|y\|_{X_{\varphi_j}}^{\rho_j})\|x-y\|_{X_{\beta_j}},
\end{aligned}
\end{equation}
a.s.\ for all $x,y\in X_{1}$, $t\in [s,T]$ such that $\|x\|_{\Xap},\|y\|_{\Xap}\leq n$.
Moreover, $\varphi_j,\beta_j,\a$ satisfy
\begin{equation}
\label{eq:HypCriticalG}
\rho_j \Big(\varphi_j-1 +\frac{1+\a}{p}\Big) + \beta_j \leq 1, \qquad j=m_F+1,\dots,m_F+m_G.
\end{equation}
\item For each $n\in \N$ there exist constants $L_{\Tr,n},C_{\Tr,n}$ such that $G_{\Tr}:[s,T]\times\Omega\times\Xap\to X_0$ satisfies
\begin{align*}
\|G_{\Tr}(t,\om,x)\|_{\g(H,X_{1/2})}&\leq C_{\Tr,n}(1+\|x\|_{\Xap}),\\
\|G_{\Tr}(t,\om,x)-G_{\Tr}(t,\om,y)\|_{\g(H,X_{1/2})}&\leq L_{\Tr,n}\|x-y\|_{\Xap},
\end{align*}
for a.a.\ $\omega\in \O$, for all $t\in [s,T]$ and $\|x\|_{\Xap},\|y\|_{\Xap}\leq n$.
\end{enumerate}

\end{enumerate}
\let\theenumi\ALTERWERTA
\let\labelenumi\ALTERWERTB

\let\ALTERWERTA\theenumi
\let\ALTERWERTB\labelenumi
\def\theenumi{(Hf)}
\def\labelenumi{(Hf)}
\begin{enumerate}
\item\label{Hf}
$f\in L^0_{\Progress}(\O;L^p(\I_T,w_{\a};X_0))$ and $g\in L^0_{\Progress}(\O;L^p(\I_T,w_{\a};\g(H,X_{1/2})))$.
\end{enumerate}
\let\theenumi\ALTERWERTA
\let\labelenumi\ALTERWERTB
Following \cite{CriticalQuasilinear} or \cite{AV19_QSEE_1}, we say that $\Xap$ is a \emph{critical space} for \eqref{eq:QSEE} if for some $j\in\{1,\dots,m_F+m_G\}$ equality holds in \eqref{eq:HypCritical} or \eqref{eq:HypCriticalG}. In this case the corresponding power of the weight $\a:=\a_{\crit}$ will be called \emph{critical}. A loose introduction to criticality can be found in \cite[Section 1.1]{AV19_QSEE_1}.

Next we recall the definitions of a strong, local, and maximal solution from \cite{AV19_QSEE_1}.
\begin{definition}[Strong solution]
\label{def:solution1}
Let Hypothesis \hyperref[H:hip]{$\Hip$} be satisfied and let $\sigma$ be a stopping time with $s\leq \sigma\leq T$. A strongly progressively measurable process $u$ on $\ll s,\sigma\rr$
satisfying
$$u\in L^p(s,\sigma,w_{\a}^s;X_1)\cap C([s,\sigma];\Xap)\ \ a.s.$$
is called an $L^p_{\a}$-{\em strong solution} of \eqref{eq:QSEE} on $\ll s,\sigma\rr$ if $F(\cdot, u)\in L^p(s,\sigma,w_{\a};X_0)$ and $G(\cdot, u)\in L^p(s,\sigma,w_{\a};\gamma(H,X_{1/2}))$ a.s.\ and the following identity holds a.s.\ and for all $t\in[s,\sigma]$,
\begin{equation}
\label{eq:identity_sol}
\begin{aligned}
u(t)-u_s &+\int_{s}^{t} A(r,u(r))u(r)dr=\int_{s}^{t} (F(r,u(r))+f(r))dr \\
&+ \int_{s}^t\one_{[s,\sigma]}(r) (B(r,u(r))u(r)+G(r,u(r))+g(r))\,dW_H(r).
\end{aligned}
\end{equation}
\end{definition}

As noted in \cite{AV19_QSEE_1}, the conditions in Definition \ref{def:solution1} ensure that the integrals in \eqref{eq:identity_sol} are well-defined.

\begin{definition}[Local and maximal solution]
\label{def:solution2}
Let hypothesis \hyperref[H:hip]{$\Hip$} be satisfied  and let $\sigma$ be a stopping time with $s\leq \sigma\leq T$ a.s. Moreover, let $u:\ll s,\sigma\rro\rightarrow X_1$ be strongly progressively measurable.
\begin{itemize}
\item $(u,\sigma)$ is called an $L^p_{\a}$-{\em local solution to \eqref{eq:QSEE} on $[s,T]$}, if there exists an increasing sequence $(\sigma_n)_{n\geq 1}$ of stopping times such that $\lim_{n\to \infty} \sigma_n =\sigma$ a.s.\ and $u|_{\ll s,\sigma_n\rr}$ is an $L^p_{\a}$-strong solution to \eqref{eq:QSEE} on $\ll s,\sigma_n\rr$. In this case, $(\sigma_n)_{n\geq 1}$ is called a {\em localizing sequence} for the $L^p_{\a}$-local solution $(u,\sigma)$;
\item an $L^p_{\a}$-local solution $(u,\sigma)$ of \eqref{eq:QSEE} is called {\em unique}, if for every $L^p_{\a}$-local solution $(v,\tau)$ for a.a.\ $\om\in \Omega$ and for all $t\in [s,\tau(\om)\wedge \sigma(\omega))$  one has $v(t,\omega)=u(t,\omega)$;
\item a unique $L^p_{\a}$-local solution $(u,\sigma)$ to \eqref{eq:QSEE} on $[s,T]$ is called an {\em $L^p_{\a}$-maximal local solution on $[s,T]$}, if for any other unique $L^p_{\a}$-local solution $(v,\tau)$ to \eqref{eq:QSEE} on $[s,T]$, we have a.s.\ $\tau\leq \sigma$ and  for a.a.\ $\om \in \Omega$ and all $t\in [s,\tau(\om))$, $u(t,\omega)=v(t,\omega)$.
\end{itemize}
\end{definition}

We will omit the ``on $[s,T]$'' whenever $s,T$ are fixed. In Subsection \ref{ss:globalgeneral} we extend the above to the case $T=\infty$. Moreover,  we omit the prefix $L^p_{\a}$ if no confusion seems likely.
The above definition can be extended verbatim to random initial times $\tau$ instead of $s$. This setting will be needed in some of the preliminary results and will, in particular, be considered in Subsection \ref{ss:QSEE_tau}.

Note that $L^p_{\a}$-maximal local solutions are unique by definition. In addition, an (unique) $L^p_{\a}$-strong solution $u$ on $\ll s,\sigma\rr$ gives an (unique) $L^{p}_{\a}$-local solution $(u,\sigma)$ to \eqref{eq:QSEE} on $[s,T]$.
Finally, under suitable assumptions, one can omit the word  ``unique'' in the definition of maximal $L^p_{\a}$-maximal local solutions, see  Remark \ref{r:uniqueness_maximal} below.


Given $u_s\in L^0_{\F_s}(\O;\Xap)$ we denote by $(u_{s,n})_{n\geq 1}$ a sequence such that
\begin{equation}
\label{eq:approximating_sequence_initial_data}
u_{s,n}\in L^{\infty}_{\F_s}(\O;\Xap),  \quad  \text{ and }\quad  u_{s,n}=u_s  \   \text{ on  }  \ \{\|u_s\|_{\Xap}\leq n\}.
\end{equation}
A possible choice would be to set $u_{s,n} = R_n(u_s)$ where
\begin{equation}
\label{eq:truncation}
R_n(x)=x,\quad \text{ if } \|x\|_{\Xap}\leq n,\quad  \text{otherwise}\quad  R_n(x):=n x/\|x\|_{\Xap}.
\end{equation}

The following is the main local well-posedness result of the first part of our work and was proved in \cite[Theorem 4.7]{AV19_QSEE_1}.
\begin{theorem}[Local well-posedness]
\label{t:local_s}
Let Hypothesis \hyperref[H:hip]{$\Hip$} be satisfied. Let $u_s\in L^0_{\F_s}(\O;\Xap)$ and that \eqref{eq:approximating_sequence_initial_data} holds for some $(u_{s,n})_{n\geq 1}$. Suppose that
\begin{equation}
\label{eq:stochastic_maximal_regularity_assumption_local_extended}
(A(\cdot,u_{s,n}),B(\cdot,u_{s,n}))\in \MRtas, \ \ \ n\geq 1.
\end{equation}
Then the following assertions hold:
\begin{enumerate}[{\rm (1)}]
\item\label{it:regularity_data_L0}{\rm (Existence and regularity)} There exists an $L^p_{\a}$-maximal local solution $(u,\sigma)$ to \eqref{eq:QSEE} such that $\sigma>s$ a.s. Moreover, for each localizing sequence $(\sigma_n)_{n\geq 1}$ for $(u,\sigma)$  one has
\begin{itemize}
\item if $p>2$ and $\a\in [0,\frac{p}{2}-1)$, then for all $\theta\in [0,\frac{1}{2})$ and $n\geq 1$,
 \[u\in H^{\theta,p}(s,{\sigma_n},w_{\a}^s;X_{1-\theta}) \cap C([s,{\sigma_n}];\Xap)  \ \ \ a.s.\]
Moreover, $u$ instantaneously regularizes to $u\in C((s,\sigma_n];\Xp)$ a.s.;
\item if $p=2$  and $\a=0$, then for all $n\geq 1$,
 \[u\in L^2(s,{\sigma_n};X_{1}) \cap C([s,{\sigma_n}];X_{1/2})  \ \ \ a.s.\]
\end{itemize}
\item\label{it:localization_L0} {\rm (Localization)} If $(v,\tau)$ is an $L^p_{\a}$-maximal local solution to \eqref{eq:QSEE} with initial data $v_s\in L^{0}_{\F_s}(\O;\Xap)$, then setting $\Gamma:=\{v_s=u_s\}$, one has
$$
\tau|_{\Gamma}=\sigma|_{\Gamma},\qquad  v|_{\Gamma\times [s,\tau)}=u|_{\Gamma\times [s,\sigma)}.
$$
\end{enumerate}
\end{theorem}

For future reference, we conclude this section with the following remark.
\begin{remark}\
\label{r:nonlinearity}
\begin{enumerate}[{\rm(1)}]
\item In \cite{AV19_QSEE_1}, $F = F_L + F_{\Tr} + F_c$, where $F_L$ is  assumed to have a small Lipschitz constant in the $X_1$-norm (see Hypothesis (HF) in \cite{AV19_QSEE_1}), and similarly for $G$. The results presented there extend to our current setting. However, since in the applications in \cite{AV20_NS, AV19_QSEE_3}, $F_L$ and $G_L$ do not play a role, we prefer to omit these terms here;
\item\label{it:trace_space_smooth} if $F_c$ satisfies \eqref{eq:F_c_estimates} with $\varphi_j,\beta_j\leq 1-\frac{1+\a}{p}$ for all $j\in \{1,\dots,m\}$, then $F_c$ satisfies \ref{HFcritical} with $\rho_j$ as it was, and $\varphi_j=\beta_j=1-\frac{1+\a}{p}+\varepsilon$ where $\varepsilon\in (0,\min\{\frac{1+\a}{(\rho_j+1) p}:1\leq j\leq m_F\})$ is fixed. Indeed, this follows from $X_{1-\frac{1+\a}{p}+\varepsilon}\hookrightarrow X_{\varphi_j}\cap X_{\beta_j}$, and since \eqref{eq:HypCritical} is equivalent to $\varepsilon\leq \frac{1+\a}{(\rho_j+1) p}$. Moreover, due to the choice of $\varepsilon$, \eqref{eq:HypCritical} holds with the {\em strict} inequality. Therefore, in this situation one even knows that $\Xap$ is not critical for \eqref{eq:QSEE}. The same applies to $G_c$.
\end{enumerate}
\end{remark}

\subsection{Main results}
\label{ss:main_result_blow_up}
In this subsection, we state our main results regarding blow-up criteria for \eqref{eq:QSEE}.
For this we will need the following assumption on the nonlinearity $(A,B)$. Recall that by Assumption \ref{ass:X}, either $\a\in [0,\frac{p}{2}-1)$, $p>2$ or $\a=0$, $p=2$.

\begin{assumption}\label{H_a_stochnew}
Suppose that Assumption \emph{\ref{HAmeasur}} holds for $(A,B)$. Let $\ell\in [0,\frac{p}{2}-1)$ (where $\ell=0$ if $p=2$). Assume that for each $M,\eta>0$ and $\theta\in [0,\frac{1}{2})\setminus \{\frac{1+\a}{p}\}$, there exists a constant $K_{M,\eta}^{\theta}$ such that, for all $t\in [s+\eta,T)$ and $v\in L^{\infty}_{\F_t}(\O;\Xellp)$ with $\|v\|_{\Xellp} \leq M$ a.s., one has
$(A(\cdot,v),B(\cdot,v))\in \mathcal{SMR}_{p,\ell}^{\bullet}(t,T)$ and
\begin{equation*}
\max\{K^{\deter,\theta,p,\ell}_{(A(\cdot,v),B(\cdot,v))}(t,T),K^{\stoc,\theta,p,\ell}_{(A(\cdot,v),B(\cdot,v))}(t,T)\}\leq K_{M,\eta}^{\theta}, \ \ \ t\in [s+\eta,T).
\end{equation*}
where the constants are as defined in \eqref{eq:constants_SMR}.
\end{assumption}

Assumption \ref{H_a_stochnew} ensures that the maximal regularity constants are uniform on balls in $\Xap$. It is important that the assumption is only formulated for non-random initial times $t$. Random initial times can be obtained afterward using Proposition \ref{prop:change_initial_time}.
In applications to semilinear equations, i.e.\ in the case that $(A(t,x),B(t,x)) = (\bar{A}(t),\bar{B}(t))$, the condition $(A(t,x),B(t,x))\in \MRtas$ already implies Assumption \ref{H_a_stochnew} for $\ell = 0$ by Proposition \ref{prop:change_initial_time}. Finally, we note that on most applications we know $\mathcal{SMR}_{p,\kappa}^{\bullet}(t,T)\neq \emptyset$ with uniform estimates in $t$. Therefore, by transference (see Proposition \ref{prop:time_transference}) it is enough to check $(A(\cdot,v),B(\cdot,v))\in \mathcal{SMR}_{p,\ell}(t,T)$ together with the above estimate for $\theta=0$.

In the quasilinear case,
Assumption \ref{H_a_stochnew} can be weakened in some situations of interest. For future convenience, we formulate this in the following remark.

\begin{remark}
Let $\closed\subseteq X_{\ell,p}^{\Tr}$ be a closed subset and assume that the maximal $L^p_{\a}$-local solution $(u,\sigma)$ to \eqref{eq:QSEE} satisfies $u(t)\in \closed$ a.s.\ for all $t\in ( 0,\sigma)$. If the previous holds, then the requirement $v\in \closed$ a.s.\ can be added in Assumption \ref{H_a_stochnew}. For instance, in the case $ X_{\ell,p}^{\Tr}$ is a function space the choice $\closed =\{v\in  X_{\ell,p}^{\Tr}\,:\, v\geq 0\}$ can be useful in applications to quasilinear SPDEs where the flow is positive preserving. For more on this see Remark \ref{rem:onlyonpath}.
\end{remark}

For our main blow-up result we need another condition which states that the conditions on $F$ and $G$ are also satisfied in the unweighted setting.

\begin{assumption}
\label{ass:FG_a_zero}
Suppose that Assumption \ref{ass:X} holds for $X_0,X_1,\a,p$. Let $F$ and $G$ be as in Assumptions \emph{\ref{HFcritical}} and \emph{\ref{HGcritical}}. Suppose that Assumption \emph{\ref{HFcritical}} and \emph{\ref{HGcritical}} hold with $\a$ replaced by $0$ and a possibly different choice of the parameters $({\rho}_j',{\varphi}_i',{\beta}_j')$ for $j\in \{0, \ldots, m_F'+m_G'\}$ for certain integers $m_F'$ and $m_G'$.
\end{assumption}

\begin{remark}
If for a given $\a\in [0,\frac{p}{2}-1)$ and for each $j\in \{1, \ldots, m_F+m_G\}$ one has either $\varphi_j = \beta_j $ or $\rho_j \geq 1$
in \ref{HFcritical} and \ref{HGcritical}, then Assumption \ref{ass:FG_a_zero} is satisfied with $\rho_j' = \rho_j$. These cases covers all  applications to SPDEs we considered in \cite{AV19_QSEE_1}. Next we explain the sufficiency of these cases in more details:
\begin{itemize}
\item \emph{Case $\varphi_j=\beta_j$}.  If $\beta_j=\varphi_j >1-1/p$, then one can choose $\varphi_j'=\beta_j'= \varphi_j$. In case $\beta_j =\varphi_j \leq 1-1/p$, one can take $\varphi_j'=\beta_j'=1-\frac{1}{p}+ \frac{1}{(\rho_j+1) p}$ (see Remark \ref{r:nonlinearity}\eqref{it:trace_space_smooth});
\item \emph{Case $\rho_j\geq 1$}. If $\beta_j,\varphi_j >1-1/p$ one can take $\varphi_j'=\beta_j'= \varphi_j$. If $\beta_j,\varphi_j\leq 1-1/p$, one can take $\beta_j'=\varphi_j'=1-\frac1p+ \frac{1}{(\rho_j+1) p}$. Till now we did not use $\rho_j\geq 1$ yet. If $\beta_j\leq 1-1/p<\varphi_j$, then one can argue as follows. Note that, \eqref{eq:HypCritical}, \eqref{eq:HypCriticalG} and $\beta_j>1-\frac{1+\a}{p}$ implies $\rho_j(\varphi_j-1+\frac{1}{p})<\frac{1}{p}+\frac{\a}{p}(1-\rho_j)\leq \frac{1}{p}$. Thus there exists $\varepsilon>0$ such that $\rho_j(\varphi_j-1+\frac{1}{p})+\beta_j'< 1$ where $\beta_j'=1-\frac{1}{p}+\varepsilon$.
\end{itemize}
\end{remark}

The main results of this section are blow-up criteria for the $L^p_{\a}$-maximal local solution of Theorem \ref{t:local_s}.
Theorems \ref{t:blow_up_criterion}-\ref{thm:semilinear_blow_up_Serrin_refined} below show that $\sigma$ is an \emph{explosion time} of the $L^p_{\a}$-maximal local solution of \eqref{eq:QSEE} in a certain norm.
For notational convenience, for $s,t\in [0,T]$ set
\begin{equation}
\label{eq:L_p_norm_nonlinearity}
\nonlinearity_c^\a(u;s,t):=\|F_c(\cdot,u)\|_{L^p(s,t,w^{s}_\a;X_0)}+
\|G_c(\cdot,u)\|_{L^p(s,t,w^{s}_\a;\g(H,X_{1/2}))},
\end{equation}
where we recall $F = F_{\Tr} + F_c$ and $G = G_{\Tr} + G_c$.

\begin{theorem}[Blow up criteria for quasilinear SPDEs]
\label{t:blow_up_criterion}
Let the Hypothesis \hyperref[H:hip]{$\Hip$} be satisfied. Let $u_s\in L^0_{\F_s}(\O;\Xap)$ and suppose that \eqref{eq:approximating_sequence_initial_data} holds. Suppose that
\begin{equation}
\label{eq:stochastic_maximal_regularity_assumption_local_extended_blow_up_quasilinear}
(A(\cdot,u_{s,n}),B(\cdot,u_{s,n}))\in \MRtas, \ \ \ n\geq 1,
\end{equation}
and that Assumption \ref{H_a_stochnew} holds for $\ell\in\{0,\a\}$ and Assumption \ref{ass:FG_a_zero} holds. Let $(u,\sigma)$ be the $L^p_{\a}$-maximal local solution to \eqref{eq:QSEE}. Then
\begin{enumerate}[{\rm(1)}]
\item\label{it:blow_up_norm_general_case_F_c_G_c}
$\dps\P\Big(\sigma<T,\,\lim_{t\uparrow \sigma} u(t)\text{ exists in }\Xap ,\,
\nonlinearity_c^{\a}(u;s,\sigma)<\infty\Big)=0$;
\item\label{it:blow_up_non_critical_Xap}
$\dps\P\Big(\sigma<T,\,\lim_{t\uparrow \sigma} u(t)\text{ exists in }\Xap\Big)=0$
provided $\Xap$ is not critical for \eqref{eq:QSEE};
\item\label{it:blow_up_norm_general_case_quasilinear_Xap_Pruss}
$\dps\P\Big(\sigma<T,\,\lim_{t\uparrow \sigma} u(t)\text{ exists in }\Xap,\,
\|u\|_{L^p(s,\sigma;X_{1-\frac{\a}{p}})}<\infty\Big)=0$.
\end{enumerate}
\end{theorem}
In Figure \ref{fig:tree} we provide a decision tree for applying Theorem \ref{t:blow_up_criterion}.

Some comments are in order. In case $\Xap$ is not critical, \eqref{it:blow_up_non_critical_Xap} is the easiest to check in applications. In critical situations \eqref{it:blow_up_norm_general_case_F_c_G_c} and \eqref{it:blow_up_norm_general_case_quasilinear_Xap_Pruss} are available. As in \cite[Theorem 2.4]{CriticalQuasilinear}, we note that the space $L^p(s,\sigma;X_{1-\frac{\a}{p}})$ appearing in \eqref{it:blow_up_norm_general_case_quasilinear_Xap_Pruss} has the (space-time) Sobolev index $1-\frac{1+\a}{p}$, which coincides with the ones of $C([s,\sigma];\Xap)$.

To apply \eqref{it:blow_up_norm_general_case_F_c_G_c}, one only needs to control  $F_c$ and $G_c$. which can be done with Lemmas \ref{l:embeddings} and \ref{l:F_G_bound_N_C_cn}. The control of $F_c$ and $G_c$ is needed only far from $t=s$. Actually one can replace $\nonlinearity_c^{\a}(u;\,s,\,\sigma)$ by $\nonlinearity_c^{0}(u;\,\tau,\,\sigma)$ for any random time $\tau\in (s,\sigma)$. Indeed, this follows from Theorem \ref{t:local_s}, and Lemmas \ref{l:embeddings} and \ref{l:F_G_bound_N_C_cn}.

As we will show in Section \ref{s:regularization}, the solution $u$ is typically smoother than its values near $t=s$ and this may simplify the proof of \textit{energy estimates}. In applications to concrete SPDEs we always use the following consequence of Theorem \ref{t:blow_up_criterion}\eqref{it:blow_up_norm_general_case_F_c_G_c}:
$$
\P\Big(s'<\sigma<T,\,\lim_{t\uparrow \sigma} u(t)\;\text{exists in}\;\Xap,\,\nonlinearity_c^0\big(u;s',\sigma)<\infty\Big)=0,\text{ for all }s'\in (s,T).
$$
Similar considerations hold for Theorem \ref{t:blow_up_criterion}\eqref{it:blow_up_non_critical_Xap}-\eqref{it:blow_up_norm_general_case_quasilinear_Xap_Pruss} and Theorems \ref{thm:semilinear_blow_up_Serrin}-\ref{thm:semilinear_blow_up_Serrin_refined} below.

In \eqref{it:blow_up_norm_general_case_quasilinear_Xap_Pruss} it suffices to estimate the $L^p(\tau,\sigma;X_{1-\frac{\a}{p}})$-norm of $u$ for some stopping time $\tau\in (s,\sigma)$. This already implies that $u$ is in $L^p$ near $t=s$ as a map with values in $X_{1-\frac{\a}{p}}$. Indeed, if $p>2$ this follows from Theorem \ref{t:local_s}\eqref{it:regularity_data_L0}, and
\begin{equation}
\label{eq:Lpa_norm_up_to_zero}
u\in H^{\frac{\a}{p},p}(s,\sigma_n,w_{\a}^s;X_{1-\frac{\a}{p}})\hookrightarrow
L^p(s,\sigma_n;X_{1-\frac{\a}{p}}),\quad \text{ a.s.\ for all }n\geq 1,
\end{equation}
where we used Proposition \ref{prop:change_p_q_eta_a}\eqref{it:Sob_embedding}. The case $p=2$ is immediate from the fact that $(u,\sigma)$ is an $L^2_0$-maximal local solution (see Definitions \ref{def:solution1}-\ref{def:solution2}).
Part \eqref{it:blow_up_norm_general_case_quasilinear_Xap_Pruss} plays a key role in proving instantaneous regularization of solutions to \eqref{eq:QSEE} in the unweighted setting (see Proposition \ref{prop:adding_weights}).

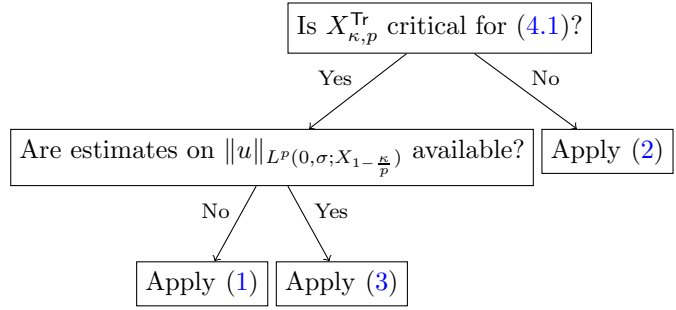
\begin{figure}[h!]
  \centering
\forestset{EL/.style 2 args={edge label={%
    node[midway, font=\footnotesize,
         inner sep=2pt, anchor=south #1]{$#2$}},
                     },
        }
  \begin{forest}
for tree={
edge={->}, draw,
  l sep= 10 mm,
  s sep= 1 mm,
        }
[Is $\Xap$ critical for \eqref{eq:QSEE}?
    [Are estimates on $\|u\|_{L^p(0,\sigma;X_{1-\frac{\a}{p}})}$ available?, EL={east}{\text{Yes}},
        [Apply \eqref{it:blow_up_norm_general_case_F_c_G_c},EL={east}{\text{No}}]
        [Apply \eqref{it:blow_up_norm_general_case_quasilinear_Xap_Pruss},EL={west}{\text{Yes}}
        ]
    ]
    [Apply \eqref{it:blow_up_non_critical_Xap}, EL={west}{\text{No}}]
    ]
]
\end{forest}
  \caption{Decision tree for applying Theorem \ref{t:blow_up_criterion} to quasilinear SPDEs.}\label{fig:tree}
\end{figure}

In applications to semilinear equations, the following improvement of Theorem \ref{t:blow_up_criterion} holds. For convenience,  for $s,t\in [0,T]$ set
\begin{equation}
\label{eq:L_p_norm_nonlinearity2}
\nonlinearity^{\a}(u;s,t):=\|F(\cdot,u)\|_{L^p(s,t,w_{\a}^s;X_0)}+
\|G(\cdot,u)\|_{L^p(s,t,w^{s}_{\a};\g(H,X_{1/2}))}.
\end{equation}

\begin{theorem}[Blow-up criteria for semilinear SPDEs]
\label{thm:semilinear_blow_up_Serrin}
Let the Hypothesis \hyperref[H:hip]{$\Hip$} be satisfied, where we suppose that $(A(t,x),B(t,x)) = (\bar{A}(t), \bar{B}(t))$ does not depend on $x$ and
\begin{equation}
\label{eq:stochastic_maximal_regularity_assumption_local_extended_blow_up_semilinear}
(\bar{A}(\cdot),\bar{B}(\cdot))\in \MRtas.
\end{equation}
Assume that Assumption \ref{H_a_stochnew} holds for $\ell=\a$ and Assumption \ref{ass:FG_a_zero} holds. Let $u_s\in L^{0}_{\F_s}(\O;\Xap)$ and let
$(u,\sigma)$ be the $L^p_{\a}$-maximal local solution to \eqref{eq:QSEE}. Then
\begin{enumerate}[{\rm(1)}]
\item\label{it:blow_up_stochastic_semilinear}
$\dps \P\Big(\sigma<T,\,\nonlinearity^{\a}(u;s,\sigma)<\infty\Big)=0$;
\item\label{it:blow_up_semilinear noncritical_stochastic}
$\dps \P\Big(\sigma<T,\,\sup_{t\in [s,\sigma)}\|u(t)\|_{\Xap}<\infty\Big)=0$
provided $\Xap$ is not critical for \eqref{eq:QSEE};
\item\label{it:blow_up_semilinear_serrin_Pruss_modified}
$\dps \P\Big(\sigma<T,\,\sup_{t\in [s,\sigma)}\|u(t)\|_{\Xap}+\|u\|_{L^p(s,\sigma;X_{1-\frac{\a}{p}})}<\infty\Big)=0$.
\end{enumerate}
\end{theorem}

Taking into account Theorem \ref{thm:semilinear_blow_up_Serrin_refined} below, a decision tree is given in Figure \ref{fig:tree2}.

The proof of Theorem \ref{thm:semilinear_blow_up_Serrin}\eqref{it:blow_up_stochastic_semilinear} does not require Assumption \ref{H_a_stochnew} for $\ell=\a$. Moreover, Assumption \ref{H_a_stochnew} for $\ell=0$ is not assumed since it follows from Assumption \ref{H_a_stochnew} for $\ell=\a$ and Proposition \ref{prop:change_initial_time}.
Theorem \ref{thm:semilinear_blow_up_Serrin}\eqref{it:blow_up_semilinear noncritical_stochastic}-\eqref{it:blow_up_semilinear_serrin_Pruss_modified} are slight improvements of Theorem \ref{t:blow_up_criterion}\eqref{it:blow_up_non_critical_Xap}-\eqref{it:blow_up_norm_general_case_quasilinear_Xap_Pruss} since only boundedness is required. As before in \eqref{it:blow_up_stochastic_semilinear} we may replace $\nonlinearity^{\a}(u;s,\sigma)$ by $\nonlinearity^{0}(u;\tau,\sigma)$ any random time $\tau\in (s,\sigma)$. The same holds for \eqref{it:blow_up_semilinear_serrin_Pruss_modified} with $\|u\|_{L^p(s,\sigma;X_{1-\frac{\a}{p}})}$ replaced by $\|u\|_{L^p(\tau,\sigma;X_{1-\frac{\a}{p}})}$.

Below we will obtain a further improvement of Theorem \ref{thm:semilinear_blow_up_Serrin}\eqref{it:blow_up_semilinear_serrin_Pruss_modified} by removing the condition $\sup_{t\in [s,\sigma)}\|u(t)\|_{\Xap}<\infty$ under suitable assumptions. In literature blow-up criteria which only require $L^p$-bounds are called of \emph{Serrin type} due to the analogy with Serrin's blow up criteria for Navier-Stokes equations (se e.g.\ \cite[Theorem 11.2]{LePi}).
\begin{theorem}
[Serrin type blow-up criteria for semilinear SPDEs]
\label{thm:semilinear_blow_up_Serrin_refined}
Let Hypothesis \hyperref[H:hip]{$\Hip$} be satisfied, where we suppose that $(A(t,x),B(t,x)) = (\bar{A}(t), \bar{B}(t))$ does not depend on $x$ and
\begin{equation}
\label{eq:stochastic_maximal_regularity_assumption_local_extended_blow_up_semilinear_refined}
(\bar{A}(\cdot),\bar{B}(\cdot))\in \MRtas,
\end{equation}
$F_{\Tr}=0$, $G_{\Tr}=0$, the constants $C_{c,n}$ in \emph{\ref{HFcritical}-\ref{HGcritical}} are independent of $n\geq 1$, and for each $j\in \{1,\dots,m_F+m_G\}$
\begin{equation}\label{eq:condSerrin}
\beta_j=\varphi_j \ \  \text{and} \ \ [(\a>0 \ \text{and}  \ \rho_j <1+\a) \ \text{or} \  (\a=0 \ \text{and} \  \rho_j\leq 1)].
\end{equation}
Suppose that Assumption \ref{H_a_stochnew} holds for $\ell=\a$ and Assumption \ref{ass:FG_a_zero} holds.
If $u_s\in L^{0}_{\F_s}(\O;\Xap)$ and  $(u,\sigma)$ is the $L^p_{\a}$-maximal local solution to \eqref{eq:QSEE}, then
\begin{equation}
\label{eq:claim_proof_Serrin_refined}
\dps \P\Big(\sigma<T,\,\|u\|_{L^p(s,\sigma;X_{1-\frac{\a}{p}})}<\infty\Big)=0.
\end{equation}
\end{theorem}

To the best of our knowledge, Theorem \ref{thm:semilinear_blow_up_Serrin_refined} is new even in the case $p=2$ and $\a=0$.
The second part of \eqref{eq:condSerrin} holds if $\rho_j=1$, and this covers the case of bilinear nonlinearities as considered in \cite{AV20_NS,CriticalQuasilinear}.
An extension of Theorem \ref{thm:semilinear_blow_up_Serrin_refined} allowing $\varphi_j\neq \beta_j$ can be found in Proposition \ref{prop:serrin_Pruss_general_form} below.

Suppose that $m_F=m_G=1$,  $\beta:=\beta_1=\beta_2=\varphi_1=\varphi_2$, and $\rho:=\rho_1=\rho_2$ are fixed for a given problem \eqref{eq:QSEE}. Let $S = \{(p,\a): \text{$\Xap$ is critical for \eqref{eq:QSEE}}\}$. Then for all $(p,\a)\in S$ the following identity holds $\frac{1+\a}{p}=\frac{\rho+1}{\rho}(1-\beta)$, which means that $\frac{1+\a}{p}$ is constant. Moreover, the second part of \eqref{eq:condSerrin} holds if and only if
$\frac{\rho}{p}<\frac{1+\a}{p}=\frac{\rho+1}{\rho}(1-\beta)$,
which holds for $p$ large enough. Moreover, $1-\frac{\a}{p}=\frac{\rho+1}{\rho}(\beta-1)+\frac{1}{p}$ which decreases in $p$. Therefore, Theorem \ref{thm:semilinear_blow_up_Serrin_refined} requires only a mild control of the regularity ``in space" of $u$ provided $p$ and thus $\a$, are large.

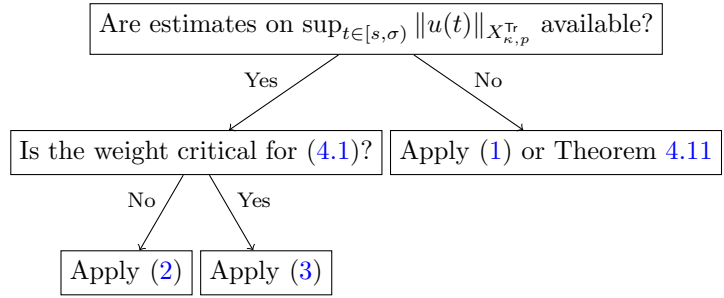
\begin{figure}[h!]
  \centering
\forestset{EL/.style 2 args={edge label={%
    node[midway, font=\footnotesize,
         inner sep=2pt, anchor=south #1]{$#2$}},
                     },
        }
  \begin{forest}
for tree={
edge={->},  draw,
  l sep= 10 mm,
  s sep= 1 mm,
        }
[Are estimates on $\sup_{t\in [s,\sigma)}\|u(t)\|_{\Xap}$ available?
    [Is the weight critical for \eqref{eq:QSEE}?, EL={east}{\text{Yes}},
        [Apply \eqref{it:blow_up_semilinear noncritical_stochastic},EL={east}{\text{No}}]
        [Apply \eqref{it:blow_up_semilinear_serrin_Pruss_modified},EL={west}{\text{Yes}},
                ]
    ]
    [Apply \eqref{it:blow_up_stochastic_semilinear} or Theorem \ref{thm:semilinear_blow_up_Serrin_refined}, EL={west}{\text{No}}]
    ]
]
\end{forest}
  \caption{Decision tree for applying Theorems \ref{thm:semilinear_blow_up_Serrin} and \ref{thm:semilinear_blow_up_Serrin_refined} to semilinear SPDEs.}\label{fig:tree2}
\end{figure}

As a key step in the proof of Theorems \ref{t:blow_up_criterion}--\ref{thm:semilinear_blow_up_Serrin_refined} we prove the following result which is of independent interest.

\begin{proposition}[Predictability of the explosion time $\sigma$]
\label{prop:predictability}
If the conditions of Theorem \ref{t:blow_up_criterion} hold, then any
localizing sequence $(\sigma_n)_{n\geq 1}$ for $(u,\sigma)$ satisfies
$$\P(\sigma<T, \sigma_n= \sigma) = 0, \  \text{ for all $n\geq 1$.}$$
\end{proposition}
The above implies that $\sigma$ is a so-called \emph{predictable} stopping time. Proposition \ref{prop:predictability} will be proven in Subsection \ref{ss:proof_blow_up_criterion}, and as its the proof shows, to obtain it we only need Assumption \ref{H_a_stochnew} for $\ell=0$.

The next simple result will allow us to reduce to integrable or even bounded data. It will be used in the proofs of Theorems \ref{t:blow_up_criterion}, \ref{thm:semilinear_blow_up_Serrin} and \ref{thm:semilinear_blow_up_Serrin_refined}, but it can also be a helpful reduction in proving global existence in concrete situations.
\begin{proposition}[Reduction to uniformly bounded data]\label{prop:redblowbounded}
Let the Hypothesis \hyperref[H:hip]{$\Hip$} be satisfied. Let $u_s\in L^0_{\F_s}(\O;\Xap)$ and suppose that \eqref{eq:approximating_sequence_initial_data} holds.
Let $f_n = f \one_{[s,\tau_n]}$ and $g = g \one_{[s,\tau_n]}$, where
\[\tau_n = \inf\{t\in [s,T]: \|f\|_{L^p(s,t,w_{\a}^s;X_0)}\geq n,  \  \|g\|_{L^p(s,t,w_{\a}^s;\gamma(H,X_{1/2}))}\geq n\}.\]
Let $(u,\sigma)$ be the $L^p_{\a}$-maximal local solution to \eqref{eq:QSEE}, and let $(u_n,\sigma_n)$ be the $L^p_{\a}$-maximal local solution to \eqref{eq:QSEE} with $(u_s,f,g)$ replaced by $(u_{s,n}, f_n, g_n)$.
For each of the statements in Theorems \ref{t:blow_up_criterion}, \ref{thm:semilinear_blow_up_Serrin} and \ref{thm:semilinear_blow_up_Serrin_refined} it suffices to prove that the corresponding probability is zero with $u$ replaced by $u_n$ for each $n\geq 1$.

Finally, if $\sigma_n=T$ a.s.\ for all $n\geq 1$, then $\sigma=T$ a.s. 
\end{proposition}

\begin{proof}
By a translation argument we may assume that $s=0$.
We present the details in case of Theorem \ref{thm:semilinear_blow_up_Serrin}\eqref{it:blow_up_stochastic_semilinear}.
The other cases can be obtained in the same way replacing the set \eqref{eq:redsetO} below by a suitable set in each case.

Set $\Gamma_n:=\{\|u_0\|_{\Xap}\leq n\}\in \F_0$.
Observe that $\P(\{\tau_n=T\}\cap \Gamma_n) \to 1$ as $n\to \infty$, and for each $n\geq 1$, $(u, \sigma\wedge \tau_n)$ is a unique $L^p_{\a}$-local solution to \eqref{eq:QSEE} with $(f, g)$ replaced by $(f_n,g_n)$. Denoting by $(v_n,\mu_n)$ the $L^p_{\a}$-maximal solution to \eqref{eq:QSEE} with $(u_0,f_n, g_n)$, we have $\tau_n \wedge \sigma\leq\mu_n$ and $u=v_n$ on $\ll 0,\tau_n \wedge \sigma\rro$ by maximality (see Theorem \ref{t:local_s}). Similarly, since $(v_n, \mu_n\wedge \tau_n)$ is a  unique $L^p_{\a}$-local solution to  \eqref{eq:QSEE} with $(u_0,f,g)$ we have $\mu_n\wedge \tau_n\leq \sigma$ and  $u=v_n$ on $\ll 0,\mu_n\wedge \tau_n\rro$. It follows that
\begin{equation}\label{eq:taunTmunsigma}
\mu_n = \sigma\ \ \text{on} \ \{\tau_n=T\}, \ \ \text{and} \ \  u = v_n \ \text{on} \  [0,\sigma)\times \{\tau_n=T\}.
\end{equation}
Moreover, by Theorem \ref{thm:semilinear_blow_up_Serrin}\eqref{it:blow_up_stochastic_semilinear},
\begin{equation}\label{eq:taunTmunsigma2}
\mu_n = \sigma_n  \ \ \text{on} \  \Gamma_n, \ \ \text{and} \ \  v_n = u_n \ \text{on} \  \Gamma_n.
\end{equation}
For a stopping time $\nu$ such that $0\leq \nu\leq \sigma$, and a process $(v(t))_{t\in (\tau,\nu)}$ set
\begin{equation}\label{eq:redsetO}
\W(v,\nu)
:=\{\nonlinearity^\a(v;0,\nu)<\infty\}.
\end{equation}
Now if Theorem \ref{thm:semilinear_blow_up_Serrin}\eqref{it:blow_up_stochastic_semilinear} holds with $(u_{0,n},f_n, g_n)$, then by \eqref{eq:taunTmunsigma} and \eqref{eq:taunTmunsigma2}
\begin{equation*}
\begin{aligned}
\P\big(\{\sigma<T\}\cap\W(u,\sigma)\big)  & = \lim_{n\to \infty}\P\big(\{\sigma<T\}\cap\W(u,\sigma)\cap \{\tau_n = T\}\cap \Gamma_n \big)
\\ & = \lim_{n\to \infty}\,\P\big(\{\sigma_n<T\} \cap \W(u_n,\sigma_n) \cap \{\tau_n = T\}\cap \Gamma_n \big)
\\ & \leq \liminf_{n\to \infty}\,\P\big(\{\sigma_n<T\} \cap \W(u_n,\sigma_n)\big)=0.
\end{aligned}
\end{equation*}
The last sentence follows from \eqref{eq:taunTmunsigma}, \eqref{eq:taunTmunsigma2} and $\P(\{\tau_n=T\}\cap \Gamma_n) \to 1$.
\end{proof}

The proofs of the blow-up criteria are given in Section \ref{s:proofs_blow_up_criteria}:
\begin{itemize}
\item Subsection \ref{ss:proof_blow_up_criterion}: Theorem \ref{thm:semilinear_blow_up_Serrin}\eqref{it:blow_up_stochastic_semilinear} and Proposition \ref{prop:predictability};
\item Subsection \ref{ss:thmblowuphardespart}: Theorem \ref{t:blow_up_criterion}\eqref{it:blow_up_norm_general_case_F_c_G_c}-\eqref{it:blow_up_non_critical_Xap} and Theorem \ref{thm:semilinear_blow_up_Serrin}\eqref{it:blow_up_semilinear noncritical_stochastic};
\item Subsection \ref{ss:thmblowuphardespart_Serrin}: Theorems \ref{t:blow_up_criterion}\eqref{it:blow_up_norm_general_case_quasilinear_Xap_Pruss}, \ref{thm:semilinear_blow_up_Serrin}\eqref{it:blow_up_semilinear_serrin_Pruss_modified} and Theorem \ref{thm:semilinear_blow_up_Serrin_refined}.
\end{itemize}
Blow-up criteria involving the space $\X$ (see \eqref{eq:def_X_space} below) will be given in Remarks \ref{r:blow_up_semilinear_X_norm} and \ref{r:blow_up_quasilinear_X_norm} below.

\subsection{Global existence}\label{ss:globalgeneral}
In this section we demonstrate how Theorem \ref{thm:semilinear_blow_up_Serrin}  can be used to prove global existence for an equation, where $F$ and $G$ satisfy a certain linear growth condition.

Definitions \ref{def:solution1} and \ref{def:solution2} can be extended to the half line case. Indeed, in Definition \ref{def:solution1} one can just take $T=\infty$ and replace $[s,T]$, $L^p(s,\sigma)$ and $C([s, \sigma])$ by $[s, \infty)$, $L^p_{\rm loc}([s,\sigma))$ and $C([s, \sigma]\cap [s, \infty))$, respectively. Definition \ref{def:solution2} extends verbatim to $T=\infty$.

One can check that $(u,\sigma)$ is an $L^p_{\a}$-(maximal) local solution to \eqref{eq:QSEE} on $[ s,\infty)$ if for each $T<\infty$, $(u|_{\ll s,\sigma\wedge T\rro},\sigma\wedge T)$ is an $L^p_{\a}$-(maximal) local solution  to \eqref{eq:QSEE}  on $[s,T]$. As  before an $L^p_{\a}$-maximal local solution on $[s,\infty)$ is unique. Conversely, one can construct $L^p_{\a}$-maximal local solutions on $[s,\infty)$ from the ones on finite time intervals. Indeed,
suppose that an $L^p_{\a}$-maximal local solution $(u^T, \sigma^T)$ exists on $[s,T]$ for every $T\in (s,\infty)$. Then by maximality (see Definition \ref{def:solution2})  $\sigma^T=\sigma^{S}\wedge T$ a.s.\ and $u^T=u^{S}$ a.e.\ on $\ll s,\sigma^T\rro$ for $s<T\leq S$. Therefore, letting $u = u^T$ on $\ll s,\sigma_T\rro$ and $\sigma := \lim_{T\to \infty} \sigma^T$, one has that $u$ is an $L^p_{\a}$-maximal local solution on $[s,\infty)$. In particular, an $L^p_{\a}$-maximal local solution on $[s,\infty)$ exists if the conditions of Theorem \ref{t:local_s} hold for all $T\in (s,\infty)$. Finally we mention that if for each $T\in (s,\infty)$, $(\sigma_n^T)_{n\geq 1}$ is a localizing sequence for $(u^T,\sigma^T)$, then letting
\[\sigma_n = \sup_{m\in \{1, \ldots, n\}}\sigma^m_n, \ \ \ n\geq 1,\]
we obtain a localizing sequence $(\sigma_n)_{n\geq 1}$ for $(u,\sigma)$.

The following roadmap can be used to prove global well-posedness and regularity.

\begin{roadmap}[Proving global existence and regularity]\label{roadcomplete}
\
\begin{enumerate}[{\rm(a)}]
\item\label{it:roadmaploc} Prove local well-posedness with Theorem \ref{t:local_s};
\item\label{it:roadmapreg} obtain instantaneous regularization from Theorem \ref{t:regularization_z} and Corollary \ref{cor:regularization_X_0_X_1} using as a first step Proposition \ref{prop:adding_weights} in the case $\a=0$;
\item\label{it:roadmapred} reduce the global existence proof to data $(u_0,f,g)$ which are uniformly bounded in $\Omega$ (see Proposition \ref{prop:redblowbounded});
\item\label{it:roadmapenergy} prove an energy estimate for a certain norm $\|u\|_{Z(s,\sigma\wedge T)}$ by applying the equation and/or It\^o's formula. In this part, the regularization proven in \eqref{it:roadmapreg} can be used to simplify and/or obtain the estimate;
\item\label{it:roadmapthmappl} combine the energy estimate with Theorem \ref{t:blow_up_criterion}, \ref{thm:semilinear_blow_up_Serrin} or \ref{thm:semilinear_blow_up_Serrin_refined} to prove $\sigma\geq T$ a.s. possibly under restrictions on the integrability parameters and weights;
\item\label{it:roadmapregred} use the instantaneous regularization phenomena of Theorem \ref{t:regularization_z} and Corollary \ref{cor:regularization_X_0_X_1} to reduce to the previous case.
\end{enumerate}
Some steps of this roadmap can be skipped in certain situations. But the steps \eqref{it:roadmaploc}, \eqref{it:roadmapenergy}, \eqref{it:roadmapthmappl} seem essential in all cases. Furthermore, we  mention that in \eqref{it:roadmapreg} and \eqref{it:roadmapregred} the use of weights is essential.
\end{roadmap}

To illustrate the above roadmap concretely, we will now prove global existence of \eqref{eq:QSEE} in the semilinear setting under linear growth assumptions on $F$ and $G$. Of course the linear growth assumptions fail to hold for many of the interesting SPDEs. So this result should only be seen as an illustration and test case. For more advanced applications where the roadmap is followed, we refer to Section \ref{s:1D_problem} and \cite{AV20_NS} on stochastic Navier-Stokes equation with transport noise.

\begin{theorem}[Global well-posedness under linear growth conditions]\label{thm:globallinear}
Let Hypothesis \hyperref[H:hip]{$\Hip$} be satisfied for all $T\in (s, \infty)$,
where we suppose that $(A(t,x),B(t,x)) = (\bar{A}(t), \bar{B}(t))$ does not depend on $x$ and
\begin{equation}
\label{eq:SMRsemilinearglobal}
(\bar{A}(\cdot),\bar{B}(\cdot))\in \MRtas \ \ \text{for all $T\in (s,\infty)$}.
\end{equation}
Assume that Assumption \ref{H_a_stochnew} holds for $\ell=\a$ and all $T\in (s,\infty)$ and Assumption \ref{ass:FG_a_zero} holds for all $T\in (s,\infty)$. Suppose that for every $\varepsilon>0$ there exist a constant $L_{\varepsilon}>0$ such that for all $t\in (s, \infty)$, $\omega\in \Omega$ and $x\in X_1$,
\begin{align}\label{eq:lineargrowth}
\|F(t,\om,x)\|_{X_{0}} + \|G(t,\om,x)\|_{\gamma(H,X_{1/2})}&\leq L_{\varepsilon}(1+ \|x\|_{X_{0}})+\varepsilon \|x\|_{X_1}.
\end{align}
Then for each $u_s\in L^0_{\F_s}(\O;\Xap)$ there is a unique $L^p_{\a}$-global solution $u$ to \eqref{eq:QSEE} s.t.
\begin{itemize}
\item If $p>2$ and $\a\in [0,\frac{p}{2}-1)$, then for all $\theta\in [0,\frac{1}{2})$,
 \[u\in H^{\theta,p}_{\rm loc}([s,\infty),w_{\a}^s;X_{1-\theta}) \cap C([s,\infty);\Xap)  \ \ \ a.s.\]
Moreover, $u$ instantaneously regularizes to $u\in C((s,\infty);\Xp)$ a.s.
\item If $p=2$ and $\a=0$, then
 \[u\in L^2_{\rm loc}(s,\infty;X_{1}) \cap C([s,\infty);X_{1/2})  \ \ \ a.s.\]
 \end{itemize}
Moreover, if additionally $u_s\in L^p_{\F_s}(\Omega;\Xap)$, $f\in L^p_{\Progress}((s,T)\times \Omega,w_{\a}^s;X_0)$ and $g\in L^p_{\Progress}((s,T)\times \Omega,w_{\a}^s;\g(H,X_{1/2}))$, then for every $T\in (s, \infty)$ and every $\theta\in [0,1/2)$ there exists a constant $C_{\theta,T}$ such that
\begin{equation}\label{eq:lineargrowthapriori}
\begin{aligned}
\|v \|_{L^p(\Omega;E_{\theta,p})} \leq C_{\theta,T}(&1 +\|u_s\|_{L^p(\Omega;\Xap)}  \\ &   + \|f\|_{L^p(\llo s,T\rro,w_{\a}^s;X_0)}  +\|g\|_{L^p(\llo s,T\rro,w_{\a}^s;\g(H,X_{1/2}))}),
\end{aligned}
\end{equation}
where we set, for $s'>s$,
\begin{align*}
E_{\theta,p} &\in \{H^{\theta,p}(s,T,w_{\a}^s;X_{1-\theta}), C([s,T];\Xap), C([s',T];\Xp)\} \ \ \ \text{if} \ \  p\in (2, \infty),
\\ E_{\theta,2} &\in  \{L^2(s,T;X_1), C([s,T];X_{1/2})\}.
\end{align*}
\end{theorem}
By standard interpolation inequalities we can replace $\|x\|_{X_0}$ by $\|x\|_{X_{1-\delta}}$ with arbitrary $\delta\in(0,1)$ in \eqref{eq:lineargrowth}. From the proof below one can actually see that it is enough to have \eqref{eq:lineargrowth} for some fixed small $\varepsilon>0$.
\begin{proof}[Proof of Theorem \ref{thm:globallinear}]
We may suppose that $s=0$. We will only prove the result for $p>2$ as the other case is simpler.

By Theorem \ref{t:local_s} and the above discussion there exists local solution $(u, \sigma)$ of \eqref{eq:QSEE} on $[0,\infty)$ with the required properties on $[0,\sigma)$ and thus we only need to show that $\sigma = \infty$ a.s. Replacing $\sigma$ by $\sigma\wedge T$ it suffices to show that $\P(\sigma<T)=0$ for all $T\in (0,\infty)$. Moreover, by Proposition \ref{prop:redblowbounded} it suffices to consider the case of $L^p(\Omega)$-integrable data $u_0$, $f$ and $g$.
To prove $\sigma = T$ a.s., we will apply Theorem \ref{thm:semilinear_blow_up_Serrin}\eqref{it:blow_up_stochastic_semilinear} (but also the slightly simpler Lemma \ref{l:lemma_correction_theorem} below suffices).
In order to do so we will first derive a suitable energy estimate.

Let $(\sigma_n)_{n\geq 1}$ be a localizing sequence for $(u,\sigma)$. Moreover, for each $n\geq 1$ define a stopping time by
\[\tau_n = \inf\{t\in [0,\sigma): \|u\|_{L^p(0,t,w_{\a};X_1)}\geq n\}\wedge \sigma_n,\]
where we set $\inf\emptyset =\sigma$. Then $u|_{\ll 0,\tau_n\rr}$ is a strong solution of \eqref{eq:QSEE} on $\ll 0,\tau_n\rr$.

Set $\wt{f}_n = \one_{[0,\tau_n]} (f+ F(\cdot, u))$ and $\wt{g}_n = \one_{[0,\tau_n]} (g+ G(\cdot, u))$. Then by \eqref{eq:lineargrowth}, $\wt{f}_n\in L^{p}(\llo 0,T\rro,w_{\a};X_0)$ and $\wt{g}_n\in
L^{p}(\llo 0,T\rro,w_{\a};\g(H,X_{1/2}))$. By \eqref{eq:SMRsemilinearglobal} for the strong solution $v$ to
\begin{align*}
\label{eq:diffAB_s}
\begin{cases}
dv(t) +A(t)v(t)dt=\wt{f}(t) dt+ (B(t)v(t)+\wt{g}(t))dW_H(t),\quad t\in \ll 0,T\rr,\\
u(0)=u_{0},
\end{cases}
\end{align*}
we have $u=v$ on $\ll 0,\tau_n\rr$, and by Proposition \ref{prop:start_at_s},
\begin{align*}
\|v\|_{L^p(\llo 0,T\rro,w_{\a};X_{1})}
\leq
C(\|u_{0}\|_{L^p(\O;\Xap)}
&+ \|\wt{f}\|_{L^p(\llo 0,T\rro,w_{\a};X_{0})}\\
&+\|\wt{g}\|_{L^p(\llo 0,T\rro,w_{\a};\g(H,X_{1/2}))}).
\end{align*}
By the linear growth assumption \eqref{eq:lineargrowth}, and $\|\one_{[0,\tau_n]} u\|_{X_i} \leq \|v\|_{X_i}$ we obtain
\[\|\wt{f}\|_{X_0} + \|\wt{g}\|_{\gamma(H,X_{1/2})}\leq \|f\|_{X_0} + \|g\|_{\gamma(H,X_{1/2})} + L_{\varepsilon} (1+\|v\|_{X_0}) + \varepsilon \|v\|_{X_1}.\]
Choose $\varepsilon = \frac{1}{2C}$ and set
\[K = \|u_{0}\|_{L^p(\O;\Xap)}+ \|f\|_{L^p(\llo 0,T\rro,w_{\a};X_{0})}+\|g\|_{L^p(\llo 0,T\rro,w_{\a};\g(H,X_{1/2}))}+L_{\varepsilon},\]
Then combining the above we obtain
\[\|v\|_{L^p(\llo 0,T\rro,w_{\a};X_{1})} \leq C K
 + CL_{\varepsilon}\|v\|_{L^p(\llo 0,T\rro,w_{\a};X_{0})}
+\frac12\|v\|_{L^p(\llo 0,T\rro,w_{\a};X_{1})},\]
and hence
\begin{align}\label{eq:vestlingrowth2}
\|v\|_{L^p(\llo 0,T\rro,w_{\a};X_{1})} \leq 2 C K
+ 2C L_{\varepsilon}\|v\|_{L^p(\llo 0,T\rro,w_{\a};X_{0})}.
\end{align}

Similarly, by Proposition \ref{prop:start_at_s}\eqref{it:start_at_s1} there exists a $\wt{C}>0$ independent of $T$ such that
\begin{align*}
\|v\|_{L^p(\O;C([0,T];\Xap))}
&\leq
\wt{C}(\|u_{0}\|_{L^p(\O;\Xap)}+ \|\wt{f}\|_{L^p(\llo 0,T\rro,w_{\a};X_{0})}\\
&\qquad \qquad +\|\wt{g}\|_{L^p(\llo 0,T\rro,w_{\a};\g(H,X_{1/2}))})
\\ & \leq \wt{C} K + \wt{C}L_{\varepsilon} \|v\|_{L^p(\llo 0,T\rro,w_{\a};X_{0})}  +  \frac12 \wt{C} \|v\|_{L^p(\llo 0,T\rro,w_{\a};X_{1})}
\\ & \stackrel{\eqref{eq:vestlingrowth2}}{\leq} \wh{C} K + \wh{C} L_{\varepsilon} \|v\|_{L^p(\llo 0,T\rro,w_{\a};X_{0})},
\end{align*}
where $\wh{C} = \wt{C}(1+C)$. Since $T>0$ was arbitrary letting $y(t) = \|v\|_{L^p(\O;C([0,t];\Xap))}^p$ it follows that for all $t\in (0,T]$,
\begin{align*}
y(t) \leq   2^{p-1} \wh{C}^p K^p + 2^{p-1} \wh{C}^p L_{\varepsilon}^p \int_0^t y(s) \, ds.
\end{align*}
Thus, Gronwall's inequality implies $y(t) \leq 2^{p-1} \wh{C}^p K^p e^{2^{p-1}
\wh{C}^p L_{\varepsilon}^pt}$.
This gives
\[\|v\|_{L^p(\O;C([0,T];\Xap))}\leq 2 \wh{C} K e^{\frac{1}p2^{p-1} \wh{C}^p L_{\varepsilon}^pt}:=K \overline{C}_{T}.\]
Therefore, by $\Xap\hookrightarrow X_0$ with embedding constant $M$, from \eqref{eq:vestlingrowth2} we obtain that
\[
\|v\|_{L^p(\llo 0,T\rro,w_{\a};X_{1})} \leq 2 C K + 2CL_{\varepsilon} M K \overline{C}_{T} T^{1/p}
\]
Since $u = v$ on $\ll 0,\tau_n\rr$ letting $n\to \infty$ the following energy estimate follows
\begin{align}\label{eq:vestlingrowth3}
\|u\|_{L^p(\llo 0,\sigma\rro,w_{\a};X_{1})} \leq 2 C K + 2CL_{\varepsilon} M K \overline{C}_{T} T^{1/p}
\end{align}

From the estimate \eqref{eq:vestlingrowth3} and \eqref{eq:lineargrowth} we obtain that for a suitable $\wt{C}_T$ \[\Big\|\|F(\cdot,u)\|_{X_0}+
\|G(\cdot,u)\|_{\g(H,X_{1/2})}\Big\|_{L^p(\llo 0,\sigma\rro,w_{\a})}\leq \wt{C}_T K <\infty.\]
Therefore, Theorem \ref{thm:semilinear_blow_up_Serrin}\eqref{it:blow_up_stochastic_semilinear}  implies $\sigma = T$ a.s. Furthermore, \eqref{eq:lineargrowthapriori} follows from the latter estimate, \eqref{eq:SMRsemilinearglobal} and Proposition \ref{prop:start_at_s}.
\end{proof}

\section{Proofs of Theorems \ref{t:blow_up_criterion}, \ref{thm:semilinear_blow_up_Serrin} and \ref{thm:semilinear_blow_up_Serrin_refined}}
\label{s:proofs_blow_up_criteria}

In this section we have collected the proofs of the blow-up criteria stated in Section \ref{s:blow_up}.
The proofs are technical and require some preparations. In Section \ref{ss:QSEE_tau} we will first obtain a local existence result for \eqref{eq:QSEE} starting at a random initial time. It plays a key role in the proof of Lemma \ref{l:lemma_correction_theorem}, which is a weaker version of Theorem \ref{thm:semilinear_blow_up_Serrin}\eqref{it:blow_up_norm_general_case_F_c_G_c}, but it is a central step in proving all the blow-up criteria.

\subsection{Local existence when starting at a random time}
\label{ss:QSEE_tau}
In this subsection, for a stopping time $\tau$, we consider
\begin{equation}
\label{eq:QSEE_tau}
\begin{cases}
du +A(\cdot,u)dt =(F(\cdot,u)+f )dt +(B(\cdot,u)u+G(\cdot,u)+g)dW_{H},\\
u(\tau)=u_{\tau};
\end{cases}
\end{equation}
on $\ll \tau,T\rr$. To define an $L^p_{\a}$-local solution to \eqref{eq:QSEE_tau} on $\ll \tau,T\rr$, one can just replace the initial time $s$ by $\tau$ in Definition \ref{def:solution2}.

The following is the natural extension of the local existence part of Theorem \ref{t:local_s} to the case of random initial times \eqref{eq:QSEE_tau}. It will be used to prove the blow-up results of Theorem \ref{t:blow_up_criterion}.
\begin{proposition}[Local existence starting at a random initial time]
\label{prop:local_sigma}
Let Hypothesis \hyperref[H:hip]{$\Hip$} be satisfied. Let $\tau$ be a stopping time with values in $[s,T]$, where we assume  that $\tau$ takes values in a finite set if $\a>0$. Assume that $u_{\tau}\in L^{\infty}_{\F_{\tau }}(\O;\Xap)$ and $(A(\cdot,u_{\tau})|_{\ll \tau,T\rr},B(\cdot,u_{\tau})|_{\ll \tau,T\rr})\in \MRtatau$. Then there exists a unique $L^p_{\a}$-local solution $(u,\sigma)$ to \eqref{eq:QSEE_tau}  on $\ll \tau,T\rr$ such that $\sigma>\tau$ a.s.\ on the set $\{\tau<T\}$.
\end{proposition}
The analogues assertions of Theorem \ref{t:local_s} for \eqref{eq:QSEE_tau} hold as well, but since these results will not be needed we do not consider this.

To prove Proposition \ref{prop:local_sigma}, we use a variation of the method in \cite[Theorem 4.6]{AV19_QSEE_1}.
As in \cite{AV19_QSEE_1}, we introduce the space $\X$ which allows to control the nonlinearity $F_c,G_c$.
If \eqref{eq:HypCritical} or \eqref{eq:HypCriticalG} holds with strict inequality for some $j\in\{1,\dots,m_F+m_G\}$ we denote by $\rhos_j$ the unique positive number such that \eqref{eq:HypCritical} or \eqref{eq:HypCriticalG} holds with equality if $\rho_j$ is replaced by $\rhos_j$. More precisely, we set
\begin{equation}
\label{eq:rhostar}
\rhos_j:=\frac{1-\beta_j}{\varphi_j-1+(1+\a)/p},\qquad j\in\{1,\dots,m_F+m_G\}.
\end{equation}
To introduce the space $\X$, set
\begin{equation}
\label{eq:rr'}
\frac{1}{\p_{j}'}:=\frac{\rhos_j(\varphi_j-1+(1+\a)/p)}{(1+\a)/p}<1, \qquad \frac{1}{\p_j}:=\frac{\beta_j -1 + (1+\a)/p}{(1+\a)/p}<1.
\end{equation}
Note that $\frac{1}{r_j}+\frac{1}{r_j'}=1$ by \eqref{eq:rhostar}.
Finally, for each $0\leq a<b\leq \infty$, we set
\begin{equation}
\label{eq:def_X_space}
\X({a,b}):=\Big(\bigcap_{j=1}^{m_F+m_G} L^{p\p_j}(a,b,w_{\a}^a;X_{\beta_j})\Big) \cap
\Big(\bigcap_{j=1}^{m_F+m_G}  L^{\rhos_j p \p_j'}(a,b,w_{\a}^a;X_{\varphi_j})\Big).
\end{equation}
Setting $\X(T):=\X(0,T)$, for all $T>0$, our notation is consistent with \cite{AV19_QSEE_1}.

The following result is proven in \cite[Lemma 4.10]{AV19_QSEE_1}.
\begin{lemma}
\label{l:embeddings}
Let Assumption \ref{ass:X}, \emph{\ref{HFcritical}} and \emph{\ref{HGcritical}} be satisfied.  Let $0<a<b\leq T<\infty$ and let $\X(a,b)$ be as in \eqref{eq:def_X_space}. Then the following hold:
\begin{enumerate}[{\rm(1)}]
\item\label{it:emb_p_grather_2} If $p>2$, $\a\in [0,\frac{p}{2}-1)$, and $\A\in \{\hz,H\}$, then for any
$\delta\in (\frac{1+\a}{p},\frac{1}{2})$,
$$
\A^{\d,p}(a,b;w_{\a}^a;X_{1-\d})\cap L^p(a,b,w_{\a}^a;X_1)
\hookrightarrow \X(a,b);
$$
\item\label{it:emb_p_2} if $p=2$ and $\a=0$, then
$
C([a,b];X_{1/2})\cap L^2(a,b;X_1)
\hookrightarrow \X(a,b).
$
\end{enumerate}
Finally, if in \eqref{it:emb_p_grather_2} $\A=\hz$, then the constants in the embeddings \eqref{it:emb_p_grather_2}-\eqref{it:emb_p_2} can be chosen to be independent of $b-a>0$.
\end{lemma}

The following lemma contains the key estimate for the proof of Proposition \ref{prop:local_sigma}. It is not immediate from Lemma \ref{l:embeddings}, since we require uniformity in the constants if $|b-a|$ tends to zero.
\begin{lemma}
\label{l:estimate_Sol_op_Y}
Let Assumption \ref{ass:X}, \emph{\ref{HFcritical}} and \emph{\ref{HGcritical}} be satisfied.
Let $0\leq a<b<T<\infty$ and let $\sigma$ be a stopping time with values in $[a,b]$, where we assume that $\sigma$ takes values in a finite set if $\a>0$.
Let $(A,B)\in \MRtasigma$. Let either $p>2$, $\a\in [0,\frac{p}{2}-1)$ and $\delta\in (\frac{1+\a}{p},\frac{1}{2})$ or $p=2$, $\a=0$ and $\delta\in (0,\frac{1}{2})$ be fixed. Set
\[K_{(A,B)}:=\max\{K^{\det,\delta,p,\a}_{(A,B)},K^{\stoc,\delta,p,\a}_{(A,B)}\}.\]
Then there exists a constant $C>0$ independent of $a,b,\sigma$ such that for each $(u_{\sigma},f,g)$ which belongs to \eqref{eq:u_sigma_f_g_max_reg},
\begin{equation}
\label{eq:estimate_Sol_X_spaces}
\begin{aligned}
&\|\Sol_{\sigma}(u_{\sigma},f,g)\|_{ L^p(\O;\X(\sigma,b)\cap L^p(\sigma,b,w_\a^{\sigma};X_1)\cap C([\sigma,b];\Xap))}\\
&\qquad\qquad \leq C(1+K_{(A,B)})(\|u_{\sigma}\|_{L^p(\O;\Xap)}+\|f\|_{L^p(\llo\sigma,b\rro,w_{\a}^{\sigma};X_0)}\\
&\qquad\qquad\qquad\qquad\qquad\qquad\qquad\qquad\qquad+
\|g\|_{L^p(\llo\sigma,b\rro,w_{\a}^{\sigma};\g(H,X_{1/2}))}),
\end{aligned}
\end{equation}
where $\Sol_{\sigma}:=\Sol_{\sigma,(A,B)}$ is the solution operator associated to $(A,B)$.
\end{lemma}

\begin{proof}
We only consider the case $p>2$. From Proposition \ref{prop:start_at_s} we see that the constants in the estimates for the $L^p(\ll \sigma,b\rr ,w_\a^{\sigma};X_1)$ and $L^p(\Omega;C([\sigma,b];\Xap))$ norm of $u$ do not depend on $a,b,\sigma$. It remains to estimate the $L^p(\O;\X(\sigma,b))$-norm of $u:=\Sol_{\sigma}(u_{\sigma},f,g)$, and for this we will reduce to the case with zero initial data. As in the proof of Proposition \ref{prop:start_at_s} we can assume that $u_{\sigma}=\sum_{j=1}^N \one_{\U_j} x_j$ is simple and set $v_1=\sum_{j=1}^N \one_{\U_j}h_j(\cdot-\sigma)$. The proof of Proposition \ref{prop:start_at_s} shows that $u = v_1+v_2$ and that on $\U_j$,
\begin{equation*}
\begin{aligned}
\|v_1\|_{\X(\sigma,b)}
&\leq
\|t\mapsto h_j(t-\sigma)\|_{\X(\sigma,\infty)}\\
&=
\|h_j\|_{\X(0,\infty)}
\lesssim_{\wt{A}} \|h_j\|_{W^{1,p}(\R_+,w_{\a};X_0)\cap
L^p(\R_+,w_{\a};X_1)}
\lesssim_{\wt{A}} \|u_{\sigma}\|_{\Xap};
\end{aligned}
\end{equation*}
where we used Lemmas \ref{l:mixed_derivative} and \ref{l:embeddings}
and the same argument as the proof of Proposition \ref{prop:start_at_s}. Taking $L^p(\O)$-norms we obtain $\|v_1\|_{L^p(\O;\X(\sigma,b))}
\lesssim \|u_{\sigma}\|_{L^p(\O;\Xap)}$,
where the implicit constant does not depend on $\sigma,a,b$.
To estimate $v_2$, choosing any $\delta\in (\frac{1+\a}{p},\frac{1}{2})$, by Lemma \ref{l:embeddings} we find
\begin{align*}
\|v_2\|_{L^p(\O;\X(\sigma,b))}
&\lesssim
\|v_2\|_{L^p(\O;\hz^{\delta,p}(\sigma,b,w_{\a}^{\sigma};X_{1-\delta})\cap L^p(\sigma,b,w_{\a}^{\sigma};X_1))}\\
&\lesssim  2\wt{C}_0 K_{(A,B)} \|u_{\sigma}\|_{L^p(\O;\Xap)}
+ K_{(A,B)} \|f\|_{L^p(\llo\sigma,b\rro,w_{\a}^{\sigma};X_0)}
\\ & \qquad +K_{(A,B)} \|g\|_{L^p(\llo\sigma,b\rro,w_{\a}^{\sigma};\g(H,X_{1/2}))},
\end{align*}
where we used \eqref{eq:v2estHtheta}. Combining the estimates for $v_1$ and $v_2$, the result follows.
\end{proof}
After these preparations we can prove Proposition \ref{prop:local_sigma}.
\begin{proof}[Proof of Proposition \ref{prop:local_sigma}]
The proof is a variation of the argument in Step 1 and 4 in the proof of \cite[Theorem 4.6]{AV19_QSEE_1}. For each $\lambda\in (0,1)$, we look at the following truncation of \eqref{eq:QSEE_tau}:
\begin{equation*}
\begin{cases}
du +A(\cdot,u_{\tau})u dt= (\wt{F}_{\lambda}(u)+\wt{f})dt +
(B(\cdot,u_{\tau})u + \wt{G}_{\lambda}(u)+\wt{g})dW_{H},\\
u(\tau)=u_{\tau},
\end{cases}
\end{equation*}
on $\ll \tau ,T\rr$, where
\begin{equation}
\label{eq:nonlinearity_truncation_tau}
\begin{aligned}
\wt{F}_{\lambda}(u)&:=
\Theta_{\lambda}(\cdot,u_{\tau},u)[F_c(\cdot,u)-F_c(\cdot,0)]\\
&\quad +\Psi_{\lambda}(\cdot,u_{\tau},u)[(A(\cdot,u_{\tau})u-A(\cdot,u)u)+F_{\Tr}(\cdot,u)-F_{\Tr}(\cdot,u_{\tau})],\\
\wt{G}_{\lambda}(u)&:=
\Theta_{\lambda}(\cdot,u_{\tau},u)[G_c(\cdot,u)-G_c(\cdot,0)]\\
&\quad
+\Psi_{\lambda}(\cdot,u_{\tau},u)[(B(\cdot,u_{\tau})u-B(\cdot,u)u)+G_{\Tr}(\cdot,u)-G_{\Tr}(\cdot,u_{\tau})],\\
\wt{f}&:=f+ F_c(\cdot,0)+F_{\Tr}(\cdot,u_{\tau}), \qquad
\wt{g}:=g+ G_c(\cdot,0)+G_{\Tr}(\cdot,u_{\tau}).
\end{aligned}
\end{equation}
Here, for $\lambda>0$,
$\xi_{\lambda}:=\xi(\cdot/\lambda)$, $\xi\in W^{1,\infty}([0,\infty))$, $\xi=1$ on $[0,1]$, $\xi=0$ on $[1,\infty)$ and linear on $[1,2]$, we have set, a.s.\ for all $t\in [\tau,T]$,
\begin{equation}
\label{eq:truncation_Theta_Psi}
\begin{aligned}
\Theta_{\lambda}(t,u_{\tau},u)&
:=\xi_{\lambda}\Big(\|u\|_{\X(\tau,t)}+\sup_{s\in[\tau,t]} \|u(s)-u_{\tau}\|_{\Xap}\Big),\\
\Psi_{\lambda}(t,u_{\tau},u)&
:=\xi_{\lambda}\Big(\|u\|_{L^p(\tau,t,w_{\a}^{\tau};X_1)}+\sup_{s\in[\tau,t]} \|u(s)-u_{\tau}\|_{\Xap}\Big).
\end{aligned}
\end{equation}
Let $\Sol_{\tau}:=\Sol_{(A(\cdot,u_{\tau}),B(\cdot,u_{\tau}))}$ be the solution operator associated to the couple $(A(\cdot,u_{\tau}),B(\cdot,u_{\tau}))$, see \eqref{eq:soloperatorR2}. For any $T'\in (0,T]$ we introduce the Banach space
$$
\z_{T'}:=L_{\Progress}^p(\O;C([\tau,\mu_{T'}];\Xap)\cap\X(\tau,\mu_{T'})
\cap L^p(\tau,\mu_{T'},w_{\a}^{\tau};X_1)),
$$
where $\mu_{T'}:=T\wedge(\tau +T')$. In the following, for notational simplicity, we write $\mu$ instead of $\mu_{T'}$ if no confusion seems likely. Let us consider the map $\yT$ defined as
\begin{equation}
\label{eq:contraction_map_starting_at_tau}
\yT(u)=\Sol_{\tau}(u_{\tau},\one_{\ll \tau,\mu \rr}(\wt{F}_{\lambda}(u)+\wt{f}),
\one_{\ll \tau,\mu \rr}(\wt{G}_{\lambda}(u)+\wt{g})).
\end{equation}
For the sake of clarity, we split the proof into two steps.

\emph{Step 1: $\yT$ maps $\z_{T'}$ into itself for all $T'\in (0,T]$ and $\lambda>0$. Moreover, there exists a $T^*\in (0,T]$ and $\lambda^*>0$ such that}
\begin{equation}
\label{eq:contractivity_T}
\|\yT(\vone)-\yT(\vtwo)\|_{\z_{T^*}}\leq \frac{1}{2}\|\vone-\vtwo\|_{\z_{T^*}},\quad \text{ for all }\vone,\vtwo\in \z_{T^*}.
\end{equation}
Let us note that by a translation argument and the pointwise estimates w.r.t.\ $\om\in\O$ in \cite[Lemmas 4.14 and 4.16]{AV19_QSEE_1}, one can check that for all $\vone,\vtwo\in \z_{T'}$,
\begin{align*}
\|\wt{F}_{\lambda}(\vone)\|_{L^{p}(\llo \tau,\mu\rro,w_{\a}^{\tau};X_0)}+
\|\wt{G}_{\lambda}(\vone)\|_{L^{p}(\llo \tau,\mu\rro,w_{\a}^{\tau};\g(H,X_{1/2}))}\leq C_{\lambda},
\end{align*}
\begin{align*}
\|\wt{F}_{\lambda}(\vone)-\wt{F}_{\lambda}(\vtwo)\|_{L^{p}(\llo \tau,\mu\rro,w_{\a}^{\tau};X_0)}
+\|\wt{G}_{\lambda}(\vone)-\wt{G}_{\lambda}(\vtwo)&\|_{L^{p}(\llo \tau,\mu\rro,w_{\a}^{\tau};\g(H,X_{1/2}))}
\\ & \leq L_{\lambda,T} \|\vone-\vtwo\|_{\z_{T'}}.
\end{align*}
In addition, for each $\varepsilon>0$ there exists a $\bar{\lambda}(\varepsilon)>0$ and $\bar{T}(\varepsilon)\in (0,T]$ such that
$$
L_{\lambda,T}< \varepsilon,\qquad \text{ for all }\lambda\in (0,\bar{\lambda}],\text{ and } T\in (0,\bar{T}].
$$
We will only prove \eqref{eq:contractivity_T}. The fact that $\yT$ maps $\z_{T'}$ into itself can be proved in a similar way. Let
$K$ be the least constant in \eqref{eq:estimate_Sol_X_spaces} with $(A,B)$, $\sigma$ replaced by $(A(\cdot,u_{\tau}),B(\cdot,u_{\tau}))$, $\tau$. Choose $\varepsilon^*>0$ such that $4 K L_{\lambda^*,T^*}\leq 1$ where $\lambda^*:=\bar{\lambda}(\varepsilon^*)$ and $T^*:=\bar{T}(\varepsilon^*)$. Thus, Lemma \ref{l:estimate_Sol_op_Y} and the previous choices yield
\begin{align*}
&\|\yT(\vone)-\yT(\vtwo)\|_{\z_{T^*}}\\
&\quad =
\|\Sol_{\tau}(0,\one_{\ll \tau,\mu \rr}(\wt{F}_{\lambda}(\vone) -\wt{F}_{\lambda}(\vtwo)),
\one_{\ll \tau,\mu \rr}(\wt{G}_{\lambda}(\vone)-\wt{G}_{\lambda}(\vtwo)))\|_{\z_{T^*}}\\
&\quad
\leq K\big(\|\wt{F}_{\lambda}(\vone)-\wt{F}_{\lambda}(\vtwo)\|_{L^{p}(\llo \tau,\mu\rro,w_{\a}^{\tau};X_0)}
+\|\wt{G}_{\lambda}(\vone)-\wt{G}_{\lambda}(\vtwo)\|_{L^{p}(\llo \tau,\mu\rro,w_{\a}^{\tau};\g(H,X_{1/2}))}\big)\\
&\quad
\leq \frac{1}{2}\|\vone-\vtwo\|_{\z_{T^*}}.
\end{align*}

\emph{Step 2: Conclusion}. Let $\lambda^*,T^*$ be as in Step 1. The conclusion of step 1 ensures that $\yT$ is a contraction on $\z_{T^*}$, and thus there exists a fixed point of the map $\yT$ on $\z_{T^*}$ which will be denoted by $U$. Setting
$$
\nu:=\inf\Big\{t\in [\tau,T]\,:\,\|U\|_{L^p(\tau,t,w_{\a}^{\tau};X_1)\cap \X(\tau,t)}
+\sup_{s\in [\tau,t)} \|U(t)-u_{\tau}\|>\lambda^*\Big\}.
$$
Then $\nu$ is a stopping time and $\nu>\tau$ a.s.\ on $\{\tau<T\}$. Moreover, as in Step 4 in the
proof of \cite[Theorem 4.5]{AV19_QSEE_1}, one obtains $u:=U|_{\ll \tau,\nu\rr}$ is an $L^p_{\a}$-local solution to \eqref{eq:QSEE_tau}. This follows since by \eqref{eq:truncation_Theta_Psi}, a.s.\ for all $t\in [\tau,\nu]$,
$$
\Theta_{\lambda^*}(t,u_{\tau},U)=1, \qquad \Psi_{\lambda^*}(t,u_{\tau},U)=1.
$$
By \eqref{eq:nonlinearity_truncation_tau} the latter implies, a.s.\ for all $t\in [\tau,\nu]$,
\begin{align*}
\tilde{F}_{\lambda^*}(U)
&=
A(\cdot,u_{\tau})U-A(\cdot,U) U + F_c(\cdot,U)-F_c(\cdot,0)+F_{\Tr}(\cdot,U)-F_{\Tr}(\cdot,u_{\tau}),\\
\tilde{G}_{\lambda^*}(U)
&=
B(\cdot,u_{\tau})U-B(\cdot,U)U + G_c(\cdot,U)-G_c(\cdot,0)+G_{\Tr}(\cdot,U)-
G_{\Tr}(\cdot,u_{\tau}).
\end{align*}
Thus, \eqref{eq:contraction_map_starting_at_tau} and Proposition \ref{prop:start_at_s} ensure that $(u,\nu)$ is an $L^p_{\a}$-local solution to \eqref{eq:QSEE} where $u=U|_{\ll \tau,\nu\rr}$ and $\nu>\tau$ a.s.
The uniqueness of the above $L^p_{\a}$-local solution can be proven as in Step 5 of \cite[Theorem 4.5]{AV19_QSEE_1}.
\end{proof}

\subsection{Proofs of Theorem \ref{thm:semilinear_blow_up_Serrin}\eqref{it:blow_up_stochastic_semilinear} and Proposition \ref{prop:predictability}}
\label{ss:proof_blow_up_criterion}
We begin by proving a blow-up criteria for \eqref{eq:QSEE} which will be an important intermediate step to obtain the blow-up criteria stated in Section \ref{ss:main_result_blow_up}.

\begin{lemma}[An intermediate blow-up criteria]
\label{l:lemma_correction_theorem}
Let Hypothesis \hyperref[H:hip]{$\Hip$} be satisfied and let $u_s\in L^0_{\F_s}(\O;\Xap)$. Suppose that \eqref{eq:approximating_sequence_initial_data} and \eqref{eq:stochastic_maximal_regularity_assumption_local_extended_blow_up_quasilinear} are satisfied. Assume that Assumption \ref{H_a_stochnew} holds for $\ell=0$ and Assumption \ref{ass:FG_a_zero} holds. Let $(u,\sigma)$ be the $L^p_{\a}$-maximal local solution to \eqref{eq:QSEE} and let $\nonlinearity^{\a}$ be as in \eqref{eq:L_p_norm_nonlinearity2}. Then
$$
\P\Big(\sigma<T,\,
\lim_{t\uparrow \sigma} u(t)\text{ exists in }\Xp ,
\,\|u\|_{L^p(s,\sigma,w_{\a}^s;X_1)}+
\nonlinearity^{\a}(u;s,\sigma)<\infty\Big)=0.
$$
\end{lemma}

The above blow-up criteria is weaker than Theorem \ref{t:blow_up_criterion}\eqref{it:blow_up_norm_general_case_F_c_G_c} since $\Xp\embed \Xap$. Moreover, as $u$ is bounded with values in $\Xap$, boundedness of $\nonlinearity^{\a}(u;s,\sigma)<\infty$ if and only if $\nonlinearity^{\a}_c(u;s,\sigma)<\infty$ (see \eqref{eq:L_p_norm_nonlinearity}). Note that the lemma does not require Assumption \ref{H_a_stochnew} for $\ell=\a$. Lemma \ref{l:lemma_correction_theorem} will be used to prove Theorems \ref{t:blow_up_criterion}\eqref{it:blow_up_norm_general_case_F_c_G_c} and \ref{thm:semilinear_blow_up_Serrin}\eqref{it:blow_up_stochastic_semilinear}.

To prove Lemma \ref{l:lemma_correction_theorem}, we argue by contradiction. If the probability would be nonzero, we will obtain a new equation on a set of positive probability at a random initial time. Using Proposition \ref{prop:local_sigma} we extend the solution which gives a contradiction with maximality.

\begin{proof}
By a translation argument we may assume that $s=0$. Moreover, we will only consider $p>2$, since the other case is simpler.

{\em Step 1: Let $M\in \N$, $\eta>0$ and $r\in [\eta,T]$. If $\mu$ is a stopping time with values in $[r,T]$, $u_{\mu} \in L^{\infty}_{\F_{\mu}}(\O;\Xp)$ and $u_{r}\in L^{\infty}_{\F_r}(\O;\Xp)$ are such that $u_{\mu},u_r\in \B_{L^{\infty}(\O;\Xap)}(M)$ and
\begin{equation*}
\|u_{\mu}-u_{r}\|_{L^{\infty}(\O;\Xap)}\leq \frac{1}{4K_{M,\eta} L_M},
\end{equation*}
then one has $(A(\cdot,u_{\mu}),B(\cdot,u_{\mu}))\in \MRtmu$, where $L_M$ and $K_{M,\eta} = K_{M,\eta}^0$ are as in \emph{\ref{HAmeasur}} and Assumption \ref{H_a_stochnew} with $\ell=0$, respectively}. To prove the result, we use the perturbation result of Corollary \ref{cor:Pert2}. Note that by \ref{HAmeasur}, for $x\in X_1$, for all $t\in (r,T)$ and a.s.
$$
\|(A(t,u_r)-A(t,u_{\mu}))x\|_{X_0}\leq L_M \|u_{\mu}-u_r\|_{\Xap}\|x\|_{X_1}\leq \frac{1}{4K_{M,\eta}}\|x\|_{X_1}.
$$
Similarly, $\|(B(t,u_r)-B(t,u_{\mu}))x\|_{\g(H,X_{1/2})}\leq1/(4K_{M,\eta})\|x\|_{X_1}$ for all $t\in (r,T)$ a.s. Assumption \ref{H_a_stochnew} for $\ell=0$ and Proposition \ref{prop:change_initial_time} imply that $(A(\cdot,u_{r})|_{\ll \mu,T\rr},B(\cdot,u_{r})|_{\ll \mu,T\rr})$ is in $\MRtmu$ with
\[\max\{K_{(A(\cdot,u_{r})|_{\ll \mu,T\rr},B(\cdot,u_{r})|_{\ll \mu,T\rr})}^{\deter,0,p,0},K_{(A(\cdot,u_{r})|_{\ll \mu,T\rr},B(\cdot,u_{r})|_{\ll \mu,T\rr})}^{\stoc,0, p,0}\}\leq K_{M,\eta},\]
The claim follows from the previous estimates and Corollary \ref{cor:Pert2}
with
\[\delta_{A,B}\leq K_{M,\eta} \frac{1}{4K_{M,\eta}} + K_{M,\eta}\frac{1}{4K_{M,\eta}}  = \frac12.\]

{\em Step 2: Conclusion}. By contradiction assume that $\P(\W)>0$, where
$$
\W:=\Big\{\sigma<T,\,\lim_{t\uparrow \sigma} u(t)\text{ exists in }\Xp ,\,\|u\|_{L^p(0,\sigma,w_{\a};X_1)}+
\nonlinearity^{\a}(u;0,\sigma)<\infty\Big\}\in \F_{\sigma}.
$$
Since $\sigma>0$ a.s.\ there exists an $\eta>0$ such that $\P(\W\cap\{\sigma>\eta\})>0$. Moreover, by Egorov's theorem, there exist $\V\subseteq \W\cap \{\sigma>\eta\}$ and $M\in \N$ such that $\V\in \F_{\sigma}$, $\P(\V)>0$ and
\begin{align}
\label{eq:blow_up_X_p_condition_pathwise_correction}
\|u\|_{C([\eta,\sigma];\Xp)}+
\|u\|_{L^p(0,\sigma,w_{\a};X_1)}+\nonlinearity^{\a}(u;0,\sigma)&\leq M, \ \text{ a.s.\ on }\V, \\
\label{eq:blow_up_X_p_conditions_egorov_theorem_M_Cpa}
\lim_{n\to \infty}\sup_{\V}\sup_{s\in [\sigma-\frac{1}{n},\sigma]}\|u(s)-u(\sigma)\|_{\Xp}&=0,
\end{align}
where we set $u(\sigma):=\lim_{t\uparrow \sigma}u (t)$ on $\W$. Let $N\in \N$ be such that $\eta>\frac{1}{N}$ and
\begin{equation}
\label{eq:blow_up_X_p_estimate_u}
\sup_{s\in [\sigma-\frac{1}{N},\sigma]}\|u(s)-u(\sigma)\|_{\Xp}\leq \frac{1}{4K_{M,\eta}L_M C_{\a,p}} \qquad \text{on }\V.
\end{equation}
Here, $C_{\a,p}$ denotes constant in the embedding $\Xp\hookrightarrow \Xap$.
Set $\maxSigma:=\esssup_{\V} \sigma$ and fix $r\in (\maxSigma-\frac{1}{N},\maxSigma)$ such that $r\geq \eta$, and set $\U:=\V\cap \{\sigma>r\}$.
Then $\U\in \F_{\sigma}\cap \F_r$, and by definition of essential supremum one has $\P( \U)>0$.

Set $\mu = \sigma \one_{\U} + r \one_{\O\setminus\U}$ and  $v_{\mu}:=\one_{\U}u(\sigma)$. Then $\mu\in [r, T]$ and $v_{\mu} \in L^{\infty}_{\F_{\mu}}(\O;\Xp)$. By \eqref{eq:blow_up_X_p_estimate_u} one has
$$
\|u(r)-v_{\mu}\|_{\Xap}\leq C_{\a,p} \|u(r)-v_{\sigma}\|_{\Xp}\leq \frac{1}{4 K_{M,\eta} L_M} \quad \text{on  }\;\U,
$$
and by \eqref{eq:blow_up_X_p_condition_pathwise_correction}, $\|\one_{\U} u(r)\|_{\Xp}\leq M$ and $\|v_{\mu}\|_{\Xp}\leq M$.
Applying Step 1 to $(\one_{\U}u(r), v_{\mu})$, we obtain that $(A(\cdot,v_{\mu}),B(\cdot,v_{\mu}))\in \MRtmu$. Thus, Assumption \ref{ass:FG_a_zero} and Proposition \ref{prop:local_sigma} ensure the existence of an $L^p_{0}$-local solution $(v,\tau)$ to
\begin{equation}
\label{eq:restart_equation_QSEE_Xp_blow_up_criteria}
\begin{cases}
dv + A(\cdot,v)v dt= (F(\cdot,v)+f)dt+(B(\cdot,v)v+G(\cdot,v)+g)dW_{H}(t),\\
v(\mu)=v_{\mu},
\end{cases}
\end{equation}
on $\ll\mu,T\rr$, and where $\tau>\mu$ a.s.\ Note that $v_{\mu}=u(\sigma)$ on $\U$. Set
\[
\wt{u}=u\one_{\ll 0,\sigma\rro}+ v\one_{\U\times [\sigma,\tau)} \ \ \text{and} \ \
\wt{\sigma}:=\one_{\O\setminus\U} \sigma + \one_{\U} \tau.\]
Then $(\wt{u},\wt{\sigma})$ is a unique $L^p_{\a}$-local  solution to \eqref{eq:QSEE} which extends $(u,\sigma)$ on $\U$. Since $\P(\U)>0$, this contradicts the maximality of $(u,\sigma)$ and this gives the desired contradiction.
\end{proof}

\begin{remark}
\label{rem:onlyonpath}
Suppose that we know that the maximal solution satisfies
\begin{equation}
\label{eq:u_lives_in_a_closed_set}
u(t)\in \closed \text{ a.s.\ for all }t\in (0,\sigma)\text{ where }\closed\subseteq \Xp \text{ is closed subset}.
\end{equation}
Then in Assumption \ref{H_a_stochnew} we only need to consider $v\in \closed$ a.s.
In this way Lemma \ref{l:lemma_correction_theorem} remains true. Indeed, one can repeat Step 1 for $u_{r},u_{\mu}$ satisfying $u_r,u_{\mu}\in \closed $ a.s.\ and replacing $v_{\mu}$ in \eqref{eq:restart_equation_QSEE_Xp_blow_up_criteria} by $\one_{\U}u(\sigma) + \one_{\O\setminus\U} x$, where $x\in \closed$. The same extension holds for the assertions in Theorem \ref{t:blow_up_criterion} and this will not be repeated later.
\end{remark}

\begin{remark}
\label{r:uniqueness_maximal}
If the assumptions of Lemma \ref{l:lemma_correction_theorem} are satisfied, then one can equivalently define $L^p_{\a}$-maximal local solution to \eqref{eq:QSEE} (see Definition \ref{def:solution2}) as:
\begin{enumerate}[{\rm(a)}]
\item\label{it:def_maximal_2}
An $L^p_{\a}$-local solution $(u,\sigma)$ to \eqref{eq:QSEE} on $[s,T]$ is called an {\em $L^p_{\a}$-maximal local solution on $[s,T]$}, if for any other $L^p_{\a}$-local solution $(v,\tau)$ to \eqref{eq:QSEE} on $[s,T]$, we have a.s.\ $\tau\leq \sigma$ and  for a.a.\ $\om \in \Omega$ and all $t\in [s,\tau(\om))$, $u(t,\omega)=v(t,\omega)$.
\end{enumerate}
Compared to Definition \ref{def:solution2}, in \eqref{it:def_maximal_2} we omit the word ``unique". In particular \eqref{it:def_maximal_2} is stronger than Definition \ref{def:solution2} since we also exclude the existence of an $L^p_{\a}$-local solution (but not unique) which extends $(u,\sigma)$, i.e.\ $\tau>\sigma$ with positive probability.

Next we show that the $L^p_{\a}$-maximal local solution in the sense of Definition \ref{def:solution2} also satisfies \eqref{it:def_maximal_2} if the assumptions of Lemma \ref{l:lemma_correction_theorem} hold. Let $(v,\tau)$ be an $L^p_{\a}$-local solution $(v,\tau)$ to \eqref{eq:QSEE}. By uniqueness of $(u,\sigma)$, we get $u=v$ a.e.\ on
$\ll0,\sigma\wedge \tau\rro$. In particular, a.s.\ on $\{\sigma<\tau\}$, $u\in L^p(s,\sigma,w_{\a}^s;X_1)\cap C([s,\sigma];\Xp)$ and $\nonlinearity^{\a}(u;\sigma)<\infty$ since $(v,\tau)$ is an $L^p_{\a}$-local solution.
Hence by Lemma \ref{l:lemma_correction_theorem}
$$
\P(\sigma<\tau)\leq
\P\Big(\sigma<T,\,
\lim_{t\uparrow \sigma} u(t)\text{ exists in }\Xp ,
\,\|u\|_{L^p(s,\sigma,w_{\a}^s;X_1)}+
\nonlinearity^{\a}(u;s,\sigma)<\infty\Big)=0.
$$
Therefore, $\tau \leq \sigma$ a.s.\ as desired.
\end{remark}

Next we will prove Theorem \ref{thm:semilinear_blow_up_Serrin}\eqref{it:blow_up_stochastic_semilinear} and Proposition \ref{prop:predictability}.
For this we need the following elementary lemma (see \cite[Lemma 4.12]{AV19_QSEE_1} for a similar result).
\begin{lemma}
\label{l:F_G_bound_N_C_cn}
Let the hypothesis \emph{\ref{HFcritical}-\ref{HGcritical}} be satisfied. Let $s\leq a<b\leq T<\infty$ and $N\in \N$ be fixed. Let $\zeta:=1+\max\{\rho_j: j\in \{1,\dots,m_F+m_G\}\}$. Then for all $h\in C([a,b];X^{\mathsf{Tr}}_{\a,p})\cap \X({a,b})$ which satisfy $\|h\|_{C([a,b];X^{\mathsf{Tr}}_{\a,p})}\leq N$, one has a.s.\
\begin{multline*}
\|F_c(\cdot,h)\|_{L^p(a,b,w_{\a}^a;X_0)}+\|G_c(\cdot,h)\|_{L^p(a,b,w_{\a}^a;\g(H,X_{1/2}))}
\\
\leq c_{a,b}(1+\|h\|_{\X(a,b)}+\|h\|_{\X(a,b)}^{\zeta}),
\end{multline*}
where $c_{a,b}=c(|a-b|,N)>0$ is independent of $f$ and satisfies $ c(\delta_1,N)\leq c(\delta_2,N)$ for all $0\leq \delta_1\leq \delta_2$.
Moreover, if \eqref{eq:F_c_estimates} and \eqref{eq:G_c_estimates} are satisfied with constants $C_{c,n}$ independent of $n\geq 1$, then $c(b-a,N)$ can be chosen independent of $N$.
\end{lemma}
\begin{proof}
We prove the estimate for $F_c$ in the case $m_F=1$. The other cases are similar. Let $N,h$ be as in the statement. Hypothesis \ref{HFcritical} ensures that for a.a.\ $\omega\in \O$ and all $t\in [a,b]$,
$$
\|F_c(t,\om,h(t))\|_{X_0}\leq C_{c,N} (1+\|h(t)\|_{X_{\varphi}}^{\rho} )\|h(t)\|_{X_{\beta}}+C_{c,N}.
$$
where $\beta:=\beta_1$, $\varphi:=\varphi_1$ and $\rho:=\rho_1$.
By H\"{o}lder's inequality with exponent $r':=r_1',r:=r_1$ (see \eqref{eq:rr'} and the text below it) we get
\begin{align*}
\|F_c(\cdot, h)\|_{L^p(a,b,w^a_{\a};X_0)}
&\leq C_{c,N}\big( \|h\|_{L^p(a,b,w_{\a}^a;X_{\beta})}+\\
&\qquad
+\|h\|_{L^{\rho p r'}(a,b,w_{\a}^a;X_{\varphi})}^{\rho} \|h\|_{L^{p r}(a,b,w_{\a}^a;X_{\beta})}+|b-a|^{1/p}\big).
\end{align*}
Since $\a\geq 0$, $r>1$ and $\rhos:=\rhos_1\geq \rho$ (see \eqref{eq:rhostar}), H\"{o}lder's inequality ensures that
\begin{align*}
\|h\|_{L^p(a,b,w_{\a}^a;X_{\beta})}&\leq C_{b-a}\|h\|_{L^{p r}(a,b,w_{\a}^a;X_{\beta})} \\ \|h\|_{L^{\rho p r'}(a,b,w_{\a}^a;X_{\varphi})}&\leq C_{b-a} \|h\|_{L^{\rhos p r'}(a,b,w_{\a}^a;X_{\varphi})},
\end{align*}
where $C_{\delta_1}\leq C_{\delta_2}$ for $0\leq \delta_1<\delta_2<\infty$ and $\sup_{\delta\in (0,T)} C_{\delta}<\infty$ due to the assumption $T<\infty$. By \eqref{eq:def_X_space} the previous inequalities imply the claimed estimate.
\end{proof}

\begin{proof}[Proof of Theorem \ref{thm:semilinear_blow_up_Serrin}\eqref{it:blow_up_stochastic_semilinear}]
As before we assume $s=0$, and we only consider $p>2$. By \eqref{eq:stochastic_maximal_regularity_assumption_local_extended_blow_up_semilinear} and Proposition \ref{prop:change_initial_time}, $(A,B)$ satisfies Assumption \ref{H_a_stochnew} for $\ell=0$. Therefore, Lemma \ref{l:lemma_correction_theorem} is applicable.

Let us argue by contradiction. Thus, we assume that $\P(\W)>0$ where
$$
\W:=\{\sigma<T,\,\nonlinearity^\a(u;0,\sigma)<\infty\}.
$$
Since $\sigma>0$ a.s.\ by Theorem \ref{t:local_s}, it follows that there exist $\eta,M>0$ such that
$\P(\V)>0$ where
\begin{equation}
\label{eq:V_def_proof_semilinear_easy_case}
\V:=\{\eta<\sigma<T, \nonlinearity^\a(u;0,\sigma)<M\}.
\end{equation}
Define a stopping time by
\begin{equation*}
\tau:=\inf\{t\in [0,\sigma)\,:\,
\nonlinearity^\a(u;0,t)\geq M\},
\end{equation*}
where we take $\inf\emptyset:=\sigma$. Note that $\tau=\sigma$ a.s.\ on $\V$ and
\begin{equation}
\label{eq:semilinear_F_G_M_bound}
\|\one_{\ll 0,\tau\rr} F(\cdot,u)\|_{L^p(\I_{T},w_{\a};X_0)}+\|\one_{\ll 0,\tau\rr} G(\cdot,u)\|_{L^p(\I_{T},w_{\a};\g(H,X_{1/2}))}\leq M.
\end{equation}
Since $u$ is an $L^p_{\a}$-maximal local solution to \eqref{eq:QSEE}, $u|_{\ll 0,\tau\rr}$ is an $L^p_{\a}$-local solution to \eqref{eq:diffAB_s} on $\ll 0,\tau\rr$. Proposition \ref{prop:causality_phi_revised_2} and \eqref{eq:semilinear_F_G_M_bound} gives
\begin{equation}
\label{eq:blow_up_un}
u=\Sol_{0,(A,B)}(u_0,\one_{\ll 0,\tau\rr}F(\cdot,u)+f,\one_{\ll 0,\tau\rr}G(\cdot,u)+g),\qquad \text{a.e.\ on }\ll 0,\tau\rr,
\end{equation}
On the other hand, by Propositions \ref{prop:change_p_q_eta_a}\eqref{it:loc_embedding} and \ref{prop:start_at_s}\eqref{it:start_at_s2}, \eqref{eq:semilinear_F_G_M_bound} ensures that
\begin{equation*}
\text{RHS\eqref{eq:blow_up_un}}\in
L^p(\O;C([\eta,T];\Xp)\cap L^p(0,T,w_{\a};X_1)).
\end{equation*}
Thus by \eqref{eq:blow_up_un} and $\tau = \sigma$ on $\V$, we obtain that $\lim_{t\uparrow \sigma} u(t)$ exists in $\Xp$ and $\|u\|_{L^p(0,\sigma,w_{\a};X_1)}<\infty$ a.s.\ on $\V$. Therefore,
\begin{align*}
0&<\P(\V)
=  \P\Big(\V\cap\{\sigma<T\}
\cap\big\{\lim_{t\uparrow \sigma} u(t)\text{ exists in }\Xp,\,\|u\|_{L^p(0,\sigma,w_{\a};X_1)}<\infty\big\}\Big)\\
&\leq  \P\Big(\sigma<T,\,
\lim_{t\uparrow \sigma} u(t)\text{ exists in }\Xp ,
\,\|u\|_{L^p(\I_{\sigma},w_{\a}^s;X_1)}+
\nonlinearity^{\a}(u;0,\sigma)<\infty\Big)= 0,
\end{align*}
where we used \eqref{eq:V_def_proof_semilinear_easy_case} and the last equality follows from Lemma \ref{l:lemma_correction_theorem}. This contradiction completes the proof.
\end{proof}

\begin{remark}\
\label{r:blow_up_semilinear_X_norm}
\begin{enumerate}[{\rm(1)}]
\item
Let the assumptions of Theorem \ref{thm:semilinear_blow_up_Serrin} be satisfied.
If $F_{\Tr}=G_{\Tr}=0$ and $C_{c,n}$ in \ref{HFcritical}-\ref{HGcritical} does not depend on $n\geq 1$, then Lemma \ref{l:F_G_bound_N_C_cn} yields
\begin{equation*}
\P(\sigma<T,\,\|u\|_{\X(s,\sigma)}<\infty) =\P(\sigma<T,\,\nonlinearity^{\a}(u;s,\sigma)<\infty) =0,
\end{equation*}
where in the last step we applied Theorem \ref{thm:semilinear_blow_up_Serrin}\eqref{it:blow_up_stochastic_semilinear};
\item\label{it:case_p_2_simpler}
If $p=2$, then Theorem \ref{t:blow_up_criterion}\eqref{it:blow_up_norm_general_case_quasilinear_Xap_Pruss} (resp.\ \ref{thm:semilinear_blow_up_Serrin}\eqref{it:blow_up_semilinear_serrin_Pruss_modified}) follows from Lemma \ref{l:lemma_correction_theorem} (resp.\ Theorem \ref{thm:semilinear_blow_up_Serrin}\eqref{it:blow_up_stochastic_semilinear})
due to Lemmas \ref{l:embeddings} and \ref{l:F_G_bound_N_C_cn}. The general case is more complicated and will be considered in Section \ref{ss:thmblowuphardespart_Serrin}.
\end{enumerate}
\end{remark}

\begin{proof}[Proof of Proposition \ref{prop:predictability}]
Let $(\sigma_n)_{n\geq 1}$ be a localizing sequence.
Suppose that there exists an $n\geq 1$ such that $\P(\sigma_n=\sigma<T)>0$. Setting $\V:=\{\sigma_n=\sigma<T\}$, by Theorem \ref{t:local_s}\eqref{it:regularity_data_L0}, one has $\sigma_n = \sigma>s$ a.s.\ on $\V$, and for all $\theta\in [0,\frac{1}{2})$,
\begin{align*}
u&\in H^{\theta,p}(s,\sigma_n,w_{\a}^s;X_{1-\theta})\cap C([s,\sigma_n];\Xap)\cap C((s,\sigma_n];\Xp) \\
&=H^{\theta,p}(s,\sigma,w_{\a}^s;X_{1-\theta})\cap C([s,\sigma];\Xap)\cap C((s,\sigma];\Xp) \ \ \   \text{ a.s.\ on }\V.
\end{align*}
In particular,
$
\lim_{t\uparrow \sigma} u(t)\text{ exists in }\Xp,
$
a.s.\ on $\V$. Thus, by Lemmas \ref{l:embeddings} and \ref{l:F_G_bound_N_C_cn},
\begin{align*}
\P(\V)
&=\P\Big(\V\cap\Big\{\lim_{t\uparrow \sigma} u(t)\text{ exists in }\Xp,\,
\|u\|_{ L^p(s,\sigma,w_{\a};X_1)}+\nonlinearity^{\a}(u;s,\sigma)<\infty\Big\}\Big)\\
&\leq \P\Big(\sigma<T,\,\lim_{t\uparrow \sigma} u(t)\text{ exists in }\Xp,\,
\|u\|_{ L^p(s,\sigma,w_{\a};X_1)}+\nonlinearity^{\a}(u;s,\sigma)<\infty\Big)=0,
\end{align*}
where in the last equality we used Lemma \ref{l:lemma_correction_theorem}. This contradicts $\P(\V)>0$ and therefore the result follows.
\end{proof}

\subsection{Proofs of Theorems \ref{t:blow_up_criterion}\eqref{it:blow_up_norm_general_case_F_c_G_c}-\eqref{it:blow_up_non_critical_Xap} and \ref{thm:semilinear_blow_up_Serrin}\eqref{it:blow_up_semilinear noncritical_stochastic}}\label{ss:thmblowuphardespart}
Using Lemma \ref{l:lemma_correction_theorem} and Proposition \ref{prop:predictability} we can now prove Theorem \ref{t:blow_up_criterion}\eqref{it:blow_up_norm_general_case_F_c_G_c}. Unfortunately, the proof is rather technical. It requires several reduction arguments and one key step is to approximate localizing sequences by stopping times taking finitely many values which in turn allow to apply the maximal regularity estimates of Proposition \ref{prop:start_at_sigma_random_time}.

\begin{proof}[Proof of Theorem \ref{t:blow_up_criterion}\eqref{it:blow_up_norm_general_case_F_c_G_c}]
By a translation argument we may assume $s=0$. We will only consider $p>2$ as the other case is similar. By Proposition \ref{prop:redblowbounded} it is enough to consider
\[u_0\in L^{\infty}_{\F_0}(\O;\Xap), \  f\in L^p(\I_T\times\O,w_{\a};X_0), \  \ \text{and} \ \
g\in L^p(\I_T\times\O,w_{\a};\g(H,X_{1/2})).\]

{\em Step 1: Setting up the proof by contradiction.}
We will prove \eqref{it:blow_up_norm_general_case_F_c_G_c} by a contradiction argument. So suppose that $\P(\W)>0$ where
$$
\W:=\Big\{\sigma<T,\,\lim_{t\uparrow \sigma} u(t)\;\text{exists in}\;\Xap,\,
\nonlinearity_c^\a(u;0,\sigma)<\infty\Big\}\in \F_{\sigma},
$$
see \eqref{eq:L_p_norm_nonlinearity} for the definition of $\nonlinearity_c^{\a}$. For each $n\geq 1$, let
\begin{equation}
\label{eq:canonical_localizing_sequence_u_sigma}
\sigma_n=\inf\Big\{t\in [0,\sigma)\,:\,\|u\|_{L^p(\I_t,w_{\a};X_1)\cap C([0,t];\Xap)}+\nonlinearity_c^{\a}(u;0,t)\geq n\Big\} \wedge \frac{n T}{n+1},
\end{equation}
and $\inf\emptyset:=\sigma$.
Then $(\sigma_n)_{n\geq 1}$ is a localizing sequence for $(u,\sigma)$.
By Egorov's theorem and the fact that $\sigma>0$ a.s., there exist $\eta> 0$, $\F_{\sigma}\ni \V\subseteq \W$, $M\in \N$ such that $\P(\V)>0$, $\sigma\geq \eta$ a.s.\ on $\V$, and
\begin{equation}
\label{eq:uniform_bound_u_k}
\begin{aligned}
\|u\|_{C([0,{\sigma}];\Xap)}+ \nonlinearity_c^{\a}(u;0,\sigma)& \leq M \ \text{on }\V, \\
\  \lim_{n\to \infty}\sup_{\V}\sup_{s\in [\sigma_n,\sigma]}\|u(s)-u(\sigma)\|_{\Xap}&= 0,
\end{aligned}
\end{equation}
where we have set $u(\sigma):=\lim_{t\uparrow \sigma} u(t)$ on $\W$. By decreasing $\eta$ if necessary, we may suppose $\P(\sigma\leq \eta)\leq \frac14 \P(\V)$.

By Proposition \ref{prop:predictability}, one has $\sigma_n<\sigma$ on $\{\sigma<T\}$ for all $n\geq 1$. Moreover, the definition of $\sigma_n$ implies $\sigma_n<\sigma$ on the set $\{\sigma= T\}$. Therefore, Lemma \ref{l:stopping_k} implies there exists a sequence of stopping times $(\wt{\sigma}_n)_{n\geq 1}$ such that for each $n\geq 1$, $\wt{\sigma}_n$ takes values in a finite subset of $[0,T]$,  $\wt{\sigma}_n\leq\wt{\sigma}_{n+1}$, $\wt{\sigma}_n\geq \sigma_n$ and $\P(\wt{\sigma}_n\geq \sigma)\leq \frac14 \P(\V)$. Moreover, we can also assume that $\sup_{\O}\wt{\sigma}_n<T$ for all $n\geq 1$. Set
\begin{align}\label{eq:V'def}
\sigma_n' = \wt{\sigma}_n\vee \eta \ \ \  \  \text{for $n\geq 1$ and} \ \ \ \ \  \V':=\V\cap (\cap_{n\geq 1}\{\sigma_n'<\sigma\}).
\end{align}
Then by Proposition \ref{prop:F_sigma_algebra_stopping_times}, $\V'\in \F_{\sigma}$, and
\begin{align}\label{eq:V'pos}
\P(\V') = \lim_{n\to \infty} \P(\V\cap \{\sigma'_{n}<\sigma\}) \geq \lim_{n\to \infty} \P(\V) - \P(\sigma'_n\geq \sigma) \geq \frac12 \P(\V)>0,
\end{align}
where in the last step we used
\begin{align*}
\P(\sigma_n'\geq \sigma) &\leq \P(\sigma'_n\geq \sigma, \sigma>\eta)  + \P(\sigma\leq \eta)
\leq \P(\wt{\sigma}_n\geq \sigma)  + \P(\sigma\leq \eta) \leq \frac12 \P(\V).
\end{align*}

{\em Step 2: In this step we will prove that $\P(\W)>0$ implies}
\begin{equation}
\label{eq:blow_up_general_X_norm_claim_Step_1}
\P\Big(\sigma<T,\,\lim_{t\uparrow \sigma} u(t)\;\text{exists in}\;\Xap,\,\|u\|_{L^p(\I_{\sigma},w_{\a};X_1)\cap\X(\sigma)}<\infty\Big)>0.
\end{equation}
To prove the above, we need some preliminary observations.
By \eqref{eq:uniform_bound_u_k}, for each $\varepsilon>0$ there exists an $N(\varepsilon)\in \N$ such that
\begin{equation}
\label{eq:eps_diff}
\sup_{s\in [\sigma_{N(\varepsilon)}, \sigma]}\|u(s)-u(\sigma)\|_{\Xap}<\varepsilon \ \ \text{on} \  \V.
\end{equation}
For each $\varepsilon>0$ set $\stopp_{\varepsilon} = \sigma_{N(\varepsilon)}$, $\stopp_{\varepsilon}' = \sigma_{N(\varepsilon)}'$ and define the stopping time $\tau_{\varepsilon}$ by
\begin{equation}
\begin{aligned}
\label{eq:tau_N_proof_4_3_first_part}
\tau_{\varepsilon}:=\inf\Big\{t\in [\stopp_{\varepsilon},\sigma)\,:\,&\|u(t)-u(\stopp_{\varepsilon})\|_{\Xap}\geq 2\varepsilon, \\
&\|u\|_{C([0,t];\Xap)}+\nonlinearity_c^{\a}(u;0,t)\geq M\Big\}
\end{aligned}
\end{equation}
where we set $\inf\emptyset:=\sigma$.
Note that $\tau_{\varepsilon}=\sigma$ on $\V\supseteq\V'$. Therefore,
\[\V'\subseteq \U_{\varepsilon}:=\{\tau_\varepsilon>\stopp_{\varepsilon}'\}\in \F_{\stopp_{\varepsilon}'}.\]

For each $\varepsilon>0$ we set
\[u_{\varepsilon}:=\one_{\U_{\varepsilon}}u(\stopp_{\varepsilon}')\in L^{\infty}_{\F_{\stopp_{\varepsilon}'}}(\O;\Xap).\]
The latter random variable is well defined since $\sigma\geq \tau_{\varepsilon}>\stopp_{\varepsilon}'$ on $\U_{\varepsilon}$, and by \eqref{eq:tau_N_proof_4_3_first_part}, we have
$
\|u_{\varepsilon}\|_{\Xap}\leq M.
$
Since $\stopp_{\varepsilon}'\geq \eta$ (see \eqref{eq:V'def}), combining Assumption \ref{H_a_stochnew} with Proposition \ref{prop:start_at_sigma_random_time}, we obtain that
$(A(\cdot,u_{\varepsilon}),B(\cdot,u_{\varepsilon}))\in \mathcal{SMR}_{p,\a}^{\bullet}(\stopp_{\varepsilon}',T)$, and for each $\theta\in [0,\frac{1}{2})\setminus \{\frac{1+\a}{p}\}$,
\begin{equation}
\label{eq:max_proof_blow_up}
\max\{ K^{\deter,\theta,p,\a}_{(A(\cdot,u_{\varepsilon}),B(\cdot,u_{\varepsilon}))},K^{\stoc,\theta,p,\a}_{(A(\cdot,u_{\varepsilon}),B(\cdot,u_{\varepsilon}))}\}
\leq K_{M,\eta}^{\theta},
\end{equation}
where $K_{M,\eta}^{\theta}$ is as in Assumption \ref{H_a_stochnew} for $\ell=\a$. Let us stress that $K_{M,\eta}^{\theta}$ does not depend on $\varepsilon$. Fix any $\theta\in (\frac{1+\a}{p},\frac{1}{2})$ and set $K_{M,\eta} = K_{M,\eta}^{0} + K_{M,\eta}^{\theta}$. For notational convenience, set $\Sol_{\varepsilon}= \Sol_{\stopp_{\varepsilon}',(A(\cdot,u_{\varepsilon}),B(\cdot,u_{\varepsilon}))}$.

Lemma \ref{l:estimate_Sol_op_Y} ensures that \eqref{eq:estimate_Sol_X_spaces} holds with $\Sol_{\sigma}$ replaced by $\Sol_{\varepsilon}$ and constant $\K_{M,\eta}:=C(1+K_{M,\eta})$ which is independent of $\varepsilon$ (see \eqref{eq:max_proof_blow_up}). Thus, for $L_n$ as in \ref{HAmeasur}, we set
$
\varepsilon=1/(16 \K_{M,\eta} L_{M})
$.
Let
\begin{equation}
\label{eq:definition_psi_proof_general_quasilinear_NFG}
\psi:=\one_{\O\setminus \U_{\varepsilon}}\stopp_{\varepsilon}'+\one_{\U_{\varepsilon}} \tau_{\varepsilon} \ \  \text{and} \ \  \psi_n:=\one_{\O\setminus \U_{\varepsilon}}\stopp_{\varepsilon}'+\one_{\U_{\varepsilon}} [(\sigma_n\vee \stopp_{\varepsilon}') \wedge  \tau_{\varepsilon} ].
\end{equation}
Note that $\psi_n\uparrow \psi$ and for each $n\geq 1$, $\ll \stopp_{\varepsilon}', \psi_n\rro\subseteq \ll \stopp_{\varepsilon}', \sigma_n\rro$.
Since $(u,\sigma)$ is an $L^p_{\a}$-maximal local solution to \eqref{eq:QSEE}, $(u|_{\ll \stopp_{\varepsilon}',\psi\rro},\psi)$ is an $L^p_{\a}$-local solution to
\begin{equation}
\label{eq:equation_start_at_u_N}
\begin{cases}
dv +A(\cdot,v)vdt =(F(\cdot,v)+f)dt+(B(\cdot,v)v+G(\cdot,v)+g)dW_{H}(t),\\
v(\stopp_{\varepsilon}')=u_{\varepsilon},
\end{cases}
\end{equation}
with localizing sequence $(\psi_n)_{n\geq 1}$. Here we used that $\sigma>\stopp_{\varepsilon}'$ on $\U_{\varepsilon}$ which follows from the definition of $\U_{\varepsilon}$. Finally, we set
\begin{align*}
\Lambda_\varepsilon &:= \ll \stopp_{\varepsilon}',\psi\rro= [\stopp_{\varepsilon}',\tau_{\varepsilon})\times \U_{\varepsilon},
\\
\Lambda_\varepsilon^n& := \ll \stopp_{\varepsilon}',\psi_n\rro= [\stopp_{\varepsilon}', (\sigma_n\vee \stopp_{\varepsilon}') \wedge  \tau_{\varepsilon} )\times \U_{\varepsilon}.
\end{align*}

Since $A(\cdot,u)=A(\cdot,u_{\varepsilon})+(A(\cdot,u)-A(\cdot,u_{\varepsilon}))$, $B(\cdot,u)=B(\cdot,u_{\varepsilon})+(B(\cdot,u)-B(\cdot,u_{\varepsilon}))$, by \eqref{eq:equation_start_at_u_N} and Proposition \ref{prop:causality_phi_revised_2} one has a.s.\ on $\Lambda_{\varepsilon}^n$
\begin{equation}
\label{eq:u_sum_v}
\begin{aligned}
\one_{\U_{\varepsilon}}u
 & = \Sol_\varepsilon(u_{\varepsilon},\one_{{\Lambda}_\varepsilon^n}f,\one_{{\Lambda}_\varepsilon^n}g)\\
&\quad+\Sol_\varepsilon(0,\one_{{\Lambda}_\varepsilon^n}(A(\cdot,u_{\varepsilon})-A(\cdot,u))u,\one_{{\Lambda}_\varepsilon^n}(B(\cdot,u_{\varepsilon})-B(\cdot,u))u)\\
&\quad+\Sol_\varepsilon(0,\one_{{\Lambda}_\varepsilon^n}F_{\Tr}(\cdot,u),\one_{{\Lambda}_\varepsilon^n}G_{\Tr}(\cdot,u))\\
&\quad+ \Sol_\varepsilon(0,\one_{{\Lambda}_\varepsilon^n}F_c(\cdot,u),\one_{{\Lambda}_\varepsilon^n}G_c(\cdot,u))\\
&=: I+II+III+IV.
\end{aligned}
\end{equation}
Next, we estimate each of the above summands. To make the formulas more readable, in this step, we denote by $\y$ the space $L^p(\O;L^p(\stopp_{\varepsilon}',T,w_{\a}^{\stopp_{\varepsilon}'};X_1)\cap \X({\stopp_{\varepsilon}',T}))$. To begin, by Lemma \ref{l:estimate_Sol_op_Y},
\begin{align*}
\|I\|_{\y}&\leq \K_{M,\eta} (\|u_{\varepsilon}\|_{L^p(\U_{\varepsilon};\Xap)}+
\|f\|_{L^p(\Lambda_\varepsilon^n,w_{\a}^{\stopp_{\varepsilon}'};X_0)}+\| g\|_{L^p(\Lambda_\varepsilon^n,w_{\a}^{\stopp_{\varepsilon}'};\g(H,X_{1/2}))})\\
&\leq C\K_{M,\eta} (M+\|f\|_{L^p(\I_T\times\O,w_{\a};X_0)}+\| g\|_{L^p(\I_T\times\O,w_{\a};\g(H,X_{1/2}))}).
\end{align*}
Again, by Lemma \ref{l:estimate_Sol_op_Y},
\begin{equation}
\label{eq:estimate_II_blow_up_proof_X}
\begin{aligned}
\|II\|_{\y}
&\leq  \K_{M,\eta}(\|(A(\cdot,u_{\varepsilon})-A(\cdot,u))u\|_{L^p(\Lambda_\varepsilon^n,w_{\a}^{\stopp_{\varepsilon}'};X_0)}\\
&+  \|(B(\cdot,u_{\varepsilon})-B(\cdot,u))u\|_{L^p(\Lambda_\varepsilon^n,w_{\a}^{\stopp_{\varepsilon}'};\g(H,X_{1/2}))})
\leq \frac{1}{2} \|u\|_{L^p(\Lambda_\varepsilon^n,w_{\a}^{\stopp_{\varepsilon}'};X_1)},
\end{aligned}
\end{equation}
where in the last inequality we used the choice of $\varepsilon$ (see the text before \eqref{eq:definition_psi_proof_general_quasilinear_NFG}) and the fact that
$$\sup_{s\in [\stopp_{\varepsilon}',\tau_{\varepsilon})}\|u(s)-u(\stopp_{\varepsilon}')\|_{\Xap}
\leq 2\sup_{s\in [\stopp_{\varepsilon},\tau_{\varepsilon})}\|u(s)-u(\stopp_{\varepsilon})\|_{\Xap}\leq 4\varepsilon, \text{ a.s.\ on }\U_{\varepsilon},$$
since $\stopp_{\varepsilon}\leq \stopp_{\varepsilon}'$. Similarly, one obtains
\begin{align*}
\|III\|_{\y}&\leq \K_{M,\eta}(\|F_{\Tr}(\cdot,u)\|_{L^p(\Lambda_\varepsilon^n,w_{\a}^{\stopp_{\varepsilon}'};X_0)}+
\|G_{\Tr}(\cdot,u)\|_{L^p(\Lambda_\varepsilon^n,w_{\a}^{\stopp_{\varepsilon}'};X_0)})\\
&\leq 2 \K_{M,\eta}(1+C_{\Tr,M} M),
\end{align*}
where in the last estimate we used \ref{HFcritical}-\ref{HGcritical} and \eqref{eq:uniform_bound_u_k}. Finally,
\begin{equation}
\begin{aligned}
\label{eq:estimates_v_4}
\|IV\|_{\y}
&\leq \K_{M,\eta}(\|F_{c}(\cdot,u)\|_{L^p(\Lambda_\varepsilon^n,w_{\a}^{\stopp_{\varepsilon}'};X_0)}+
\|G_{c}(\cdot,u)\|_{L^p(\Lambda_\varepsilon^n,w_{\a}^{\stopp_{\varepsilon}'};X_0)})\\
&\leq \K_{M,\eta} CM,
\end{aligned}
\end{equation}
in the last inequality we used \eqref{eq:L_p_norm_nonlinearity} and the bound in \eqref{eq:uniform_bound_u_k}. By \eqref{eq:u_sum_v}, and the previous estimates, one obtains that for some $C_1>0$ for all $n\geq 1$,
\begin{equation}
\label{eq:conclusion_Step_1_proof_Thm_4_3}
\begin{aligned}
\|u\|_{L^p(\U_{\varepsilon};L^p(\stopp_{\varepsilon}',(\sigma_n\vee \stopp_{\varepsilon}') \wedge  \tau_{\varepsilon} ,w_{\a}^{\stopp_{\varepsilon}'};X_1)\cap \X(\stopp_{\varepsilon}',(\sigma_n\vee \stopp_{\varepsilon}') \wedge  \tau_{\varepsilon} ))}\leq C_1,
\end{aligned}
\end{equation}
and by Fatou's lemma, \eqref{eq:conclusion_Step_1_proof_Thm_4_3} also holds with $(\sigma_n\vee \stopp_{\varepsilon}') \wedge  \tau_{\varepsilon} $ replaced by $\tau_{\varepsilon}$.

Recall that $\tau_{\varepsilon}|_{\V'}=\sigma|_{\V'}$. Since $\stopp_{\varepsilon}'<\tau_{\varepsilon}=\sigma$ on $\V'$,  \eqref{eq:conclusion_Step_1_proof_Thm_4_3} (with $n\to \infty$) implies
\begin{equation*}
\begin{aligned}
&\V'\cap \Big\{\sigma<T,\;\lim_{t\uparrow \sigma} u(t)\;\text{exists in}\;\Xap,\,\,\|u\|_{L^p(\I_{\sigma},w_{\a};X_1)\cap\X(\sigma)}<\infty\Big\} \\
&= \V'\cap \Big\{\sigma<T,\;
\lim_{t\uparrow \sigma} u(t)\;\text{exists in}\;\Xap,\,
\,\|u\|_{L^p(\stopp_{\varepsilon}',\sigma,w_{\a}^{\stopp_{\varepsilon}'};X_1)\cap\X(\stopp_{\varepsilon}',\sigma)}<\infty\Big\}=\V',
\end{aligned}
\end{equation*}
By \eqref{eq:V'pos}, \eqref{eq:blow_up_general_X_norm_claim_Step_1} follows, and this completes the proof of the claim in step 2.

{\em Step 3: In this step we will prove that $\P(\W)>0$ implies}
\begin{equation}
\label{eq:positive_measure_step_2}
\P\Big(\sigma<T,\lim_{t\uparrow \sigma}u(t)\,\text{exists in }\Xp,
\,\|u\|_{L^p(\I_{\sigma},w_{\a};X_1)}+\nonlinearity^{\a}(u;0,\sigma)<\infty\Big)>0.
\end{equation}
Here $\nonlinearity^{\a}$ is as in \eqref{eq:L_p_norm_nonlinearity2}.
Clearly, \eqref{eq:positive_measure_step_2} contradicts Lemma \ref{l:lemma_correction_theorem} and thus implies $\P(\W)=0$.
To prove \eqref{eq:positive_measure_step_2}, by Step 2 we can find $\wt{M}\geq M$ (see \eqref{eq:uniform_bound_u_k}), such that $\P(\ee)>0$, where
$$
\ee:=\big\{\sigma<T,\,\|u\|_{L^p(\I_{\sigma},w_{\a};X_1)\cap\X(\sigma)\cap C(\overline{I}_{\sigma};\Xap)}<\wt{M}\big\}.
$$
Note that Lemma \ref{l:F_G_bound_N_C_cn} yields
\begin{equation}
\label{eq:ee_bounds_non_X_1_norm_correction}
\nonlinearity^{\a}(u;0,\sigma)<\infty,\ \ \text{ a.s.\ on }\ee.
\end{equation}
Let $\mu$ be the stopping time given by
\begin{equation}
\label{eq:def_m_blow_up_2_step_2}
\mu:=\inf\{t\in [0,\sigma)\,:\, \|u\|_{L^p(\I_{t},w_{\a};X_1)\cap \X({t})\cap C([0,t];\Xap)}\geq  M\},
\quad \inf\emptyset:=\sigma.
\end{equation}
By construction and \eqref{eq:uniform_bound_u_k}, $\{\mu=\sigma\}\supseteq \ee$ and $\mu>0$ a.s. Since we already reduced to bounded initial values, we have $(A(\cdot,u_0),B(\cdot,u_0))\in \MRta$ by \eqref{eq:stochastic_maximal_regularity_assumption_local_extended_blow_up_quasilinear}. Set $v:=\Sol_{0,(A(\cdot,u_0),B(\cdot,u_0))}(u_0,\ff,\ggg)
$ on $\ll 0,T\rr$. Here $\ff$ and $\ggg$ are defined by
\begin{equation*}
\begin{aligned}
&\ff:=\one_{\ll 0,\mu\rr}((A(\cdot,u_0)-A(\cdot,u))u+F(\cdot,u))+f \in L^p_{\Progress}(\I_T\times \O,w_{\a};X_0),\\
&\ggg:=\one_{\ll 0,\mu\rr}((B(\cdot,u_0)-B(\cdot,u))u+G(\cdot,u))+g\in L^p_{\Progress}(\I_T\times \O,w_{\a};\g(H,X_{1/2})),
\end{aligned}
\end{equation*}
where we used Lemma \ref{l:F_G_bound_N_C_cn}, \eqref{eq:def_m_blow_up_2_step_2} and \ref{HAmeasur} to check the required $L^p$-integrability.
Since $(u,\sigma)$ is an $L^p_{\a}$-maximal local solution to \eqref{eq:QSEE} with $s=0$ it follows from Proposition \ref{prop:causality_phi_revised_2} that $u=v$ on $\ll 0,\mu\rr$. Since $\sigma>0$ a.s., there exists an $\wt{\eta}>0$ such that $\P(\{\sigma>\wt{\eta}\}\cap\ee)>0$. Set $\U:=\{\sigma>\wt{\eta}\}\cap \ee$.
Using the regularity estimate of Proposition \ref{prop:start_at_s}\eqref{it:start_at_s2} (and \eqref{eq:soloperatorR2}) we obtain
\begin{align*}
\|u&\|_{L^p(\U;C([\wt{\eta},\mu];\Xp))}
\leq
\|v\|_{L^p(\U;C([\wt{\eta},T];\Xp))}\\
&\lesssim_{\wt{\eta}}\|u_0\|_{L^p(\O;\Xap)}+
\|\ff\|_{L^p(\I_T\times \O,w_{\a};X_0)}+
\|\ggg\|_{L^p(\I_T\times \O,w_{\a};\g(H,X_{1/2}))}<\infty.
\end{align*}
Since $\{\mu=\sigma\}\supseteq \U$, $\U\subseteq \ee$ and $\P(\U)>0$, the above estimate and \eqref{eq:ee_bounds_non_X_1_norm_correction} imply \eqref{eq:positive_measure_step_2}.
This finishes the proof of Step 3, and therefore the proof of Theorem \ref{t:blow_up_criterion}\eqref{it:blow_up_norm_general_case_F_c_G_c}.
\end{proof}

\begin{remark}\
\label{r:blow_up_quasilinear_X_norm}
\begin{itemize}
\item
The arguments in the proof of Theorem \ref{t:blow_up_criterion}\eqref{it:blow_up_norm_general_case_F_c_G_c} can be extended to prove Theorem \ref{t:blow_up_criterion}\eqref{it:blow_up_non_critical_Xap} in the case $F_c=G_c=0$. The only difference is in Step 2, where $IV=0$ by assumption. Of course the latter situation is also covered by Theorem \ref{t:blow_up_criterion}\eqref{it:blow_up_non_critical_Xap} which is proved below;
\item similar to Remark \ref{r:blow_up_semilinear_X_norm}, under the assumptions of Theorem \ref{t:blow_up_criterion},
\begin{equation}
\label{eq:quasilinear_blow_up_X_norm}
\P\big(\sigma<T,\,\lim_{t\uparrow \sigma} u(t)\;\text{exists in}\;\Xap,\,\|u\|_{\X(s,\sigma)}<\infty\big)=0.
\end{equation}
To obtain \eqref{eq:quasilinear_blow_up_X_norm} one can repeat the argument of Theorem \ref{t:blow_up_criterion}\eqref{it:blow_up_norm_general_case_F_c_G_c}
using Lemma \ref{l:F_G_bound_N_C_cn} to get the estimate of $IV$ in \eqref{eq:estimates_v_4}.
\end{itemize}
\end{remark}

To prove Theorem \ref{t:blow_up_criterion}\eqref{it:blow_up_non_critical_Xap}, we need the next result (see \cite[Lemma 4.12]{AV19_QSEE_1}).
\begin{lemma}
\label{l:F_G_bound_N}
Let the hypothesis \emph{\ref{HFcritical}-\ref{HGcritical}} be satisfied. Let $0<a<b<T<\infty$ and $N\in \N$ be fixed. Assume that $\Xap$ is not critical for \eqref{eq:QSEE}. Then there exists a $\zeta>1$ such that for all $N\geq 1$ and $h\in C([a,b];\Xap)\cap \X({a,b})$ which satisfies $\|h\|_{C([a,b];\Xap)}\leq N$, one has a.s.
\begin{multline*}
\|F_c(\cdot,h)-F_c(\cdot,0)\|_{L^p(a,b,w_{\a}^a;X_0)}+\|G_c(\cdot,h)-G_c(\cdot,0)\|_{L^p(a,b,w_{\a}^a;\g(H,X_{1/2}))}
\\
\leq C_{a,b}(\|h\|_{\X(a,b)}+\|h\|_{\X(a,b)}^{\zeta}),
\end{multline*}
where $C_{a,b}=C(|a-b|,N)>0$ is independent of $f$ and satisfies $ C(\delta_1,N)\leq C(\delta_2,N)$ for all $0\leq \delta_1\leq \delta_2$ and $\lim_{\delta \downarrow 0}C(\delta,N)= 0$.
\end{lemma}

To derive the remaining part \eqref{it:blow_up_non_critical_Xap} of Theorem \ref{t:blow_up_criterion}, we will exploit the non-criticality of $\Xap$ by using Lemma \ref{l:F_G_bound_N}. Also Theorem \ref{t:blow_up_criterion}\eqref{it:blow_up_norm_general_case_F_c_G_c} is applied in a technical but essential step in the proof below.

\begin{proof}[Proof of Theorem \ref{t:blow_up_criterion}\eqref{it:blow_up_non_critical_Xap}]
As before in part \eqref{it:blow_up_norm_general_case_F_c_G_c} we assume $s=0$ and $p>2$. By Proposition \ref{prop:redblowbounded} we may assume $u_0$ is bounded, and $f$ and $g$ are $L^p$-integrable.  Set
\begin{equation}
\label{eq:def_V_blow_up_non_critical_Xap}
\W:=\Big\{\sigma<T,\;\lim_{t\uparrow \sigma} u(t)\;\text{exists in}\;\Xap\Big\}
\end{equation}
and suppose that $\P(\W)>0$. We will show that this leads to a contradiction.

Let $(\sigma_n)_{n\geq 1}$ be the localizing sequence for $(u,\sigma)$ defined in \eqref{eq:canonical_localizing_sequence_u_sigma}. By Egorov's theorem and the fact that $\sigma>0$ a.s., there exist $\eta>0$ and $\F_{\sigma}\ni \V\subseteq \W$, $M\in \N$ such that $\P(\V)>0$, $\sigma>\eta$ a.s.\ on $\V$ and
\begin{align*}
\sup_{\V}\|u\|_{C(\overline{I}_{\sigma};\Xap)}& \leq M \ \ \text{on $\V$},
\\ \lim_{n\to \infty} \sup_{\V} |\sigma_n-\sigma|  &= 0, \ \ \ \  \lim_{n\to \infty}\sup_{\V}\sup_{s\in [\sigma_n,\sigma]}\|u(s)-u(\sigma)\|_{\Xap} = 0.
\end{align*}
Here, as usual, we have set $u(\sigma):=\lim_{t\uparrow \sigma}u(t)\in \Xap$ on $\V$. Moreover, we may suppose that $\P(\sigma\leq \eta)\leq \frac14 \P(\V)$.

As in Step 1 of the proof of Theorem \ref{t:blow_up_criterion}\eqref{it:blow_up_norm_general_case_F_c_G_c},
there exists a sequence of stopping times $(\wt{\sigma}_n)_{n\geq 1}$ such that for each $n\geq 1$, $\wt{\sigma}_n$ takes values in a finite subset of $[0,T]$,  $\wt{\sigma}_n\leq\wt{\sigma}_{n+1}$, $\wt{\sigma}_n\geq \sigma_n$ and $\P(\wt{\sigma}_n\geq \sigma)\leq \frac14 \P(\V)$. Defining $\sigma_n'$ and $\V'$ as in \eqref{eq:V'def},
we have $\V'\in \F_{\sigma}$, and $\P(\V')>0$ as before.
Moreover, for each $\varepsilon>0$ there exists an $N(\varepsilon)$ such that on the set $\V$, we have
\begin{equation}
\label{eq:non_critical_epsilon_assumption}
\sup_{t\in [0,\sigma)}\|u(t)\|_{\Xap}<M, \quad |\sigma_{N(\varepsilon)}-\sigma|< \varepsilon,\quad
\sup_{s\in [\sigma_{N(\varepsilon)},\sigma]}\|u(s)-u(\sigma)\|_{\Xap}<\varepsilon.
\end{equation}

For each $\varepsilon>0$ set $\stopp_{\varepsilon} = \sigma_{N(\varepsilon)}$, $\stopp_{\varepsilon}' = \sigma_{N(\varepsilon)}'$ and define the stopping time $\tau_{\varepsilon}$ by
\begin{align*}
\tau_{\varepsilon}:=\inf\Big\{t\in [\stopp_{\varepsilon},\sigma)\,
&:\,
\sup_{s\in [\stopp_{\varepsilon},t]}\|u(s)-u(\stopp_{\varepsilon})\|_{\Xap}\geq 2\varepsilon,\\
&\ \ \ \sup_{s\in [0,t]}\|u(t)\|_{\Xap}\geq M,\ \  |t-\lambda_{\varepsilon}|\geq \varepsilon\Big\},
\end{align*}
where $\inf\emptyset:=\sigma$. As in the proof of Theorem \ref{t:blow_up_criterion}\eqref{it:blow_up_norm_general_case_F_c_G_c}, $\tau_{\varepsilon}=\sigma$ a.s.\ on $\V\supseteq\V'$ for each $\varepsilon>0$. Moreover, setting $\U_{\varepsilon}:=\{\tau_{\varepsilon}>\stopp'_{\varepsilon}\}\in \F_{\stopp'_{\varepsilon}}\cap \F_{\tau_\varepsilon}$ and $u_{\varepsilon}:=\one_{\U_{\varepsilon}}u(\stopp_{\varepsilon}')$, one has $\U_{\varepsilon}\supseteq \V'$, $u_{\varepsilon}$ is $\F_{\stopp_{\varepsilon}'}$-measurable and
\begin{equation}
\label{eq:u_N_bounded_non_critical_proof}
u_{\varepsilon}\in \B_{L^{\infty}(\O;\Xap)}(M), \ \ \  \text{ for all }\varepsilon>0,\qquad  N=N(\varepsilon).
\end{equation}
Again, as in the proof of Theorem \ref{t:blow_up_criterion}\eqref{it:blow_up_norm_general_case_F_c_G_c}, combining Assumption \ref{H_a_stochnew} for $\ell=\a$, \eqref{eq:u_N_bounded_non_critical_proof} and Proposition \ref{prop:start_at_sigma_random_time}, one obtains
$(A(\cdot,u_{\varepsilon}),B(\cdot,u_{\varepsilon}))\in \mathcal{SMR}_{p,\a}^{\bullet}(\stopp_{\varepsilon}',T)$, and for each $\theta\in [0,\frac{1}{2})\setminus \{\frac{1+\a}{p}\}$,
\begin{equation}
\label{eq:max_proof_blow_up_not_critical}
\max\{ K^{\deter,\theta,p,\a}_{(A(\cdot,u_{\varepsilon}),B(\cdot,u_{\varepsilon}))}, K^{\stoc,\theta,p,\a}_{(A(\cdot,u_{\varepsilon}),B(\cdot,u_{\varepsilon}))}\}\leq K_{M,\eta}^{\theta}.
\end{equation}
Here $K_{M,\eta}^{\theta}$ does not depend on $\varepsilon$. Fix $\theta\in (\frac{1+\a}{p},\frac{1}{2})$ and set $K_{M,\eta} = K_{M,\eta}^{0}+K_{M,\eta}^{\theta}$. Let $\Sol_\varepsilon:=\Sol_{\stopp_{\varepsilon},(A(\cdot,u_{\varepsilon}),B(\cdot,u_{\varepsilon}))}$ be the solution operator associated with $(A(\cdot,u_{\varepsilon}),B(\cdot,u_{\varepsilon}))$. By Lemma \ref{l:estimate_Sol_op_Y} and \eqref{eq:max_proof_blow_up_not_critical}, the estimate \eqref{eq:estimate_Sol_X_spaces} holds with $\Sol_{\sigma}$ replaced by $\Sol_{\varepsilon}$ and $\K_{M,\eta}$ independent of $\varepsilon>0$.

{\em Step 1: There exist $\Constant$, $\bar{\varepsilon}>0$, $\zeta>1$ such that for all $\varepsilon\in (0,\bar{\varepsilon})$ and all stopping time $\tau$ satisfying}
\begin{equation}
\label{eq:step_1_critical_condition_tau}
\stopp_{\varepsilon}'\leq \tau\leq \tau_{\varepsilon} \ \text{ a.s.\ on }\U_{\varepsilon},
\end{equation}
\emph{one has}
\begin{align}\label{eq:eststep1uX}
\E\big[\one_{\U_{\varepsilon}}\|u\|_{\X(\stopp_{\varepsilon}',\tau)}^p\big] \leq \Constant+
C_{\varepsilon}\E\big[\one_{\U_{\varepsilon}}\|u\|_{\X(\stopp_{\varepsilon}',\tau)}^{p\zeta}\big],
\end{align}
{\em for some $C_{\varepsilon}>0$ independent of $u,\tau$ such that $\lim_{\varepsilon\downarrow 0}C_{\varepsilon}=0$.}

It suffices to prove the result with $\tau$ replaced by $\stopp_{\varepsilon}'\vee (\tau\wedge \sigma_n)$ for each $n\geq 1$. This has the advantage that all norms which appear here will be finite.

Set ${\varepsilon}_1:=1/(32\,\K_{M,\eta} L_M)$. Let $\varepsilon_2>0$ be such that $C(\varepsilon_2,M)\leq 1/(4\K_{M,\eta})$, where $C(\varepsilon_2,M)$ is as in Lemma \ref{l:F_G_bound_N}. Here we used the fact that since $F_c$ and $G_c$ are noncritical $\lim_{\varepsilon\downarrow 0}C(\varepsilon,M)  = 0$. Let $\bar{\varepsilon}:=\varepsilon_1\wedge \varepsilon_2$ and
fix $\varepsilon\in (0,\bar{\varepsilon})$. Set $\psi:=\one_{\O\setminus\U_{\varepsilon}}\stopp_{\varepsilon}'+\one_{\U_{\varepsilon}}\tau$. Since $\U_{\varepsilon}\in \F_{\tau_{\varepsilon}}\cap \F_{\stopp_{\varepsilon}'}$ and $\tau_{\varepsilon}\geq \tau$ a.s.\ on $\U_{\varepsilon}$, $\psi$ is a stopping time.
Let $\Lambda_\varepsilon:=\ll \stopp_{\varepsilon}',\tau\rro = [\stopp_{\varepsilon}',\tau)\times \U_{\varepsilon}$.
Reasoning as in the proof of Theorem \ref{t:blow_up_criterion} (see \eqref{eq:equation_start_at_u_N}-\eqref{eq:u_sum_v}), by Proposition \ref{prop:causality_phi_revised_2} and the linearity of $\Sol_\varepsilon$, a.s.\ on $\Lambda_{\varepsilon}$, one has
\begin{equation}\label{eq:RNcalculationtimeomega}
\begin{aligned}
\one_{\U_{\varepsilon}} u &= \Sol_\varepsilon(u_{\varepsilon},\one_{\Lambda_\varepsilon}F_c(\cdot,0)+f,\one_{\Lambda_\varepsilon}G_c(\cdot,0)+g)\\
& \quad +\Sol_\varepsilon(0,\one_{\Lambda_\varepsilon}(A(\cdot,u_{\varepsilon})-A(\cdot,u))u,\one_{\Lambda_\varepsilon}(B(\cdot,u_{\varepsilon})-B(\cdot,u))u)\\
&\quad +\Sol_\varepsilon(0,\one_{\Lambda_\varepsilon}F_{\Tr}(\cdot,u),\one_{\Lambda_\varepsilon}G_{\Tr}(\cdot,u))\\
&\quad + \Sol_\varepsilon(0,\one_{\Lambda_\varepsilon}(F_c(\cdot,u)-F_c(\cdot,0)),\one_{\Lambda_\varepsilon}(G_c(\cdot,u)-G_c(\cdot,0))\\
&:=I+II+III+IV.
\end{aligned}
\end{equation}
It remains to estimate each part separately. For notational convenience, we set $\y:=L^p(\O; L^p(\stopp_{\varepsilon}',T,w_{\a}^{\stopp_{\varepsilon}'};X_1)\cap \X(\stopp_{\varepsilon}',T))$.
By Lemma \ref{l:estimate_Sol_op_Y},
\begin{align*}
\|I\|_{\y}
&\leq \K_{M,\eta} (\|u_{\varepsilon}\|_{L^p(\O;\Xap)}+ \|F_c(\cdot,0)\|_{L^p(\Lambda_\varepsilon,w_{\a}^{\stopp_{\varepsilon}'};X_0)}+\|f\|_{L^p(\Lambda_\varepsilon,w_{\a}^{\stopp_{\varepsilon}'};X_0)}\\
&\qquad \quad+ \|G_c(\cdot,0)\|_{L^p(\Lambda_\varepsilon,w_{\a}^{\stopp_{\varepsilon}'};\g(H,X_{1/2}))}+\|g\|_{L^p(\Lambda_\varepsilon,w_{\a}^{\stopp_{\varepsilon}'};\g(H,X_{1/2}))})\\
& \leq \K_{M,\eta}C_{\eta}(C+M),
\end{align*}
where we used that \eqref{eq:u_N_bounded_non_critical_proof}. Moreover, as in \eqref{eq:estimate_II_blow_up_proof_X} and \eqref{eq:estimates_v_4} in the proof of Theorem \ref{t:blow_up_criterion}\eqref{it:blow_up_norm_general_case_F_c_G_c} and using that $\tau\leq \tau_{\varepsilon}$ on $\U_{\varepsilon}$, one easily obtains
\begin{equation*}
\|II\|_{\y}\leq \frac{1}{4}\|u\|_{L^p(\Lambda_\varepsilon,w_{\a}^{\stopp_{\varepsilon}'};X_1)},\qquad
\|III\|_{\y}\leq \K_{M,\eta}C(1+M).
\end{equation*}
To estimate $IV$ we employ Lemma \ref{l:F_G_bound_N}.

For $\delta>0$, set $C_{\delta}:=C(\delta,M)$, where $C(\cdot,\cdot)$ is as in Lemma \ref{l:F_G_bound_N}. By the choice of $\varepsilon_2$, we know that $C_{|b-a|}\K_{M,\eta}\leq\frac{1}{4}$ for each $a<b$ with $|b-a|\leq {\varepsilon}_2$.
By  Lemma \ref{l:estimate_Sol_op_Y},
\begin{align*}
\|IV\|_{\y}
&\leq \K_{M,\eta} (\|F_c(\cdot,u)-F_c(\cdot,0)\|_{L^p(\Lambda_\varepsilon,w_{\a}^{\stopp_{\varepsilon}'};X_0)}\\
&\qquad \qquad\qquad \quad+\|G_c(\cdot,u)-G_c(\cdot,0)\|_{L^p(\Lambda_\varepsilon,w_{\a}^{\stopp_{\varepsilon}'};\g(H,X_{1/2}))})\\
&\leq \frac{1}{4}\|u\|_{L^p(\U_{\varepsilon};\X(\stopp_{\varepsilon}',\tau))}+
C_{\varepsilon}(\E[\one_{\U_{\varepsilon}}\|u\|_{\X(\stopp_{\varepsilon}',\tau)}^{p\zeta}])^{1/p},
\end{align*}
where the last estimate follows from Lemma \ref{l:F_G_bound_N} and $\tau-\stopp_{\varepsilon}'\leq \tau_{\varepsilon}-\lambda_{\varepsilon}\leq \varepsilon$ (here we used \eqref{eq:step_1_critical_condition_tau}, the definition of $\tau_{\varepsilon}$ below \eqref{eq:non_critical_epsilon_assumption} and $\lambda_{\varepsilon}'\geq \lambda_{\varepsilon}$).

Combining the estimates we obtain the claim of Step 1.

{\em Step 2: Conclusion}. Fix $\varepsilon>0$ and set
\begin{equation*}
\ee_{\varepsilon}:=\V'\cap \big\{\sigma<T,\,
\|u\|_{\X(\stopp_{\varepsilon}',\sigma)}<\infty\big\},
\end{equation*}
Recall that $\stopp_{\varepsilon}'<\sigma$ on $\V'$. We claim that $\P(\ee_{\varepsilon})=0$. Indeed, by \eqref{eq:def_V_blow_up_non_critical_Xap}, one has $\lim_{t\uparrow \sigma} u(t)\;\text{exists in}\;\Xap$ a.s.\ on $\V'\subseteq \W$. Therefore,
\begin{align*}
\P(\ee_{\varepsilon})
&=\P\Big(\V'\cap\Big\{\sigma<T,\;\lim_{t\uparrow \sigma} u(t)\;\text{exists in}\;\Xap,\,
\|u\|_{\X(\stopp_{\varepsilon}',\sigma)}<\infty\Big\} \Big)\\
&=
\P\Big(\V'\cap\Big\{\sigma<T,\;\lim_{t\uparrow \sigma} u(t)\;\text{exists in}\;\Xap,\,
\|u\|_{\X(\sigma)}<\infty\Big\} \Big)\\
&\leq \P\Big(\sigma<T,\;\lim_{t\uparrow \sigma} u(t)\;\text{exists in}\;\Xap,\,
\|u\|_{\X(\sigma)}<\infty\Big)= 0,
\end{align*}
where in the last step we used Theorem \ref{t:blow_up_criterion}\eqref{it:blow_up_norm_general_case_F_c_G_c} and Lemma \ref{l:F_G_bound_N_C_cn} (or \eqref{eq:quasilinear_blow_up_X_norm}).

Next let $\bar{\varepsilon}, (C_{\varepsilon})_{\varepsilon\in (0,\bar{\varepsilon})}, \Constant,\zeta$ be as in Step 1. For each $\varepsilon\in (0,\bar{\varepsilon})$ and $x\in\R_+$, set $\varphi_{\varepsilon}(x):=x-C_{\varepsilon}x^{\zeta}$. Standard considerations show that $\varphi_{\varepsilon}$ has a unique maximum on $\R_+$ and $\lim_{\varepsilon\downarrow 0}M_{\varepsilon}=\infty$ where $\max_{\R_+}\varphi_{\varepsilon} =:M_{\varepsilon}$. Let $m_{\varepsilon}>0$ be the unique number such that $\varphi_{\varepsilon}(m_{\varepsilon})=M_{\varepsilon}$ and note that $\varphi_{\varepsilon}\geq 0$ on $[0,m_{\varepsilon}]$. Since $\lim_{\varepsilon\downarrow 0}M_{\varepsilon}=\infty$ and $\P(\V')>0$, we can choose $\varepsilon\in (0,\bar{\varepsilon})$ such that $M_{\varepsilon}\P(\V')>\Constant$.

Since $\P(\ee_{\varepsilon})=0$, for a.a.\ $\om\in \V'$ there exists a $t<\sigma(\om)$ such that
\begin{equation}
\label{eq:over_u_xi}
\|u(\cdot,\om)\|_{\X(\stopp_{\varepsilon}',t)}>m_{\varepsilon}^{1/p}.
\end{equation}
Define the stopping time $\mu_{\varepsilon}$ by
\begin{equation*}
\mu_{\varepsilon}:=\left\{
             \begin{array}{ll}
               \inf\Big\{t\in [\stopp_{\varepsilon}',\tau_{\varepsilon})\,:\,
\|u\|_{\X(\stopp_{\varepsilon}',t)}\geq m_{\varepsilon}^{1/p} \Big\}, & \hbox{on $\U_{\varepsilon}$;} \\
               \stopp_{\varepsilon}', & \hbox{on $\O\setminus \U_{\varepsilon}$.}
             \end{array}
           \right.
\end{equation*}
Here we set $\inf\emptyset :=\tau_{\varepsilon}$.
In this way $\|u\|_{\X(\stopp_{\varepsilon}',\mu_{\varepsilon})} \leq m_{\varepsilon}^{1/p}$ a.s.\ 

By the definition of $\tau_{\varepsilon}$ and \eqref{eq:over_u_xi}, one has $\mu_{\varepsilon}<\tau_{\varepsilon}=\sigma$ and  $\|u\|_{\X(\stopp_{\varepsilon}',\mu_{\varepsilon})} = m_{\varepsilon}^{1/p}$ a.s.\ on $\V'$.
Since $\varphi_{\varepsilon}|_{[0,m_{\varepsilon}]}\geq 0$ and $\U_{\varepsilon}\supseteq \V'$, we find
\begin{equation*}
\E[\one_{\U_{\varepsilon}} \varphi_{\varepsilon}(\|u\|_{\X(\stopp_{\varepsilon}',\mu_{\varepsilon})}^p)]
\geq \E[\one_{\V'} \varphi_{\varepsilon}(\|u\|_{\X(\stopp_{\varepsilon}',\mu_{\varepsilon})}^p)]
=
\E[\one_{\V'} \varphi_{\varepsilon}(m_{\varepsilon})]=
M_{\varepsilon}\P(\V').
\end{equation*}

Next observe that $\tau=\mu_{\varepsilon}$ satisfies condition \eqref{eq:step_1_critical_condition_tau}, and the quantities appearing in \eqref{eq:eststep1uX} are finite. Therefore, Step 1 implies the following converse estimate
\begin{equation}
\label{eq:C_1_uniform_bound_f}
\E[\one_{\U_{\varepsilon}} \varphi_{\varepsilon}(\|u\|_{\X(\stopp_{\varepsilon}',\mu_{\varepsilon})}^{p})]
\leq \Constant.
\end{equation}
This leads to a contradiction since $\Constant<M_{\varepsilon}\P(\V')$. Therefore $\P(\W)=0$ and this completes the proof of Theorem \ref{t:blow_up_criterion}\eqref{it:blow_up_non_critical_Xap}.
\end{proof}

\begin{proof}[Proof of Theorem \ref{thm:semilinear_blow_up_Serrin}\eqref{it:blow_up_semilinear noncritical_stochastic}]
We use the same method as in the proof of Theorem \ref{t:blow_up_criterion}\eqref{it:blow_up_non_critical_Xap}.
Suppose that $\P(\W)>0$ where
$$
\W:=\big\{\sigma<T,\,\sup_{t\in [0,\sigma)}\|u(t)\|_{\Xap}<\infty\big\}.
$$
As before (see \eqref{eq:non_critical_epsilon_assumption}) one can check that there exists a set $\V$ with positive probability such that for all $\varepsilon>0$ there exists an $N(\varepsilon)\in\N$ for which
\begin{equation}
\label{eq:uniformly_bound_M_u_blow_up_semilinear}
\sup_{t\in [0,\sigma)}\|u(t)\|_{\Xap}<M,\quad \text{ and }\quad |\sigma_{N(\varepsilon)}-\sigma|< \varepsilon.
\end{equation}
Now the estimate \eqref{eq:eststep1uX} holds again. Indeed, in the proof the fact that $\lim_{t\uparrow \sigma}u(t)$ exists in $\Xap$ was only used to estimate $II$. In the semilinear case, $II=0$, and the bound in \eqref{eq:uniformly_bound_M_u_blow_up_semilinear} can be used to reproduce the estimates for $I,III,IV$. The proof of Step 2 of Theorem \ref{t:blow_up_criterion}\eqref{it:blow_up_non_critical_Xap} can be used to complete the proof.
\end{proof}

\subsection{Proofs of Theorems \ref{t:blow_up_criterion}\eqref{it:blow_up_norm_general_case_quasilinear_Xap_Pruss}, \ref{thm:semilinear_blow_up_Serrin}\eqref{it:blow_up_semilinear_serrin_Pruss_modified} and \ref{thm:semilinear_blow_up_Serrin_refined}}\label{ss:thmblowuphardespart_Serrin}
In this subsection, we prove the remaining results stated in Subsection \ref{ss:main_result_blow_up}. We begin with the proof of Theorem \ref{thm:semilinear_blow_up_Serrin}\eqref{it:blow_up_semilinear_serrin_Pruss_modified} which will guide us into the remaining ones. The advantage is that in the semilinear case the argument used to control the nonlinearities is more transparent. Additional changes are needed to get Theorems \ref{t:blow_up_criterion}\eqref{it:blow_up_semilinear_serrin_Pruss_modified} and \ref{thm:semilinear_blow_up_Serrin_refined}.

Before starting with the proofs we introduce some notation which will be used only in this subsection and allows us to give an extension of Theorem \ref{thm:semilinear_blow_up_Serrin_refined}, i.e.\ Proposition \ref{prop:serrin_Pruss_general_form} below.
Let \ref{HFcritical}-\ref{HGcritical} be satisfied and fix $j\in \{1,\dots,m_F+m_G\}$. If $\rho_j>0$, then we define $\beta_j^{\star}, \varphi_j^{\star}\in (1-\frac{1+\a}{p},1)$ as follows:
\begin{itemize}
\item If $\rho_j (\varphi_j-1+\frac{1+\a}{p})+\varphi_j\geq 1$, then $\varphi^{\star}_j=\varphi_j$, and $\beta_j^{\star}=1-\rho_j (\varphi_j-1+\frac{1+\a}{p})\in [\beta_j,\varphi_j]$;
\item  If $\rho_j (\varphi_j-1+\frac{1+\a}{p})+\varphi_j< 1$, then $\beta_j^{\star}=\varphi^{\star}_j=1-\frac{\rho_j}{\rho_j+1}\frac{1+\a}{p}\in (\varphi_j,1)$.
\end{itemize}
The previous definition implies that
\begin{equation}
\label{eq:star_varphi_beta_j_equality_proof_serrin}
\rho_j \Big(\varphi_j^{\star}-1+\frac{1+\a}{p}\Big) +\beta_j^{\star} =1 \ \  \ \text{ for all }\ j\in \{1,\dots,m_F+m_G\}.
\end{equation}
If $\rho_j=0$, then with a slight abuse of notation we replace $\rho_j$ by $\varepsilon_j>0$ such that $\varepsilon_j(\varphi_j -1+\frac{1+\a}{p})+\varphi_j<1,$ and $\varphi_j^{\star} := \beta_j^{\star} := 1-\frac{\varepsilon_j}{\varepsilon_j+1}\frac{1+\a}{p}\geq \varphi_j$. By choosing $\varepsilon_j$ small enough (e.g.\ $\varepsilon_j<\a+1$) one always has $\beta_j^{\star}=\varphi_j^{\star}>1-\frac{1+\a}{p}\frac{1+\a}{2+\a}$, which is needed in Lemma \ref{lem:interpolationineqMR0}\eqref{it:interpolationineqMR01} below.

In all cases we have $\varphi_j^\star\geq \varphi_j$ and $\beta_j^\star\geq \beta_j$. Therefore, by $X_{\Phi}\hookrightarrow X_{\phi}$ for $0<\phi\leq \Phi<1$, and by \ref{HFcritical}-\ref{HGcritical} for all $n\geq 1$ a.s.\ for all $x\in X_1$ s.t.\ $\|x\|_{\Xap}\leq n$,
\begin{equation}
\label{eq:F_c_G_c_estimate_star_preliminary}
\|F_c(\cdot,x)\|_{X_{0}}
+
\|G_c(\cdot, x)\|_{\g(H,X_{1/2})}
\leq C_{c,n}' \sum_{j=1}^{m_F+m_G}(1+\|x\|_{X_{\varphi^{\star}_j}}^{\rho_j})\|x\|_{X_{\beta^{\star}_j}}+C_{c,n}'.
\end{equation}
where $C'_{c,n}=C' C_{c,n}$ with $C'$ depending only on $X_0,X_1,\beta_j,\beta_j^{\star},\varphi_j,\varphi_j^{\star}$.

Next, we partially repeat the construction of the $\X$-space (see \eqref{eq:def_X_space}) using $(\rho_j,\beta_j^{\star},\varphi_j^{\star})$ instead of $(\rho_j,\beta_j,\varphi_j)$. As in Lemma \ref{l:F_G_bound_N_C_cn} this will be needed to estimate the nonlinearities $F_c,G_c$. Similar to \eqref{eq:rr'}, for all $j\in \{1,\dots,m_F+m_G\}$ we set
\begin{equation}
\label{eq:xi_xi'}
\frac{1}{\xi_j' }:=\frac{\rho_j (\varphi_j^{\star}-1+(1+\a)/p )}{(1+\a)/p},
\ \ \ \ \text{and }\ \ \ \
\frac{1}{\xi_j}:=\frac{\beta_j^{\star}-1+(1+\a)/p }{(1+\a)/p}.
\end{equation}
Since $\varphi_j^{\star},\beta_j^{\star}\in (1-\frac{1+\a}{p},1)$, we get $\frac{1}{\xi_j'},\frac{1}{\xi_j}>0$. Moreover, \eqref{eq:star_varphi_beta_j_equality_proof_serrin} yields $\frac{1}{\xi_j' }+\frac{1}{\xi_j}=1$ and therefore $\frac{1}{\xi_j'},\frac{1}{\xi_j}<1$. Parallel to \eqref{eq:def_X_space}, for all $0\leq a<b\leq T$ we define
\begin{equation}
\label{eq:def_X_star_space}
\X^{\star}({a,b}):=\Big(\bigcap_{j=1}^{m_F+m_G} L^{p \xi_j}(a,b,w_{\a}^a;X_{\beta^{\star}_j})\Big) \cap
\Big(\bigcap_{j=1}^{m_F+m_G}  L^{\rho_j p \xi_j'}(a,b,w_{\a}^a;X_{\varphi^{\star}_j})\Big).
\end{equation}
By \eqref{eq:F_c_G_c_estimate_star_preliminary} and H\"{o}lder's inequality we obtain that for all $M\geq 1$ and all $h\in \X^{\star}(a,b)\cap C([a,b];\Xap)$ which satisfy $\|h\|_{ C([a,b];\Xap)}\leq M$,
\begin{equation}
\begin{aligned}
\label{eq:F_c_G_c_estimate_star}
&\|F_c(\cdot,h)\|_{L^p(a,b,w_{\a}^{a};X_{0})}
+
\|G_c(\cdot,h)\|_{L^p(a,b,w_{\a}^{a};\g(H,X_{1/2}))}\\
&\
\leq C_{c,M}''\Big[\sum_{j=1}^{m_F+m_G}(1+\|h\|_{L^{\rho_j p \xi_j' }(a,b,w_{\a}^{a};X_{\varphi^{\star}_j})}^{\rho_j})\|h\|_{L^{p \xi_j}(a,b,w_{\a}^{a};X_{\beta^{\star}_j})}+1\Big],
\end{aligned}
\end{equation}
where $C_{c,M}'' = C_{c,M}' (1\vee \|1\|_{L^p(0,T,w_{\a})}\vee \max_j \|1\|_{L^{p\xi_j'}(0,T,w_{\a})})$.

The key to the proofs of the blow-up criteria is interpolation inequalities. In order to simplify the notation for $\theta\in [0,1]$ and $0\leq a<b\leq T$, we set
\begin{equation}
\label{eq:Sz_def}
\Sz^{\theta,\a}(a,b):=
\hz^{\theta,p}(a,b,w_{\a}^{a};X_{1-\theta})\cap L^p(a,b,w_{\a}^{a};X_1).
\end{equation}
The reason for using the space $\hz^{\theta,p}$ instead of $H^{\theta,p}$ is that Proposition \ref{prop:start_at_s}\eqref{it:start_at_s4} allows to obtain uniformity of the estimates in $b-a$ in the proof of Theorem \ref{thm:semilinear_blow_up_Serrin}\eqref{it:blow_up_semilinear_serrin_Pruss_modified} below. By \eqref{eq:LMVresult} there are no trace restrictions when $\theta<\frac{1+\a}{p}$.
\begin{lemma}[Interpolation inequality]\label{lem:interpolationineqMR0}
Let $p\in (1, \infty)$, $\kappa\in [0,p-1)$, $\psi\in (1-\frac{1+\a}{p},1)$, and set $\zeta=(1+\a)/\big(\psi-1+\frac{1+\a}{p}\big)$. Then there exists a $\theta_0\in [0,\frac{1+\a}{p})$ such that for all $\theta\in [\theta_0, 1)$ there is a constant $C>0$ such that the following estimate holds for all $0\leq a<b\leq T$ and all $f\in \Sz^{\theta,\a}(a,b)\cap L^{\infty}(a,b;\Xap)\cap L^p(a,b;X_{1-\frac{\a}{p}})$,
\begin{equation}
\label{eq:interpolation_inequality_serrin_zeta}
\|f\|_{L^{\zeta}(a,b,w_{\a};X_{\psi})}\leq C \|f\|_{L^{\infty}(a,b;\Xap)}^{1-\phi}\|f\|_{\Sz^{\theta,\a}(a,b)}^{(1-\delta)\phi}\|f\|_{L^{p}(a,b;X_{1-\frac{\a}{p}})}^{\delta\phi},
\end{equation}
where we can take $\delta,\phi\in [0,1]$ as follows:
\begin{enumerate}[{\rm (1)}]
\item\label{it:interpolationineqMR01} $\delta = 1-\frac{p}{1+\a}\Big(\psi-1+\frac{1+\a}{p}\Big)$ and $\phi=1$ if $\psi\in (1-\frac{1+\a}{p}\frac{1+\a}{2+\a},1)$ and $\kappa>0$;
\item\label{it:interpolationineqMR02} $\delta = \frac{\kappa}{\kappa+1}$ and $\phi=p\Big(\psi-1+\frac{1+\a}{p}\Big)$ if $\psi\in (1-\frac{1+\a}{p},1-\frac{\a}{p}]$ and $\kappa>0$;
\item\label{it:interpolationineqMR03} $\delta=1$ and $\phi=p\Big(\psi-1+\frac{1}{p}\Big)$ if $\kappa=0$.
\end{enumerate}
In particular, in each of the above cases
\begin{equation}\label{eq:identity1mindeltaphi}
(1-\delta)\phi \leq \frac{p}{1+\a}\Big(\psi-1+\frac{1+\a}{p}\Big).
\end{equation}
\end{lemma}

Note that \eqref{it:interpolationineqMR01}  and \eqref{it:interpolationineqMR02} have a nontrivial overlap since $1-\frac{\kappa}{p}>1-\frac{1+\a}{p}\frac{1+\a}{2+\a}$.

\begin{proof}
By a translation argument we can suppose that $a=0$ and we write $t$ instead of $b$ below.
Let us begin by making some reductions. By Lemma \ref{l:mixed_derivative}\eqref{it:mixed_derivative_H} for all $\theta\in [\theta_0,1)$ we have
$$
\hz^{\theta,p}(\I_t,w_{\a};X_{1-\theta})\cap L^p(\I_t;w_{\a};X_{1})
\hookrightarrow
\hz^{\theta_0,p}(\I_t,w_{\a};X_{1-\theta_0})\cap L^p(\I_t;w_{\a};X_{1}).
$$
Therefore, it suffices to consider $\theta = \theta_0$ in all cases.

\eqref{it:interpolationineqMR01}: First consider $\psi\in (1-\frac{\kappa}{p},1)$. For $\theta\in (0,1-\psi)$ one has $\theta<\a/p<(\a+1)/p$. Setting $\delta:=(1-\theta-\psi)/(\frac{\a}{p}-\theta)\in(0,1]$, we find
\begin{equation}
\label{eq:emb_serrin_1_interpolation_proof}
\begin{aligned}
\|f\|_{L^{\zeta}(\I_t,w_{\a};X_{\psi})} & \stackrel{(i)}{\lesssim} \|f\|_{\hz^{\theta(1-\delta),p}(\I_t,w_{\a(1-\delta)};X_{\psi})}
\stackrel{(ii)}{\lesssim} \|f\|_{\hz^{\theta,p}(\I_t,w_{\a};X_{1-\theta})}^{1-\delta} \|f\|_{L^p(\I_t;X_{1-\frac{\a}{p}})}^{\delta}.
\end{aligned}
\end{equation}
In $(ii)$ we used Lemma \ref{l:mixed_derivative}\eqref{it:mixed_derivative_H}. In $(i)$ we used Proposition \ref{prop:change_p_q_eta_a}\eqref{it:Sob_embedding} with Sobolev exponents
\[\theta(1-\delta) - \frac{\a(1-\delta)+1}{p} = -\Big(\frac{\a}{p}-\theta\Big) (1-\delta) -\frac{1}{p} \stackrel{(a)}{=} -\Big(\psi-1+\frac{\a+1}{p}\Big) \stackrel{(b)}{=}  -\frac{\a+1}{\zeta},\]
where we used $1-\delta=(\psi+\frac{\a}{p}-1)/(\frac{\a}{p}-\theta)$ in (a) and the definition of $\zeta$ in (b). The condition $\frac{\a}{ \zeta}\leq \frac{\a}{p}(1-\delta)$ of Proposition \ref{prop:change_p_q_eta_a}\eqref{it:Sob_embedding} gives the following restriction on the parameter $\theta$:
\begin{equation}
\label{eq:weight_condition_Sob_embeddings}
\frac{1}{1+\a}\Big(\psi-1+\frac{1+\a}{p}\Big)\leq
\frac{1}{p}\frac{\psi-1+\frac{\a}{p}}{\frac{\a}{p}-\theta}.
\end{equation}
One can check that \eqref{eq:weight_condition_Sob_embeddings} is satisfied with strict inequality for $\theta = 1-\psi$ (since $\psi<1$), and \eqref{eq:weight_condition_Sob_embeddings} is not satisfied for $\theta=0$. Therefore, by continuity and linearity there is a unique $\theta\in (0,1-\psi)$ such that equality holds in \eqref{eq:weight_condition_Sob_embeddings}.
Now \eqref{eq:emb_serrin_1_interpolation_proof} implies \eqref{eq:interpolation_inequality_serrin_zeta} with $\phi=1$, and
\begin{equation*}
1-\delta=\frac{\big(\psi-1+\frac{\a}{p}\big)}{\frac{\a}{p}-\theta}\stackrel{\eqref{eq:weight_condition_Sob_embeddings}}{=}
\frac{p}{1+\a}\Big(\psi-1+\frac{1+\a}{p}\Big)
\end{equation*}
which coincides with the choice of $\delta$ in \eqref{it:interpolationineqMR01}.

Next consider $\psi\in (1-\frac{1+\a}{p}\frac{1+\a}{2+\a},1-\frac{\kappa}{p})$. Then $1-\psi>\kappa/p$ and we can apply the same argument starting with $\theta\in (1-\psi,(1+\kappa)/p)$ and setting $\delta:=(\theta+\psi-1)/(\theta-\frac{\a}{p})$.
Now the following variant of \eqref{eq:weight_condition_Sob_embeddings} has to be considered
\begin{equation}\label{eq:weight_condition_Sob_embeddingsalt}
\frac{1}{1+\a}\Big(\psi-1+\frac{1+\a}{p}\Big)\leq
\frac{1}{p}\frac{1-\psi-\frac{\a}{p}}{\theta - \frac{\a}{p}}.
\end{equation}
This time \eqref{eq:weight_condition_Sob_embeddingsalt} is satisfied with strict inequality for $\theta = 1-\psi$ (since $\psi<1$), and \eqref{eq:weight_condition_Sob_embeddingsalt} is not satisfied for $\theta=\frac{1+\a}{p}$ (since $\psi>1-\frac{1+\a}{p}\frac{1+\a}{2+\a}$). Therefore, there is a unique $\theta\in (1-\psi,\frac{1+\a}{p})$ such that equality holds in \eqref{eq:weight_condition_Sob_embeddingsalt}. The rest of the proof is identical to \eqref{it:interpolationineqMR01}.

The case $\psi = 1-\frac{\a}{p}$ is contained in \eqref{it:interpolationineqMR02} and will be proved below.

\eqref{it:interpolationineqMR02}: First consider $\psi\in (1-\frac{1+\a}{p},1-\frac{\a}{p})$. Here we use a two step interpolation.
By real interpolation \cite[Theorems 3.5.3 and 4.7.1]{BeLo} and the definition of $\phi$ we obtain
\begin{equation}
\label{eq:Xap_X_phi_interpolation}
\begin{aligned}
(\Xap,X_{1-\frac{\a}{p}})_{\phi,1}& \hookrightarrow \Big((X_0, X_1)_{1-\frac{1+\a}{p},p},(X_0, X_1)_{1-\frac{\a}{p},\infty}\Big)_{\phi,1}
\\ & = (X_0,X_1)_{\psi,1} \hookrightarrow X_{\psi},
\end{aligned}
\end{equation}
and hence $\|x\|_{X_{\psi}}\lesssim \|x\|_{\Xap}^{1-\phi}\|x\|_{X_{1-\frac{\a}{p}}}^{\phi}$ for all $x\in X_{1-\frac{\a}{p}}$. Since $\phi \zeta=p(1+\a)$, the latter estimate implies
\begin{equation}
\label{eq:estimate_Xap_varphi_small_trace_Serrin}
\|f\|_{L^{\zeta}(\I_t;w_{\a};X_{\psi})}\lesssim
\|f\|_{L^{\infty}(\I_t;\Xap)}^{1-\phi}
\|f\|_{L^{p(1+\a)}(\I_t,w_{\a};X_{1-\frac{\a}{p}})}^{\phi}.
\end{equation}
Reasoning as in \eqref{eq:emb_serrin_1_interpolation_proof} we get
\begin{equation}
\begin{aligned}
\label{eq:hapone_embedding_sharp}
\|f\|_{L^{p(\a+1)}(\I_t,w_{\a};X_{1-\frac{\a}{p}})}&\lesssim \|f\|_{\hz^{\frac{\a}{p}(1-\delta),p}(\I_t,w_{\a(1-\delta)};X_{1-\frac{\a}{p}})}
\\ & \lesssim \|f\|_{\hz^{\frac{\a}{p},p}(\I_t,w_{\a};X_{1-\frac{\a}{p}})}^{1-\delta} \|f\|_{L^p(\I_t;X_{1-\frac{\a}{p}})}^{\delta},
\end{aligned}
\end{equation}
where for the Sobolev embedding of Proposition \ref{prop:change_p_q_eta_a}\eqref{it:Sob_embedding} we used
\[\frac{\a}{p}(1-\delta) - \frac{\a(1-\delta)+1}{p} = -\frac{1}{p}= -\frac{\a+1}{p(\a+1)},\]
and
$\frac{\a}{p(\a+1)} = \frac{\a}{p}(1-\delta)$ (since $\delta = \frac{\a}{\a+1}$). Combining \eqref{eq:estimate_Xap_varphi_small_trace_Serrin} and \eqref{eq:hapone_embedding_sharp} gives \eqref{eq:interpolation_inequality_serrin_zeta}.

Finally, the case $\psi=1-\frac{\a}{p}$ of \eqref{it:interpolationineqMR02} follows from \eqref{eq:hapone_embedding_sharp} with $\phi=1$ and $\delta$ as before.

\eqref{it:interpolationineqMR03}: This follows in a similar way as in \eqref{eq:Xap_X_phi_interpolation} and \eqref{eq:estimate_Xap_varphi_small_trace_Serrin}.
\end{proof}

\begin{proof}[Proof of Theorem \ref{thm:semilinear_blow_up_Serrin}\eqref{it:blow_up_semilinear_serrin_Pruss_modified}]
As usual, we set $s=0$ and we mainly focus on the case $p>2$ as the case $p=2$ follows from Theorem \ref{thm:semilinear_blow_up_Serrin}\eqref{it:blow_up_stochastic_semilinear}, cf.\ Remark \ref{r:blow_up_semilinear_X_norm}\eqref{it:case_p_2_simpler}.
For the reader's convenience we split the proof into several steps. In Step 1 we apply Lemma \ref{lem:interpolationineqMR0} to obtain interpolation inequalities. In Step 2 we set-up the proof by contradiction and in Step 3 we derive the contradiction. Recall that from Theorem \ref{t:local_s} we obtain that a.s.\ for all $\theta\in [0,\frac{1}{2})$ and $0\leq a<b<\sigma$,
\begin{align}\label{eq:regublowup}
u\in H^{\theta,p}(a,b,w_{\a}^a;X_{1-\theta}) \cap C([a,b];\Xap).
\end{align}
In case $a>0$, we also used Proposition \ref{prop:change_p_q_eta_a}\eqref{it:loc_embedding}.

\textit{Step 1: Interpolation inequalities}.
Since $\varphi^{\star}_j, \beta^{\star}_j\in (1-\frac{1+\a}{p},1)$ and
\begin{equation}
\label{eq:parameters_serrin_integrability_X_space}
 \rho_j p \xi'_j=\frac{1+\a}{(\varphi_j^{\star}-1+\frac{1+\a}{p})}
\ \ \text{ and }\ \
p \xi_j = \frac{1+\a}{(\beta^{\star}_j-1+\frac{1+\a}{p})},
\end{equation}
the exponents $\rho_j p \xi'_j$ and $p\xi_j$ satisfy the conditions of Lemma \ref{lem:interpolationineqMR0}. Therefore, there exist $\theta\in [0,\frac{1+\a}{p})$ and $C>0$ such that for all $j\in \{1,\dots,m_F+m_G\}$ there are $\phi_{1,j},\phi_{2,j},\delta_{1,j},\delta_{2,j}\in (0,1]$ such that a.s.\ for all $0\leq a<b<\sigma$
\begin{align}
\label{eq:interpolation_inequality_serrin_1}
\|u\|_{L^{\rho_j p \xi_j'} (a,b,w_{\a}^{a};X_{\varphi^{\star}_j})}
&\leq C\|u\|_{L^{\infty}(a,b;\Xap)}^{1-\phi_{1,j}}
\|u\|_{\Sz^{\theta,\a}(a,b)}^{(1-\delta_{1,j})\phi_{1,j}}
\|u\|_{L^p(a,b;X_{1-\frac{\a}{p}})}^{\delta_{1,j} \phi_{1,j}},\\
\label{eq:interpolation_inequality_serrin_2}
\|u\|_{L^{p \xi_j} (a,b,w_{\a}^{a};X_{\beta^{\star}_j})}
&\leq C\|u\|_{L^{\infty}(a,b;\Xap)}^{1-\phi_{2,j}}
\|u\|_{\Sz^{\theta,\a}(a,b)}^{(1-\delta_{2,j})\phi_{2,j}}
\|u\|_{L^p(a,b;X_{1-\frac{\a}{p}})}^{\delta_{2,j} \phi_{2,j}},
\end{align}
Moreover, by \eqref{eq:star_varphi_beta_j_equality_proof_serrin} and \eqref{eq:identity1mindeltaphi}, $\rho_j \phi_{1,j} (1-\delta_{1,j})+\phi_{2,j}(1-\delta_{2,j})\leq 1$. Note that in
\eqref{eq:interpolation_inequality_serrin_1} and \eqref{eq:interpolation_inequality_serrin_2} we used \eqref{eq:regublowup} and \eqref{eq:LMVresult}.

\textit{Step 2: Setting up the proof by contradiction}.
By contradiction we assume that $\P(\W)>0$ where
\begin{equation}
\label{eq:W_semilinear_pruss_modified}
\W:=\Big\{\sigma<T,\,\sup_{t\in [0,\sigma)}\|u(t)\|_{\Xap}<\infty,
\,\|u\|_{L^p(\I_{\sigma};X_{1-\frac{\a}{p}})}<\infty\Big\}.
\end{equation}
By Egorov's theorem and the fact that $\sigma>0$ a.s., there exist $\eta>0$, $M\geq 1$, $\F_{\sigma}\ni \V\subseteq \W$ such that $\P(\V)>0$, $\sigma>\eta$ a.s. on $\V$,
\begin{equation}
\label{eq:equal_zero_by_limit_serrin_modified}
\sup_{\V}\sup_{t\in [0,\sigma)}\|u(t)\|_{\Xap}\leq M,\ \ \text{ and }\ \
\lim_{n\to \infty} \sup_{\V}\|u\|_{L^{p}(\sigma_n,\sigma;X_{1-\frac{\a}{p}})}=0.
\end{equation}
Reasoning as in the proof of Theorem \ref{t:blow_up_criterion}\eqref{it:blow_up_norm_general_case_F_c_G_c},
employing Proposition \ref{prop:predictability}, there exist a sequence of stopping times $(\sigma_n')_{n\geq 1}$ and a set $\F_{\sigma}\ni\V'\subseteq \V$ with positive measure such that $\sigma_n\leq \sigma_n'$, $\sigma_n'\geq \eta$ a.s., and $\sigma_n'<\sigma$ a.s.\ on $\V'$ for all $n\geq 1$. Finally, by \eqref{eq:equal_zero_by_limit_serrin_modified} and the fact that $\sigma_n'\geq \sigma_n$, for all $\varepsilon>0$ there exists an $N(\varepsilon)>0$ such that
\begin{equation}
\label{eq:uniform_bound_X_phi_serrin_modified}
\sup_{t\in [0,\sigma)}\|u(t)\|_{\Xap}\leq M,\ \ \text{ and }\ \
\|u\|_{L^p(\sigma_{N(\varepsilon)}',\sigma;X_{1-\frac{\a}{p}})}\leq \varepsilon,\ \ \text{ a.s.\ on }\V'.
\end{equation}

\textit{Step 3: In this step we prove the desired contradiction}.
We begin by partially repeating the argument used in Step 2 in the proof of Theorem \ref{t:blow_up_criterion}\eqref{it:blow_up_norm_general_case_F_c_G_c}. Let $\theta$ and $M,\eta$ be as in Steps 1 and 2, respectively.
By Assumption \ref{H_a_stochnew} there exists a $K_{M,\eta}>0$ such that for all $t\in (\eta,T)$ one has
\begin{equation}
\label{eq:costant_stochastic_maximal_regularity_Serrin_constant_initial_t}
\max\{ K^{\deter,\theta,p,\a}_{(A,B)}(t,T),K^{\stoc,\theta,p,\a}_{(A,B)}(t,T)\}\leq K_{M,\eta}.
\end{equation}
Since $\sigma'_n$ takes values in a finite set, \eqref{eq:costant_stochastic_maximal_regularity_Serrin_constant_initial_t} and Proposition \ref{prop:start_at_sigma_random_time} imply that $(A,B)\in \mathcal{SMR}^{\bullet}_{p,\a}(\sigma_n',T)$ and
\begin{equation}
\label{eq:costant_stochastic_maximal_regularity_Serrin}
\max\{ K^{\deter,\theta,p,\a}_{(A,B)}(\sigma_n',T),
K^{\stoc,\theta,p,\a}_{(A,B)}(\sigma_n',T)\}\leq K_{M,\eta},
\quad \text{ for all }n\geq 1 .
\end{equation}

For notational convenience, for each $\varepsilon>0$ we set $\stopp_{\varepsilon}:=\sigma_{N(\varepsilon)}$ and $\stopp_{\varepsilon}':=\sigma_{N(\varepsilon)}'$.
For each $\ell\geq 1$, let us define the following stopping time
\begin{equation}
\begin{aligned}
\label{eq:tau_varepsilon_j}
\tau_{\varepsilon,\ell}:=\inf\big\{t\in [\stopp_{\varepsilon}',\sigma)
\,
&:\,
\|u\|_{L^p(\stopp_{\varepsilon}',t;X_{1-\frac{\a}{p}})}\geq \varepsilon,
\\
&\;\;
\|u(t)\|_{\Xap}\geq M,\,
\,\|u\|_{\X(\stopp_{\varepsilon}',t)\cap L^p(\stopp_{\varepsilon}',t,w_{\a};X_1)}\geq \ell\big\},
\end{aligned}
\end{equation}
on $\{\stopp_{\varepsilon}'<\sigma\}$ and $
\tau_{\varepsilon,\ell}=\stopp_{\varepsilon}'$ on $\{\stopp_{\varepsilon}'\geq \sigma\}$. Due to \eqref{eq:uniform_bound_X_phi_serrin_modified}, $\lim_{\ell \to \infty}\tau_{\varepsilon,\ell}=\sigma$ a.s.\ on $\V'$. Next, fix $L\geq 1$ so large enough that
\begin{equation}
\label{eq:V_double_prime}
\P(\V'')>0,
 \text{ where }
\V'':=\V'\cap \U,\ \
\U:=\{\sigma>\stopp_{\varepsilon}',\,\|u(\stopp_{\varepsilon}')\|_{\Xap}\leq L\}\in \F_{\stopp_{\varepsilon}'}.
\end{equation}
By \eqref{eq:tau_varepsilon_j}, for all $\varepsilon>0$ and $\ell\geq 1$, we have a.s.\
\begin{equation}
\label{eq:Serrin_proof_u_bounds}
\|u\|_{L^p(\stopp_{\varepsilon}',\tau_{\varepsilon,\ell};X_{1-\frac{\a}{p}})}\leq \varepsilon, \ \ \
 \|u\|_{\X(\stopp_{\varepsilon}',\tau_{\varepsilon,\ell}) \cap L^p(\stopp_{\varepsilon}',\tau_{\varepsilon,\ell},w_{\a};X_1)}\leq \ell, \ \ \
\|u\|_{L^{\infty}(\stopp_{\varepsilon}',\tau_{\varepsilon,\ell};\Xap)}\leq M .
\end{equation}
Combining the last inequality in \eqref{eq:Serrin_proof_u_bounds} and \ref{HFcritical}-\ref{HGcritical}, we get $\|F_{\Tr}(\cdot,u)\|_{X_0}+\|G_{\Tr}(\cdot,u)\|_{\g(H,X_{1/2})}\leq C_{\Tr,M}(1+M)$ a.e.\ on $\ll \stopp_{\varepsilon}',\tau_{\varepsilon,\ell}\rro$. In particular, for some $R_{\Tr,M}>0$ independent of $\ell\geq 1$, we have a.s.
\begin{equation}
\label{eq:F_G_tr_semilinear_Serrin_modified}
\|F_{\Tr}(\cdot,u)\|_{L^p(\stopp_{\varepsilon}',\tau_{\varepsilon,\ell},w_{\a}^{\stopp_{\varepsilon}'};X_0)}+\|G_{\Tr}(\cdot,u)\|_{L^p(\stopp_{\varepsilon}',\tau_{\varepsilon,\ell},w_{\a}^{\stopp_{\varepsilon}'};\g(H,X_{1/2}))}\leq R_{\Tr,M}.
\end{equation}
Finally, we estimate $F_c,G_c$. By \eqref{eq:F_c_G_c_estimate_star} we get, for all $\varepsilon>0$, $\ell\geq 1$ and a.s.
\begin{equation}
\label{eq:F_G_c_bounds_serrin_1}
\begin{aligned}
&\|F_c(\cdot,u)\|_{L^p(\stopp_{\varepsilon}',\tau_{\varepsilon,\ell},w_{\a}^{\stopp_{\varepsilon}'};X_0)}
+
\|G_c(\cdot,u)\|_{L^p(\stopp_{\varepsilon}',\tau_{\varepsilon,\ell},w_{\a}^{\stopp_{\varepsilon}'};\g(H,X_{1/2}))}\\
&\leq
C_{c,M}''\Big[\sum_{j=1}^{m_F+m_G}
(1+\|u\|_{L^{\rho_j p \xi_j'}(\stopp_{\varepsilon}',\tau_{\varepsilon,\ell},w_{\a}^{\stopp_{\varepsilon}'};X_{\varphi^{\star}_j})}^{\rho_j})
\|u\|_{L^{p \xi_j}(\stopp_{\varepsilon}',\tau_{\varepsilon,\ell},w_{\a}^{\stopp_{\varepsilon}'};X_{\beta^{\star}_j})}+1\Big].
\end{aligned}
\end{equation}
Fix $j\in \{1,\dots,m_F+m_G\}$.
From \eqref{eq:interpolation_inequality_serrin_1} and \eqref{eq:interpolation_inequality_serrin_2} we get a.s.\
\begin{equation}
\label{eq:F_G_c_bounds_serrin_2}
\begin{aligned}
&\|u\|_{L^{\rho_j p \xi_j'}(\stopp_{\varepsilon}',\tau_{\varepsilon,\ell},w_{\a}^{\stopp_{\varepsilon}'};X_{\varphi_j^{\star}})}^{\rho_j}
\|u\|_{L^{p \xi_j}(\stopp_{\varepsilon}',\tau_{\varepsilon,\ell},w_{\a}^{\stopp_{\varepsilon}'};X_{\beta_j^{\star}})}\\
&\leq C \one_{\stopp_{\varepsilon}'<\tau_{\varepsilon,\ell}} \|u\|_{L^{\infty}(\stopp_{\varepsilon}',\tau_{\varepsilon,\ell};\Xap)}^{\rho_j (1-\phi_{1,j})+(1-\phi_{2,j})}
\|u\|_{\Sz^{\theta,\a}(\stopp_{\varepsilon}',\tau_{\varepsilon,\ell})}^{\rho_j \phi_{1,j}(1-\delta_{1,j})+\phi_{2,j}(1-\delta_{2,j})}
\|u\|_{L^{p}(\stopp_{\varepsilon}',\tau_{\varepsilon,\ell};X_{1-\frac{\a}{p}})}^{\rho_j \delta_{1,j}\phi_{1,j}+\delta_{2,j}\phi_{2,j}}\\
&\stackrel{(i)}{\leq} C_{M,j}\Upsilon_j(\varepsilon)
\|u\|_{\Sz^{\theta,\a}(\stopp_{\varepsilon}',\tau_{\varepsilon,\ell})}^{\rho_j \phi_{1,j}(1-\delta_{1,j})+\phi_{2,j}(1-\delta_{2,j})} \\
&\stackrel{(ii)}{\leq} R_{M,j} + C_{M,j} \|u\|_{\Sz^{\theta,\a}(\stopp_{\varepsilon}',\tau_{\varepsilon,\ell})} \Upsilon_j (\varepsilon),
\end{aligned}
\end{equation}
where $R_{M,j}>0$, $\Upsilon_j \in C([0,\infty))$ are independent of $\ell\geq 1$ and $\lim_{\varepsilon\downarrow 0} \Upsilon_j(\varepsilon)=0$ (since $\delta_{k,j}\phi_{k,j}>0$) and in $(i)$ we used the first and the last inequality in \eqref{eq:Serrin_proof_u_bounds} and in $(ii)$ the Young's inequality and the fact that $\rho_j \phi_{1,j}(1-\delta_{1,j})+\phi_{2,j}(1-\delta_{2,j})\leq 1$. Using the same argument one can provide a similar estimate for $\|u\|_{L^{p \xi_j}(\stopp_{\varepsilon}',\tau_{\varepsilon,\ell},w_{\a}^{\stopp_{\varepsilon}'};X_{\beta^{\star}_j})}$.
Thus, the latter and \eqref{eq:F_G_c_bounds_serrin_1}-\eqref{eq:F_G_c_bounds_serrin_2} yield, a.s.\ for all $\ell\geq 1$,
\begin{multline}
\label{eq:F_G_c_bounds_serrin}
\|F_c(\cdot,u)\|_{L^p(\stopp_{\varepsilon}',\tau_{\varepsilon,\ell},w_{\a}^{\stopp_{\varepsilon}'};X_0)}
+
\|G_c(\cdot,u)\|_{L^p(\stopp_{\varepsilon}',\tau_{\varepsilon,\ell},w_{\a}^{\stopp_{\varepsilon}'};\g(H,X_{1/2}))}\\
\leq R + \Upsilon(\varepsilon) \|u\|_{\Sz^{\theta,\a}(\stopp_{\varepsilon}',\tau_{\varepsilon,\ell})}
\end{multline}
where  $R>0$ and $\Upsilon\in C([0,\infty))$ are independent of $\ell\geq 1$ and $\lim_{\varepsilon\downarrow 0} \Upsilon(\varepsilon)=0$.

To proceed further, note that by \eqref{eq:costant_stochastic_maximal_regularity_Serrin_constant_initial_t} there exists a constant $\wt{K}_{M,\eta}>0$ independent of $\varepsilon>0$ such that \eqref{eq:MaxRegu_0random} and Proposition \ref{prop:start_at_s}\eqref{it:start_at_s4} holds with $C$, $(A,B)$ replaced by $\wt{K}_{M,\eta}$, $(A|_{\ll \sigma'_n,T\rr},B|_{\ll \sigma'_n,T\rr})$, respectively.
Let $\varepsilon^*>0$ be such that
\begin{equation}
\label{eq:choice_varepsilon_star_serrin_semilinear}
\wt{K}_{M,\eta} \Upsilon(\varepsilon^*)\leq \frac{1}{4}
\end{equation}
and set $\tau_{*,\ell}:=\tau_{\varepsilon^*,\ell}$, $\stopp_{*}':=\stopp_{\varepsilon^*}'$ and $\Sol_{*}:=\Sol_{\stopp_{\varepsilon^*}',(A,B)}$ (see \eqref{eq:soloperatorR2}). Fix $\ell\geq 1$ an recall that $F_c=F_{\Tr}+F_c$, $G=G_{\Tr}+G_c$. Thus, by \eqref{eq:F_G_tr_semilinear_Serrin_modified} and combining the second and the second estimate of \eqref{eq:Serrin_proof_u_bounds} with Lemma \ref{l:F_G_bound_N_C_cn} we get
\begin{equation}
\begin{aligned}
\label{eq:integrability_nonlinearities_proof_Pruss}
\one_{[\stopp_{*}',\tau_{*,\ell}]\times \U}F (\cdot,u)&\in L^p_{\Progress}(\O;L^p(\stopp_{*}',\tau_{*,\ell},w_{\a}^{\stopp_{*}'};X_0)),\\
\one_{[\stopp_{*}',\tau_{*,\ell}]\times \U}G(\cdot,u)&\in L^p_{\Progress}(\O;L^p(\stopp_{*}',\tau_{*,\ell},w_{\a}^{\stopp_{*}'};\g(H,X_{1/2}))),
\end{aligned}
\end{equation}
with norms depending possibly on $\ell\geq  1$.
To get the desired contradiction we need to prove an estimate which is uniform in $\ell\geq 1$.

To this end, recall that $\|u(\stopp_{*}')\|_{\Xap}\leq L$ a.s.\ on $\U$ by \eqref{eq:V_double_prime}.
Reasoning as in the proof of Theorem \ref{t:blow_up_criterion} (see \eqref{eq:equation_start_at_u_N}-\eqref{eq:u_sum_v}), by Proposition \ref{prop:causality_phi_revised_2}, the second estimate of \eqref{eq:Serrin_proof_u_bounds} and \eqref{eq:integrability_nonlinearities_proof_Pruss} one has, a.e.\ on $[\stopp_{*}',\tau_{*,\ell}]\times \U$,
\begin{equation}
\label{eq:representation_formula_u_serrin_proof}
u=\Sol_{*}(\one_{\U}u(\stopp_{*}'),\one_{[ \stopp_{*}',\tau_{*,\ell}]\times\U} F(\cdot,u),
\one_{[ \stopp_{*}',\tau_{*,\ell}]\times\U} G(\cdot,u)).
\end{equation}
Thus, by \eqref{eq:F_G_tr_semilinear_Serrin_modified}, \eqref{eq:F_G_c_bounds_serrin} with $\varepsilon=\varepsilon^*$ and Proposition \ref{prop:start_at_s} applied with $\theta$ as in Step 1, for all $\ell\geq 1$ we obtain
\begin{equation*}
\begin{aligned}
&\|u\|_{L^p(\U;\Sz^{\theta,\a}(\stopp'_{*},\tau_{*,\ell}))}\\
&\qquad
\leq
\|\Sol_{*}
(\one_{\U}u(\stopp_{*}'),\one_{[ \stopp_{*}',\tau_{*,\ell}]\times\U} F(\cdot,u),
\one_{[ \stopp_{*}',\tau_{*,\ell}]\times\U} G(\cdot,u))
\|_{L^p(\O;\Sz^{\theta,\a}(\stopp_{*}',T))}\\
&\qquad
\leq
  2\wt{K}_{M,\eta}\big( L+R_{\Tr,M}
  +
\| F_c(\cdot,u)\|_{L^p((\stopp_{*}',\tau_{*,\ell})\times \U,w_{\a}^{\stopp_{*}'};X_0)}\\
&\qquad \qquad \qquad\qquad \qquad \qquad +
\| G_c(\cdot,u)\|_{L^p((\stopp_{*}',\tau_{*,\ell})\times \U,w_{\a}^{\stopp_{*}'};\g(H,X_{1/2}))}\big)\\
&\qquad
\leq 2K_{M,\eta,L}+\frac{1}{2} \|u\|_{L^p(\U;\Sz^{\theta,\a}(\stopp_{*}',\tau_{*,\ell}))},
\end{aligned}
\end{equation*}
where $ K_{M,\eta,L}$ does not depend on $\ell\geq 1$ and in the last estimate we used the choice of $\varepsilon^*$ in \eqref{eq:choice_varepsilon_star_serrin_semilinear}. Let us stress that $L^p(\U)$-norms in the previous inequality are well-defined due to the measurability result \cite[Lemma 2.15]{AV19_QSEE_1} and the fact that $\stopp_{\varepsilon}'$ takes values in a finite set.

Therefore, $\|u\|_{L^p(\U;\Sz^{\theta,\a}(\stopp_{*}',\tau_{*,\ell}))}\leq C$ for a constant $C$ independent of $\ell\geq 1$. Since $\lim_{\ell\to \infty}\tau_{*,\ell}=\sigma$ a.s.\ on $\V'$, by \eqref{eq:F_G_c_bounds_serrin} with $\varepsilon=\varepsilon^*$, by Fatou's lemma we get (recall that $\V''=\V'\cap \U$)
\begin{multline*}
\|F_c(\cdot,u)\|_{L^p((\stopp_{*}',\sigma)\times \V'',w_{\a}^{\stopp_{*}'};X_0)}
+
\|G_c(\cdot,u)\|_{L^p((\stopp_{*}',\sigma)\times \V'',w_{\a}^{\stopp_{*}'};\g(H,X_{1/2}))}\\
\leq \sup_{\ell\geq 1} \Big[
\|F_c(\cdot,u)\|_{L^p((\stopp_{*}',\tau_{*,\ell})\times \V'',w_{\a}^{\stopp_{*}'};X_0)}
+
\|G_c(\cdot,u)\|_{L^p((\stopp_{*}',\tau_{*,\ell})\times \V'',w_{\a}^{\stopp_{*}'};\g(H,X_{1/2}))}\Big]
<\infty.
\end{multline*}
Combining this with \eqref{eq:F_G_tr_semilinear_Serrin_modified}, we have
$
\mathcal{N}^{\a}(u;\stopp_{*}',\sigma)<\infty$
a.s.\ on $\V''$
(see \eqref{eq:L_p_norm_nonlinearity2} for $\mathcal{N}^{\a}$). The former and $\sigma<T$ a.s.\ on $\W\supseteq \V''$ imply
\begin{equation*}
\begin{aligned}
\P(\V'')
&=\P\Big(\V''\cap \big\{\sigma<T,\,
\mathcal{N}^{\a}(u;\stopp_{*}',\sigma)<\infty\big\}\Big)\\
&=\P\Big(\V''\cap \big\{\sigma<T,\,
\mathcal{N}^{\a}(u;0,\sigma)<\infty\big\}\Big)
\leq \P\Big(\sigma<T,\,\mathcal{N}^{\a}(u;0,\sigma)<\infty\Big)
=0,
\end{aligned}
\end{equation*}
where in the last equality we used Theorem \ref{thm:semilinear_blow_up_Serrin}\eqref{it:blow_up_stochastic_semilinear}. The previous yields the desired contradiction with \eqref{eq:V_double_prime}.
\end{proof}

The proof of Theorem \ref{t:blow_up_criterion}\eqref{it:blow_up_norm_general_case_quasilinear_Xap_Pruss} combines the argument used above and the one used in the Step 2 in the proof of Theorem \ref{t:blow_up_criterion}\eqref{it:blow_up_norm_general_case_F_c_G_c}.

\begin{proof}[Proof of Theorem \ref{t:blow_up_criterion}\eqref{it:blow_up_norm_general_case_quasilinear_Xap_Pruss}]
As usual $s=0$ and we prove the claim by contradiction. Assume that $\P(\W)>0$ where
$$
\W:=\Big\{\sigma<T,\,\lim_{t\uparrow \sigma} u(t)\text{ exists in }\Xap,\,\|u\|_{L^{p}(0,\sigma;X_{1-\frac{\a}{p}})}<\infty\Big\}.
$$
Reasoning as in the proof of Theorem \ref{t:blow_up_criterion}\eqref{it:blow_up_norm_general_case_F_c_G_c} there exist  $\eta,M>0$, a sequence of stopping times $(\sigma_n')_{n\geq 1}$ taking values in a finite set and $\V'\in \F_{\sigma}$ such that $\eta\leq \sigma_n'<\sigma$ a.s.\ on $\V'$ and for each $\varepsilon>0$ there exists an $N(\varepsilon)\geq 1$
$$
\|u\|_{C(\overline{I}_{\sigma};\Xap)}\leq M, \ \
\sup_{t\in [\sigma'_{N(\varepsilon)},\sigma]}\|u(t)-u(\sigma)\|_{\Xap}<\varepsilon
 \ \text{ and } \
\|u\|_{L^{p}(\sigma_{N(\varepsilon)}',\sigma;X_{1-\frac{\a}{p}})}<\varepsilon.
$$
where $u(\sigma):=\lim_{t\uparrow \sigma} u(t)$ a.s.\ on $\W$.

For each $\varepsilon>0$, $\ell\geq 1$ define $\tau_{\varepsilon,\ell}:=\sigma_{N(\varepsilon)}'$ on $\{\sigma_{N(\varepsilon)}'\geq \sigma\}$ and on $\{\sigma_{N(\varepsilon)}'< \sigma\}$ as
\begin{equation*}
\begin{aligned}
\tau_{\varepsilon,\ell}:=\inf\Big\{t\in [\sigma_{N(\varepsilon)}',\sigma)\,:\,&\|u(t)-u(\sigma_{N(\varepsilon)}')\|_{\Xap} \geq 2\varepsilon, \,  \\
&\|u\|_{\X(\sigma_{N(\varepsilon)}',t)\cap L^p(\sigma_{N(\varepsilon)}',t,w_{\a},X_1)}\geq \ell,\\
&\|u\|_{L^{p}(\sigma_{N(\varepsilon)}',t;X_{1-\frac{\a}{p}})}\geq \varepsilon,\,\|u\|_{C([0,t];\Xap)}\geq M\Big\}.
\end{aligned}
\end{equation*}
Choose $\varepsilon^*>0$ such that the condition in the proof of Theorem \ref{t:blow_up_criterion}\eqref{it:blow_up_norm_general_case_F_c_G_c} (see the text before \eqref{eq:definition_psi_proof_general_quasilinear_NFG}) and \eqref{eq:choice_varepsilon_star_serrin_semilinear} both hold. At this point, one can repeat the estimates of $u$ using the splitting in \eqref{eq:u_sum_v} with $\tau_{\varepsilon}$ replaced by $\tau_{\varepsilon^*,\ell}$. Note that $I-III$ (see \eqref{eq:u_sum_v}) can be estimated as in Theorem \ref{t:blow_up_criterion}\eqref{it:blow_up_norm_general_case_F_c_G_c} and the remaining terms as in the proof of Theorem \ref{thm:semilinear_blow_up_Serrin}\eqref{it:blow_up_semilinear_serrin_Pruss_modified} above. The claim follows similar to Theorem \ref{thm:semilinear_blow_up_Serrin}\eqref{it:blow_up_semilinear_serrin_Pruss_modified} by contradiction with Theorem \ref{t:blow_up_criterion}\eqref{it:blow_up_norm_general_case_F_c_G_c}.
\end{proof}

It remains to prove Theorem \ref{thm:semilinear_blow_up_Serrin_refined}.
As announced in Subsection \ref{ss:main_result_blow_up}, we prove a generalization of Theorem \ref{thm:semilinear_blow_up_Serrin_refined}, where we do not require $\varphi_j = \beta_j$. An example of such a situation is provided by stochastic reaction-diffusion equations with gradient nonlinearities, see \cite[Subsection 5.4]{AV19_QSEE_1}.

Recall that $\varphi_j^{\star}$ and $\beta_j^{\star}$ are defined at the beginning of Subsection \ref{ss:thmblowuphardespart_Serrin}.
\begin{proposition}[Serrin type blow-up criteria for semilinear SPDEs - revised]
\label{prop:serrin_Pruss_general_form}
Let the assumptions of Theorem \ref{thm:semilinear_blow_up_Serrin_refined} be satisfied replacing \eqref{eq:condSerrin} by the following condition:
For each $j\in\{1,\dots,m_F+m_G\}$ such that $\rho_j>0$ one of the following is satisfied
\begin{itemize}
\item $\a>0$ and $\beta_j^{\star},\varphi_j^{\star}>1-\frac{1+\a}{p}\frac{1+\a}{2+\a}$;
\item $\a=0$ and $\rho_j\leq 1$.
\end{itemize}
Then the $L^p_{\a}$-maximal local solution $(u,\sigma)$ to \eqref{eq:QSEE} satisfies
\[
\P\big(\sigma<T,\,\|u\|_{L^p(s,\sigma;X_{1-\frac{\a}{p}})}<\infty\big)=0.
\]
\end{proposition}

Before we prove Proposition \ref{prop:serrin_Pruss_general_form}, we first show that it implies Theorem \ref{thm:semilinear_blow_up_Serrin_refined}.

\begin{proof}[Proof of Theorem \ref{thm:semilinear_blow_up_Serrin_refined}]
It is enough to check the assumptions of Proposition \ref{prop:serrin_Pruss_general_form}. In case $\a=0$ the assumptions coincide and hence this case is clear.

Next we assume $\a>0$ and fix $j\in \{1,\dots,m_F+m_G\}$. If $\rho_j=0$, then as agreed at the beginning of Subsection \ref{ss:thmblowuphardespart_Serrin} we replaced $\rho_j$ by $\varepsilon_j$, and the corresponding $\varphi_j^{\star},\beta_j^{\star}$ satisfy $\varphi_j^{\star},\beta_j^{\star}>1-\frac{1+\a}{p}\frac{1+\a}{2+\a}$ as assumed in Proposition \ref{prop:serrin_Pruss_general_form}. If $\rho_j>0$, then note that the definition of $\varphi^{\star}_j,\beta^{\star}_j$ and the fact that $\varphi_j=\beta_j$ imply $\varphi^{\star}_j=\beta^{\star}_j=1-\frac{\rho_j}{\rho_j+1}\frac{1+\a}{p}$. Since $\rho_j<1+\a$ is equivalent to $\varphi_j^{\star},\beta^{\star}_j>1-\frac{1+\a}{p}\frac{1+\a}{2+\a}$, the assumptions of Proposition \ref{prop:serrin_Pruss_general_form} are satisfied also in this case.
\end{proof}

\begin{proof}[Proof of Proposition \ref{prop:serrin_Pruss_general_form}]
As usual, we consider $s=0$ and we split the proof into several cases. The proof follows a similar argument as Theorem \ref{thm:semilinear_blow_up_Serrin}\eqref{it:blow_up_semilinear_serrin_Pruss_modified}.

Suppose that $\P(\W)>0$ where
\begin{equation}
\label{eq:W_semilinear_pruss}
\W:=\{\sigma<T,
\|u\|_{L^p(\I_{\sigma};X_{1-\frac{\a}{p}})}<\infty\}.
\end{equation}
Then there exist $\eta,M>0$ and $\F_{\sigma}\ni \V\subseteq \W$ such that $\P(\V)>0$ and
$$
\sigma >\eta
\  \ \ \text{ and } \ \ \
\|u\|_{L^p(\I_{\sigma};X_{1-\frac{\a}{p}})}\leq M \ \ \ \  \text{ a.s.\ on }\V.
$$
Note that in contrast to the proof of Theorem \ref{thm:semilinear_blow_up_Serrin}\eqref{it:blow_up_semilinear_serrin_Pruss_modified} we do not have an $L^{\infty}$-bound for $u$ in the trace space $\Xap$. However, the assumption of Theorem \ref{thm:semilinear_blow_up_Serrin_refined} is that $F_{\Tr}=G_{\Tr}=0$ and $C_{c,n}'$ in \eqref{eq:F_c_G_c_estimate_star} are independent of $n\in\N$.

As in Step 1 of the proof of Theorem \ref{thm:semilinear_blow_up_Serrin}\eqref{it:blow_up_semilinear_serrin_Pruss_modified}, from
Lemma \ref{lem:interpolationineqMR0}\eqref{it:interpolationineqMR01} and \eqref{it:interpolationineqMR03}, it follows that
for all $j\in \{1,\dots,m_F+m_G\}$, \eqref{eq:interpolation_inequality_serrin_1}-\eqref{eq:interpolation_inequality_serrin_2} hold with $\phi_{1,j},\phi_{2,j},\delta_{1,j},\delta_{2,j}\in (0,1]$ such that
\begin{equation}
\label{eq:equation_phi_1_delta_1_phi_2_delta_2}
\rho_j[(1-\phi_{1,j})+\phi_{1,j}(1-\delta_{1,j})]+[(1-\phi_{2,j})+\phi_{2,j}(1-\delta_{2,j})]\leq 1.
\end{equation}
Note that $\phi_{k,j}=1$ if $\a>0$ and $\delta_{k,j}=1$ if $\a=0$, by Lemma  \ref{lem:interpolationineqMR0}.
Therefore, we can extend the proof of Theorem \ref{thm:semilinear_blow_up_Serrin}\eqref{it:blow_up_semilinear_serrin_Pruss_modified} to the present case. Indeed, by
\eqref{eq:equation_phi_1_delta_1_phi_2_delta_2} we can repeat the estimate \eqref{eq:F_G_c_bounds_serrin} replacing the term $\|u\|_{\Sz^{\theta,\a}(\lambda_{\varepsilon},\tau_{\varepsilon,\ell})}$ by $\|u\|_{\Sz^{\theta,\a}(\lambda_{\varepsilon},\tau_{\varepsilon,\ell})\cap L^{\infty}(\lambda_{\varepsilon},\tau_{\varepsilon,\ell};\Xap)}$. After this modification, one can repeat the argument of Step 3 of the proof of Theorem \ref{thm:semilinear_blow_up_Serrin_refined}.
\end{proof}

\section{Instantaneous regularization}
\label{s:regularization}

In this section we study regularization effects for \eqref{eq:QSEE}.
We present a space and time regularity result for the solution to \eqref{eq:QSEE}. We already saw that due to the presence of non-trivial weights one can obtain instantaneous regularization of solutions (see Theorem \ref{t:local_s}\eqref{it:regularity_data_L0}). Below we state conditions under which regularity can be bootstrapped for any positive time.

There are three main results. In Theorem \ref{t:regularization_z} we present a general iteration scheme to bootstrap regularity in time and space. In Corollary \ref{cor:regularization_X_0_X_1} we specialize to time regularity. In both results weights play an essential role. Finally, in Proposition \ref{prop:adding_weights} we present a result which allows to introduce weights after starting from an unweighted situation. The latter is important in several interesting situations.
An example where this occurs will be given in Section \ref{s:1D_problem} (see Roadmap \ref{roadcriticalreg} for a simple explanation).

\subsection{Assumptions}
Below we state abstract conditions which can be used for bootstrapping arguments. We first present our main assumptions on the spaces $Y_1\hookrightarrow Y_0$ in which we bootstrap regularity.

\begin{assumption}\label{assum:HY}
Suppose that Hypothesis \hyperref[H:hip]{$\Hip$} is satisfied.
Let $Y_0$ and $Y_1$ be UMD Banach spaces with type $2$, and such that $Y_1\hookrightarrow Y_0$ densely. Let either $r=2$ and $\alpha=0$, or $r\in (2, \infty)$ and $\alpha\in [0,\frac{r}{2}-1)$.  We say that hypothesis {\em \textbf{H}}$(Y_0, Y_1, r,\alpha)$ holds if
\begin{enumerate}[{\rm(1)}]
\item\label{it:assum_Yi_Xi} $X_0$ and $Y_0$ are compatible, and $Y_1\cap X_1 \hookrightarrow Y_1$ is dense;
\item\label{it:assum_Y0_Y1_AY} There exist maps $A_{Y}:\ll s,T\rr \times \Yaaar \to \calL(Y_1,Y_0)$, $F_Y:\ll s,T\rr \times Y_1\to Y_0$, $B_{Y}:\ll s,T\rr \times \Yaaar\to \calL(Y_1,\g(H,Y_{1/2}))$ and $G_Y:\ll s,T\rr \times Y_1\to \g(H,Y_{1/2})$ such that a.s.\ for all $t\in (s,T)$,
\begin{equation*}
\begin{aligned}
A_{Y}(t,z)v&=A(t,z)v, &\qquad  B_{Y}(t,z)v&=B(t,z)v,
\\ F_Y(t,v)&=F(t,v), &\qquad G_Y(t,v)&=G(t,v),
\end{aligned}
\end{equation*}
for all $z,v\in X_1\cap Y_1$. Moreover, the following hold:
\begin{itemize}
\item $A_{Y},B_{Y}$ verify \emph{\ref{HAmeasur}} with $(X_0,X_1,p,\a)$ replaced by $(Y_0,Y_1,r,\aaa)$;
\item $F_Y,G_Y$ satisfy \emph{\ref{HFcritical}-\ref{HGcritical}} with $(X_0,X_1,p,\a)$ replaced by $(Y_0,Y_1,r,\aaa)$ and (possibly) different parameters $(\wt{\rho}_j,\wt{\varphi}_j,\wt{\beta}_j,\wt{m}_F,\wt{m}_G)$.
\end{itemize}
\item\label{it:assum_fg_smooth} $f\in L^0_{\Progress}(\O;L^r(s+\varepsilon,T;Y_0))$, $g\in L^0_{\Progress}(\O;L^r(s+\varepsilon,T;\g(H,Y_{1/2})))$ for all $\varepsilon>0$.
\end{enumerate}
\end{assumption}

As before in case \eqref{eq:QSEE} is semilinear, we write $(A(\cdot,x),B(\cdot,x))=(\bar{A}(\cdot),\bar{B}(\cdot))$ and $(\bar{A}_Y(\cdot),\bar{B}_Y(\cdot))$ instead of $(A_Y(\cdot,x),B_{Y}(\cdot,x))$.
If it is necessary to make the dependency on $(\alpha,r)$ explicit as well, then we will write $(A_{Y,\alpha,r},B_{Y,\alpha,r},F_{Y,\alpha,r},G_{Y,\alpha,r})$ instead of $(A_{Y},B_{Y},F_Y,G_Y)$. In some applications (see \cite{AV19_QSEE_3}), we will need a generalization of Assumption \ref{assum:HY} in which lower order parts of $A$ and $B$ are moved to $F$ and $G$ (see Remark \ref{r:regularization_z_splitting} below).

Let $(u,\sigma)$ be the maximal $L^p_{\a}$-solution given by Theorem \ref{t:local_s}. Now the idea is as follows. The above setting allows to consider \eqref{eq:QSEE} in {\em the $(Y_0, Y_1, r, \alpha)$-setting}, i.e.\
\[\text{replace} \ \  (X_0,X_1,p,\a,A,B,F,G) \ \ \text{by}  \ \ (Y_0, Y_1,r,\alpha,A_{Y},B_{Y},F_Y,G_Y) \ \ \text{in \eqref{eq:QSEE}}.\]
Now if Assumption \ref{H_a_stochnew} holds in the $(Y_0, Y_1, r, \alpha)$-setting for $\ell=\alpha$, it follows that all conditions of Theorem \ref{t:local_s} also hold on $[s+\varepsilon,T]$ for $\varepsilon>0$ arbitrary. Therefore, if $u(s+\varepsilon)\in \Yaaar$ a.s.\ there exists an $L^r_{\alpha}$-maximal local solution $(v,\tau)$ to \eqref{eq:QSEE} in the $(Y_0, Y_1, r, \alpha)$-setting with $(s,u_s)$ replaced by $(s+\varepsilon,u(s+\varepsilon))$ and $\tau:\Omega\to (s+\varepsilon,T]$. Now one would expect that $\tau = \sigma$ and $u=v$ on $[\varepsilon, \sigma]$, and this typically improves the space-time regularity of $u$.
In order to make the above bootstrap argument precise we need to be able to connect the $(Y_0,Y_1,r,\alpha)$-setting to the $(X_0,X_1,p,\a)$-setting to assure:
\begin{enumerate}[(a)]
\item $\Xp \subseteq \Yaaar$;
\item\label{it:YXsettingb} $v = u$ on $[\varepsilon, \tau]$ and $\tau\leq \sigma$;
\item $\tau\geq \sigma$ via a blow-up criterium in the $(Y_0,Y_1,r,\alpha)$-setting.
\end{enumerate}

Below we will actually use an abstract $(Y_0, Y_1, r, \alpha)$-setting and $(\wh{Y}_0,\wh{Y}_1,\wh{r},\wh{\alpha})$-setting to be able to iteration the bootstrap argument.
One important ingredients in the proof will be to show uniqueness (see \eqref{it:YXsettingb} in the above), and this will be done by presuming the following inclusion:
\begin{equation}
\label{eq:emb_uniqueness_Z}
\bigcap_{\theta\in [0,1/2)} H^{\theta,\wh{r}}(s,T,w_{\wh{\alpha}}^s;\wh{Y}_{1-\theta})
\subseteq
L^r(s,T,w_{\aaa}^s;Y_1)\cap \Y(s,T)\cap C([s,T];\Yaaar).
\end{equation}
Here $\Y(s,T)$ is defined as in \eqref{eq:def_X_space} with the above new parameters $(r,\alpha,\wt{\rho}_j,\wt{\varphi}_j,\wt{\beta}_j)$ and $X_{\theta}$ replaced by $Y_{\theta}$.

By translation and scaling \eqref{eq:emb_uniqueness_Z} extends to all other bounded intervals.  If $u$ is in the space on RHS\eqref{eq:emb_uniqueness_Z}, then $u$ has the required regularity for being an $L^{r}_{\alpha}$-strong solution to \eqref{eq:QSEE} in the $(Y_0,Y_1,r,\alpha)$-setting by Definition \ref{def:solution1}. This follows from Lemma \ref{l:F_G_bound_N_C_cn} in the $(Y_0,Y_1,r,\alpha)$-setting.

The following lemma gives sufficient conditions for \eqref{eq:emb_uniqueness_Z} and is strong enough to cover all applications which we have studied so far. In particular, we never need to consider $\Y$ explicitly in applications. We only consider the case $\wh{Y}_i\hookrightarrow Y_i$ for $i\in \{0,1\}$, but there are also variations which avoid this condition.
\begin{lemma}
\label{l:emb_z}
Suppose that Hypothesis \hyperref[assum:HY]{\Hiep}$(Y_0, Y_1, r,\alpha)$ holds (see Assumption \ref{assum:HY}) and that
\begin{itemize}
\item $\wh{Y}_0$ and $\wh{Y}_1$ are Banach spaces such that $\wh{Y}_1\hookrightarrow \wh{Y}_0$;
\item $\wh{Y}_i\hookrightarrow Y_i$ for $i\in \{0,1\}$, $\wh{r}\in [r, \infty)$,  and $\wh{\alpha}\in [0,\frac{\wh{r}}{2}-1)$.
\end{itemize}
Then  \eqref{eq:emb_uniqueness_Z} holds in each of the following cases:
\begin{enumerate}[{\rm(1)}]
\item\label{it:regularization_z_equality} $r=\wh{r}$ and $\aaa=\wh{\alpha}$;
\item\label{it:regularization_z_strict_inequality} $\frac{1+\wh{\alpha}}{\wh{r}}<\frac{1+\aaa}{r}$;
\item\label{it:regularization_z_sharp_condition_refined}
$\frac{1+\wh{\alpha}}{\wh{r}}<\frac{1+\aaa}{r}+\varepsilon$ provided $\wh{Y}_{1-\varepsilon}\hookrightarrow Y_1$ and $\wh{Y}_{0}\hookrightarrow Y_{\varepsilon}$, for some
$\varepsilon\in (0,\frac{1}{2}-\frac{1+\aaa}{r})$;

\item \label{it:regularization_z_sharp_condition_refined_q_p} $r=\wh{r}$ and
$\frac{1+\wh{\alpha}}{r}=\frac{1+\alpha}{r}+\varepsilon$ provided $\wh{Y}_{1-\varepsilon}\hookrightarrow Y_1$ and $\wh{Y}_{0}\hookrightarrow Y_{\varepsilon}$, for some
$\varepsilon\in (0,\frac{1}{2}-\frac{1+\alpha}{r})$.
\end{enumerate}
\end{lemma}

\begin{proof}
In all cases, it is enough to consider the case $s=0$.

\eqref{it:regularization_z_equality}-\eqref{it:regularization_z_strict_inequality}: By Proposition \ref{prop:change_p_q_eta_a}\eqref{it:change_a_p_eta_a} and the fact that $\wh{Y}_{1-\theta}\hookrightarrow Y_{1-\theta}$, for all $\theta\in [0,\frac{1}{2})$,
$$
H^{\theta,\wh{r}}(\I_T,w_{\wh{\alpha}};\wh{Y}_{1-\theta})
\hookrightarrow
H^{\theta,{r}}(\I_T,w_{{\alpha}};\wh{Y}_{1-\theta})
\hookrightarrow
H^{\theta,{r}}(\I_T,w_{{\alpha}};Y_{1-\theta}).
$$
Therefore, the inclusion follows by the former, Lemma \ref{l:embeddings} and Proposition \ref{prop:continuousTrace}\eqref{it:trace_with_weights_Xap}.

\eqref{it:regularization_z_sharp_condition_refined}: Due to Lemma \ref{l:embeddings}, it is enough to show that for some $\theta_1,\theta_2\in [0,\frac{1}{2})$, $\nu\in (\frac{1+\aaa}{r},\frac{1}{2})$,
\begin{equation}
\label{eq:sufficient_condition_sharp_case_emb_stochastic}
\bigcap_{\theta\in \{\theta_1,\theta_2\}}
H^{\theta,\wh{r}}(\I_T,w_{\wh{\alpha}};\wh{Y}_{1-\theta})
\subseteq
H^{\nu,r}(\I_t,w_{\aaa};Y_{1-\nu})\cap L^{r}(\I_T,w_{\aaa};Y_{1}).
\end{equation}
The reiteration theorem for the complex interpolation (see e.g.\ \cite[Theorem 4.6.1]{BeLo}) ensures that $\wh{Y}_{\theta(1-\varepsilon)}\hookrightarrow Y_{\varepsilon+\theta(1-\varepsilon)}$, for each $\theta\in (0,1)$. Therefore
\begin{equation}
\label{eq:Z_theta_Y_sharp_case_emb_stochastic}
\wh{Y}_{1-\theta}\hookrightarrow Y_{1-\theta+\varepsilon}, \qquad \text{ for all }\theta\in [\varepsilon,1).
\end{equation}
Since $\varepsilon\in (0,\frac{1}{2}-\frac{1+\aaa}{r})$, there exists a $\delta\in (0,\frac{1}{2})$ such that $\delta>\varepsilon+\frac{1+\aaa}{r}>\frac{1+\wh{\alpha}}{\wh{r}}$, where the last inequality follows by assumption.
By \eqref{eq:Z_theta_Y_sharp_case_emb_stochastic} and the fact that $\delta>\varepsilon$ we obtain
\begin{equation*}
\begin{aligned}
H^{\delta,\wh{r}}(\I_T,w_{\wh{\alpha}};\wh{Y}_{1-\delta})
\hookrightarrow
H^{\delta,\wh{r}}(\I_T,w_{\wh{\alpha}};Y_{1-(\delta-\varepsilon)})
\hookrightarrow
H^{\delta-\varepsilon,r}(\I_T,w_{\aaa};Y_{1-(\delta-\varepsilon)})
\end{aligned}
\end{equation*}
where the last embedding follows from Corollary \ref{cor:technicalweight}.
Similarly,
\begin{equation*}
\begin{aligned}
H^{\varepsilon,\wh{r}}(\I_T,w_{\wh{\alpha}};\wh{Y}_{1-\varepsilon})
\hookrightarrow
H^{\varepsilon,\wh{r}}(\I_T,w_{\wh{\alpha}};Y_{1})
&\hookrightarrow
L^{r}(\I_T,w_{\aaa};Y_1).
\end{aligned}
\end{equation*}
The above embeddings imply
\eqref{eq:sufficient_condition_sharp_case_emb_stochastic} with $\theta_1=\varepsilon$, $\theta_2=\delta$ and $\nu=\delta-\varepsilon$.

\eqref{it:regularization_z_sharp_condition_refined_q_p}: The proof is similar to \eqref{it:regularization_z_sharp_condition_refined} using the last claim in Corollary \ref{cor:technicalweight} in the case that $\frac{1+\wh{\alpha}}{r}=\varepsilon+\frac{1+\alpha}{r}$.
\end{proof}

\subsection{Bootstrapping using weights}
\label{ss:bootstrapping_using_weights}
In this section we state our main result. The statement below is quite technical because the list of conditions is rather long. The strength of the result will be demonstrated in Section \ref{s:1D_problem} and in the follow-up papers \cite{AV20_NS,AV19_QSEE_3,AV22_reaction_diffusion} where we use the results below to show
\begin{itemize}
\item H\"older regularity results with rough initial data;
\item weak solutions immediately become strong solutions.
\end{itemize}
To obtain this type of regularization, we build a general scheme which only depends on the structure of the SPDE through the parameters $p,\a,X_0,X_1$ in which the scaling properties of the underlined SPDEs is encoded. An example will be given in Section \ref{s:1D_problem} (see Roadmap \ref{roadcriticalreg} for a short overview).

The assumptions below have a considerable overlap with Theorem \ref{t:blow_up_criterion}, which plays a key role in the proof.
Let us remind that critical spaces for \eqref{eq:QSEE} are defined below Hypothesis \hyperref[H:hip]{$\Hip$}. The picture one should have in mind is that $Y$-regularity and $L^r$-integrability is given, and $\wh{Y}$-regularity and $L^{\wh{r}}$-integrability are deduced as a consequence.
\begin{theorem}[Bootstrapping regularity]
\label{t:regularization_z}
Let Hypothesis \hyperref[H:hip]{$\Hip$} be satisfied. Let $u_s\in L^0_{\F_s}(\O;\Xap)$ and suppose that \eqref{eq:approximating_sequence_initial_data} for a sequence $(u_{s,n})_{n\geq 1}$. Suppose that
\begin{equation*}
(A(\cdot,u_{s,n}),B(\cdot,u_{s,n}))\in \MRtas, \ \ \ n\geq 1,
\end{equation*}
and let $(u,\sigma)$ be the $L^p_{\a}$-maximal local solution to \eqref{eq:QSEE} given by Theorem \ref{t:local_s}. Suppose that Assumption \ref{H_a_stochnew} holds for $\ell\in \{0,\a\}$ and Assumption \ref{ass:FG_a_zero} holds with parameters $(\varphi_j')_{j\in \{1,\dots,m_F'+m_G'\}}$. Further suppose the following:
\begin{enumerate}[{\rm(1)}]
\item\label{it:assumption_u_regular_Y1} Hypothesis \hyperref[assum:HY]{\Hiep}$(Y_0, Y_1, r,\alpha)$ holds  with $r\in [p,\infty)$ (see Assumption \ref{assum:HY}), Assumption \ref{H_a_stochnew} holds in the $(Y_0, Y_1, r, \alpha)$-setting for $\ell=\alpha$, and
\begin{itemize}
\item  $\Yr \embed \Xp$;
\item $Y_{\delta}\embed X_{0}$, $Y_{1}\embed X_{1-\delta}$ for some $\delta\in [0,1-\max_j \varphi_j')$ and $\frac{1}{r}+\delta \leq \frac{1}{p}$;
\item $u:\llo s,\sigma\rro\to Y_1$ is strongly progressively measurable and
\begin{equation*}
u\in \bigcap_{\theta\in [0,1/2)}H^{\theta,r}_{\emph{loc}}(s,{\sigma};Y_{1-\theta}) \ \ \text{a.s.};
\end{equation*}
\end{itemize}

\item\label{it:assumption_u_regular_Y2} Hypothesis \hyperref[assum:HY]{\Hiep}$(\wh{Y}_0,\wh{Y}_1,\wh{\alpha},\wh{r})$ holds with $\wh{r}\in [r,\infty)$ and $\wh{\alpha}\in [0,\frac{\wh{r}}{2}-1)$, and the space $\wh{Y}^{\Tr}_{\wh{\alpha},\wh{r}}$ is not critical for \eqref{eq:QSEE}, Assumption \ref{H_a_stochnew} for $\ell\in \{0,\wh{\alpha}\}$ and \ref{ass:FG_a_zero} both hold in the $(\wh{Y}_0,\wh{Y}_1,\wh{r},\wh{\alpha})$-setting;
\item\label{it:assumption_Yr_embeds_Yar_hat} $\wh{Y}_i\hookrightarrow Y_i$ for $i\in \{0,1\}$, $\Yr\hookrightarrow \wh{Y}^{\Tr}_{\wh{\alpha},\wh{r}}$ and
\eqref{eq:emb_uniqueness_Z} holds.
\end{enumerate}
Then $(u,\sigma)$ instantaneously regularizes in spaces and time in the sense that $u:\llo s,\sigma\rro\to \wh{Y}_1$ is strongly progressively measurable and
\begin{equation}
\label{eq:regularity_u_Z}
u\in \bigcap_{\theta\in [0,1/2)} H^{\theta,\wh{r}}_{\emph{loc}}(s,{\sigma};\wh{Y}_{1-\theta})
\subseteq
C((s,{\sigma});\wh{Y}_{\wh{r}}^{\Tr}) \ \ \text{a.s.}
\end{equation}
\end{theorem}
Observe that if Hypothesis \hyperref[H:hip]{$\Hip$} and Assumption \ref{H_a_stochnew} for $\ell=\a$ hold, then by Theorem \ref{t:local_s}, condition \eqref{it:assumption_u_regular_Y1} is always satisfied for $(Y_0, Y_1, r, \alpha) = (X_0, X_1, p, \a)$. Also note that $Y^{\Tr}_{\alpha,r}$ can be critical for \eqref{eq:QSEE} in the $(Y_0, Y_1, r, \alpha)$-setting.

The second bullet of \eqref{it:assumption_u_regular_Y1} holds with $\delta = 0$ if $Y_i\hookrightarrow X_i$, which is the case in most applications to SPDEs. In the latter situation, $\Yr\embed \Xp$ holds provided $r\geq p$. The case $\delta>0$ will be needed in combination with Proposition \ref{prop:adding_weights} where we consider the case $\a=0$ and therefore $\varphi_j'=\varphi_j$ for all $j$.

\begin{remark}\
\label{r:regularization_det}
\begin{enumerate}[{\rm(1)}]
\item\label{it:regularization_det1} Theorem \ref{t:regularization_z} yields an improvement in regularity in time if $\wh{r}>r$ and in space if $\wh{Y}_i\hookrightarrow Y_i$ is strict for some $i\in \{0,1\}$;
\item\label{it:regularization_det2} often Theorem \ref{t:regularization_z} can be applied iteratively where $(\wh{Y}_0, \wh{Y}_1,\wh{r},\wh{\alpha})$ takes over the role of $({Y}_0, {Y}_1,{r},{\alpha})$, and another quadruple takes over the role of $(\wh{Y}_0, \wh{Y}_1,\wh{r},\wh{\alpha})$. In this way for concrete SPDEs, in finitely many steps one can often derive $(\frac12-\varepsilon)$-H\"older regularity in time and higher order H\"older regularity in space for rough initial data.
\end{enumerate}
\end{remark}

\begin{proof}[Proof of Theorem \ref{t:regularization_z}]
To prepare the proof, we collect some useful facts. It suffices to consider $s=0$. Fix $\varepsilon>0$ and set
\begin{equation}
\label{eq:def_V_regularization}
\V:=\{\sigma>\varepsilon\}\in \F_{\varepsilon}.
\end{equation}
By \eqref{it:assumption_u_regular_Y1},
Proposition \ref{prop:continuousTrace}, and \eqref{it:assumption_Yr_embeds_Yar_hat},
\begin{equation}
\label{eq:initial_data_instantaneous_regularization_stochastic}
\one_{\V}u(\varepsilon)\in L^0_{\F_{\varepsilon}}(\O;\Yr)\subseteq L^0_{\F_{\varepsilon}}(\O;\wh{Y}^{\mathsf{Tr}}_{\wh{\alpha},\wh{r}}).
\end{equation}
As explained below Assumption \ref{assum:HY}, by \eqref{it:assumption_u_regular_Y1} and Theorem \ref{t:local_s} applied in the $(Y_0, Y_1, r, \alpha)$-setting we find that there exists an $L^r_{\alpha}$-maximal local solution to $(v,\tau)$ to \eqref{eq:QSEE} in the $(Y_0, Y_1, r, \alpha)$-setting with $s=\varepsilon$ and initial data $\one_{\V}u(\varepsilon)$, with localizing sequence  $(\tau_k)_{k\geq 1}$. Similarly, arguing in $(\wh{Y}_0, \wh{Y}_1,\wh{r}, \wh{\alpha})$-setting (thus using \eqref{eq:initial_data_instantaneous_regularization_stochastic} and \eqref{it:assumption_u_regular_Y2}), we obtain an $L^{\wh{r}}_{\wh{\alpha}}$-maximal local solution $(\wh{v},\wh{\tau})$ to \eqref{eq:QSEE} in the $(\wh{Y}_0, \wh{Y}_1, \wh{r}, \wh{\alpha})$-setting with $s=\varepsilon$, initial data $\one_{\V}u(\varepsilon)$ and localizing sequence $(\wh{\tau}_k)_{k\geq 1}$.

{\em Step 1: $\tau=\sigma$ on $\V$ and $v=u$ a.e.\ on $[\varepsilon,\tau)\times\V$}.
By Theorem \ref{t:local_s} and Proposition \ref{prop:continuousTrace}\eqref{it:trace_without_weights_Xp}, for all $k\geq 1$,
\begin{equation}
\label{eq:regularity_v_n}
v\in \bigcap_{\theta\in [0,1/2)}H^{\theta,r}(\varepsilon,\tau_k,w_{\aaa}^{\varepsilon};Y_{1-\theta})
\subseteq C((\varepsilon,\tau_k];\Yr).
\end{equation}
By condition \eqref{it:assumption_u_regular_Y1} and Theorem \ref{t:local_s}\eqref{it:regularity_data_L0}, $u\in Y_1\cap X_1$ a.e.\ on $\ll \varepsilon,\sigma\rro$. Thus, $A_{Y}(\cdot,u)u=A(\cdot,u)u$, $B_{Y}(\cdot,u)u=B(\cdot,u)u$, $F_Y(\cdot,u)=F(\cdot,u)$, $G_Y(\cdot,u)=G(\cdot,u)$ a.e.\ on $\ll \varepsilon,\sigma\rro$. This implies that $ (\one_{\V} u|_{[\varepsilon,\sigma)},\one_{\V} \sigma+\one_{\O\setminus\V}\varepsilon)$ is an $L^r_{\alpha}$-local solution to \eqref{eq:QSEE} in the $(Y_0, Y_1, r, \alpha)$-setting with $s=\varepsilon$ and initial data $\one_{\V}u(\varepsilon)$. Therefore, maximality of $(v,\tau)$ implies (see also Remark \ref{r:uniqueness_maximal})
\begin{equation}
\label{eq:uniqueness_u_u_n}
 \sigma\leq \tau ,\quad \text{a.e.\ on }\V, \qquad \qquad
u=v, \quad \text{ a.e.\ on } [\varepsilon,\sigma)\times \V.
\end{equation}
It remains to prove that $\sigma\geq \tau$ a.e.\ on $\V$. For this it is enough to show
\begin{equation}
\label{eq:claim_step_1_regularization_tau_sigma}
\P(\V\cap\{\sigma<\tau\})=0.
\end{equation}
For convenience, we divide the proof into two substeps. In case $\delta=0$ (i.e. $Y_i\hookrightarrow X_i$), the claim in Step 1a follows immediately from Lemma \ref{l:embeddings} and $r\geq p$.

\emph{
Step 1a: Let $\X'$ be defined as in \eqref{eq:def_X_space} with $(m_F,m_G,\varphi_j,\beta_j,\rhos_j,\a)$ replaced by $(m_F',m_G',\varphi_j',\beta_j',\wt{\rho}^{\star}_j,0)$ (see Assumption \ref{ass:FG_a_zero}). In particular $\wt{\rho}^{\star}_j=(1-\beta_j')/(\varphi_j'-1+\frac{1}{p})$, cf.\ \eqref{eq:rhostar}.
Then, for all $0\leq a<b<\infty$,}
\begin{equation}
\label{eq:claim_step_1a_correction_proof}
\bigcap_{\theta\in [0,1/2)} H^{\theta,r} (a,b;Y_{1-\theta})\subseteq \X'(a,b).
\end{equation}
Fix $k\in \{1,\dots,m_F'+m_G'\}$. By \eqref{eq:rhostar}-\eqref{eq:rr'} with $(m_F,m_G,\varphi_j,\beta_j,\a)$ replaced by $(m_F',m_G',\varphi_j',\beta_j',0)$, it is enough to show that the LHS in \eqref{eq:claim_step_1a_correction_proof} embeds into
\begin{equation}
\label{eq:X_prime_simplified}
L^{ p \wt{r}_k}(a,b;X_{\beta_k'})\cap L^{ \wt{\rho}^{\star}_k p \wt{r}_k'}(a,b;X_{\varphi_k'})
\end{equation}
where $\wt{r}_k,\wt{r}_k'$ satisfy
\begin{equation}
\label{eq:parameters_correction_proof_Step_1a}
\frac{1}{\wt{\rho}^{\star}_k p \wt{r}_k'}=\varphi_k'-1+\frac{1}{p}\quad \text{ and }\quad \frac{1}{ p \wt{r}_k}=\beta_k'-1+\frac{1}{p}.
\end{equation}
By Assumption \ref{ass:FG_a_zero},  for $\phi\in \{\beta_k',\varphi_k'\}$ we have $\delta<1-\phi$ (since $\delta<1-\varphi_k'$ and $\varphi_k'\geq \beta_k'$)
and $1-\phi-\delta<\frac{1}{2}$ (since $\phi>1-\frac{1}{p}>\frac{1}{2}$). Using that $Y_{\theta+\delta}\embed X_{\theta}$ by complex reiteration (see e.g.\ \cite[Theorem 4.6.1]{BeLo}), we have
\begin{align*}
\bigcap_{\theta\in [0,1/2)}H^{\theta,r}(a,b;Y_{1-\theta})
&\subseteq \bigcap_{\phi\in \{\beta_k',\varphi_k'\}}H^{1-\phi-\delta,r}(a,b;X_{\phi}).
\end{align*}
To complete the proof of Step 1a, it suffices to show that the latter space embeds in \eqref{eq:X_prime_simplified}. We only consider the case $\phi = \beta_k$ since the other case is similar. If $p \wt{r}_k>r$, then
$H^{1-\beta_k'-\delta,r}(a,b;X_{\beta_k'})\embed L^{p \wt{r}_k}(a,b;X_{\beta_k'})$
follows from Proposition \ref{prop:change_p_q_eta_a}\eqref{it:Sob_embedding} and the fact that $1-\beta_k'-\delta-\frac{1}{r}\geq -\frac{1}{p \wt{r}_k}=1-\beta_k'-\frac{1}{p}$ by \eqref{eq:parameters_correction_proof_Step_1a} and $\frac{1}{r}+\delta\leq \frac{1}{p}$. In case $p\wt{r}_k\leq r$, then the embedding follows from $H^{1-\beta_k'-\delta,r}\hookrightarrow L^r\hookrightarrow  L^{p\wt{r}_k}$, where we used $\delta<1-\varphi_k'\leq 1-\beta_k'$.

\emph{Step 1b: Proof of \eqref{eq:claim_step_1_regularization_tau_sigma}}.
Since $(u,\sigma)$ is an $L^p_{\a}$-maximal local solution to \eqref{eq:QSEE} in the $(X_0,X_1,p,\a)$-setting, $\nonlinearity_c^{\a}(u;0,t)<\infty$ for all for all $t<\sigma$ a.s.\ (here $\nonlinearity_c^{\a}$ is as in \eqref{eq:L_p_norm_nonlinearity}). We claim that
\begin{equation}
\label{eq:claim_step_1b_correction_trace_nonlinearity}
\lim_{t\uparrow \sigma} u(t)\text{ exists in }\Xap \ \  \text{ and }\ \  \nonlinearity_{\a}^c(u;0,\sigma)<\infty \ \ \text{on $\V\cap \{\sigma<\tau\}$.}
\end{equation}
Here the first part of the claim follows from \eqref{eq:regularity_v_n}, \eqref{eq:uniqueness_u_u_n}, $\lim_{k\to\infty}\tau_k=\tau$ and $\tau>\varepsilon$ a.s., and the embedding $\Yr\embed \Xp\embed \Xap$.
To obtain the second part of \eqref{eq:claim_step_1b_correction_trace_nonlinearity}, note that
by Step 1a and Lemma \ref{l:F_G_bound_N_C_cn} applied to $v$ in the unweighted case, we have  $\nonlinearity_{0}^c(u;\sigma/2,\sigma)<\infty$ a.s.\ on $\V\cap \{\sigma<\tau\}$. Combining this with $\nonlinearity_c^{\a}(u;0,\sigma/2)<\infty$ a.s.\ the claim follows.

By \eqref{eq:claim_step_1b_correction_trace_nonlinearity} and $\sigma<\tau\leq T$ a.s.\ on $\V\cap \{\sigma<\tau\}$, we obtain
\begin{align*}
&\P(\V\cap \{\sigma<\tau\})\\
&=\P\Big(\V \cap \{\sigma<\tau\}\cap \{\sigma<T\} \cap \big\{\lim_{t\uparrow \sigma}u(t)\text{ exists in }\Xap,\,\nonlinearity^{c}_{\a}(u;0,\sigma)<\infty\big\}\Big)\\
&\leq \P\Big(\sigma<T,\,\lim_{t\uparrow \sigma}u(t)\text{ exists in }\Xap,\,\nonlinearity^{c}_{\a}(u;0,\sigma)<\infty\Big)=0,
\end{align*}
where the last equality follows from Theorem \ref{t:blow_up_criterion}\eqref{it:blow_up_norm_general_case_F_c_G_c} (here we used Assumption \ref{H_a_stochnew} for $\ell\in \{0,\a\}$ and Assumption \ref{ass:FG_a_zero}). Therefore \eqref{eq:claim_step_1_regularization_tau_sigma} is proved.

{\em Step 2: $\tau = \wh{\tau}$ a.s.\ and $v=\wh{v}$ on $\ll \varepsilon,\tau\rro$}.
Since $(\wh{\tau}_k)_{k\geq 1}$ is a localizing sequence, for all $k\geq 1$, one has
\begin{equation}
\label{eq:regularity_z_n}
\wh{v}\in \bigcap_{\theta\in [0,1/2)}
H^{\theta,\wh{r}}(\varepsilon,	\wh{\tau}_k,w_{\wh{\alpha}}^{\varepsilon};\wh{Y}_{1-\theta})  \ \ \text{a.s.}
\end{equation}
Next, as in Step 1, we show that $(\wh{v},\wh{\tau})$ is an $L^{r}_{\alpha}$-local solution to \eqref{eq:QSEE} in the $(Y_0,Y_1,\alpha,r)$-setting. To this end, note that thanks to Hypothesis \hyperref[assum:HY]{\Hiep}$({Y}_0,{Y}_1,{\alpha},{r})$ and \hyperref[assum:HY]{\Hiep}$(\wh{Y}_0,\wh{Y}_1,\wh{\alpha},\wh{r})$,
and by density,
\begin{align*}
A_{\wh{Y},\wh{\alpha},\wh{r}}(\cdot,\wh{v})\wh{v} =A_{{Y},{\alpha},r}(\cdot,\wh{v})\wh{v},&  \qquad F_{\wh{Y},\wh{\alpha},\wh{r}}(\cdot,\wh{v})  =F_{{Y},{\alpha},r}(\cdot,\wh{v}),
\\ B_{\wh{Y},\wh{\alpha},\wh{r}}(\cdot,\wh{v})\wh{v} =B_{{Y},{\alpha},r}(\cdot,\wh{v})\wh{v}, & \qquad G_{\wh{Y},\wh{\alpha},\wh{r}}(\cdot,\wh{v})=G_{{Y},{\alpha},r}(\cdot,\wh{v})
\end{align*}
a.e.\ on $\ll \varepsilon,\wh{\tau}\rro$. The latter, \eqref{eq:emb_uniqueness_Z}, and \eqref{eq:regularity_z_n}, ensure that $(\wh{v},\wh{\tau})$ is also an $L^{r}_{\alpha}$-local solution to \eqref{eq:QSEE} in the $(Y_0, Y_1, r, \alpha)$-setting  with $s=\varepsilon$ and initial data $\one_{\V} u(\varepsilon)$. The maximality of $(v,\tau)$ gives  (see also Remark \ref{r:uniqueness_maximal})
\begin{equation}
\label{eq:uniqueness_v_z_n}
\wh{\tau}\leq \tau,\quad \text{ a.s.}, \qquad v=\wh{v}, \qquad \text{ a.e.\ on }\ll\varepsilon,\wh{\tau}\rro.
\end{equation}
It remains to prove $\tau\leq \wh{\tau}$ a.s. By \eqref{it:assumption_Yr_embeds_Yar_hat}, $\Yr\hookrightarrow \wh{Y}^{\Tr}_{\wh{\alpha},\wh{r}}$ and thus by \eqref{eq:regularity_v_n} and \eqref{eq:uniqueness_v_z_n},
$$
v=\wh{v}\in C((\varepsilon,\wh{\tau}];\Yr)\subseteq C((\varepsilon,\wh{\tau}];\wh{Y}^{\mathsf{Tr}}_{\wh{\alpha},\wh{r}}), \qquad \text{a.s.\ on }\{\wh{\tau}<\tau\}.
$$
Therefore, $\lim_{t\uparrow \wh{\tau}} \wh{v}(t)$ exists in $\wh{Y}^{\mathsf{Tr}}_{\wh{\alpha},\wh{r}}$ a.e.\ on $\{\wh{\tau}<\tau\}$. Since $\wh{\tau}<\tau\leq T$ on $\{\wh{\tau}<\tau\}$,
\begin{align*}
\P(\wh{\tau}<\tau)&
=\P\Big(\{\wh{\tau}<\tau\}\cap \{\wh{\tau}<T\}\cap \big\{\lim_{t\uparrow \wh{\tau}}\wh{v}(t)\;\text{exists in}\;\wh{Y}^{\mathsf{Tr}}_{\wh{\alpha},\wh{r}}\big\}\Big)\\
&\leq \P\Big(\wh{\tau}<T,\,\lim_{t\uparrow \wh{\tau}}\wh{v}(t)\;\text{exists in}\;\wh{Y}^{\mathsf{Tr}}_{\wh{\alpha},\wh{r}}\Big)=0,
\end{align*}
where in the last step we used condition \eqref{it:assumption_u_regular_Y2} in order to apply Theorem \ref{t:blow_up_criterion}\eqref{it:blow_up_non_critical_Xap} in the $(\wh{Y}_0, \wh{Y}_1,\wh{r}, \wh{\alpha})$-setting.

{\em Step 3: Conclusion}.
By Steps 1-2, $\sigma=\tau=\wh{\tau}$ a.s.\ on $\V$ and $u=v=\wh{v}$ on $\V\times [\varepsilon,\sigma)=\ll \varepsilon,\sigma\rro$. Let $(\sigma_n)_{n\geq 1}$ be the localizing sequence for $u$ defined in \eqref{eq:canonical_localizing_sequence_u_sigma}.
Then we have already seen that one has $\sigma_n<\sigma$ for all $n\geq 1$. Thus, by \eqref{eq:regularity_z_n} and the previous consideration, for all $n\geq 1$ and $\delta>0$,
\begin{equation}\label{eq:conclusion_regularity_varepsilon_varepsilon_prime}
\one_{\V} u \in  \bigcap_{\theta\in [0,1/2)} H^{\theta,\wh{r}}(\varepsilon,\sigma_n,w_{\wh{\alpha}}^{\varepsilon};\wh{Y}_{1-\theta})\subseteq \bigcap_{\theta\in [0,1/2)} H^{\theta,\wh{r}}(\varepsilon+\delta,\sigma_n;\wh{Y}_{1-\theta}) \  \ \text{a.s.}
\end{equation}
where we used that $\sigma_n<\sigma= \lim_{k\to \infty}\wh{\tau}_k$ a.s., Proposition \ref{prop:change_p_q_eta_a}\eqref{it:loc_embedding}.

Now let $\varepsilon_k=\delta_k=\frac{1}{2k}$, $\V_k = \{\sigma>\frac{1}{2k}\}$ and set $\Omega_0 = \bigcup_{k\geq 1} \V_k$. Let $(\wh{v}_k)_{k\geq 1}$ denote the corresponding $\wh{Y}_1$-valued solutions  defined on $\ll 1/k, \sigma\rro$. Since $\sigma>0$ a.s., $\P(\Omega_0) = 1$, and therefore, a.s.\ for all $k\geq 1$ and all $n\geq 1$,
\[u\in \bigcap_{\theta\in [0,1/2)} H^{\theta,\wh{r}}(\tfrac{1}{k},\sigma_n;\wh{Y}_{1-\theta}).\]
This implies the first part of \eqref{eq:regularity_u_Z}.
The final part of \eqref{eq:regularity_u_Z} follows from Proposition \ref{prop:continuousTrace}\eqref{it:trace_with_weights_Xap} in the unweighted case.

Finally, to check the progressive measurability of $u$ as a $\wh{Y}_1$-valued function, note that $u|_{\ll 1/k,\sigma\rro} = \hat{v}_k$ on $\V_k$ a.s. Since $\one_{[1/k,\sigma)\times \V_k}\hat{v}_k$ is strongly progressively measurable as a $\wh{Y}_1$-valued process and converges pointwise to $u$ a.s.\ we find that $u$ has the same property.
\end{proof}

In the special case $\wh{Y}_i = {Y}_i$, the above result simplifies and can be used to derive time-regularity.

\begin{corollary}[Bootstrapping time regularity]
\label{cor:regularization_X_0_X_1}
Let Hypothesis \hyperref[H:hip]{$\Hip$} be satisfied. Let $u_s\in L^0_{\F_s}(\O;\Xap)$ and that \eqref{eq:approximating_sequence_initial_data} holds. Suppose that
\begin{equation*}
(A(\cdot,u_{s,n}),B(\cdot,u_{s,n}))\in \MRtas, \ \ \ n\geq 1,
\end{equation*}
and let $(u,\sigma)$ be the $L^p_{\a}$-maximal local solution to \eqref{eq:QSEE} given by Theorem \ref{t:local_s}.
Suppose that Assumption \ref{H_a_stochnew} holds for $\ell\in \{0,\a\}$ and Assumption \ref{ass:FG_a_zero} holds.
\begin{enumerate}[{\rm(1)}]
\item\label{it:assumption_u_regular_Y2_cor_alpha_grather_than_zero} Suppose that Hypothesis \hyperref[assum:HY]{\Hiep}$(Y_0, Y_1, r,\alpha)$ holds with $r\in [p,\infty)$ and $\alpha\in (0,\frac{r}{2}-1)$,  Assumption \ref{H_a_stochnew} holds in the $(Y_0, Y_1, r, \alpha)$-setting for $\ell=\alpha$, and
\begin{itemize}
\item  $\Yr \embed \Xp$;
\item $Y_{\delta}\embed X_{0}$, $Y_{1}\embed X_{1-\delta}$ for some $\delta\in [0,1-\max_j \varphi_j')$ and $\frac{1}{r}+\delta \leq \frac{1}{p}$;
\item $u:\llo s,\sigma\rro\to Y_1$ is strongly progressively measurable and
\begin{equation*}
u\in \bigcap_{\theta\in [0,1/2)}H^{\theta,r}_{\emph{loc}}(s,{\sigma};Y_{1-\theta}) \ \ \text{a.s.};
\end{equation*}
\end{itemize}
\item\label{it:assumption_u_regular_Y2cor} let $\wh{r}\in [r,\infty)$ and suppose
Assumption \ref{H_a_stochnew} for $\ell\in \{0,\wh{\alpha}\}$ and Assumption \ref{ass:FG_a_zero} both hold in the $({Y}_0, {Y}_1,\wh{r}, \wh{\alpha})$-setting for all $\wh{\alpha}\in [0,\frac{\wh{r}}{2}-1)$ satisfying $\frac{1+\wh{\alpha}}{\wh{r}}<\frac{1+\alpha}{r}$.
\end{enumerate}
Then
\begin{equation}
\label{eq:regularity_u_xi}
u\in \bigcap_{\theta\in [0,1/2)}
H^{\theta,\wh{r}}_{\emph{loc}}(s,{\sigma};Y_{1-\theta}) \subseteq
C(s,{\sigma};Y^{\mathsf{Tr}}_{\wh{r}}) \ \ \text{a.s.}
\end{equation}
\end{corollary}
Note that if $f=0$ and $g=0$, then the above result may be applied with $\wh{r}$ arbitrary.
Recall that by Theorem \ref{t:local_s}, \eqref{it:assumption_u_regular_Y1} is satisfied in the case $X_i=Y_i$ (for $i\in\{0,1\}$), $r = p$ and $\alpha = \a>0$. In the proof of Corollary \ref{cor:regularization_X_0_X_1} we will see that it is enough to assume \eqref{it:assumption_u_regular_Y2cor} for a particular value $\wh{\alpha}$ such that $\frac{1+\wh{\alpha}}{\wh{r}}<\frac{1+\alpha}{r}$, and the corresponding trace space $\wh{Y}^{\Tr}_{\wh{\alpha}, \wh{r}}$ is not critical for \eqref{eq:QSEE} in the $(\wh{Y}_0,\wh{Y}_1,\wh{r},\wh{\alpha})$-setting.

\begin{proof}[Proof of Corollary \ref{cor:regularization_X_0_X_1}]
The idea is to apply Theorem \ref{t:regularization_z} with $\wh{Y}_0=Y_0,\wh{Y}_1=Y_1$, $\wh{r}>r$ and $\wh{\alpha}$  such that $\frac{1+\wh{\alpha}}{\wh{r}}<\frac{1+\alpha}{r}$ which will be chosen below. It remains to check Theorem \ref{t:regularization_z}\eqref{it:assumption_u_regular_Y1}-\eqref{it:assumption_Yr_embeds_Yar_hat}.
Note that \eqref{it:assumption_u_regular_Y1} holds by assumption.
Next, we check \eqref{it:assumption_u_regular_Y2}. Since hypothesis \hyperref[assum:HY]{\Hiep}$(Y_0,Y_1,\alpha,r)$ holds, there exist $\wt{m}_F,\wt{m}_G$, $(\wt{\varphi}_j)_{j=1}^{\wt{m}_F+\wt{m}_G}\subseteq (1-\frac{1+\alpha}{r},1)$, $(\wt{\beta}_j)_{j=1}^{\wt{m}_F+\wt{m}_G}$ such that $\wt{\beta}_j\in (1-\frac{1+\alpha}{r},\wt{\varphi}_j]$ and \ref{HFcritical}-\ref{HGcritical} hold with $(p,\a)$ replaced by $(r,\aaa)$. Set
$$
2\varepsilon:=\min_{j\in\{1,\dots,\wt{m}_F+\wt{m}_G\}} \Big\{\wt{\beta}_j -1+\frac{1+\aaa}{r},\frac{\aaa}{r}\Big\}>0,
$$
where we used that $\alpha>0$ by \eqref{it:assumption_u_regular_Y2_cor_alpha_grather_than_zero}.
In particular,
\begin{equation}
\label{eq:choice_r_cor_X_0_X_1_det}
\min_{j\in \{1,\dots,\wt{m}_F+\wt{m}_G\}}\{\wt{\beta}_j,\wt{\varphi}_j\}> 1-\frac{1+\aaa}{r}+\varepsilon,
\ \ \ \text{ and }\ \ \
 \frac{1}{r}<\frac{1+\aaa}{r}-\varepsilon.
\end{equation}
Since $\wh{\alpha}\in [0,\frac{\wh{r}}{2}-1)$ if and only if $\frac{1+\wh{\alpha}}{\wh{r}}\in [ \frac{1}{\wh{r}},\frac{1}{2})\supseteq [\frac{1}{r},\frac{1}{2})$ (where we used that $\wh{r}\geq r$), there exists an $\wh{\alpha}\in [0,\frac{\wh{r}}{2}-1)$ such that
\begin{equation}
\label{eq:choice_alpha_r}
\frac{1+\alpha}{r}-\varepsilon<\frac{1+\wh{\alpha}}{\wh{r}}<\frac{1+\aaa}{r}.
\end{equation}
Note that the above choice of $\wh{\alpha}$ and the second in \eqref{eq:choice_r_cor_X_0_X_1_det} yield
\begin{equation}
\label{eq:relation_Y_alpha_Y_wh_corollary_regularization_time}
Y_r\hookrightarrow Y_{\wh{\alpha},\wh{r}}^{\Tr}=\wh{Y}_{\wh{\alpha},\wh{r}}^{\Tr}
\ \  \  \
\text{ and }
\ \  \  \
Y_{\wh{\alpha},\wh{r}}^{\Tr}=\wh{Y}_{\wh{\alpha},\wh{r}}^{\Tr} \hookrightarrow Y_{\alpha,r}^{\Tr}.
\end{equation}
We claim that $F_{Y},G_Y$ satisfy \ref{HFcritical}-\ref{HGcritical} with $(p,\a)$ replaced by $(\wh{r},\wh{\alpha})$. To see this, note that by \eqref{eq:choice_r_cor_X_0_X_1_det}-\eqref{eq:relation_Y_alpha_Y_wh_corollary_regularization_time}, $\wt{\varphi}_j,\wt{\beta}_j>1-\frac{1+\aaa}{r}+\varepsilon>1-\frac{1+\wh{\alpha}}{\wh{r}}$, and thus \eqref{eq:HypCritical}, \eqref{eq:HypCriticalG} hold with $(p,\a,\rho_j,\varphi_i,\beta_j,m_F,m_G)$ replaced by $(\wh{r},\wh{\alpha},\wt{\rho}_j,\wt{\varphi}_i,\wt{\beta}_j,\wt{m}_F,\wt{m}_G)$.
Moreover, due to the fact that $\frac{1+\wh{\alpha}}{\wh{r}}<\frac{1+\aaa}{r}$, \eqref{eq:HypCritical} and \eqref{eq:HypCriticalG} hold with the \emph{strict} inequality and thus $\wh{Y}_{\wh{\alpha},\wh{r}}^{\Tr}$ is not critical for \eqref{eq:QSEE} in the $(Y_0,Y_1,\wh{r},\wh{\alpha})$-setting. Finally, \ref{HAmeasur} holds by the second inclusion in \eqref{eq:relation_Y_alpha_Y_wh_corollary_regularization_time}.

Due to the first inclusion in \eqref{eq:relation_Y_alpha_Y_wh_corollary_regularization_time}, to check Theorem \ref{t:regularization_z}\eqref{it:assumption_Yr_embeds_Yar_hat} it remains to note that \eqref{eq:emb_uniqueness_Z} with the above choice of $(\wh{Y}_0,\wh{Y}_1,\wh{\alpha},\wh{r})$ follows from Lemma \ref{l:emb_z}\eqref{it:regularization_z_strict_inequality} and the upper bound in \eqref{eq:choice_alpha_r}.
\end{proof}

\begin{remark}
In the deterministic case, part of the arguments used in Theorem \ref{t:regularization_z} appears in \cite{P17_geostrophic,PW18,hieber2020primitive}. Our systematic treatment appears to be new. Let us note that an essential step in the proof is to use blow-up criteria to show the invariance of the explosion time $\sigma$ in the different settings.
\end{remark}

\begin{remark}
\label{r:regularization_z_splitting}
In some application (see \cite{AV19_QSEE_3}), we need a straightforward extension of Theorem \ref{t:regularization_z} and Corollary \ref{cor:regularization_X_0_X_1}, where in Assumptions \ref{H_a_stochnew}, \ref{ass:FG_a_zero} and \ref{assum:HY} we assume that there is a splitting in a principle and a lower order part of the form
\begin{equation}
\begin{aligned}
\label{eq:A_B_splitting}
A(\cdot,\vone)\vone&=A_{Y,\text{princ}}(\cdot,\vone)\vone+A_{Y,\text{lower}}(t,\vone),
\quad
\\ B(\cdot,\vone)\vone& =B_{Y,\text{princ}}(t,\vone)\vone+B_{Y,\text{lower}}(\cdot,\vone),
\end{aligned}
\end{equation}
on $Y_0$. If \eqref{eq:A_B_splitting} holds, then it can be useful to consider the corresponding assumptions on $(A_{Y,\text{princ}},B_{Y,\text{princ}},F_Y+A_{\text{lower}},G_Y+B_{\text{lower}})$.
One can check that Theorem \ref{t:regularization_z} and Corollary \ref{cor:regularization_X_0_X_1} extend to this setting.
\end{remark}

\subsection{The emergence of weights}
\label{ss:emergence_weights}
In Theorem \ref{t:regularization_z}\eqref{it:assumption_u_regular_Y2}, there are two difficulties:
\begin{itemize}
\item it is not applicable in the critical case;
\item it is often not applicable in the unweighted setting.
\end{itemize}
In this subsection we show how to create a \emph{weighted} situations from an unweighted one, which also allows criticality. For simplicity we only consider the case where we add a weight near $t=s$ as this is what is needed to start a bootstrapping argument. Moreover, we only consider the semilinear setting, as the extension to the quasilinear setting is quite cumbersome and harder to state. Unlike Theorem \ref{t:regularization_z} the case $\a=0$ is allowed, which is central in many applications.

Recall that $(\bar{A}_Y,\bar{B}_Y)$ is as below Assumption \ref{assum:HY}. We need an additional condition on $F$ and $G$. Fix $r\in [1, \infty]$. Suppose that for each $n\geq 1$ there is a constant $C_{n}$ such that for a.a.\ $\omega\in \O$, for all $t\in [s,T]$ and $\|x\|_{(X_0,X_1)_{1-\frac{1}{p},r}},\|y\|_{(X_0,X_1)_{1-\frac{1}{p},r}}\leq n$,
\begin{align}
\nonumber \|F_c(t,\om,x)\|_{X_{0}}&\leq C_{n} \sum_{j=1}^{m_F}(1+\|x\|_{X_{\varphi_j}}^{\rho_j})\|x\|_{X_{\beta_j}}+C_{n}
\\ \label{eq:condFGstrange}
\|G_c(t,\om,x)\|_{\g(H,X_{1/2})}&\leq C_{n} \sum_{j=m_F+1}^{m_F+m_G}(1+\|x\|_{X_{\varphi_j}}^{\rho_j})\|x\|_{X_{\beta_j}}+C_{n},\\
\|F_{\Tr}(t,\om,x)\|_{X_0} + \|G_{\Tr}(t,\om,x)\|_{\g(H,X_{1/2})}&\leq C_{n}(1+\|x\|_{(X_0,X_1)_{1-\frac{1}{p},r}}).\nonumber
\end{align}
This coincides with the growth condition in \ref{HFcritical} and \ref{HGcritical} if $r=p$ and $\a=0$.

The main result of this section is the following:
\begin{proposition}[Adding weights at the initial time]
\label{prop:adding_weights}
Let Hypothesis \hyperref[H:hip]{$\Hip$} be satisfied with $\a=0$.
Let $r\in [p,\infty)$, $r>2$, $\alpha\in [0,\frac{r}{2}-1)$, and suppose that \eqref{eq:condFGstrange} holds.
Let $u_s\in L^0_{\F_s}(\O;\Xp)$ and suppose that
\begin{equation}
\label{eq:stochastic_maximal_regularity_semilinear_adding_lemma}
(A(\cdot,x),B(\cdot,x))\equiv (\bar{A}(\cdot),\bar{B}(\cdot))\in \MRts, \ \ \text{ for all }x\in X_1.
\end{equation}
Let $(u,\sigma)$ be the $L^p_{0}$-maximal local solution to \eqref{eq:QSEE} of Theorem \ref{t:local_s}. Suppose that
$\delta\in [0,1-\max_j \varphi_j)$, where $p>2$ in case $\delta=0$, and the following are satisfied:
\begin{enumerate}[{\rm(1)}]
\item\label{it:assumption_u_regular_Y2_lemma_adding_weights1} Hypothesis \hyperref[assum:HY]{\Hiep}$({Y}_0,{Y}_1,{\alpha},{r})$, $(\bar{A}_{Y},\bar{B}_{Y})\in \mathcal{SMR}^{\bullet}_{r,\alpha}(t,T)$ for all $t\in (s,T)$ and Assumption \ref{ass:FG_a_zero} hold in the $(Y_0,Y_1,\alpha,r)$-setting,
$$
Y_{\delta}=X_0, \quad Y_{1}=X_{1-\delta}, \quad \frac{1}{p}=\frac{1+\alpha}{r}+\delta,
\quad \text{and} \quad
 \frac{1}{r}\geq \max_j \varphi_j-1+\frac{1}{p};
 $$
\item\label{it:assumption_u_regular_Y2_lemma_adding_weights2} $(\bar{A}_{Y},\bar{B}_{Y})\in \mathcal{SMR}_{q}(t,T)$ for all $t\in (s,T)$ and $q\in [2,r]$.
\end{enumerate}
Then
\begin{equation}
\label{eq:regularity_u_adding_weights}
u
\in
\bigcap_{\theta\in [0,1/2)}
H^{\theta,r}_{\emph{loc}}(s,{\sigma};X_{1-\delta-\theta}) \subseteq
C(s,{\sigma};(X_0,X_1)_{1-\delta-\frac{1}{r},r}) \ \ \text{a.s.}
\end{equation}
\end{proposition}

Proposition \ref{prop:adding_weights} allows to bootstrap regularity in time provided $r>p$ is not too big at the expense of reducing the regularity `in space' in the case that $\delta>0$. Moreover, if $\alpha>0$, then $(X_0,X_1)_{1-\delta-\frac{1}{r},r} = (X_0,X_1)_{1-\frac{1}{p}+\frac{\alpha}{r},r}\hookrightarrow \Xp$ and therefore Proposition \ref{prop:adding_weights} also yields a regularization in space.
One of the interesting features of Proposition \ref{prop:adding_weights} is that in applications to SPDEs one can fix $r$ and choose $\delta$ so small that $\frac{\alpha}{r}=\frac{1}{p}-\frac{1}{r}-\delta>0$. In that case, the conditions in Corollary \ref{cor:regularization_X_0_X_1}\eqref{it:assumption_u_regular_Y2_cor_alpha_grather_than_zero} follow from $\varphi'_j=\varphi_j$ and Proposition \ref{prop:adding_weights}\eqref{it:assumption_u_regular_Y2_lemma_adding_weights1}. Therefore, if Corollary \ref{cor:regularization_X_0_X_1}\eqref{it:assumption_u_regular_Y2cor} holds as well, then one can obtain further integrability in time afterward. After that one can bootstrap further regularity via Theorem \ref{t:regularization_z}.

In the case $p>2$, one usually takes $\delta=0$. This is not allowed if $p=2$, since $\frac{1}{p}=\frac{1}{2}>\frac{1+\alpha}{r}$ for all $r\in [2,\infty)$, $\alpha\in [0,\frac{r}{2}-1)$ and therefore \eqref{it:assumption_u_regular_Y2_lemma_adding_weights1} can hold if and only if $\delta>0$. In this case, $Y_{0}$ can be thought as ``$X_{-\delta}$" and (typically) can be defined as $X_{-\delta,\wt{A}}$, i.e.\ the extrapolated space  (see e.g.\ \cite[Appendix A]{AV19_QSEE_1}) constructed via a sectorial operator $\wt{A}$ on $X_0$ such that $\Do(\wt{A})=X_1$.

Actually, Proposition \ref{prop:adding_weights} holds under more general assumptions, and it has a version for a quasilinear equations.
However, we prefer to state Proposition \ref{prop:adding_weights} in its current simple form as this is enough for many of the applications we have in mind and is less technical.

\begin{proof}[Proof of Proposition \ref{prop:adding_weights}]
As usual, we set $s=0$.
Due to \cite[Theorems 4.6.1 and 4.7.2]{BeLo} and the fact that $Y_{\delta}=X_0$, $Y_{1}=X_{1-\delta}$ we have
\begin{equation}
\label{eq:whX_X_relation_delta}
Y_{\theta}=X_{\theta-\delta}, \text{ and }
(Y_0,Y_1)_{\theta,\zeta}=(X_0,X_1)_{\theta-\delta,\zeta}  \ \text{ for all }\theta\in (\delta,1),\,\zeta\in [1,\infty].
\end{equation}
The former, the fact that $r\geq p$ and $\frac{1+\alpha}{r}=\frac{1}{p}-\delta$ imply that for all $\varepsilon>0$
\begin{equation}
\label{eq:initial_data_adding_weights_lemma}
 \one_{\V} u(\varepsilon)\in L^0_{\F_{\varepsilon}}(\O;X^{\Tr}_{p})
\subseteq
L^0_{\F_{\varepsilon}}(\O;Y_{\alpha,r}^{\Tr}) \ \ \text{ where } \ \ \V:=\{\sigma>\varepsilon\}.
\end{equation}
As explained below Assumption \ref{assum:HY}, by \eqref{eq:initial_data_adding_weights_lemma} and Hypothesis \hyperref[assum:HY]{\Hiep}$(Y_0,Y_1,\alpha,r)$, Theorem \ref{t:local_s} gives existence of an $L^r_{\alpha}$-maximal local solution $(v,\tau)$ to \eqref{eq:QSEE} on $[\varepsilon,T]$ with initial data $\one_{\V}u(\varepsilon)$ in the $(Y_0,Y_1,\alpha,r)$-setting. Since $r>2$ by assumption, Theorem \ref{t:local_s}\eqref{it:regularity_data_L0} ensures that a.s.
\begin{equation}
\label{eq:regularity_v_adding_weight_lemma}
v\in \bigcap_{\theta\in [0,1/2)}
H^{\theta,r}_{\rm loc}([\varepsilon,\tau),w_{\alpha}^{\varepsilon};Y_{1-\theta})\subseteq C([\varepsilon,\tau);Y^{\mathsf{Tr}}_{\alpha,r})=C([\varepsilon,\tau);(X_0, X_1)_{1-\frac1p,r}),
\end{equation}
where the latter is not the ``right'' trace space in the $(X_0,X_1,p,\a)$-setting. As in the proof of Theorem \ref{t:regularization_z}, to prove \eqref{eq:regularity_u_adding_weights} it remains to show that
\begin{equation}
\label{eq:uniqueness_u_v_lemma_adding_weights}
\tau= \sigma \ \  \text{a.s.\ }\ \ \text{ and } \ \ u=v \ \ \text{ a.e.\ on }\ll \varepsilon,\sigma\rro.
\end{equation}
Indeed, if \eqref{eq:uniqueness_u_v_lemma_adding_weights} holds, then \eqref{eq:regularity_u_adding_weights} follows from \eqref{eq:regularity_v_adding_weight_lemma}, the arbitrariness of $\varepsilon>0$, and the argument in Step 3 of the proof of Theorem \ref{t:regularization_z}.

For the reader's convenience, we split the proof of \eqref{eq:uniqueness_u_v_lemma_adding_weights} into several steps. In Step 1 we prove that $\tau\leq \sigma$ a.s.\ and $u=v$ a.e.\ on $\V\times [\varepsilon,\tau)$ assuming that
\begin{equation}
\label{eq:emb_uniqueness_Z_serrin_case}
\bigcap_{\theta\in [0,1/2)}H^{\theta,r}(a,b,w_{\alpha}^a;Y_{1-\theta})
\subseteq \X(a,b)
,\quad \text{ for all }0\leq a<b<\infty,
\end{equation}
in Step 2 we prove \eqref{eq:uniqueness_u_v_lemma_adding_weights}, and in Step 3 we prove \eqref{eq:emb_uniqueness_Z_serrin_case}. Here $\X$ is as in \eqref{eq:def_X_space}.

\textit{Step 1: $\tau\leq \sigma$ a.s.\ and $u=v$ a.e.\ on $\V\times [ \varepsilon,\tau)$}. By uniqueness of the $L^p_0$-maximal local solution $(u,\sigma)$ (see also Remark \ref{r:uniqueness_maximal}), to prove the claim of this step it remains to check that $(v,\tau)$ is an $L^p_0$-local solution to \eqref{eq:QSEE} on $[\varepsilon,T]$ in the $(X_0,X_1,p,0)$-setting with initial data $\one_{\V}u(\varepsilon)$. Since Hypothesis \hyperref[assum:HY]{\Hiep}$(Y_0,Y_1,\alpha,r)$ holds, it is enough to check that the process $v$ has the required regularity for being an $L^p_{0}$-local solution to \eqref{eq:QSEE} on $[\varepsilon,T]$ in the $(X_0,X_1,p,0)$-setting (see Definitions \ref{def:solution1}-\ref{def:solution2} and Lemma \ref{l:F_G_bound_N_C_cn}), i.e.\
\begin{equation}
\label{eq:v_regularity_for_being_Lp0_solution}
v\in L^p(\varepsilon,\tau_k;X_1)\cap C([\varepsilon,\tau_k];\Xp)\cap \X(\varepsilon,\tau_k) \ \ \text{ a.s.\ for all }k\geq 1,
\end{equation}
for a suitable localizing sequence $(\tau_k)_{k\geq 1}$ for $(v,\tau)$.
By \eqref{eq:regularity_v_adding_weight_lemma} and \eqref{eq:emb_uniqueness_Z_serrin_case}, $v\in \X(\varepsilon,\tau_k)$ a.s., and
thus it remains to prove the first two parts of \eqref{eq:v_regularity_for_being_Lp0_solution}.

To proceed, we need a localization argument.
For $j\geq 1$, set
\begin{equation}
\label{eq:V_n_proof_Lp0}
\V_j:=\V\cap \{\|u(\varepsilon)\|_{\Xp}\leq j\}\in \F_{\varepsilon}.
\end{equation}
By \eqref{eq:regularity_v_adding_weight_lemma} and \eqref{eq:emb_uniqueness_Z_serrin_case}, we can define a localizing sequence by
\begin{equation}
\label{eq:tau_j_localizing_valphar}
\tau_j:=\inf\{t\in [\varepsilon,\tau)\,:\,
\|v\|_{L^r(\varepsilon,t,w_{\a}^{\varepsilon};Y_1)}+\|v(t)\|_{(X_0, X_1)_{1-\frac1p,r}}+\|v\|_{\X(\varepsilon,t)}\geq j\},
\end{equation}
where $\inf\emptyset:=\tau$,
and moreover, $(\tau_j)_{j\geq 1}$ is a localizing sequence for $(v,\tau)$.
Due to \eqref{eq:condFGstrange} one can check that Lemma \ref{l:F_G_bound_N_C_cn} is also valid if $X^{\mathsf{Tr}}_{\a,p}$ is replaced by $(X_0, X_1)_{1-\frac1p,r}$ everywhere. Therefore, by \eqref{eq:tau_j_localizing_valphar}, we obtain that for all $j\geq 1$,
\begin{equation}
\begin{aligned}
\label{eq:F_v_composition_integrability_adding_lemma}
F_j:=\one_{ [ \varepsilon,\tau_j]\times \V_j }F(\cdot,v)&\in L^{\infty}(\O;L^p(\varepsilon,T;X_0)),\\
G_j:=\one_{ [ \varepsilon,\tau_j]\times \V_j } G(\cdot,v)&\in L^{\infty}(\O;L^p(\varepsilon,T;\g(H,X_{1/2}))).
\end{aligned}
\end{equation}
Due to \eqref{eq:stochastic_maximal_regularity_semilinear_adding_lemma}, \eqref{eq:V_n_proof_Lp0} and \eqref{eq:F_v_composition_integrability_adding_lemma}, for each $j\geq 1$ there exists a strong solution
$
z_j\in L^p_{\Progress} (\O;L^{p}(\varepsilon,T;X_{1})\cap C([\varepsilon,T];\Xp))
$
to the following (see Definition \ref{def:strong_linear})
\begin{equation}
\label{eq:w_k}
\begin{cases}
d z_j +\bar{A}(\cdot) z_j dt&=
F_j dt+
(\bar{B}(\cdot)z_j+ G_j)dW_H, \ \ \text{ on }[\varepsilon,T],\\
 z_j(\varepsilon)&=\one_{\V_j}u(\varepsilon).
\end{cases}
\end{equation}
Recall that $(v,\tau)$ is an $L^r_{\alpha}$-maximal local solution to \eqref{eq:QSEE} on $[\varepsilon,T]$ in the $(Y_0,Y_1,\alpha,r)$-setting with initial data $\one_{	\V}u(\varepsilon)$. By \eqref{eq:tau_j_localizing_valphar},  $v\in L^{\infty}(\O;L^r(\varepsilon,\tau_j,w_{\alpha}^{\varepsilon};Y_1))$. For all $j\geq 1$, set
\begin{equation}
\label{eq:Vj_regularity_adding_weights}
v_j:=\one_{\V_j}(v-z_j)\in L^p(\O;L^p(\varepsilon,\tau_j;X_1) + L^r(\varepsilon,\tau_j,w_{\alpha}^{\varepsilon};Y_1)).
\end{equation}
Then $v_j$ is a strong solution to the following problem on $[\varepsilon,\tau_j]\times \V_j$
\begin{equation}
\label{eq:V_proof_adding_weights}
\begin{cases}
d v_j +\bar{A}_{Y}(\cdot)v_j dt= \bar{B}_{Y}(\cdot)v_j  dW_H,\ \  \text{ on }\ll \varepsilon,T\rr,\\
 v_j (\varepsilon)=0,
\end{cases}
\end{equation}
where $(\bar{A}_{Y},\bar{B}_{Y})$ are as below Assumption \ref{assum:HY}. Due to the regularity of $z_j$, it remains to prove that $v_j= 0$ a.e.\ on $\ll \varepsilon,\tau_j\rr$. To this end, we apply assumption \eqref{it:assumption_u_regular_Y2_lemma_adding_weights2}. For the sake of clarity we divide the argument into two cases.
\begin{enumerate}[{\rm(1)}]
\item \emph{Case $\delta=0$, $p>2$}. Recall that $r\geq p$ and $\frac{1+\alpha}{r}=\frac{1}{p}$ by \eqref{it:assumption_u_regular_Y2_lemma_adding_weights1}. Fix $q\in (2, p)$. Then $\frac{1+\alpha}{r}<\frac{1}{q}$ and Proposition \ref{prop:change_p_q_eta_a}\eqref{it:change_a_p_eta_a} yields $L^r(\I_t,w_{\alpha};Y_{1})
\hookrightarrow L^q(\I_t;Y_1)$ for all $t>0$. Recalling that $X_1= Y_1$ (due to $\delta=0$), we have
$$
L^p(\I_t;X_1)+L^r(\I_t,w_{\alpha};Y_{1})
\hookrightarrow L^q(\I_t;X_1) \ \ \text{ for all }t>0.
$$
The former and \eqref{eq:Vj_regularity_adding_weights} ensure $v_j\in L^q(\O;L^q(\varepsilon,\tau_j;X_1))$. Therefore, $v_j\equiv 0$ by \eqref{eq:V_proof_adding_weights}, $(\bar{A}_Y,\bar{B}_Y)=(\bar{A},\bar{B})\in \mathcal{SMR}_q(\varepsilon,T)$, and Proposition \ref{prop:causality_phi_revised_2}.
\item \emph{Case $\delta>0$}. Since $\frac{1+\alpha}{r}=\frac{1}{p}-\delta<\frac{1}{p}$ we have $L^r(\I_t,w_{\alpha};Y_{1})
\hookrightarrow L^p(\I_t;Y_1)$ for all $t>0$ by Proposition \ref{prop:change_p_q_eta_a}\eqref{it:change_a_p_eta_a}. Recalling that $X_1\hookrightarrow Y_1$, we have
$$
L^p(\I_t;X_1)+
L^r(\I_t,w_{\alpha};Y_{1})
\hookrightarrow L^p(\I_t;Y_1)\ \ \text{ for all }t>0.
$$
As above, the former and \eqref{eq:Vj_regularity_adding_weights} imply that $v_j\in L^p(\O;L^p(\varepsilon,\tau_j;Y_1))$. Therefore, $v_j\equiv 0$ by \eqref{eq:V_proof_adding_weights}, $(\bar{A}_Y,\bar{B}_Y)\in \mathcal{SMR}_p(\varepsilon,T)$ and Proposition \ref{prop:causality_phi_revised_2}.
\end{enumerate}

\textit{Step 2: \eqref{eq:uniqueness_u_v_lemma_adding_weights} holds}. By Step 1 and \eqref{eq:regularity_v_adding_weight_lemma} it is enough to show that $\tau\geq \sigma$ a.s.\ on $\V$ and this will be done via Theorem \ref{thm:semilinear_blow_up_Serrin}\eqref{it:blow_up_semilinear_serrin_Pruss_modified}. To this end, we claim that it suffices to show that
\begin{equation}
\label{eq:v_regularity_adding_weights_lemma_0_probability}
v\in L^{r}(\varepsilon,\tau;Y_{1-\frac{\alpha}{r}})\cap C([\varepsilon,\tau];Y^{\Tr}_{\alpha,r}) \ \
\text{a.s.\ on }\V\cap\{\tau<\sigma\}.
\end{equation}
Indeed, if \eqref{eq:v_regularity_adding_weights_lemma_0_probability} holds, then
\begin{align*}
\P(\V\cap \{\tau<\sigma\})&\stackrel{(i)}{=}\P\Big(\V\cap \{\tau<\sigma\}\cap
\Big\{\sup_{t\in [\varepsilon,\tau)}\|v(t)\|_{Y^{\Tr}_{\alpha,r}}+
\|v\|_{L^{r}(\varepsilon,\tau;Y_{1-\frac{\alpha}{r}})}<\infty\Big\}\Big)\\
&\stackrel{(ii)}{\leq} \P\Big(\tau<T,\,\sup_{t\in [\varepsilon,\tau)}\|v(t)\|_{Y^{\Tr}_{\alpha,r}}+
\|v\|_{L^{r}(\varepsilon,\tau;Y_{1-\frac{\alpha}{r}})}<\infty\Big)=0,
\end{align*}
where in $(i)$ we used  \eqref{eq:v_regularity_adding_weights_lemma_0_probability}, and in $(ii)$ we used Theorem \ref{thm:semilinear_blow_up_Serrin}\eqref{it:blow_up_semilinear_serrin_Pruss_modified}. As usual we also used that $\tau<T$ on $\{\tau<\sigma\}$.

To prove \eqref{eq:v_regularity_adding_weights_lemma_0_probability}, recall that $\tau\leq \sigma$ a.s.\ on $\V$ and $u=v$ a.e.\ on $[\varepsilon,\tau)\times \V$ by Step 1. The latter, \eqref{eq:whX_X_relation_delta} and the fact that $r\geq p$ ensure
\begin{equation}
\label{eq:v_u_p_unweighted_case_larger_than_2_C}
v=u\in C([\varepsilon,\tau];\Xp) \subseteq C([\varepsilon,\tau];Y^{\Tr}_{\alpha,r}), \ \ \text{ a.s.\ on }\V\cap\{\tau<\sigma\}.
\end{equation}
To complete the proof of \eqref{eq:v_regularity_adding_weights_lemma_0_probability}, we need to prove $v\in L^r(\varepsilon,\tau;Y_{1-\frac{\alpha}{r}})$ a.s.\ on $\V\cap\{\tau<\sigma\}$. To this end, we consider $p>2$ and $p=2$ separately.

If $p>2$, then by Step 1 and Theorem \ref{t:local_s}\eqref{it:regularity_data_L0} applied with $\theta = \frac{\alpha}{r}+\delta = \frac1p-\frac1r<\frac12$ we have, a.s.\ on $\V\cap \{\tau<\sigma\}$,
\begin{equation*}
\begin{aligned}
v=u&\in
H^{\frac{\alpha}{r}+\delta,p}(\varepsilon,\tau;X_{1-\frac{\alpha}{r}-\delta})
\stackrel{\eqref{eq:whX_X_relation_delta} }{=}
H^{\frac{\alpha}{r}+\delta,p}(\varepsilon,\tau;Y_{1-\frac{\alpha}{r}})
\stackrel{(iii)}{\hookrightarrow}
L^{r}(\varepsilon,\tau;Y_{1-\frac{\alpha}{r}})
\end{aligned}
\end{equation*}
where $(iii)$ follows from Proposition \ref{prop:change_p_q_eta_a}\eqref{it:Sob_embedding}.

If $p=2$, then instead of Sobolev embedding, we can use the following standard interpolation inequality for $0\leq a<b<\infty$ and $\theta\in (0,1)$:
\begin{equation}
\label{eq:interpolation_inequalities_p_2}
C([a,b];X_{1/2})\cap L^2(a,b;X_1)\hookrightarrow L^{2/\theta}(a,b;X_{(1+\theta)/2}).
\end{equation}
By Theorem \ref{t:local_s}\eqref{it:regularity_data_L0} with $p=2$ and Step 1 we have, a.s.\ on $\V\cap \{\tau<\sigma\}$,
\begin{align*}
v=u\in C([\varepsilon,\tau];X_{1/2})\cap L^2(\varepsilon,\tau;X_1)
\stackrel{(iv)}{\hookrightarrow}
 L^r(\varepsilon,\tau;X_{1-\frac{\alpha}{r}-\delta})
  \stackrel{\eqref{eq:whX_X_relation_delta}}{=}
  L^r(\varepsilon,\tau;Y_{1-\frac{\alpha}{r}}),
\end{align*}
where in $(iv)$ we used \eqref{eq:interpolation_inequalities_p_2} with $\theta=1-2(\frac{\alpha}{r}+\delta)=\frac{2}{r}\in (0,1)$.

\textit{Step 3: \eqref{eq:emb_uniqueness_Z_serrin_case} holds}. By translation and scaling, it is enough to prove \eqref{eq:emb_uniqueness_Z_serrin_case} for $a=0$ and $b=T$.
 Fix $k\in \{1,\dots,{m}_F+{m}_G\}$. Recall that $\a=0$ and by \eqref{eq:rhostar}-\eqref{eq:rr'}
\begin{equation}
\label{eq:parameters_unweighted}
\frac{1}{\rhos_k p r_k'}=\varphi_k-1+\frac{1}{p}
\ \ \ \text{ and }\ \ \
\frac{1}{ p r_k}=\beta_k-1+\frac{1}{p}.
\end{equation}
By Hypothesis \hyperref[H:hip]{$\Hip$}  for $\phi\in \{\beta_k,\varphi_k\}$ we have $\delta<1-\phi$ (since $\delta<1-\varphi_k$ and $\varphi_k\geq \beta_k$)
and $1-\phi-\delta<\frac{1}{2}$ (since $\phi>1-\frac{1}{p}>\frac{1}{2}$).
Thus, to prove \eqref{eq:emb_uniqueness_Z_serrin_case} note that
\begin{equation}
\label{eq:emb_step_3_lemma_adding_weights_k}
\begin{aligned}
\bigcap_{\theta\in [0,1/2)}H^{\theta,r}(\I_T,w_{\alpha};Y_{1-\theta})
&\subseteq \bigcap_{\phi\in \{\varphi_k,\beta_k\}}
H^{1-\phi-\delta,r}(\I_T,w_{\alpha};Y_{\phi+\delta})\\
&\stackrel{(i)}{=}
\bigcap_{\phi\in \{\varphi_k,\beta_k\}}  H^{1-\phi-\delta,r}(\I_T,w_{\alpha};X_{\phi})\\
& \stackrel{(ii)}{\hookrightarrow}
L^{\rhos_k p r'_k}(\I_T;X_{\varphi_k})\cap
L^{p r_k }(\I_T;X_{\beta_k})
\end{aligned}
\end{equation}
where in $(i)$ we used \eqref{eq:whX_X_relation_delta}, and in $(ii)$ we used \eqref{eq:parameters_unweighted}, Proposition \ref{prop:change_p_q_eta_a}\eqref{it:Sob_embedding} and
$$
r\leq \min \{ p r_k, \rhos_k p r_k'\},  \ \
\ 1-\varphi_k -\delta-\frac{1+\alpha}{r}=-\frac{1}{\rhos_k p r'_k},\ \ 1-\beta_k-\delta-\frac{1+\alpha}{r}=-\frac{1}{p r_k}.
$$
Note that $r\leq \min \{ p r_k, \rhos_k p r_k'\}$ is equivalent to $\frac{1}{r}\geq  \frac{1}{\rhos_k p r_k'}={\varphi}_k-1+\frac{1}{p}$ and $\frac{1}{r}\geq \frac{1}{p r_k}= \beta_k-1+\frac{1}{p}$ (see \eqref{eq:parameters_unweighted}), which hold by the last condition in \eqref{it:assumption_u_regular_Y2_lemma_adding_weights1} and the fact that $\beta_k\leq \varphi_k$. Since $k$ was arbitrary \eqref{eq:emb_uniqueness_Z_serrin_case} follows from \eqref{eq:def_X_space} and \eqref{eq:emb_step_3_lemma_adding_weights_k}.
\end{proof}

\begin{remark}
\label{r:regularity_up_to_zero_adding_weights}
If additionally in Proposition \ref{prop:adding_weights}, $(\bar{A}_Y,\bar{B}_Y)\in \mathcal{SMR}^{\bullet}_{r,\alpha}(s,T)$, $f\in L^0_{\Progress}(\O;L^r(s,T,w_{\alpha};Y_0))$ and $g\in L^0_{\Progress}(\O;L^r(s,T,w_{\alpha};\g(H,Y_{1/2})))$, then
\begin{equation}
\label{eq:improved_regularity_up_to_0_weighted}
u\in
\bigcap_{\theta\in [0,1/2)}
H^{\theta,r}_{\rm loc}([s,{\sigma}),w_{\alpha}^s;X_{1-\delta-\theta})\ \ \text{a.s.}
\end{equation}
Indeed, this follows by taking $\varepsilon=0$ in \eqref{eq:initial_data_adding_weights_lemma}, and using $(\bar{A}_Y,\bar{B}_Y)\in \mathcal{SMR}_{r,\alpha}^{\bullet}(s,T)$ and $\Xp\hookrightarrow Y^{\Tr}_{\alpha,r}$.
\end{remark}

The previous regularization results allow us to prove instantaneous regularization for $L^p_{\a}$-maximal local solutions to \eqref{eq:QSEE}. In applications to SPDEs, one can employ the following extrapolation result to transfer the regularity and life-span of solutions for a given setting to another one.

\begin{lemma}[Extrapolating regularity and life-span]
\label{lem:extrapolation}
Let Hypothesis \hyperref[H:hip]{$\Hip$} be satisfied. Let $u_s\in L^0_{\F_s}(\O;\Xap)$
and suppose that $(u,\sigma)$ the $L^p_{\a}$-maximal local solution to \eqref{eq:QSEE} exists. Suppose that Assumption \ref{H_a_stochnew} for $\ell\in \{0,\kappa\}$ and \ref{ass:FG_a_zero} are satisfied, and that the following conditions hold for a given $\varepsilon\in (0,T-s)$:
\begin{enumerate}[{\rm(1)}]
\item\label{it:assumption_u_regularizes_to_Xhat_setting} Hypothesis \hyperref[assum:HY]{\Hiep}$(Y_0,Y_1,r,\alpha)$ and Assumption \ref{H_a_stochnew} for $\ell=\alpha$ hold in the $(Y_0,Y_1,r,\alpha)$-setting, and one of the following holds:
    \begin{itemize}
    \item $u\in C((s,\sigma);Y_{1/2})\cap L^2_{{\rm loc}}(s,\sigma;Y_1)$ a.s.\ and $r=2$,
    \item
    $u\in \bigcap_{\theta\in [0,1/2)} H^{\theta,r}_{\loc}(s,\sigma;Y_{1-\theta})$ a.s.\ and $r>2$;
    \end{itemize}
\item\label{it:assumption_existence_of_global_solutions_Xhat_setting} Hypothesis \hyperref[assum:HY]{\Hiep}$(\wh{Y}_0,\wh{Y}_1,\wh{r},\wh{\alpha})$ holds, $\wh{Y}_i\hookrightarrow Y_i$, $\wh{r}\in [r\vee p ,\infty)$, $
\wh{Y}_{\wh{r}}^{\Tr}\hookrightarrow \Xap $, $\wh{Y}_{1}\embed X_{1-\a/p}$
and the $L^{{r}}_{{\alpha}}$-maximal local solution $(v,\tau)$ to \eqref{eq:QSEE} on $[s+\varepsilon,T]$ in the $(Y_0,Y_1,r,\alpha)$-setting with initial value $v_{s+\varepsilon} = \one_{\sigma>\varepsilon}u(s+\varepsilon)$ satisfies
$$v \in \bigcap_{\theta\in [0,1/2)} H^{\theta,\wh{r}}_{\loc}(s+\varepsilon,\tau;\wh{Y}_{1-\theta}).$$
\end{enumerate}
Then $\sigma = \tau$ a.s.\ on the set $\{\sigma>\varepsilon\}$ and $u = v$ a.e.\ on $\ll s+\varepsilon,\sigma\rro$.
\end{lemma}

Conditions \eqref{it:assumption_u_regularizes_to_Xhat_setting} and \eqref{it:assumption_existence_of_global_solutions_Xhat_setting} can be checked using the results in Subsections \ref{ss:bootstrapping_using_weights}-\ref{ss:emergence_weights}. Typically the lemma can be applied for every $\varepsilon\in (0,T-s)$, and in this case we obtain that $u\in \bigcap_{\theta\in [0,1/2)} H^{\theta,\wh{r}}_{\loc}(s,\sigma;\wh{Y}_{1-\theta})$.

\begin{remark}\
\label{r:extrapolation}
\begin{enumerate}[{\rm(1)}]
\item\label{it:extrapolation_global_existence} In applications to SPDEs, Lemma \ref{lem:extrapolation} allows to extrapolate \emph{global existence} result from a given $(Y_0,Y_1,r,\alpha)$-setting where $\tau=T$. Typically, this yields an improvement in the choice of the initial data (see Theorem \ref{t:1d_rough_initial_data} and the text below it).
\item
In the case that $\Xap$ is critical, Theorems \ref{t:blow_up_criterion}\eqref{it:blow_up_non_critical_Xap} and \ref{thm:semilinear_blow_up_Serrin}\eqref{it:blow_up_semilinear noncritical_stochastic} are not applicable. Using Lemma  \ref{lem:extrapolation} we can change into a $(Y_0,Y_1,r,\alpha)$-setting, and in the case $Y_{\alpha,r}^{\Tr}$ is not critical, then one can often apply those result to find $\tau=T$, and therefore $\sigma =T$. See \cite{AV22_reaction_diffusion} for an application to reaction-diffusion equations.
\end{enumerate}
\end{remark}

\begin{proof}[Proof of Lemma \ref{lem:extrapolation}]
As usual, we set $s=0$. Here we employ the arguments used in Step 1 and 3 in the proof of Theorem \ref{t:regularization_z} with minor modifications.
Note that, due to \eqref{it:assumption_u_regularizes_to_Xhat_setting} and Proposition \ref{prop:continuousTrace},
$v_{\varepsilon}:=\one_{\sigma>\varepsilon}u(\varepsilon)\in L^0_{\F_{\varepsilon}}(\O;Y^{\Tr}_{{\alpha},{r}})$.
By \eqref{it:assumption_existence_of_global_solutions_Xhat_setting} there exists a $L^{r}_{\alpha}$-maximal local solution $(v,\tau)$ to \eqref{eq:QSEE} in the $(Y_0,Y_1,\alpha,r)$-setting.

Reasoning as in Step 1 in Theorem \ref{t:regularization_z}, by \eqref{it:assumption_u_regularizes_to_Xhat_setting} and Lemma \ref{l:F_G_bound_N_C_cn} applied in the $(Y_0,Y_1,\alpha,r)$-setting, one can check that $(u|_{\ll \varepsilon,\sigma\rro},\sigma \one_{\sigma>\varepsilon}+\varepsilon\one_{\sigma\leq \varepsilon})$ is an $L^{r}_{\alpha}$-local solution to \eqref{eq:QSEE} with initial data $v_{\varepsilon}$ in the $(Y_0,Y_1,\alpha,r)$-setting. The maximality of $(v,\tau)$ ensures that $\sigma\leq \tau$ a.s.\ on $\{\sigma>\varepsilon\}$, and $u=v$ a.e.\ on $\ll \varepsilon,\sigma \rro$.
It remains to prove $\P(\varepsilon<\sigma<\tau)=0$.
By Proposition \ref{prop:continuousTrace}, and \eqref{it:assumption_existence_of_global_solutions_Xhat_setting}, we have
\begin{align*}
u=v &\in C((\varepsilon,\sigma];\wh{Y}_{\wh{r}}^{\Tr})\cap L^{\wh{r}}_{{\rm loc}}((\varepsilon,\sigma];\wh{Y}_1)\\
&\subseteq C((\varepsilon,\sigma];\Xap)\cap L^p_{{\rm loc}}((\varepsilon,\sigma];X_{1-\frac{\a}{p}})
\quad
\text{a.s.\ on $\{\varepsilon<\sigma<\tau\}$}.
\end{align*}
where we also used that $\wh{r}\geq p$. Since $\tau\leq T$ a.s.,
\begin{align*}\P(\varepsilon<\sigma<\tau)
&=
\P\Big(\{\varepsilon<\sigma<\tau\}\cap \big\{\lim_{t\uparrow \sigma}u(t) \text{ exists in }\Xap,\,\|u\|_{L^p(0,\sigma;X_{1-\frac{\a}{p}})}<\infty \big\}\Big)\\
&\leq \P\Big(\sigma<T,\,\lim_{t\uparrow \sigma}u(t) \text{ exists in }\Xap,\,\|u\|_{L^p(0,\sigma;X_{1-\frac{\a}{p}})}<\infty \Big)=0,
\end{align*}
where we used \eqref{eq:Lpa_norm_up_to_zero} and Theorem \ref{t:blow_up_criterion}\eqref{it:blow_up_norm_general_case_quasilinear_Xap_Pruss}.
This completes the proof.
\end{proof}

\section{A 1D problem with cubic nonlinearities and colored noise}
\label{s:1D_problem}
The aim of this subsection is to demonstrate our main results in a fairly simple situation. In particular, we created this section to illustrate how Sections \ref{s:blow_up} and \ref{s:regularization} can be used to transfer results in an $L^2(L^2)$-setting to $L^p(L^q)$. The arguments used in this simple 1d case  can be extended to other situations, and this will be done in \cite{AV20_NS,AV19_QSEE_3}.

Below we study the existence and regularity of global solutions to
\begin{equation}
\label{eq:1D_problem}
\begin{cases}
du-\partial_{x}^2 u\,dt =\partial_x(f(\cdot,u))dt+ g(\cdot,u)dw^c_t,& \text{on }\Tor,\\
u(0)=u_0,& \text{on }\Tor,
\end{cases}
\end{equation}
where $u:[0,\infty)\times \O\times \Tor \to \R$ is the unknown process and $w^c_t$ is a colored noise on $\Tor$, i.e.\ an $H^{\lambda}(\Tor)$-cylindrical Brownian motion (see e.g. \cite[Definition 2.11]{AV19_QSEE_1}). Here, for the sake of simplicity we will assume $\lambda\in (\frac{1}{2},1)$. Throughout this section we write $H^{s}(\Tor):=H^{s,2}(\Tor)$ for $s\in \R$.

\subsection{Statement of the main results}
\label{ss:1D_main_result_statements}
Let us begin by listing our assumptions.

\begin{assumption} $\lambda\in (\frac{1}{2},1)$.
\label{ass:1D_stochastic}
\begin{enumerate}[{\rm(1)}]
\item $f:\R_+\times \O\times\Tor \times \R\to \R$ and $g:\R_+\times \O\times \Tor\times \R\to \R$ are $\Progress\otimes \Borel(\Tor)\otimes \Borel(\R)$-measurable;
\item $f(\cdot,0)\in L^{\infty}(\R_+\times \O\times \Tor)$ and $g(\cdot,0)\in L^{\infty}(\R_+\times \O\times \Tor)$. Moreover, there exists $\nu\in(0,2]$ such that a.s.\ for all $t\in \R_+$, $x\in \Tor$ and $y,y'\in \R$,
\begin{align*}
|f(\cdot,y)-f(\cdot,y')|&\lesssim (1+|y|^2+|y'|^2)|y-y'|,\\
|g(\cdot,y)-g(\cdot,y')|&\lesssim (1+|y|^{2-\nu}+|y'|^{2-\nu})|y-y'|.
\end{align*}
\end{enumerate}
\end{assumption}

Next, we define \emph{weak solutions to \eqref{eq:1D_problem} on }$\overline{I}_T$ where $T\in (0,\infty]$. To this end, we suitably interpret the term $g(\cdot,u)dw^c_t$ in \eqref{eq:1D_problem}. The operator $M_{g(\cdot,u)}$ denotes multiplication by $g(\cdot,u)$. Since $\lambda>\frac{1}{2}$, by Sobolev embeddings $\iota :H^{\lambda}(\Tor)\to L^{\pownoise}(\Tor)$ for all $\pownoise\in (1,\infty)$ and therefore, by H\"{o}lder's inequality, we may consider $M_{g(\cdot,u)}$ as a multiplication operator from $L^{\pownoise}(\Tor)$ into $L^{2}(\Tor)$, where $u\in H^1(\Tor)$. Typical examples of nonlinearities $f$ and $g$ which satisfy Assumption \ref{ass:1D_stochastic} are given by
\[f(y) = a y^3,  \ \ \ \text{and } \ \ \ g(y) = b y^{3-\nu}, \ \ \ a,b\in \R.\]

For $T\in (0,\infty]$. We say that $(u,\sigma)$ is a \emph{weak solution to \eqref{eq:1D_problem} on }$\overline{I}_T$ if $(u,\sigma)$ is an $L^2_0$-maximal local solution to \eqref{eq:QSEE} on $\overline{\I}_{T}$ (see Definitions \ref{def:solution1}, \ref{def:solution2}, and Subsection \ref{ss:globalgeneral} for the extension to $[0,\infty)$) with $p=2$, $\a=0$, $H=H^{\lambda}(\Tor)$, $X_0=H^{-1}(\Tor)$, $X_1=H^{1}(\Tor)$, and for $v\in X_1$,
\begin{equation}
\label{eq:colored_noise_1D_choice_ABFG}
\begin{aligned}
A(\cdot)v&=-\partial_x^2 v, &\qquad  B(\cdot)v &=0,
\\ F(\cdot,v)&=\partial_x(f(\cdot,v)), &\qquad G(\cdot,v)&=M_{g(\cdot,v)}.
\end{aligned}
\end{equation}
Weak solutions are \emph{unique} by maximality. We say that $(u,\sigma)$ (or simply $u$) is a \emph{global weak solution to \eqref{eq:1D_problem}} provided $(u,\sigma)$ is a weak solution to \eqref{eq:1D_problem} on $[0,\infty)$ with $\sigma = \infty$ a.s. Note that in the above the term {\em weak} is meant in the analytic sense and is motivated by the choice $X_0 = H^{-1}(\Tor)$.

For $s_1, s_2\in (0,1)$, $C^{s_1, s_2}([a,b]\times \Tor)$ denotes the space continuous functions on $u:[a,b]\times \Tor\to \R$ for which there exists a $C\geq 0$ such that
\[|u(t_1, x_1) - u(t_2, x_2)|\leq C(|t_1-t_2|^{s_1} + |x_1-x_2|^{s_2}), \ \ t_1, t_2\in [a,b], \ x_1, x_2\in \Tor.\]
\begin{theorem}[Local existence and regularity]
\label{t:local_1D}
Let Assumption \ref{ass:1D_stochastic} be satisfied. Then for any $u_0\in L^0_{\F_0}(\O;L^2(\Tor))$, \eqref{eq:1D_problem} has a weak solution on $[0,\infty)$ such that
\begin{equation}
\label{eq:L_2_regularity_1D}
u\in L^2_{\rm loc}([0,\sigma);H^1(\Tor))\cap C([0,\sigma);L^2(\Tor)) \ \ \text{ a.s. }
\end{equation}
Moreover, $u$ instantaneously regularizes in time and space:
\begin{equation}
\label{eq:H_regularity_1D}
u\in \bigcap_{\theta\in [0,1/2)} H^{\theta,r}_{\rm loc}(\I_{\sigma};H^{1-2\theta,\zeta}(\Tor)) \ \
\text{ a.s. \ for all }r,\zeta\in (2,\infty).
\end{equation}
In particular,
\begin{equation}
\label{eq:C_regularity_1D}
u\in \bigcap_{\theta\in (0,1/2)} C^{\theta}_{{\rm loc}}(\I_{\sigma};C^{1-2\theta}(\Tor))
\subseteq
\bigcap_{\theta_1\in (0,1/2),\,\theta_2\in (0,1)} C_{{\rm loc}}^{\theta_1,\theta_2}(\I_{\sigma}\times \Tor)\ \ \text{a.s.}
\end{equation}
\end{theorem}

Under additional assumptions on the nonlinearities $f$ and $g$ but keeping  still keeping $u_0\in L^0_{\F_0}(\O;L^2(\Tor))$, one can prove higher order regularity result by using the bootstrap argument of Section \ref{s:regularization} (cf.\ \cite[Theorem 2.7]{AV20_NS}) or by using Schauder theory. We emphasize that the main difficulty is to pass from \eqref{eq:L_2_regularity_1D} to \eqref{eq:H_regularity_1D}.
The regularization effect in \eqref{eq:H_regularity_1D}-\eqref{eq:C_regularity_1D} is also non-trivial if $g\equiv 0$, and even in that case it appears to be new (see the discussion related to \eqref{eq:critproblemex} for details).

Next we will prove a global existence result under a sublinearity assumption on $g$ (but without further growth conditions on $f$). It is possible to further weaken the growth condition on $g$ under dissipativity conditions on $f$, and we will consider this in a higher dimensional setting in \cite{AV22_reaction_diffusion}.

\begin{theorem}[Global existence and regularity]
\label{t:global_1D}
Let Assumption \ref{ass:1D_stochastic} be satisfied. Assume that
$f(t,x,y)$ does not depend on $x$,
and there exists a $C_g>0$ such that
\begin{equation}
\label{eq:1D_g_sublinearity}
|g(t,x,y)|\leq C_g(1+|y|)
\ \ \text{ a.s.\ for all }t\in \R_+,\,x\in \Tor \text{ and }y\in\R.
\end{equation}
Then for any $u_0\in L^0_{\F_0}(\O;L^2(\Tor))$, \eqref{eq:1D_problem} has a \emph{global} weak solution $u$. In particular, $u$ satisfies \eqref{eq:L_2_regularity_1D}-\eqref{eq:C_regularity_1D} with $\sigma=\infty$ a.s. Moreover, if $u_0\in L^2(\O;L^2(\Tor))$, then for each $T\in \R_+$ there exists a $C>0$ independent of $u_0$ such that
\begin{align}\label{eq:energysimple1dest}
\E\Big[\sup_{s\in \I_T} \|u(s)\|_{L^2(\Tor)}^2\Big]
+\E\| u\|_{L^2(\I_T;H^1(\Tor))}^2\leq C(1+\E\|u_0\|_{L^2(\Tor)}^2).
\end{align}
\end{theorem}

By Lemma \ref{lem:extrapolation} and Theorem \ref{t:global_1D} we can extrapolate global existence of solutions to \eqref{eq:1D_problem} with \emph{rough} initial data. To this end we need to introduce $(s,q,p,\a)$-weak solutions to \eqref{eq:1D_problem}. Let $T\in (0,\infty]$ and $s\in (0,1)$. We say that $(u,\sigma)$ is a (unique) \emph{$(s,q,p,\a)$-weak solution} to \eqref{eq:1D_problem} on $\overline{I}_T$ if $(u,\sigma)$ is an $L^p_{\a}$-maximal local solution to \eqref{eq:QSEE} with the choice \eqref{eq:colored_noise_1D_choice_ABFG}, $X_0=H^{-1-s,q}(\Tor)$, $X_0=H^{1-s,q}(\Tor)$ and $H=H^{\lambda}(\Tor)$. As above, $(u,\sigma)$ (or simply $u$) is a \emph{global $(s,q,p,\a)$-weak solution}  to \eqref{eq:1D_problem} if $(u,\sigma)$ is a $(s,q,p,\a)$-weak solution  to \eqref{eq:1D_problem} on $[0,\infty)$ with $\sigma=\infty$ a.s.

\begin{theorem}[Global existence and regularity with rough initial data]
\label{t:1d_rough_initial_data}
Suppose that Assumption \ref{ass:1D_stochastic} and \eqref{eq:1D_g_sublinearity} hold. Let $s\in (0,\frac{1}{3})$, $p,q\in (2,\infty)$ be such that
$$
q\in \Big(2,\frac{2}{1-2s}\Big) \ \ \ \text{and} \ \ \ \frac{1}{p}+\frac{1}{2q}\leq \frac{3-2s}{4}.
$$
Set $\a_{\crit}=-1+\frac{p}{2}(\frac{3}{2}-s-\frac{1}{q})$.
Then for any $u_0\in L^0_{\F_0}(\O;B^{\frac{1}{q}-\frac{1}{2}}_{q,p}(\Tor))$, \eqref{eq:1D_problem} has a global $(s,q,p,\a_{\crit})$-weak solution $u$ on $[0,\infty)$ such that
\begin{equation}
\label{eq:L_2_regularity_1Dextra}
u\in L^p_{\rm loc}([0,\infty),w_{\a_{\crit}};H^{1-s,q}(\Tor))\cap
C([0,\infty);B^{\frac{1}{q}-\frac{1}{2}}_{q,p}(\Tor)) \ \ \text{ a.s. }
\end{equation}
and $u$ satisfies \eqref{eq:H_regularity_1D}-\eqref{eq:C_regularity_1D} with $\sigma=\infty$.
\end{theorem}

Letting $s\in (0,\frac{1}{3})$ and $q\in (2,6)$ be large, Theorem \ref{t:1d_rough_initial_data} ensures global existence for initial data in critical spaces with negative smoothness up to $-\frac{1}{3}$. Since $L^2(\Tor)\hookrightarrow B^{\frac{1}{q}-\frac{1}{2}}_{q,p}(\Tor)$ this improves Theorem \ref{t:global_1D}.
The proof of Theorem \ref{t:1d_rough_initial_data} also yields instantaneous regularization results for $(s,q,p,\a_{\crit})$-weak solutions to \eqref{eq:1D_problem} without condition \eqref{eq:1D_g_sublinearity}.

\subsection{Proofs of Theorems \ref{t:local_1D}-\ref{t:1d_rough_initial_data}}
Throughout this subsection, to
abbreviate the notation, we often write $L^q$, $H^{s,q}$, $B^s_{q,p}$ etc. instead of $L^q(\Tor)$,
$H^{s,q}(\Tor)$, $B^s_{q,p}(\Tor)$.
We begin by proving Theorem \ref{t:local_1D}. In Roadmap \ref{roadcriticalreg} we summarized the strategy to obtain \eqref{eq:H_regularity_1D}-\eqref{eq:C_regularity_1D} using the results in Section \ref{s:regularization}.

\begin{roadmap}[Instantaneous regularity for critical problems with $p=q=2$]\label{roadcriticalreg} \
\begin{enumerate}[{\rm(a)}]
\item Consider \eqref{eq:1D_problem} in the $(H^{-1-\varepsilon},H^{1-\varepsilon},r,\alpha)$-setting with $\varepsilon\geq 0$ small, and for $\varepsilon=0$ obtain a local solution from Theorem \ref{t:local_s} (see Step 1 below);
\item\label{it:exploit_epsilon_negative} exploit the case $\varepsilon>0$ to bootstrap time regularity via Proposition \ref{prop:adding_weights} (Step 2a below) and Corollary \ref{cor:regularization_X_0_X_1} (Step 2b below). Here we create a weighted setting in time, but we lose some regularity in space. We recover space regularity via Theorem \ref{t:regularization_z} still in case of an $L^2$-setting in space (Step 2c below);
\item apply Theorem \ref{t:regularization_z} once more to bootstrap regularity in space by considering \eqref{eq:1D_problem} in the $(H^{-1,\zeta},H^{1,\zeta},r,\alpha)$-setting where $\zeta>2$ (Steps 3 and 4 below).
\end{enumerate}
\end{roadmap}

After this brief overview we will now actually start the proof.
\begin{proof}[Proof of Theorem \ref{t:local_1D}]
The proof will be divided into several steps. Recall that the operator $-\Delta_{s,q}: H^{2+s,q}(\Tor)\subseteq H^{s,q}(\Tor)\to H^{s,q}(\Tor)$ has a bounded $H^{\infty}$-calculus of angle $0$ for all $s \in \R$ and $q\in (1, \infty)$ (see \cite[Theorem 10.2.25]{Analysis2} for the case of $\R^d$). Thus, by Theorem \ref{t:SMR_H_infinite}, for all $r\in (2,\infty)$, $q\in [2, \infty)$ and $\alpha\in [0,\frac{r}{2}-1)$ (allowing $r=q=2$ and $\alpha=0$ as well) we have
\begin{equation}
\label{eq:SMR_tor_Delta}
-\Delta_{s,q} \in \mathcal{SMR}^{\bullet}_{r,\alpha}(s,T) \ \ \ \ \text{ for all }0\leq s<T<\infty
\end{equation}
and that Assumption \ref{H_a_stochnew} holds for $\ell=\alpha$. The complex and real interpolation spaces below will be obtained via \cite[Theorem 6.4.5]{BeLo}.

\emph{Step 1: For each $u_0\in L^0_{\F_0}(\O;L^2)$ there exists a weak solution $(u,\sigma)$ to \eqref{eq:1D_problem} on $[0,\infty)$. Moreover,
for all $r\in (2,\infty)$, $\alpha \in [0,\frac{r}{2}-1)$ and $\varepsilon\in [0,\frac{1}{2})$ Hypothesis \hyperref[assum:HY]{\Hiep}$(H^{-1-\varepsilon},H^{1-\varepsilon},\alpha,r)$ holds, and Assumption \ref{ass:FG_a_zero} holds in the $(H^{-1-\varepsilon},H^{1-\varepsilon},r,\alpha)$-setting provided}
\begin{equation}
\label{eq:critical_condition_L2setting}
\frac{1+\alpha}{r} \leq \frac{1}{2}-\frac{\varepsilon}{2},
\end{equation}
{\em where the corresponding trace space $B^{1-\varepsilon-2\frac{1+\alpha}{r}}_{2,r}$ is critical for \eqref{eq:1D_problem} if and only if \eqref{eq:critical_condition_L2setting} holds with equality.}
In this step we set $X_0 = H^{-1-\varepsilon}$ and $X_1 = H^{1-\varepsilon}$. Thus, $X_{1/2} = H^{-\varepsilon}$ and $X^{\mathsf{Tr}}_{\alpha,r} = B^{1-\varepsilon-2\frac{1+\alpha}{r}}_{2,r}$.
To estimate $F$, note that by Assumption \ref{ass:1D_stochastic}, for all $\vone,\vtwo\in H^{1}$,
\begin{equation}
\begin{aligned}
\label{eq:estimates_f_1D}
\|\partial_x(f(\cdot,\vone))-\partial_x(f(\cdot,\vtwo))\|_{H^{-1-\varepsilon}}
&\lesssim \|f(\cdot,\vone)-f(\cdot,\vtwo)\|_{H^{-\varepsilon}}\\
&\stackrel{(i)}{\lesssim}  \|f(\cdot,\vone)-f(\cdot,\vtwo)\|_{L^{\xi}}\\
&\stackrel{(ii)}{\lesssim}
(1+\|\vone\|_{L^{3\xi}}^2+\|\vtwo\|_{L^{3\xi}}^2)\|\vone-\vtwo\|_{L^{3\xi}}\\
&\stackrel{(iii)}{\lesssim}
(1+\|\vone\|_{H^{\theta}}^2+\|\vtwo\|_{H^{\theta}}^2)
\|\vone-\vtwo\|_{H^{\theta}},
\end{aligned}
\end{equation}
where $\xi=\frac{2}{1+2\varepsilon}\in (1,2)$ (here we used $\varepsilon<\frac{1}{2}$), $\theta=\frac{1}{3}-\frac{\varepsilon}{3}$ and in $(i)$, $(iii)$ we used the Sobolev embeddings and in $(ii)$ H\"{o}lder's inequality with exponent $(3,\frac{3}{2})$.  Since $[H^{-1-\varepsilon},H^{1-\varepsilon}]_{\phi}=H^{-1-\varepsilon+2\phi}$ for all $\phi\in (0,1)$, setting $p=r$, $\a=\alpha$, $m_F=1$, $\rho_1=2$, and $\beta_1=\varphi_1=\frac{1+\varepsilon+\theta}{2}=\frac{2}{3}+\frac{\varepsilon}{3}$, the condition \eqref{eq:HypCritical} becomes
$$
\frac{1+\alpha}{r}\leq \frac{3}{2}(1-\varphi_1)=\frac{1}{2}-\frac{\varepsilon}{2},
$$
which coincides with \eqref{eq:critical_condition_L2setting}.

Next we estimate $G$. Since $\lambda>\frac{1}{2}$ by Assumption \ref{ass:1D_stochastic}, it follows from \cite[Example 9.3.4]{Analysis2} that $\iota: H^{\lambda} \to L^{\pownoise}$ belongs to $\g(H^{\lambda},L^{\pownoise})$ for all $\pownoise\in [1,\infty)$. By the ideal-property of $\g$-radonifying operators (see e.g.\ \cite[Theorem 9.1.10]{Analysis2}), for all $\vone,\vtwo\in H^{1-\varepsilon}$,
\begin{equation}
\begin{aligned}
\label{eq:boundedness_G_1D}
\|G(\cdot,\vone)-G(\cdot,\vtwo)\|_{\g(H^{\lambda},H^{-\varepsilon})}
& \lesssim \|G(\cdot,\vone)-G(\cdot,\vtwo)\|_{\g(H^{\lambda},L^{\xi})}
\\ &\leq \|\iota\|_{\gamma(H^{\lambda},L^{\pownoise})} \|M_{g(\cdot,\vone)}-M_{g(\cdot,\vtwo)}\|_{\calL(L^{\pownoise},L^{\xi})}
\\ & \leq C_{\lambda,\pownoise} \|g(\cdot,\vone)-g(\cdot,\vtwo)\|_{L^{\varrho}},
\end{aligned}
\end{equation}
where $\xi=\frac{2}{1+2\varepsilon}$ is as above, and we applied H\"{o}lder's inequality with $\frac{1}{\varrho}+\frac{1}{\pownoise}=\frac{1}{\xi}$. Therefore, by Assumption \ref{ass:1D_stochastic} and H\"{o}lder's inequality with exponents $(3-\nu, \frac{3-\nu}{2-\nu})$,
\begin{align*}
\|G(\cdot,\vone)-G(\cdot,\vtwo)\|_{\g(H^{\lambda},H^{-\varepsilon})}
&\lesssim \big\|(1+|\vone|^{2-\nu}+|\vtwo|^{2-\nu})|\vone-\vtwo|\big\|_{L^{\varrho}} \\
&\leq  (1+\|\vone\|_{L^{(3-\nu)\varrho}}^{2-\nu}+\|\vtwo\|_{L^{(3-\nu)\varrho}}^{2-\nu})
\|\vone-\vtwo\|_{L^{(3-\nu)\varrho}}
\end{align*}
Setting $\varrho=3\xi/(3-\nu)$, by \eqref{eq:estimates_f_1D}, the latter and $X_{1/2}=H^{-\varepsilon}$, it follows that \ref{HGcritical} holds with $m_G=1$, $\rho_2=2-\nu$, $\varphi_2=\beta_2=\varphi_1$ and \eqref{eq:HypCriticalG} holds with strict inequality.

Therefore, if $\varepsilon = 0$, Theorem \ref{t:local_s} with $p=2$, $\a=0$, $F_c = F$ and $G_c = G$ implies existence and uniqueness of a weak solution to \eqref{eq:1D_problem}. The other assertions of Step 1 follow from the above considerations for general $\varepsilon\in [0,\frac12)$.

\emph{Step 2: The weak solution $(u,\sigma)$ provided by Step 1 verifies}
\begin{equation}
\label{eq:step_2_regularization_cubic_nonlinearity}
u\in \bigcap_{\theta\in [0,1/2)} H^{\theta,r}_{\rm loc}(\I_{\sigma};H^{1-2\theta}(\Tor)), \text{ a.s. for all }r\in (2,\infty).
\end{equation}
To prove this regularization effect in time, we will first use
Proposition \ref{prop:adding_weights} to create a weighted setting with a slight increase in integrability. After that we will apply Corollary \ref{cor:regularization_X_0_X_1} to extend the integrability to arbitrary order. In the above procedure we lose some space regularity, and this will be recovered by applying Theorem \ref{t:regularization_z}. Observe that it suffices to consider $r$ large. The proof is split into several sub-steps.

\textit{Step 2a: For each $\varepsilon\in (0,\frac{1}{2})$, \eqref{eq:step_2_regularization_cubic_nonlinearity} holds with $H^{\theta,r}_{\rm loc}(\I_{\sigma};H^{1-2\theta}(\Tor))$ replaced by $H^{\theta,6}_{\rm loc}(\I_{\sigma};H^{1-2\theta-\varepsilon}(\Tor))$.}
It suffices to apply Proposition \ref{prop:adding_weights},
with $Y_0=H^{-1-\varepsilon}$, $Y_1=H^{1-\varepsilon}$, $X_0=H^{-1}$, $X_1=H^1$, $\delta=\frac{\varepsilon}{2}$, $p=2$, and $\alpha>0$ such that $\frac12 = \frac{1+\alpha}{6} + \delta$. Since $\varepsilon\in (0,1/2)$, we have $\alpha\in (0,2)$. Recall from Step 1 that $\varphi_1 = \varphi_2 = \frac23$. The requirements of Proposition \ref{prop:adding_weights} are now clear from the above choices and Step 1.

\textit{Step 2b:
For each $\varepsilon\in (0,\frac{1}{2})$ and $\wh{r}\in [6, \infty)$, \eqref{eq:step_2_regularization_cubic_nonlinearity} holds with $H^{\theta,r}_{\rm loc}(\I_{\sigma};H^{1-2\theta}(\Tor))$ replaced by $H^{\theta,\wh{r}}_{\rm loc}(\I_{\sigma};H^{1-2\theta-\varepsilon}(\Tor))$.}
This is immediate from Step 1, Step 2a, and Corollary \ref{cor:regularization_X_0_X_1}  applied with $X_i = H^{-1+i}$, $p=2$, $\a=0$, $Y_i = H^{-1+2i-\varepsilon}$, $r=6$, $\alpha = 2-3\varepsilon$ and $\wh{r}\in [6,\infty)$ arbitrary. Note that assumption \eqref{it:assumption_u_regular_Y2cor} is satisfied due to Step 1.
 Condition \eqref{it:assumption_u_regular_Y2_cor_alpha_grather_than_zero} is satisfied with $\delta=\frac{\varepsilon}{2}$, $r = 6$, $\varepsilon<\frac{1}{3}$ and $\a=0$. Finally, note that $ \Yr = B^{1-\frac{2}{r} - \varepsilon}_{2,r}\hookrightarrow L^2 = \Xp$ since $1-\frac{2}{r} - \varepsilon>0$.

\textit{Step 2c: Proof of \eqref{eq:step_2_regularization_cubic_nonlinearity}.}
Set $X_{i} = \wh{Y}_i = H^{-1+2i}$, $Y_i=H^{-1+2i-\varepsilon}$,
\begin{align}\label{eq:choicealphawh}
\alpha=0, \ \wh{\alpha}>0,  \ r=\wh{r}\geq 6, \ \varepsilon\in (0,\tfrac12), \ \ \text{such that} \ \   \frac{1+\wh{\alpha}}{r}=\frac{1}{r}+\frac{\varepsilon}{2}.
\end{align}
To gain space regularity, it suffices to check the conditions \eqref{it:assumption_u_regular_Y1}-\eqref{it:assumption_Yr_embeds_Yar_hat} of  Theorem \ref{t:regularization_z}.

\eqref{it:assumption_u_regular_Y1}: By Step 1 and \eqref{eq:SMR_tor_Delta}, Hypothesis \hyperref[assum:HY]{\Hiep}$(H^{-1-\varepsilon},H^{-1-\varepsilon},\alpha,r)$ and Assumption \ref{H_a_stochnew} hold provided \eqref{eq:critical_condition_L2setting} is satisfied for $(r,\alpha,\varepsilon)$. The required regularity of $u$ also follows from Step 2b. The remaining conditions follow as in Step 2b.

\eqref{it:assumption_u_regular_Y2}: This follows from Step 1 with $\varepsilon=0$ and the fact that $\frac{1+\wh{\alpha}}{r}<\frac{1}{2}$. Moreover, by \eqref{eq:critical_condition_L2setting} with $\varepsilon=0$ the space $Y^{\Tr}_{\wh{\alpha},\wh{r}} = B^{1-2\frac{1+\wh{\alpha}}{r}}_{2,r}$ is not critical for \eqref{eq:1D_problem} in the $(\wh{Y}_0,\wh{Y}_1,\wh{r},\wh{\alpha})$-setting.

\eqref{it:assumption_Yr_embeds_Yar_hat}: By \eqref{eq:choicealphawh}, one has $\Yr = B^{1-\varepsilon-\frac2r}_{2,r} =B^{1-2\frac{1+\wh{\alpha}}{r}}_{2,r}= \wh{Y}^{\Tr}_{\wh{\alpha},\wh{r}}$.
Also by \eqref{eq:choicealphawh} and Lemma \ref{l:emb_z}\eqref{it:regularization_z_sharp_condition_refined_q_p} applied with $\varepsilon$ replaced by $\varepsilon/2$, the embedding condition \eqref{eq:emb_uniqueness_Z} holds.

\textit{Step 3: For all $\zeta,r\in (2,\infty)$ and $\alpha \in [0,\frac{r}{2}-1)$, Hypothesis \hyperref[assum:HY]{\Hiep}$(H^{-1,\zeta},H^{1,\zeta},\alpha,r)$ holds, Assumption \ref{ass:FG_a_zero} holds in the $(H^{-1,q},H^{1,q},\alpha,r)$-setting and the corresponding trace space $B^{1-2\frac{1+\alpha}{r}}_{\zeta,r}$ is not critical for \eqref{eq:1D_problem}}. Let us begin by estimating $F$. Note that by Assumption \ref{ass:1D_stochastic}, for all $\vone,\vtwo\in H^{1,\zeta}$,
\begin{align*}
\|\partial_x(f(\cdot,\vone))-\partial_x(f(\cdot,\vtwo))\|_{H^{-1,\zeta}}
&\lesssim \|f(\cdot,\vone)-f(\cdot,\vtwo)\|_{L^{\zeta}}\\
&\stackrel{(i)}{\lesssim}
(1+\|\vone\|_{L^{3\zeta}}^2+\|\vtwo\|_{L^{3\zeta}}^2)\|\vone-\vtwo\|_{L^{3\zeta}}\\
&\stackrel{(ii)}{\lesssim}
(1+\|\vone\|_{H^{\frac{2}{3}\zeta,\zeta}}^2+\|\vtwo\|_{H^{\frac{2}{3}\zeta,\zeta}}^2)\|\vone-\vtwo\|_{H^{\frac{2}{3}\zeta,\zeta}},
\end{align*}
where in $(i)$ we used H\"{o}lder's inequality with exponent $(3,\frac{3}{2})$ and in $(ii)$ the Sobolev embedding $H^{\frac{2}{3}\zeta,\zeta}\hookrightarrow L^{3\zeta}$. Since $[H^{-1,\zeta},H^{1,\zeta}]_{\theta}=H^{-1+2\theta,\zeta}$ for all $\theta\in (0,1)$, setting $m_F=1$, $\rho_1=2$, $\beta_1=\varphi_1=\frac{1}{2}+
\frac{1}{3q}$, condition \eqref{eq:HypCritical} becomes
$$
\frac{1+\alpha}{r}\leq \frac{3}{2}(1-\varphi_1)=\frac{3}{4}-\frac{1}{2\zeta}.
$$
Since $\frac{3}{4}-\frac{1}{2\zeta}> \frac{1}{2}$ due to $\zeta>2$ and $\frac{1+\alpha}{r}<\frac{1}{2}$ for all $\alpha\in [0,\frac{r}{2}-1)$, the above estimate is always strict and hence noncriticallity follows. As in \eqref{eq:boundedness_G_1D} with $L^{\xi}$ and $H^{-\varepsilon}$ replaced by $L^{\zeta}$, one can estimate $G$ to show that \ref{HGcritical} holds in the $(H^{-1,\zeta},H^{1,\zeta},\alpha,r)$-setting with $m_G = 1$, $\rho_2 = 2-\nu$, $\varphi_2 = \beta_2 = \varphi_1$ . Moreover, since $\frac{1+\alpha}{r}<\frac{1}{2}$, one can check that \eqref{eq:HypCriticalG} holds with the strict inequality.

\textit{Step 4: $u$ satisfies \eqref{eq:H_regularity_1D} and \eqref{eq:C_regularity_1D}}. Note that \eqref{eq:C_regularity_1D}  follows from \eqref{eq:H_regularity_1D}, Sobolev embedding, and standard considerations. To prove \eqref{eq:H_regularity_1D}, we apply Theorem \ref{t:regularization_z} with $Y_0=H^{-1+2i}$, $\wh{Y}_i=H^{-1+2i,\zeta}$,
\[r=\wh{r}>4, \ \wh{\alpha} = \wh{r}/4, \ \ \text{and} \ \   \alpha\in (\wh{\alpha}, \frac{r}{2}-1) \ \text{arbitrary}.\]
Note that $\wh{\alpha} = \frac{r}{4}<\frac{r}{2}-1$. It remains to check Theorem \ref{t:regularization_z}\eqref{it:assumption_u_regular_Y1}-\eqref{it:assumption_Yr_embeds_Yar_hat}:

\eqref{it:assumption_u_regular_Y1}: all conditions are clear from Steps 1 and 2, and $\Yr = B^{1-\frac{2}{r}}_{2,r} \hookrightarrow L^2 = \Xp$;

\eqref{it:assumption_u_regular_Y2}: all conditions follow from Step 3;

\eqref{it:assumption_Yr_embeds_Yar_hat}: to check
$\Yr=B^{1-\frac{2}{r}}_{2,r}\hookrightarrow B^{1-2\frac{1+\wh{\alpha}}{\wh{r}}}_{\zeta,r} = \wh{Y}^{\Tr}_{\wh{r},\wh{\alpha}}$,
by Sobolev embeddings we need to show that
\begin{equation}
\label{eq:choice_r_wh_alpha_bootstrap_integrability}
1-\frac{2}{r}-\frac{1}{2}\geq 1-2\frac{1+\wh{\alpha}}{r}-\frac{1}{\zeta}
\ \  \
\Leftrightarrow
\ \ \
2\frac{\wh{\alpha}}{r}+\frac{1}{\zeta}\geq \frac{1}{2}.
\end{equation}
Since $\wh{\alpha}=\frac{r}{4}$, \eqref{eq:choice_r_wh_alpha_bootstrap_integrability} holds for all $\zeta\in (2,\infty)$. Finally, \eqref{eq:emb_uniqueness_Z} follows from $\wh{Y}_i\hookrightarrow Y_i$ for $i\in \{0,1\}$, Lemma \ref{l:emb_z}\eqref{it:regularization_z_strict_inequality} and the choice $\alpha\in (\wh{\alpha},\frac{r}{2}-1)$.
\end{proof}

To prove global existence for \eqref{eq:1D_problem} under the assumptions of Theorem \ref{t:global_1D}, we follow the roadmap provided in Subsection \ref{ss:globalgeneral}. Note that \eqref{it:roadmaploc}-\eqref{it:roadmapreg} are contained in the proof of Theorem \ref{t:local_1D}. Our next step is to provide energy estimates under integrability assumptions on $u_0$ (see  \eqref{it:roadmapred}-\eqref{it:roadmapenergy} and Proposition \ref{prop:redblowbounded}). The proof is based on an integration by parts argument. As noticed in \eqref{it:roadmapenergy}, we can take advantage of the regularization results in Theorem \ref{t:local_s} in the proof below. Indeed, due to \eqref{eq:H_regularity_1D}-\eqref{eq:C_regularity_1D}, if we stay away from $t=0$, then we have integrability in time and space of arbitrary order (see \eqref{eq:Gamma_s_n} and \eqref{eq:f_u_g_u_1D} below).

\begin{lemma}[Energy estimates]
\label{l:1D_energy_estimates}
Let Assumption \ref{ass:1D_stochastic} be satisfied and suppose that $u_0\in L^2_{\F_0}(\O;L^2)$. Let $(u,\sigma)$ be the weak solution to \eqref{eq:1D_problem} on $[0,\infty)$ provided by Theorem \ref{t:local_1D}.
If \eqref{eq:1D_g_sublinearity} holds, then for each $T>0$ there exists a $C>0$ independent of $u,u_0$ such that
$$
\E\Big[\sup_{s\in [0,\sigma\wedge T)} \|u(s)\|_{L^2}^2\Big]
+\E\|\nabla u\|_{L^2(0,\sigma\wedge T;L^2)}^2\leq C(1+\E\|u_0\|_{L^2}^2).
$$
\end{lemma}

\begin{proof}
Let $T>0$ be fixed. By replacing $\sigma$ by $\sigma\wedge T$, we may assume that $\sigma$ takes values in $[0,T]$. Recall that $(u,\sigma)$ is the unique weak solution to \eqref{eq:1D_problem}, and
\begin{equation}
\label{eq:L2_1D_regularity}
u\in L^2_{\rm loc}([0,\sigma);H^1)\cap C([0,\sigma);L^2)\ \ \text{a.s.}
\end{equation}
Let $s>0$ and $n\geq 1$ be arbitrary (later on we let $s\downarrow 0$ and $n\to \infty$). By \eqref{eq:C_regularity_1D}, the following stopping time is well-defined
\begin{equation*}
\tau_n:=\inf\big\{t\in [s,\sigma)\,:\,\|u\|_{L^2(s,t;H^1)}+\|u(t)-u(s)\|_{C(\Tor)}\geq n\big\}
\end{equation*}
if $\sigma>s$ and $\tau_n=s$ if $\sigma \leq s$. Here $\inf\emptyset:=\sigma$. Note that $\lim_{n\to\infty}\tau_n=\sigma$ a.s.\ on $\{\sigma>s\}$. Moreover, we set
\begin{equation}
\label{eq:Gamma_s_n}
\Gamma_{s,n}:=\{\tau_n>s,\, \|u(s)\|_{C(\Tor)}\leq n\}\in \F_s.
\end{equation}
Let
$$
y(t):=\sup_{r\in [s,t\wedge \tau_n) }\one_{\Gamma_{s,n}}\|u(r)\|_{L^2}^2
+\int_s^{t}\int_{\Tor} \one_{[s,t\wedge \tau_n)\times \Gamma_{s,n}} |\nabla u(r)|^2\,dxdr, \ \  t \in [s,T].
$$
It is enough to prove the existence of $C>0$ independent of $u_0,s,n$ such that
\begin{equation}
\label{eq:y_estimate_Tor_before_gronwall}
\E y(t)\leq C(1+t-s+\E y(s))+C\int_s^t \E y(r)dr, \ \ t\in [s,T].
\end{equation}
Indeed, by Grownall's inequality \eqref{eq:y_estimate_Tor_before_gronwall} implies that for all $t\in [s,T]$,
\begin{equation}
\label{eq:y_estimate_Tor}
\E y(t)\leq  C e^{C(t-s)} (1+t-s+\E[\one_{\Gamma_n}\|u(s)\|_{L^2}^2]).
\end{equation}
The required a priori estimate, follows by letting $s\downarrow 0$ and $n\to\infty$ in \eqref{eq:y_estimate_Tor}.

For the reader's convenience, we split the remaining argument into several steps and we simply write $\Gamma,\tau$ instead of $\Gamma_{s,n},\tau_n$, since $s$ and $n$ will be fixed.

\textit{Step 1: We apply \Ito's formula to obtain the identity \eqref{eq:Ito_L2_1D} below}.
To this end, we extend $u$ to a process $v$ on $[s,T]\times \O$ in the following way. Let $v\in L^2((s,T)\times \O;H^1)\cap L^2(\I_T;C([s,T];L^2))$ be the strong solution to the problem
\begin{equation}
\label{eq:1D_v_problem}
dv=\Delta v dt+ f^u dt+ g^u dW_{H^{\lambda}} \ \ \  \text{ and }\ \  \ v(s)=\one_{\Gamma} u(s)
\end{equation}
where by \eqref{eq:H_regularity_1D}-\eqref{eq:C_regularity_1D}, and the definition of $\tau$ and $\Gamma$, for all $q\in (1,\infty)$ one has
\begin{equation}
\label{eq:f_u_g_u_1D}
\begin{aligned}
f^u(t)&:=\one_{\Gamma\times [ s,\tau)} \partial_x(f(\cdot,u))\in L^2((s,T)\times \O;H^{-1,q}(\Tor)),\\
g^u(t)&:=\one_{\Gamma\times [ s,\tau)}  g(\cdot,u)\in L^2((s,T)\times \O;L^{q}(\Tor)).
\end{aligned}
\end{equation}
The existence of $v$ is ensured by \eqref{eq:SMR_tor_Delta}, \eqref{eq:f_u_g_u_1D}, and Proposition \ref{prop:start_at_s}.
Note that since $(u,\sigma)$ is a weak solution to \eqref{eq:1D_problem} and $v$ satisfies \eqref{eq:1D_v_problem}, by maximality of $(u,\sigma)$ we get
$
v=u
$
a.e.\ on $\Gamma\times [ s,\tau)$.

Applying \Ito's formula to $\|v\|_{L^2}^2$ (see \cite[Theorem 4.2.5]{LR15}), we obtain, a.s.\ for all $n\geq 1$ and $t\in [s,T]$,
\begin{equation}
\begin{aligned}
\label{eq:Ito_L2_1D}
\|v(t)&\|_{L^{2}}^2 - \one_{\Gamma}\|u(s)\|_{L^{2}}^2
+2\int_s^t \|\nabla v(r)\|_{L^2}^2 dr
\\ & =
-2\int_s^t \int_{\Tor}\one_{\Gamma\times [ s,\tau)} f(r,u(r))\partial_x u(r)\,dxdr\\
& \ \ +\int_s^t \one_{\Gamma\times [ s,\tau)}  \|M_{g(r,u(r))}\|_{\g(H^{\lambda},L^2)}^2 \,dr\\
&\ \ + 2\int_s^t \one_{\Gamma\times [ s,\tau)} (u(r),M_{g(r,u(r))}(\cdot))_{L^2}\,dW_{H^{\lambda}}(r)
=:I_t+II_t+III_t.
\end{aligned}
\end{equation}

\textit{Step 2: There exists $C$ independent of $u,u_0,s,n$ such that}
$$
\E\int_s^t \|\nabla v(r)\|^2_{L^2} \,dr
\leq \one_{\Gamma}\|u(s)\|_{L^2}^2+
C\Big(1+t-s+ \E\int_s^t \one_{\Gamma\times [ s,\tau)} \|u(r)\|_{L^2}^2 \,dr\Big).
$$
The idea is to take expectations in \eqref{eq:Ito_L2_1D}. Clearly, $\E[III_t]=0$ for all $t\in [s,T]$. We claim that $I_t=0$. To see this, it is enough to show that $\int_{\Tor} f(t,\phi)\partial_x \phi\,dx\equiv 0$ for any $\phi\in H^{1,q}$ with $q\geq 6$ suitably large. Here we used that $u$ is smooth (see \eqref{eq:H_regularity_1D}).

Let $(\phi_k)_{k\geq 1}$ in $C^{\infty}(\Tor)$ be such that $\phi_k\to \phi$ in $H^{1,q}$.  By Assumption \ref{ass:1D_stochastic} and
$f(\cdot,\phi_k)\to f(\cdot,\phi)$ in $L^{q/3}$. Therefore, $f(\cdot,\phi_k)\partial_x \phi_k\to f(\cdot,\phi)\partial_x \phi$ in $L^1$, and hence
$$
\int_{\Tor} f(\cdot,\phi)\partial_x \phi dx=
\lim_{k\to \infty}
\int_{\Tor} f(\cdot,\phi_k)\partial_x \phi_k dx=
\lim_{k\to \infty}
\int_{\Tor} \partial_{x}[F(\cdot,\phi_k)] dx\equiv 0,
$$
where $F$ is such $\partial_z F(\cdot, z) = f(\cdot,z)$. Hence $I_t\equiv 0$.

It remains to estimate $II$. Here we argue as in \eqref{eq:boundedness_G_1D}. Fix $\xi>2$ and let $\pownoise<\infty$ be such that $\frac{1}{\pownoise}+\frac{1}{\xi}=\frac{1}{2}$. Using \eqref{eq:1D_g_sublinearity}, a.s.\ for all $r\in [s,\tau)$,
\begin{equation}
\begin{aligned}
\label{eq:M_g_1D_estimate_energy}
\|M_{g(r, u(r))}\|_{\g(H^{\lambda},L^2)}
&\leq
C \|M_{g(r, u(r))}\|_{\calL(L^{\pownoise},L^{2})}\\
&\leq  C\|g(r,u(r))\|_{L^{\xi}}\\
&\leq C C_g(1+\|u(r)\|_{L^{\xi}})\\
&\leq C C_g(1+\|u(r)\|_{H^1}^{1-\frac2{\xi}}\|u(r)\|_{L^2}^{\frac2{\xi}})\\
&\leq \|u(r)\|_{H^1} + C'C_g (1+\|u(r)\|_{L^2})
\\ &\leq  \|\nabla u(r)\|_{L^2} + (C'C_g+1) (1+\|u(r)\|_{L^2})
\end{aligned}
\end{equation}
where $C,C'$ only depend on $z,\xi$ and we used $H^1\hookrightarrow L^{\infty}$. Thus, for all $t\in [s,T]$,
\begin{equation}
\begin{aligned}
\label{eq:expectation_II_1D}
|II_t|
&\leq  \int_s^t \one_{\Gamma\times [s,\tau)}\|\nabla u\|_{L^2}^2 \,dr+c \Big(t-s+\int_s^t  \one_{\Gamma\times [s,\tau)}  \|u\|_{L^2}^2 \, dr\Big)
\end{aligned}
\end{equation}
where $c$ depends only on $C,C_g$, and where we used the definition of $y$. Therefore, taking expectations in \eqref{eq:Ito_L2_1D} and \eqref{eq:expectation_II_1D}, and using $u = v$ on $\Gamma\times [s,\tau)$, we obtain the required estimate by comparison with LHS\eqref{eq:Ito_L2_1D}.

\textit{Step 3: Proof of \eqref{eq:y_estimate_Tor_before_gronwall}}. We take absolute values and the supremum over time in \eqref{eq:Ito_L2_1D}, and then expectations. We already saw that $I\equiv 0$ on $[s,T]$. Moreover, $\E\big[\sup_{r\in [s,t]} |II_r|\big]\leq  \E\big[|II_t|\big]$ which can be estimate by the expectation of RHS\eqref{eq:expectation_II_1D}.
To conclude, it remains to estimate $III$. By the scalar Burkholder-Davis-Gundy inequality, we get
\begin{align*}
\E\Big[&\sup_{r\in [s,t]} |III_r|\Big]
\leq C \E\Big[ \int_s^t \one_{\Gamma\times [s,\tau)}(r)
\Big\|(u(r),M_{g(r,u(r))}(\cdot))_{L^2}\Big\|_{\g(H^{\lambda},\R)}^2 \, dr\Big]^{1/2}\\
 &\leq C  \E\Big[  \int_s^t \one_{\Gamma \times [s,\tau)}(r)
 \|u(r)\|_{L^2}^2\|M_{g(r,u(r))}\|_{\g(H^{\lambda},L^2)}^2 \, dr\Big]^{1/2}\\
&\leq  C
\E\Big[ \Big(\sup_{r\in [s,t\wedge \tau)} \one_{\Gamma}\|u(r)\|_{L^2}^2 \Big)^{1/2}
\Big( \int_s^t \one_{\Gamma\times [ s,\tau)}(r) \|M_{g(r,u(r))}\|_{\g (H^{\lambda},L^2)}^2\,dr \Big)^{1/2}\Big]\\
&\leq \frac{1}{2}\E \Big(\sup_{r\in [s,t\wedge \tau) }\one_{\Gamma}\|u(r)\|_{L^2}^2\Big) + C'
\E\Big[ \int_s^t \one_{[s,\tau)}(r) \|M_{g(r,u(r))}\|_{\g (H^{\lambda},L^2)}^2\,dr\Big].
\end{align*}
where the last term coincides with $\E |II_t|$. Thus, by \eqref{eq:expectation_II_1D} and Step 2,
\begin{equation*}
\E\Big[\sup_{r\in [s,t]} |III_r|\Big]\leq
\frac{1}{2}\E \Big(\sup_{r\in [s,t\wedge \tau) }\one_{\Gamma}\|u(r)\|_{L^2}^2\Big) +  c''\Big(1+t-s+\E\int_s^t  \|u(r)\|_{L^2}^2\, dr\Big),
\end{equation*}
where $c''$ is independent of $u_0,s,n$. Combining the estimates with \eqref{eq:Ito_L2_1D}, using $u=v$ on $\Gamma\times [ s,\tau)$, and using the definition of $y$, we obtain \eqref{eq:y_estimate_Tor_before_gronwall}.
\end{proof}

By Lemma \ref{l:1D_energy_estimates} we can prove the global well-posedness to \eqref{eq:1D_problem} following Roadmap \ref{roadcomplete}\eqref{it:roadmapthmappl} in Section \ref{ss:globalgeneral}. Here the criticality of the $L^2$-setting (see \eqref{eq:critical_condition_L2setting} with $r=2$ and $\varepsilon=\alpha=0$), forces us to use Theorem \ref{thm:semilinear_blow_up_Serrin}\eqref{it:blow_up_semilinear_serrin_Pruss_modified} in the proof below.

\begin{proof}[Proof of Theorem \ref{t:global_1D}]
Let $(u,\sigma)$ be the weak solution to \eqref{eq:1D_problem} on $[0,\infty)$.
By Theorem \ref{t:local_1D} and Lemma \ref{l:1D_energy_estimates} it remains to prove that $\sigma=\infty$ a.s.

Let $T\in (0,\infty)$. Replacing $(u,\sigma)$ by $(u|_{\ll 0,\sigma \wedge T\rro}, \sigma\wedge T)$ it suffices to consider weak solutions to \eqref{eq:1D_problem} on $[0,T]$ and to show that $\sigma=T$ a.s.\
For this we will use Theorem \ref{thm:semilinear_blow_up_Serrin}\eqref{it:blow_up_semilinear_serrin_Pruss_modified} with $p=2,\a=0,X_0=H^{-1},X_1=H^1$, and therefore $\Xap=L^2$. Note that
\eqref{eq:SMR_tor_Delta} holds, and that Assumption \ref{ass:FG_a_zero} is satisfied by Step 1 of the proof of Theorem \ref{t:local_1D}. By Proposition \ref{prop:redblowbounded} we may assume that $u_0\in L^{2}(\O;L^2)$. Thus, Lemma \ref{l:1D_energy_estimates} yields
\begin{equation*}
\sup_{s\in [0,\sigma)}\|u(s)\|_{L^2}^2+\int_0^{\sigma} \| u(s)\|_{H^1}^2 ds<\infty\ \ \ \ \ \text{a.s.}
\end{equation*}
Therefore, applying Theorem  \ref{thm:semilinear_blow_up_Serrin}\eqref{it:blow_up_semilinear_serrin_Pruss_modified}, we obtain
\begin{align*}
\P(\sigma<T)=
\P\Big(\sigma<T,\,\sup_{s\in [0,\sigma)}\|u(s)\|_{\Xap}+\| u(s)\|_{L^2(\I_{\sigma};X_1)}<\infty\Big)=0.
\end{align*}
Hence $\sigma=\infty$ a.s.
Finally, Lemma \ref{l:1D_energy_estimates} also implies \eqref{eq:energysimple1dest}.
\end{proof}

It remains to prove Theorem \ref{t:1d_rough_initial_data}. The idea of the proof is similar as in Roadmap \ref{roadcriticalreg}, but since $p>2$ we can use Proposition \ref{prop:adding_weights} with $\delta=0$. Moreover, we will use the extrapolation technique of Lemma \ref{lem:extrapolation}.

\begin{proof}[Proof of Theorem \ref{t:1d_rough_initial_data}]
Let $T\in (0,\infty)$.
To prove global well-posedness we apply Lemma \ref{lem:extrapolation} (see also Remark \ref{r:extrapolation}\eqref{it:extrapolation_global_existence}) with $X_i = H^{-1-s+2i,\zeta}$, $Y_i=H^{-1+2j}$, $r=2$, $\alpha=0$, $\wh{Y}_j=H^{-1+2j,\zeta}$, $\zeta,\wh{r}$ large, and $\wh{\alpha} = \frac{\wh{r}}{4}$ as in Step 4 of Theorem \ref{t:local_1D}.

First we check Lemma \ref{lem:extrapolation}\eqref{it:assumption_existence_of_global_solutions_Xhat_setting} with $\tau= T$. As in the proof of Theorem \ref{t:local_1D} one can check that \hyperref[assum:HY]{\Hiep}$(\wh{Y}_0,\wh{Y}_1,\wh{r},\wh{\alpha})$ holds. The global well-posednes in the $({Y}_0,{Y}_1,2,0)$-setting follows from Theorem \ref{t:global_1D} and a translation argument, and the remaining conditions are clear.

It remains to check the local well-posedness and regularity assertions of Lemma \ref{lem:extrapolation}\eqref{it:assumption_u_regularizes_to_Xhat_setting}, which is the
existence of a $(s,q,p,\a_{\crit})$-weak solution to \eqref{eq:1D_problem} and the instantaneous regularization requirement in \eqref{it:assumption_u_regularizes_to_Xhat_setting}. It is enough to show that for any $u_0\in L^{0}_{\F_0}(\O;B^{\frac{1}{q}-\frac{1}{2}}_{q,p}(\Tor))$ there exists a $(s,q,p,\a_{\crit})$-weak solution $(u,\sigma)$ to \eqref{eq:1D_problem} on $[0,T]$ such that
\begin{equation}
\label{eq:smoothness_u_1d_global_rough_data}
u\in C(\I_{\sigma};L^2(\Tor))\cap L^2_{\loc}(\I_{\sigma};H^{1}(\Tor)) \ \ \text{ a.s. }
\end{equation}

\textit{Step 1: For all $s\in (0,\frac13)$, $r\in [2,\infty)$, $\alpha\in [0,\frac{r}{2}-1)$ and $\zeta\in (2,\infty)$, Hypothesis \hyperref[assum:HY]{\Hiep}$(H^{-1-s,\zeta},H^{1-s,\zeta},r,\alpha)$, and Assumption \ref{ass:FG_a_zero} hold in the $(H^{-1-s,\zeta},H^{1-s,\zeta},r,\alpha)$-setting provided}
\begin{equation}
\label{eq:critical_equations_1d_rough_data}
\zeta< \frac{2}{s}, \ \ \text{ and }\ \ \frac{1+\alpha}{r}+\frac{1}{2\zeta}\leq \frac{3-2s}{4}.
\end{equation}
\emph{The corresponding trace space $B^{1-s-2\frac{1+\alpha}{r}}_{\zeta,r}$ is critical for \eqref{eq:1D_problem} if and only if \eqref{eq:critical_equations_1d_rough_data} holds with equality. In particular, for $u_0\in L^0_{\F_0}(\Omega;B^{\frac{1}{q}-\frac{1}{2}}_{q,p})$ there exists a $(s,q,p,\a_{\crit})$-weak solution $(u,\sigma)$ to \eqref{eq:1D_problem} on $\overline{I}_T$.}
To prove local well-posedness we use Theorem \ref{t:local_s} with $X_i = H^{-1-s+2i,\zeta}$. We first check \ref{HFcritical} in the $(H^{-1-s,\zeta},H^{1-s,\zeta},r,\alpha)$-setting. Fix $\vone,\vtwo\in H^{1-s,\zeta}$ and note that
\begin{align*}
\|\partial_x (f(\cdot,\vone))-\partial_x (f(\cdot,\vtwo))\|_{H^{-1-s,\zeta}}
&\lesssim \|f(\cdot,\vone))-f(\cdot,\vtwo)\|_{H^{-s,\zeta}}\\
&\stackrel{(i)}{\lesssim} \|f(\cdot,\vone))-f(\cdot,\vtwo)\|_{L^{\psi}}\\
&\lesssim (1+\|\vone\|_{L^{3\psi}}^{2}+\|\vtwo\|_{L^{3\psi}}^{2})\|\vone-\vtwo\|_{L^{3\psi}}\\
&\stackrel{(ii)}{\lesssim} (1+\|\vone\|_{H^{\theta,q}}^{2}+\|\vtwo\|_{H^{\theta,\zeta}}^{2})\|\vone-\vtwo\|_{H^{\theta,\zeta}}
\end{align*}
where in $(i)$-$(ii)$ we used the Sobolev embedding with $-s-\frac{1}{\zeta}=-\frac{1}{\psi}$ and $\theta-\frac{1}{\zeta}= -\frac{1}{3\psi}$, where $\theta>0$. To ensure that $\psi\in (1,\infty)$ one needs $\zeta>\frac{1}{1-s}$ which holds since $\zeta>2$ and $s<\frac{1}{3}$.  Combining the above identities we have $\theta=\frac{2}{3\zeta}-\frac{s}{3}$. To ensure $\theta>0$ we need $\zeta<\frac{2}{s}$. Since $H^{\theta,\zeta}=X_{\beta}$, we obtain $\beta=\frac{1}{3}(\frac{1}{\zeta}+s)+\frac{1}{2}$, and one can check that $\beta\in (1/2,1)$. Setting $p=r$, $m_F=1$, $\rho_1=2$, and $\beta_1=\varphi_1=\beta$ the condition \eqref{eq:HypCritical} becomes
$$
\frac{1+\alpha}{r}\leq \frac{3}{2}(1-\beta)=\frac{3-2s}{4}-\frac{1}{2\zeta}.
$$
which coincides with the second condition in \eqref{eq:critical_equations_1d_rough_data}. As in the proof of Theorem \ref{t:local_1D}, one can check that condition \ref{HGcritical} holds in the $(H^{-1-s,\zeta},H^{1-s,\zeta},r,\alpha)$-setting with $m_G=1$, $\rho_2= 2-\nu$ $\varphi_2=\beta_2=\varphi_1$.

By the above, we can apply Theorem \ref{t:local_s}. It only remains to investigate criticality. By \eqref{eq:critical_equations_1d_rough_data}, criticality occurs if and only if
\begin{equation}
\label{eq:critical_equation_a_1d_rough}
\frac{1+\a}{p}+\frac{1}{2 q}= \frac{3-2s}{4}.
\end{equation}
Since $\frac{1+\a}{p}\in [\frac{1}{p},\frac{1}{2})$, \eqref{eq:critical_equation_a_1d_rough} is admissible if and only if
$$
\frac{1}{p}+\frac{1}{2 q}\leq \frac{3-2s}{4}, \  \ \ \ \text{ and }\ \ \ \ \frac{3-2s}{4}-\frac{1}{2 q}<\frac{1}{2}.
$$
The second inequality in the previous displayed formula yields the limitation $q<\frac{2}{1-2s}$. Since $\frac{2}{1-2s}<\frac{2}{s}$ due to $s<\frac{1}{3}$, the first condition in \eqref{eq:critical_equations_1d_rough_data} with $q=\zeta$ holds. By \eqref{eq:critical_equation_a_1d_rough}, $\a_{\crit}=-1+\frac{p}{2}( \frac{3}{2}-s-\frac{1}{q})$ and the corresponding trace space becomes
$$
\Xapcrit=B^{1-s-2\frac{1+\a_{\crit}}{p}}_{q,p}=B^{1-s-\frac{3}{2}+s+\frac{1}{q}}_{q,p}=B^{\frac{1}{q}-\frac{1}{2}}_{q,p},
$$
which finished this step.

\textit{Step 2: The $(s,q,p,\a_{\crit})$-weak solution provided by Step 1 verifies}
\begin{equation}
\label{eq:regularization_u_Step_2_1d_rough_data}
u\in \bigcap_{\theta\in [0,1/2)} H^{\theta,r}_{\loc}(\I_{\sigma};H^{1-s-2\theta,q}(\Tor)), \ \ \text{ a.s.\ for all }r\in (2,\infty).
\end{equation}
As in Step 2 in the proof of Theorem \ref{t:local_1D} we use Proposition \ref{prop:adding_weights} and after that Corollary \ref{cor:regularization_X_0_X_1}. In case $\a_{\crit}>0$, Step 2a is not needed.

\emph{Step 2a: There exists an $r>p$ such that \eqref{eq:regularization_u_Step_2_1d_rough_data} holds}. Let $\varphi_j=\beta$ where $\beta\in (1/2, 1)$ is as in Step 1. Let $r>p$ be such that $\frac{1}{r}\geq \max_j \varphi_j-1+\frac{1}{p}$. Then the claim follows by applying Proposition \ref{prop:adding_weights} with $\delta=0$.

\emph{Step 2b: \eqref{eq:regularization_u_Step_2_1d_rough_data} holds}. If $\a_{\crit}>0$, then the claim follows from Corollary \ref{cor:regularization_X_0_X_1} applied to $Y_i=X_i=H^{-1-s+2i,q}$, $r=p$ and $\alpha=\a_{\crit}$. Next we consider the case $\a_{\crit}=0$. Let $r$ be as in Step 2a and let $\alpha\in (0,\frac{r}{2}-1)$ be such that $\frac{1}{p}=\frac{1+\alpha}{r}$. By Step 2a, the assumptions of Corollary \ref{cor:regularization_X_0_X_1} are satisfies and this concludes the required regularity.

\textit{Step 3: The $(s,q,p,\a_{\crit})$-weak solution provided by Step 1 verifies}
\begin{equation*}
u\in \bigcap_{\theta\in [0,1/2)} H^{\theta,r}_{\loc}(\I_{\sigma};H^{1-2\theta,q}(\Tor)), \ \ \text{ a.s.\ for all }r\in (2,\infty).
\end{equation*}
\emph{In particular, \eqref{eq:smoothness_u_1d_global_rough_data} holds}. To conclude, it is enough to apply Theorem \ref{t:regularization_z} to $Y_i=H^{1-s-2i,q}$, $\widehat{Y}_i=H^{-1+2i,q}$,
$$
\alpha=0, \ \wh{\alpha}>0, \  r=\wh{r}>p \text{ large enough and  }\frac{1+\wh{\alpha}}{r}=\frac{1}{r}+\frac{s}{2}.
$$
To check the assumptions of Theorem \ref{t:regularization_z} one can use Step 2 and argue in a similar way as in Theorem \ref{t:local_1D} Step 2c.
\end{proof}

\bibliographystyle{alpha-sort}
\bibliography{literature}

\end{document}